\documentclass[12pt, dvipsnames]{article}

\usepackage{ifthen}

\usepackage{amsmath,amssymb,amsfonts,amsthm}
\usepackage{amscd}
\usepackage{bbm,bm} 
\usepackage{cases}
\usepackage[mathscr]{euscript}

\usepackage{breakcites}
\usepackage{natbib}
\RequirePackage[colorlinks,citecolor={blue!60!black},urlcolor={blue!70!black},linkcolor={red!60!black},breaklinks,hypertexnames=false]{hyperref}

\usepackage[ruled,vlined]{algorithm2e}
\usepackage{array}
\usepackage[small]{caption}
\usepackage{epic,eepic, epsfig, psfrag}
\usepackage{enumerate}
\usepackage{fullpage}
\usepackage{graphicx}
\usepackage[labelformat=simple]{subcaption}

\usepackage{setspace}
\usepackage{tikz}
\usetikzlibrary{fit, shapes, backgrounds, positioning, calc}
\usetikzlibrary{positioning}
\usetikzlibrary{patterns.meta}
\usepackage{xcolor}
\usepackage{soul}
\usepackage{cancel}
\usepackage{float}
\usepackage{etoolbox}

\newcommand{\red}[1]{{\color{red}#1}}

\newcommand{\green}[1]{{\color{green}#1}}

\newcommand{\blue}[1]{{\color{blue}#1}}

\newtheorem{theorem}{Theorem}
\newtheorem{example}{Example}

\newtheorem{lemma}[theorem]{Lemma}

\newtheorem{corollary}[theorem]{Corollary}
\newtheorem{defn}[theorem]{Definition}
\newtheorem{prop}[theorem]{Proposition}

\def\mathbi#1{\textbf{\em #1}}

\newcommand{\N}{\mathbb{N}}
\newcommand{\R}{\mathbb{R}}
\newcommand{\Z}{\mathbb{Z}}

\DeclareMathOperator*{\argmin}{argmin}

\DeclareMathOperator*{\sargmin}{sargmin}

\DeclareMathOperator{\an}{\mathrm{an}}
\DeclareMathOperator{\pa}{\mathrm{pa}}
\DeclareMathOperator{\ch}{\mathrm{ch}}
\DeclareMathOperator{\de}{\mathrm{de}}

\newcommand{\E}{\mathbb{E}}
\newcommand{\one}{\mathbbm{1}}

\newcommand{\Prob}{\mathbb{P}}
\newcommand{\Lebesgue}{\mathcal{L}_d}

\newcommand{\rejectionSet}{\mathcal{R}}

\newcommand{\closedEuclideanMetricBall}[2]{{B}_2(#1,#2)}

\newcommand{\closedSupNormMetricBall}[2]{{B}_\infty(#1,#2)}

\newcommand{\holderConstant}{C_{\mathrm{S}}}
\newcommand{\holderExponent}{\beta}

\newcommand{\subGaussianFactor}{\sigma}
\newcommand{\measureClass}{\mathcal{P}}

\newcommand{\sample}{\mathcal{D}}
\newcommand{\sampleX}{\sample_X}

\newcommand{\regressionFunction}{\eta}

\newcommand{\regularityConstant}{\upsilon}

\newcommand{\powerSet}{\mathrm{Pow}}

\newcommand{\lowerDensity}{\omega}
\newcommand{\superLevelSet}[1]{\mathcal{X}_{#1}(\regressionFunction)}

\newcommand{\approximableDensityExponent}{\kappa}
\newcommand{\approximableMarginExponent}{\gamma}
\newcommand{\approximableSetsConstant}{C_{\mathrm{App}}}

\newcommand{\support}{\mathrm{supp}}

\renewcommand{\regressionFunction}{\eta}
\newcommand{\marginalDistribution}{\mu}

\newcommand{\etaIncreasingExponent}{\gamma}
\newcommand{\etaIncreasingConstant}{\lambda}
\newcommand{\densityConstant}{\theta}
\renewcommand{\superLevelSet}[2]{\mathcal{X}_{#1}(#2)}

\let\oldnu=\nu
\let\oldbeta=\beta
\let\oldxi=\xi
\newcommand{\marginConstant}{\oldnu}
\newcommand{\marginExponent}{\oldbeta}
\newcommand{\marginDummy}{\oldxi}

\newcommand\distributionClassNonDecreasingRegressionFunction[1][]{%
  \ifstrempty{#1}{%
    \mathcal{P}_{\mathrm{Mon}, d}(\sigma)
  }{%
    \mathcal{P}_{\mathrm{Mon}, #1}(\sigma)
  }%
}

\newcommand\distributionClassUnivariateMarginCondition[1][]{%
  \ifstrempty{#1}{%
    \mathcal{P}_{\mathrm{Mar}, d}(\tau,\marginExponent,\marginConstant)
  }{%
    \mathcal{P}_{\mathrm{Mar}, #1}(\tau,\marginExponent,\marginConstant)
  }%
}

\newcommand\distributionClassMultivariateCondition[1][]{%
  \ifstrempty{#1}{%
    \mathcal{P}_{\mathrm{Reg}, d}(\tau, \densityConstant, \etaIncreasingExponent,\etaIncreasingConstant)
  }{%
    \mathcal{P}_{\mathrm{Reg}, #1}(\tau, \densityConstant, \etaIncreasingExponent,\etaIncreasingConstant)
  }%
}

\newcommand{\MGselectionset}{\hat{A}^{\mathrm{ISS}, \omega, \bm{v}}_{\sigma,\tau,\alpha,m}(\sample)}

\usepackage[draft,inline,nomargin,index]{fixme}
\fxsetup{theme=color,mode=multiuser}
\FXRegisterAuthor{tc}{atc}{\color{blue} TC}
\FXRegisterAuthor{hr}{ahr}{\color{red} HR}
\FXRegisterAuthor{rs}{ars}{\color{green} RS}
\FXRegisterAuthor{mm}{amm}{\color{purple} MM}

\usepackage[normalem]{ulem}
\newcommand{\msout}[1]{\text{\sout{\ensuremath{#1}}}}

\title{Isotonic subgroup selection}
\author{Manuel M. M\"uller$^\ast$, Henry W. J. Reeve$^\dagger$, Timothy I. Cannings$^\ddagger$ \\
and Richard J. Samworth$^\ast$\\ \\
$^\ast$Statistical Laboratory, University of Cambridge\\
$^\dagger$School of Mathematics, University of Bristol\\
$^\ddagger$School of Mathematics, University of Edinburgh
}
\date{\today}

\begin{document}

\maketitle

\begin{abstract}
    Given a sample of covariate-response pairs, we consider the subgroup selection problem of identifying a subset of the covariate domain where the regression function exceeds a pre-determined threshold.  We introduce a computationally-feasible approach for subgroup selection in the context of multivariate isotonic regression based on martingale tests and multiple testing procedures for logically-structured hypotheses.  Our proposed procedure satisfies a non-asymptotic, uniform Type I error rate guarantee with power that attains the minimax optimal rate up to poly-logarithmic factors.  Extensions cover classification, isotonic quantile regression and heterogeneous treatment effect settings.  Numerical studies on both simulated and real data confirm the practical effectiveness of our proposal, which is implemented in the \texttt{R} package \texttt{ISS}.
\end{abstract}

\section{Introduction}
\label{Sec:Introduction}

In regression settings, \emph{subgroup selection} refers to the challenge of identifying a subset of the covariate domain on which the regression function satisfies a particular property of interest.  This is a post-selection inference problem, since the region is to be selected after seeing the data, and yet we still wish to claim that with high probability, the regression function satisfies this property on the selected set.  Important applications can be found in precision medicine, for instance, where the chances of a desirable health outcome may be highly heterogeneous across a population, and hence the risk for a particular individual may be masked in a study representing the entire population. 

A natural strategy for identifying such group-specific effects is to divide a study into two stages, where the first stage is used to identify a potentially interesting subset of the covariate domain, and the second attempts to verify that it does indeed have the desired property \citep{stallard2014adaptive}.  However, such a two-stage process may often be both time-consuming and potentially expensive due to the inefficient use of the data, and moreover the binary second-stage verification may fail.  In such circumstances, we are unable to identify a further subset of the original selected set on which the property does hold.

In many applications, heterogeneity across populations may be characterised by monotonicity of a regression function in individual covariates.  For instance, age, smoking, hypertension and obesity are among known risk factors for coronary heart disease \citep{torpy2009coronary}, while for individuals with hypertrophic cardiomyopathy, risk factors for sudden cardiac death (SCD) include family history of SCD, maximal heart wall thickness and left atrial diameter \citep{omahony2014novel}.  It is frequently of interest to identify a subset of the population deemed to be at low or high risk, for instance to determine an appropriate course of treatment.  This amounts to identifying an appropriate  superlevel set of the regression function.

In this paper, we introduce a framework that allows the identification of the $\tau$-superlevel set of an isotonic regression function, for some pre-determined level $\tau$.  A key component of our formulation of the problem is to recognise that often there is an asymmetry to the two errors of including points that do not belong to the superlevel set, and failing to include points that do.  For instance, in the case of hypertrophic cardiomyopathy, a false conclusion that an individual is at low risk of sudden cardiac death within five years, and hence does not require an implantable cardioverter defibrillator \citep{omahony2014novel}, is more serious than the opposite form of error, which obliges a patient to undergo surgery and deal with the inconveniences of the implanted device.  

To introduce our isotonic subgroup selection setting, suppose that we are given $n$ independent copies of a covariate-response pair $(X,Y)$ having a distribution on $\R^d \times \R$ with coordinate-wise increasing regression function~$\regressionFunction$ given by $\eta(x) := \mathbb E(Y | X=x)$ for $x\in \R^d$.  Thus $Y = \eta(X) + \varepsilon$, where we additionally assume that $\varepsilon$ is sub-Gaussian conditional on $X$.  Given a threshold $\tau \in \R$, and writing $\superLevelSet{\tau}{\eta} := \{x \in \R^d:\eta(x) \geq \tau\}$ for the $\tau$-superlevel set of $\regressionFunction$, we seek to output an estimate $\hat{A}$ of $\superLevelSet{\tau}{\eta}$ with the first priority that it guards against the more serious of the two errors mentioned above.  Without loss of generality, we take this more serious error to be including points in $\hat{A}$ that do not belong to $\superLevelSet{\tau}{\eta}$, and we therefore require Type I error control in the sense that $\hat{A} \subseteq \superLevelSet{\tau}{\eta}$ with probability at least $1-\alpha$, for some pre-specified $\alpha \in (0,1)$.  Subject to this constraint, we would like $\marginalDistribution(\hat{A})$ to be as large as possible, where $\marginalDistribution$ denotes the marginal distribution of~$X$.

\begin{figure}[htbp]
    \center
    \includegraphics[trim={2cm 2cm 2cm 5cm}, clip, width=0.85\textwidth]{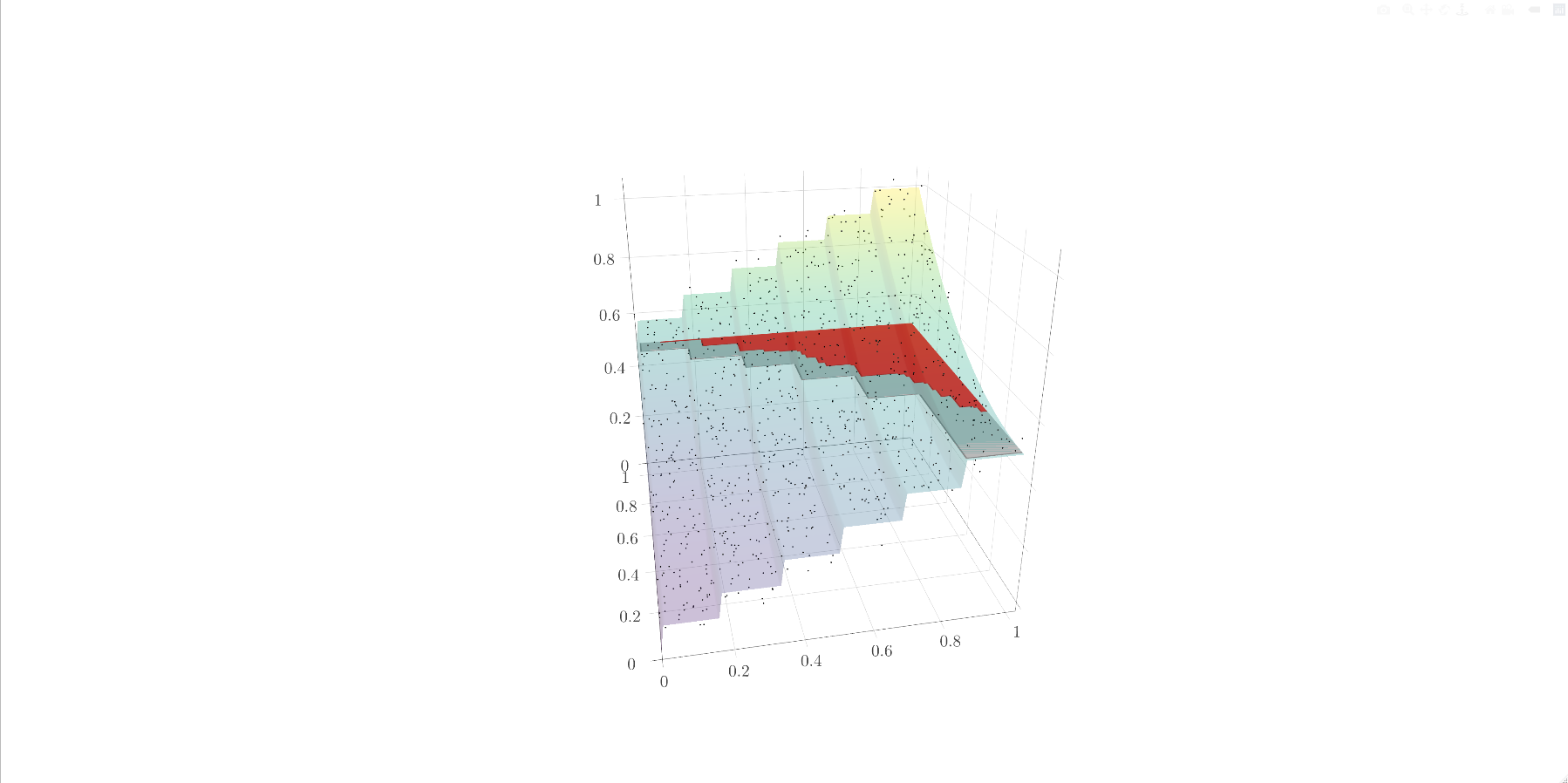}
    \caption{A visualisation with $d=2$ and $n = 1000$. The unknown regression function is $(x^{(1)}, x^{(2)})^\top \mapsto \lceil 6x^{(1)}\rceil/6 + \bigl(x^{(2)}\bigr)^2$ rescaled to the interval $[0,1]$, depicted by the multi-coloured surface. The unknown grey surface gives the $0.5$-superlevel set, of which the red area is selected by our proposed procedure $\hat{A}^\mathrm{ISS}$. }
    \label{fig:simple_illustration}
\end{figure}

One plausible strategy to achieve this goal is to construct a one-sided, uniform confidence band for $\regressionFunction$, and output the set on which the lower confidence limit is at least $\tau$.  Unfortunately, however, such an approach tends to have sub-optimal empirical performance (see Section~\ref{sec:simulations}), because the lower confidence bound is required to protect against exceeding $\regressionFunction(x)$ at all points~$x$ in the covariate domain, whereas it is only points close to the boundary of the $\tau$-superlevel set for which there is significant doubt about their inclusion.  We therefore adopt a different approach, and seek to compute at each observation a $p$-value for the null hypothesis that the regression function is below $\tau$ based on an anytime-valid martingale procedure \citep{duan2020interactive,howard2021uniform}.  The monotonicity of the regression function implies logical relationships between these hypotheses, but it is far from obvious how to combine the $p$-values effectively, particularly in the multivariate case, where we do not have a natural total ordering on $\R^d$.  Our strategy is to introduce a tailored multiple testing procedure with familywise error rate control, building on ideas of \cite{goeman2010sequential} and \cite{meijer2015multiple}.   This allows us to construct our final output set $\hat{A}^{\mathrm{ISS}}$ as the upper hull of the observations corresponding to the rejected hypotheses; see Section~\ref{Sec:Methodology} for a more formal description of our proposed procedure, which is both computationally feasible and does not require the choice of any smoothing parameters.  Our methodology is implemented in the \texttt{R} package \texttt{ISS} \citep{Mueller2023ISS}; an illustration in a bivariate example is given in Figure~\ref{fig:simple_illustration}.

Our first theoretical result, in Section~\ref{sec:theory_validity}, verifies that $\hat{A}^{\mathrm{ISS}}$ does indeed control Type~I error in the sense outlined above.  We then turn our attention to power in Section~\ref{sec:theory_power}, and provide both high-probability and expectation bounds on $\marginalDistribution\bigl(\superLevelSet{\tau}{\eta} \setminus \hat{A}^{\mathrm{ISS}}\bigr)$.   Our bound decomposes as a sum of two terms, where the first reflects the error incurred in determining whether each data point belongs to $\superLevelSet{\tau}{\eta}$, and depends on the growth rate of the regression function as we move further into the $\tau$-superlevel set from its boundary.  The second term represents the error arising from the uncertainty of whether or not regions between the data points belong to this superlevel set.  Our final theoretical contribution, in Section~\ref{sec:lowerBound}, reveals that $\hat{A}^{\mathrm{ISS}}$ attains the optimal power in the sense of minimising $\mathbb{E}\bigl\{\marginalDistribution\bigl(\superLevelSet{\tau}{\eta} \setminus \hat{A}^{\mathrm{ISS}}\bigr)\bigr\}$ up to poly-logarithmic factors, among all procedures that control the Type~I error.

In Section~\ref{Sec:Extensions}, we present various extensions that broaden the scope of our methodology.  First, in Section~\ref{SubSec:pvalues}, we describe alternative $p$-values that can be used, and that may yield more power for small and moderate sample sizes.  Section~\ref{Sec:Extension_AlternativeDistributions} introduces three variants of $\hat{A}^{\mathrm{ISS}}$ that are tailored to specific settings including Gaussian errors, classification and heavy-tailed errors. 
Finally, in Section~\ref{SubSec:HTE}, we show how our proposal can be extended to cover heterogeneous treatment effect settings.  

Section~\ref{sec:simulations} is devoted to a study of the empirical performance of $\hat{A}^{\mathrm{ISS}}$ in a wide range of settings, with 14 regression functions chosen to illustrate different characteristics of interest, as well as different sample sizes and dimensions.  
The broad conclusion across these many scenarios is that, compared with various alternative approaches,~$\hat{A}^{\mathrm{ISS}}$ has the most power for isotonic subgroup selection.  In Section~\ref{sec:application}, we illustrate the performance of $\hat{A}^{\mathrm{ISS}}$ on two real datasets, the first of which is taken from the AIDS Clinical Trials Group Study 175 (ACTG 175) \citep{speff2trial}.  Here, we consider two problems: first, we seek to identify a low-risk subgroup and, second, in the context of heterogeneous treatment effects,  we aim to identify a subgroup of patients for whom a new therapy is at least as effective as the baseline medication.  The second dataset concerns fuel consumption \citep{quinlan1993combining}, where we seek to identify fuel-efficient cars based on their weight and engine displacement.  The appendix consists of proofs of all of our main results, as well as statements and proofs of auxiliary results, further simulations and a discussion of an alternative and general approach to combining the $p$-values due to \citet{meijer2015multiple}.  Although this strategy is often highly effective in multiple testing problems, we show that, surprisingly, it has sub-optimal worst-case performance in our isotonic subgroup selection setting.

Isotonic regression has a long history dating back to  \citet{ayer1955empirical}, \citet{brunk1955maximum} and \citet{vaneeden1956maximum}.  Much recent interest has focused on risk bounds and oracle inequalities, which have been derived by \citet{meyer2000degrees}, \citet{zhang2002risk}, \citet{chatterjee2014new}, \citet{chatterjee2015risk}, \citet{bellec2018sharp}, \citet{han2019isotonic}, \citet{deng2020isotonic}, \citet{fokianos2020integrated} and \citet{pananjady2022isotonic}.  Pointwise asymptotic confidence intervals in multivariate isotonic regression have been proposed by \citet{deng2021confidence}, while confidence bands in the univariate case have been studied by \citet{yang2019contraction}. 

In the clinical trials community, the dangers of the naive approach to subgroup selection that ignores the key post-selection inference issue have been well understood for many years \citep{senn1997wisdom,feinstein1998problem,rothwell2005subgroup,wang2007Statistics,kaufman2013these,altman2015subgroup,zhang2015subgroup,gabler2016no,lipkovich2017tutorial,watson2020machine}.  Valid approaches that control Type I error in the sense above have been proposed by \citet{ballarini2018subgroup} and \citet{wan2022confidence} in the context of linear regression, and \citet{reeve2021optimal} for a smoothly-varying regression function.

The asymmetry of the two losses in our framework has some similarities with that of Neyman--Pearson classification \citep{cannon2002learning,scott2005neyman,tong2016survey,xia2021intentional}.  There, covariate-response pairs $(X,Y)$ take values in $\mathbb{R}^d \times \{0,1\}$, and we seek a classifier $C:\mathbb{R}^d \rightarrow \{0,1\}$ that minimises $\mathbb{P}\bigl(C(X) = 0|Y=1\bigr)$ subject to an upper bound on $\mathbb{P}\bigl(C(X) = 1|Y=0\bigr)$.   In addition to allowing continuous responses, another key difference of our paradigm is that we incur a Type I error whenever our selected set $\hat{A}$ contains a single point that does not belong to the $\tau$-superlevel set of the regression function.  In other words, instead of controlling averages over sub-populations, our framework provides guarantees at an individual level, which is ethically advantageous, e.g.~in medical contexts.





To conclude the introduction, we collect some notation used throughout the paper.

\paragraph{Notation.} 
For $n\in\N$, let $[n] := \{1,\ldots, n\}$ and let $[0] := \emptyset$.  Write $x \wedge y := \min(x, y)$ and $x \vee y := \max(x, y)$ for $x, y \in \R$. Further, let $x_+ := x \vee 0$ and $\log_+x := \log(x \vee e)$ for $x\in\R$.  Denote by  $\|\cdot\|_\infty$ the supremum norm on $\R^d$, and given $x \in \R^d$ and $r >0$, define the closed supremum norm ball by $\closedSupNormMetricBall{x}{r}:=\{z \in \R^d : \|z-x\|_\infty\leq r\}$. 
For $x_1 = (x_1^{(1)},\ldots, x_1^{(d)})^\top, x_2 = (x_2^{(1)},\ldots, x_2^{(d)})^\top \in \R^d$, we write $x_1 \preccurlyeq x_2$ (or, equivalently, $x_2 \succcurlyeq x_1$) if $x_1^{(j)} \leq x_2^{(j)}$ for all $j\in[d]$. 
A function $f:\mathbb{R}^d \rightarrow \mathbb{R}$ is said to be (coordinate-wise) \emph{increasing} if $f(x_0)\leq f(x_1)$ whenever $x_0 \preccurlyeq x_1$.  A set $U \subseteq \mathbb{R}^d$ is called an \emph{upper set} if, whenever $x \in U$ and $x \preccurlyeq x'$, we have $x' \in U$.  Given $A \subseteq \R^d$, the \emph{upper hull} of $A$ is the intersection of all upper sets that contain~$A$.  
For a Borel probability measure $\mu$ on~$\R^d$, we let $\support(\mu)$ denote the \emph{support} of $\mu$, i.e., the intersection of all closed sets $C \subseteq \R^d$ with $\mu(C) = 1$.  For $\tau \in \mathbb{R}$ and $f:\R^d \rightarrow \R$, we let $\superLevelSet{\tau}{f}:= \{ x \in \R^d :  f(x) \geq\tau\}$.  

A \emph{graph} $G = (I, E)$ consists of a non-empty, finite set $I$ of \emph{vertices} and a set $E \subseteq I\times I$ of \emph{edges}.  We say $G$ is \emph{directed} if $(i,j)\in E$ does not imply $(j, i)\in E$. A \emph{directed path} from $i \in I$ to $j\in I$ is a collection of distinct vertices $i_0,i_1, \ldots, i_m \in I$ for some $m\in \N$ with $i_0 = i$ and $i_m = j$ such that $(i_{k-1}, i_k) \in E$ for all $k \in [m]$. A \emph{cycle} is a directed path from $i \in I$ to itself. A \emph{directed acyclic graph (DAG)} is a directed graph that does not contain any cycles. Given a DAG $G = (I, E)$, we write $L(G) := \{i\in I: (i, i') \notin E\text{ for all }i' \in I\}$ for the set of its \emph{leaf} nodes and $\{i\in I: (i', i) \notin E\text{ for all }i' \in I\}$ for its \emph{root} nodes.  For $i\in I$, let $\pa_G(i) := \{i' \in I: (i', i)\in E\}$ denote the set of \emph{parents} of node $i$ and, similarly, write $\ch_G(i) := \{i' \in I: i \in \pa_G(i')\}$ for the set of \emph{children} of node $i$.  Further, defining $\an^1_G(i) := \pa_G(i)$ and $\an^{k+1}_G(i) := \bigcup_{j \in \an^{k}_G(i)} \pa_G(j)$ for $k\in\N$, we can define $\an_G(i) := \bigcup_{k\in\N} \an^k_G(i)$ to be the set of \emph{ancestors} of node $i$. Similarly, let $\de_G(i) := \{i' \in I: i \in \an_G(i')\}$ denote the set of \emph{descendants} of node $i$. A \emph{reverse topological ordering} of a DAG $G = (I, E)$ with $I = [m]$ is a permutation $\pi_G: I \rightarrow I$ such that if $i \in I$ and $i' \in \an_G(i)$, then $\pi_G(i) < \pi_G(i')$.  Any directed graph is acyclic if and only if it has a reverse topological ordering. We remark that all of these definitions remain unchanged when applied to a \emph{weighted DAG} $G = (I, E, \bm{w})$, i.e.~a DAG $(I, E)$ equipped with edge weights $\bm{w} = (w_e \geq 0: e \in E)$, and that any unweighted DAG may implicitly be assumed to be a weighted DAG with unit weights.  A DAG $G = (I,E)$ is a \emph{polyforest} if $|\pa_G(i)| \leq 1$ for all $i \in I$, and a weighted DAG $G = (I, E, \bm{w})$ is a \emph{polyforest-weighted DAG} if $F := (I, \{e \in E: w_e > 0\})$ is a polyforest. 

\section{Methodology}
\label{Sec:Methodology}

Let $P$ denote a distribution on $\R^d \times \R$, and let $(X,Y) \sim P$.  Suppose that the regression function $\eta:\mathbb{R}^d \rightarrow \mathbb{R}$, defined by $\eta(x) := \mathbb{E}(Y|X=x)$, is increasing, and that the conditional distribution of $Y - \eta(X)$ given $X$ is sub-Gaussian\footnote{Recall that a random variable $Z$ is \emph{sub-Gaussian with variance parameter $\sigma^2$} if $\mathbb{E}(e^{tZ}) \leq e^{\sigma^2t^2/2}$ for every $t \in \mathbb{R}$.} with variance parameter~$\sigma^2$.  Given an independent and identically distributed sample $(X_1, Y_1),\ldots, (X_n,Y_n) \sim P$, a threshold $\tau \in \R$ and a nominal Type I error rate $\alpha \in (0, 1)$, we would like to identify a Borel measurable set $\hat{A} \subseteq \R^d$ such that with probability at least $1-\alpha$, we have $\eta(x) \geq \tau$ for all $x \in \hat{A}$.  Subject to this constraint, we would like $\mu(\hat{A})$ to be as large as possible, where $\mu$ denotes the marginal distribution of~$X$.  An important observation is that, since $\superLevelSet{\tau}{\eta}$ is an upper set, replacing~$\hat{A}$ with its upper hull does not increase the Type I error probability, and may increase its $\mu$-measure.

Our general strategy is initially to focus on a subset of observations, and seek to compute a $p$-value at each of these observations for the null hypothesis that the regression function is below $\tau$.  We can then carefully combine these $p$-values using a multiple testing procedure having familywise error rate control over our structured hypotheses, and finally output the upper hull of the covariate observations corresponding to the rejected hypotheses.  More precisely, for $m \in [n]$, and writing $\mathcal{D}_{X,m} := (X_1,\ldots,X_m)$ for our reduced sample with the shorthand $\mathcal{D}_X := \mathcal{D}_{X,n}$, we will construct $p$-values $\bm{p} := (p_i)_{i \in [m]}$ having the property that $\mathbb{P}(p_i \leq t|\mathcal{D}_X) \leq t$ for every $t \in [0,1]$ whenever $\eta(X_i) < \tau$.  Next, we show how to exploit these $p$-values to obtain a set $\mathcal{R}_\alpha(\mathcal{D}_{X,m},\bm{p}) \subseteq [m]$ of rejected hypotheses having the familywise error control property that with probability at least $1-\alpha$, we have $\eta(X_i) \geq \tau$ for every $i \in \mathcal{R}_\alpha(\mathcal{D}_{X,m},\bm{p})$.  Our final output set, $\hat{A}$, is the upper hull of $\{X_i:i \in \mathcal{R}_\alpha(\mathcal{D}_{X,m},\bm{p})\}$.

It remains to describe how we propose to construct the $p$-values, and to control the familywise error, and to this end, we first focus on the case $d=1$ for simplicity of exposition.  One difference between the univariate and multivariate cases is that we take $m=n$ in the former.  Consider the null hypothesis that $\regressionFunction(x) < \tau$ for some $x\in\R$, so that $\regressionFunction(X_i) < \tau$ whenever $X_i \leq x$.  Write $\mathcal{I}(x) := \{i\in[n]: X_i \leq x\}$ and $n(x) := |\mathcal{I}(x)|$, and, for $j \in [n(x)]$, let $X_{(j)}(x)$ denote the $j$th nearest neighbour of $x$ among $\{X_i:i \in \mathcal{I}(x)\}$, where for definiteness ties are broken by retaining the original ordering.  Writing $Y_{(1)}(x),\ldots, Y_{(n)}(x)$ for the concomitant responses, under the null and conditional on~$\sample_X$, the process 
\begin{equation}
    \label{Eq:Sk}
S_k \equiv S_k(x,\sigma, \tau, \sample) := \sum_{j=1}^k \frac{Y_{(j)}(x) - \tau}{\sigma}
\end{equation}
for $k\in[n(x)]$ is a supermartingale with negative mean. 
Thus, large values of $S_k$ for some $k \in [n(x)]$ provide evidence against the null.  One could also consider the alternative supermartingale $S_k' := \sum_{j=1}^k \bigl(Y_{(n(x) + 1 - j)}(x) - \tau\bigr)/\sigma$ for $k\in[n(x)]$, but our approach has the advantage that $S_k$ stochastically dominates $S_k'$, so will have at least the same power. 
Tests based on supermartingales such as $(S_k)$ are known as \emph{martingale tests} \citep{duan2020interactive} and time-uniform upper boundaries $(v_k)$ are known for a variety of families of increment distributions \citep{howard2021uniform}; thus, $v_k \equiv v_k(\alpha)$ has the property that $\mathbb{P}\bigl(\max_{k \in [n(x)]} (S_k - v_k) \geq 0\bigr) \leq \alpha$ under the null hypothesis $\regressionFunction(x) < \tau$.  These inequalities can be inverted to yield a $p$-value; see Figure~\ref{fig:martingale_test}.  Definition~\ref{defn:pValue} below extends these ideas to the general multivariate case.

\begin{figure}
    \centering

\tikzstyle{black dot}=[fill=black, draw=black, shape=circle, minimum size=1pt,inner sep=1pt]
\tikzstyle{red dot}=[fill=red, draw=red, shape=circle, minimum size=1pt,inner sep=1pt]
\tikzstyle{cyan dot}=[fill=Melon, draw=Melon, shape=circle, minimum size=1pt,inner sep=1pt]

\tikzstyle{solid black line}=[-, draw=black, very thick]
\tikzstyle{solid red line}=[-, draw=red, very thick]
\tikzstyle{solid blue line}=[-, draw=blue, very thick]

\resizebox{\textwidth}{!}{\begin{tikzpicture}
    \def\yoffset{2.25}
    \def\xoffset{11}
    \def\tauValue{\yoffset}
    \def\xNull{8}
    \def\rightXRescale{10/6.75}

    
        \draw[->,ultra thick] (-0.5,0)--(10,0) {};
        \draw[->,ultra thick] (0, -0.5)--(0,5.5) {};

        \draw[-,ultra thick] (\xNull, -0.15)--(\xNull,0.15) {};
        \node [label = 270:{$x$}] (x) at (\xNull, 0) {};

        \draw[->,ultra thick] (-0.5 + \xoffset,\yoffset)--(10 + \xoffset,\yoffset) {};
        \draw[->,ultra thick] (\xoffset,-0.5)--(\xoffset, 5.5) {};
    
        \node [label = 270:{$k$}] (j) at (\rightXRescale * 6.75 + \xoffset - 0.1, \yoffset) {};

        \draw[-, very thick, draw = OliveGreen] (-0.15, \tauValue)--(10, \tauValue) {};
        \node [label = {[text = OliveGreen] 180:$\tau$}] (j) at (0, \tauValue) {};

\begin{scope}
\clip (-0.5,-0.5) rectangle (10 + \xoffset, 3.25 + \yoffset); 
        \def\rescale{3.5}
        \foreach \i in {1,...,7}{
            \node [style = red dot] (ub\i) at (\rightXRescale * \i + \xoffset, {\yoffset + 1.7 * (\i*(ln(ln(2*\i)) + 0.72 * ln(5.2/0.05)))^0.5 / \rescale}) {};
        }

        \foreach \i in {2,3,4,5,6,7}{
            \pgfmathtruncatemacro{\iMinusOne}{\i - 1} 
            \draw[-, draw = red] (ub\i)--(ub\iMinusOne) {};
        }

        \foreach \i in {1,...,7}{
            \node [style = cyan dot] (pb\i) at (\rightXRescale * \i + \xoffset, {\yoffset + 1.7 * (\i*(ln(ln(2*\i)) + 0.72 * ln(5.2/0.0031)))^0.5 / \rescale}) {};
        }
        \foreach \i in {2,3,4,5,6,7}{
            \pgfmathtruncatemacro{\iMinusOne}{\i - 1} 
            \draw[-, draw = Melon] (pb\i)--(pb\iMinusOne) {};
        }
\end{scope}

    \draw[-, color = gray] (-0.5, 0.75)--(3.5, 0.75) {};
    \draw[-, color = gray] (3.5, 2)--(5, 2.75) {};
    \draw[-, color = gray] (5, 2.75)--(9.25, 2.75) {};
    \draw[-, color = gray] (9.25, 2.75)--(10, 3.75) {};
    \node[label = {[text = gray] $\eta$}] at (10, 3.75) {};
    
    \draw[-, draw = gray] (-0.15, 2.75)--(0.15, 2.75) {};
    \node [label = {[text = gray] 180:$\eta(x)$}] at (0, 2.75) {};

    \node [style = black dot, label = 135:$Z_0$] (rightMostPoint) at (9, \tauValue + 1.3) {};
    
    \def\xAndDiffList{7.5/-0.75, 7/1.9, 6.25/-0.55, 4.2/1.8, 2.9/-0.8, 1/-2} 
    \edef\j{0}
    \edef\S{0}
    
    \foreach \x/\d in \xAndDiffList {
        \xdef\jPrev{\j}
        \xdef\SPrev{\S}

        \pgfmathparse{int(\j + 1)}
        \xdef\j{\pgfmathresult}
        \pgfmathparse{\S + \d}
        \xdef\S{\pgfmathresult}

        \node [style = black dot, label = 180:$Z_{\j}$] (x\x) at (\x, \tauValue + \d) {};
        \draw[-, draw = blue] (\x, \tauValue)--(\x, \tauValue + \d) {};
        \node [label = {[text = blue] 0:$\Delta_\j$}] at (\x - 0.2, \tauValue + \d/2) {};

        \node [style = black dot] (j\x) at (\xoffset + \rightXRescale * \j, \yoffset + \S) {};

        \draw[-, draw = blue] (\xoffset + \rightXRescale * \j, \yoffset + \S)--(\xoffset + \rightXRescale * \j, \yoffset + \S - \d) {};
        \node [label = {[text = blue] 0:$\Delta_\j$}] at (\xoffset + \rightXRescale * \j - 0.2, \yoffset + \S - \d/2 + 0.05) {};

        \draw[-, draw = gray] (\xoffset + \rightXRescale * \jPrev, \yoffset + \SPrev)--(\xoffset + \rightXRescale * \j, \yoffset + \S) {};

        \draw[-,ultra thick] (\rightXRescale * \j + \xoffset, -0.15 + \yoffset)--(\rightXRescale * \j + \xoffset, 0.15 + \yoffset) {};
        \node [label = {135: $\j$}] (c\j) at (\rightXRescale * \j + \xoffset + 0.15, \yoffset - 0.1) {};
        
    }    

    \node [label = {[text = gray] 90:$S_k$}] at (0.1 + \xoffset + \rightXRescale * \jPrev/2 + \rightXRescale * \j/2, \yoffset + \SPrev/2 + \S/2) {};

    \node [label = {[text = red] 90:$v_k(\alpha)$}] at (\xoffset + \rightXRescale * 6.5, \yoffset + 1.8) {};
    
    \node [label = {[text = Melon] 90:$v_k(p)$}] at (\xoffset + \rightXRescale * 6.5, \yoffset + 3) {};
\end{tikzpicture}}

    \caption{A schematic illustration of the proposed martingale test with $d=1$ and $\sigma = 1$.  Left: Raw data, where we denote $Z_j := \bigl(X_{(j)}(x), Y_{(j)}(x)\bigr)$ and $\Delta_j := Y_{(j)}(x) - \tau$ for $j\in [n(x)]$; note that the illustrated point $Z_0$ does not enter the martingale because its first component exceeds $x$.  Right: The martingale $(S_k)$, where $\bigl(v_k(\alpha)\bigr)_{k \in [n(x)]}$ denotes a suitable time-uniform upper boundary; here, $S_4 \geq v_4(\alpha)$, so we would reject the null hypothesis $\eta(x) < \tau$ with a $p$-value satisfying $v_4(p) = S_4$.}
    \label{fig:martingale_test}
\end{figure}


Turning now to familywise error rate (FWER) control, and working conditional on $\sampleX$ with $d=1$, consider $p$-values $\bm{p} := (p_i)_{i\in[n]}$ constructed as above for testing the null hypotheses $H_i: \regressionFunction(X_i) < \tau$ for $i \in [n]$.  
One approach to controlling the FWER at level~$\alpha$ is to reject only hypotheses $H_i$ with $i \in \mathcal{R}^\mathrm{FS}_\alpha(\sample_{X}, \bm{p})$, where
\[
\mathcal{R}^\mathrm{FS}_\alpha(\sample_{X}, \bm{p}) := \{j\in [n] \!:\! p_k \leq \alpha \text{ for all } k \text{ with either } X_k > X_j \text{ or both } X_k = X_j \text{ and } k \leq j\}.
\]
Controlling the FWER by employing an a priori ordering is known as a \emph{fixed sequence procedure} \citep{westfall2001optimally, hsu1999stepwise}, and explains the superscript $\mathrm{FS}$ in the notation $\mathcal{R}^\mathrm{FS}_\alpha(\sample_{X}, \bm{p})$.  Writing $H_{(i)}: \regressionFunction\bigl(X_{(i)}(\max_{j \in [n]} X_j)\bigr) < \tau$ and $p_{(i)}$ for the corresponding $p$-value, this approach can be seen as a sequential procedure that, starting with $i = 1$, stops and does not reject $H_{(i)}$ if $p_{(i)} > \alpha$ and otherwise rejects $H_{(i)}$ before proceeding to $i+1$, where the step is repeated.  Here, the order in which we decide whether hypotheses should be rejected is motivated by the fact that  $H_{(i)} \subseteq H_{(i+1)}$ for $i\in [n-1]$. A computationally-efficient implementation of this procedure only needs to calculate $p_{(i+1)}$ if $H_{(i)}$ has been rejected. 

We now extend the presented ideas to the general case $d \in \N$. The construction of the $p$-values follows a similar approach as above, but the order in which the responses enter the supermartingale sequence is now determined by the supremum norm distance of the corresponding covariates from $x$.
\begin{defn}\label{defn:pValue} Given $x \in \R^d$, $\tau \in \R$, $\sigma > 0$ and $\sample = \bigl((X_1,Y_1),\ldots,(X_n,Y_n)\bigr) \in (\R^d \times \R)^n$ write $\mathcal{I}(x) \equiv \mathcal{I}(x, \sample_X) := \{i \in [n]: X_i \preccurlyeq x\}$ and again $n(x) \equiv n(x, \sample_X) := |\mathcal{I}(x, \sample_X)|$.  Further, for $j \in [n(x)]$, let $X_{(j)}(x)$ denote the $j$th nearest neighbour in $\{X_i:i \in \mathcal{I}(x)\}$ of~$x$ in supremum norm, with ties broken by retaining the original ordering of the indices, and let $Y_{(1)}(x),\ldots,Y_{(n(x))}(x)$ denote the concomitant responses.  Defining $S_k \equiv S_k(x,\sigma,\tau, \sample)$ for $k \in [n(x)]$ as in \eqref{Eq:Sk}, we then set
\begin{align*}
    \hat{p}_{\sigma, \tau}(x) \equiv \hat{p}_{\sigma, \tau}(x, \sample) := 1 \wedge \min_{k \in [n(x)]} 5.2\exp\biggl\{-\frac{(S_k \vee 0)^2}{2.0808k} +\frac{ \log\log(2k)}{0.72}\biggr\},
\end{align*}
whenever $n(x) > 0$, and $\hat{p}_{\sigma, \tau}(x, \sample) := 1$ otherwise.
\end{defn}

Lemma~\ref{lemma:validLocalPValStouffer} below shows that $\hat{p}_{\sigma, \tau}(x, \sample)$ is indeed a $p$-value for the null hypothesis $\regressionFunction(x) < \tau$.  We now proceed to the issue of FWER control in the multivariate setting, where we again condition on $\sampleX$. Recall that we are interested in testing the hypotheses $H_i: \regressionFunction(X_i) < \tau$ for $i\in[m]$ and a pre-specified $m \in [n]$.  The fact that $\preccurlyeq$ induces only a partial order on $\mathbb{R}^d$ when $d > 1$ means that there is no natural generalisation of the univariate fixed sequence testing procedure.  Instead, we structure the hypotheses in a directed acyclic graph (DAG), with the edges in the graph representing logical relationships between hypotheses; such an approach has been studied in the literature to control both the FWER \citep{meijer2015multiple} and the false discovery rate \citep{ramdas2019sequential}\footnote{In a different but related approach, a graph structure can be used to encode a ranking of hypotheses beyond a strict logical ordering \citep{bretz2009graphical}.}.  The following definitions will be useful in the construction of an efficient multiple testing procedure. 
\begin{defn}[Induced DAGs and polyforests]\label{def:induced_graph}
Let $\bm{z}=(z_1,\ldots,z_m) \in (\R^d)^m$.
\begin{enumerate}[(i)]
\item The \emph{induced DAG} $\mathcal{G}(\bm{z})=\bigl([m],\mathcal{E}(\bm{z})\bigr)$ is the graph with nodes $[m]$ and edges
\begin{align*}
\mathcal{E}(\bm{z}):=\bigl\{(i_0,i_1) \in [m]^2: i_0 &\neq i_1 \text{ and } z_{i_1}\preccurlyeq z_{i_0} \text{, and if } z_{i_1}\preccurlyeq z_{i_2}\preccurlyeq z_{i_0}\text{ then either }\\z_{i_2} &= z_{i_0}\text{ and } i_0 \leq i_2\text{, or }z_{i_2} = z_{i_1}\text{ and }i_2 \leq i_1 \bigr\}.
\end{align*}
\item The \emph{induced polyforest} $\mathcal{G}_\mathrm{F}(\bm{z})=\bigl([m],\mathcal{E}_\mathrm{F}(\bm{z})\bigr)$ is the subgraph of $\mathcal{G}(\bm{z})$ with nodes $[m]$ and edges\footnote{Here, $\sargmin$ refers to the smallest element of the $\argmin$ set.}
\begin{align*}
\mathcal{E}_\mathrm{F}(\bm{z}) := \Bigl\{(i_0, i_1) \in \mathcal{E}(\bm{z}): i_0 = \sargmin_{i: (i, i_1) \in \mathcal{E}(\bm{z})} \| X_{i} - X_{i_1}\|_\infty \Bigr\}.
\end{align*}
\item The \emph{induced polyforest-weighted DAG} is $\mathcal{G}_\mathrm{W}(\bm{z}) := \bigl([m], \mathcal{E}(\bm{z}), \bm{w}(\bm{z})\bigr)$, where $\bm{w}(\bm{z}) = (w_e)_{e \in \mathcal{E}(\bm{z})}$ is given by $w_e := \mathbbm{1}_{\{e \in \mathcal{E}_\mathrm{F}(\bm{z})\}}$.
\end{enumerate}
\end{defn}
From the definition, we see that the induced polyforest-weighted DAG encodes the complete information of both $\mathcal{G}(\bm{z})$ and $\mathcal{G}_\mathrm{F}(\bm{z})$, as illustrated by Figure~\ref{fig:sparseMG_step1}, where $d = 2$ and each node represents the hypothesis corresponding to the observation at its location. 



\begin{defn}[DAG testing procedure]\label{def:graphMultipilictyProcedure}A \emph{DAG testing procedure} $\rejectionSet$ is a function that takes as input a significance level $\alpha \in (0,1]$, a weighted DAG $G = (I, E, \bm{w})$ and $\bm{p}=(p_i)_{i \in I} \in (0,1]^I$, and outputs a subset $\rejectionSet_\alpha(G,\bm{p})\subseteq I$.
\end{defn}
The fixed sequence procedure presented for $d=1$ is a DAG testing procedure since it only exploits the natural ordering information in $\sampleX$, though in that case we wrote the first argument of $\mathcal{R}_\alpha^{\mathrm{FS}}$ as the set of nodes in the DAG rather than the full DAG for simplicity.  In arbitrary dimensions, the methods proposed by \citet{bretz2009graphical}, \citet{meijer2015multiple} and \citet{ramdas2019sequential} are DAG testing procedures.  While the  \cite{meijer2015multiple} procedure both controls the FWER and accounts for logical relationships between the hypotheses, theoretical and empirical power considerations lead us to propose a new approach that can be regarded as a sparsified version of the \citet{meijer2015multiple} procedure or as an extension of the sequential rejection procedures of \cite{bretz2009graphical}.

In order to describe our proposed iterative DAG testing procedure $\mathcal{R}^{\mathrm{ISS}}$, write $G := \mathcal{G}(\sample_{X,m})$ for the induced DAG and $F := \mathcal{G}_\mathrm{F}(\sample_{X,m})$ for the induced polyforest.  We begin by splitting our $\alpha$-budget across the root nodes, with each such node receiving budget proportional to its number of leaf node descendants in $F$ (including the node itself if it is a leaf node).  We reject each root node hypothesis whose $p$-value is at most its $\alpha$-budget, and whenever we do so, we also reject its ancestors in the original $G$ (which does not inflate the Type I error, due to the logical ordering of the hypotheses).  The rejected root nodes are then removed from $F$, and we repeat the process iteratively, stopping when either we have rejected all hypotheses, or if we fail to reject any additional hypotheses at a given iteration.   Formal pseudocode to compute $\mathcal{R}^{\mathrm{ISS}}$ is given in Algorithm~\ref{algo:R_ISS}; see also Figure~\ref{fig:sparseMG} for an illustration.

\begin{figure}

    \tikzstyle{hyp}=[draw=black, shape=circle, minimum size = 0.85cm, align = center, inner sep = 0cm, font=\fontsize{5pt}{5pt}\selectfont, outer sep = 1pt]
    \tikzstyle{rej hyp}=[fill=black!60, draw=black, shape=circle,  minimum size = 0.85cm, align = center, inner sep = 0cm, font=\fontsize{5pt}{5pt}\selectfont, outer sep = 1pt, text = white]
    
    \tikzstyle{arrowG}=[->, draw=black, dashed]
    \tikzstyle{arrowF}=[->, draw=black]

    \tikzstyle{boundary box}=[draw = white]
    
     \centering
     \begin{subfigure}[b]{0.45\textwidth}
         \centering
         \resizebox{\textwidth}{!}{\begin{tikzpicture}
    \draw[boundary box] (-1.5, 0.25) rectangle (7.75, 5.75) {};

    \def\nodeList{1/0/5/0.01/0, 
                  2/1/2.5/0.1/0, 
                  3/2.5/1/0.3/0, 
                  4/5.5/2/0.04/0, 
                  5/3/3/0.01/0, 
                  6/7/3.5/0.1/0, 
                  7/4/4.25/0.03/0} 
    \foreach \k/\x/\y/\p/\rej in \nodeList {
        \ifthenelse{\rej=0}{
            \node [hyp] (\k) at (\x, \y) {\normalsize $\k$\\ \tiny ($\p$)};
        }{
            \node [rej hyp] (\k) at (\x, \y) {\normalsize $\k$\\ \tiny ($\p$)};
        }
    }

    \def\edgeList{7/5/1, 5/2/1, 5/3/1, 6/4/1, 6/5/0, 4/3/0} 
    \foreach \s/\e/\w in \edgeList {
        \ifthenelse{\w=1}{
        \draw[arrowF] (\s)--(\e);
        }{
        \draw[arrowG] (\s)--(\e);
        }
    }

    \def\alphaList{1/0.0125, 7/0.025, 6/0.0125} 
    \foreach \k/\a in \alphaList {
        \node [text = purple, above left=0.01cm of \k] {\scriptsize $\a$};
    }
    
\end{tikzpicture}}
         \caption{In the first iteration, no hypothesis has been rejected yet and only root nodes are assigned positive $\alpha$-budget. Here, nodes $1, 6$ and $7$ are current rejection candidates, and $1$ will be rejected, as $p_1 = 0.01 \leq 0.0125$.}
         \label{fig:sparseMG_step1}
     \end{subfigure}
     \hspace{0.5cm}
     \begin{subfigure}[b]{0.45\textwidth}
         \centering
         \resizebox{\textwidth}{!}{\begin{tikzpicture}
    \draw[boundary box] (-1.5, 0.25) rectangle (7.75, 5.75) {};

    \def\nodeList{1/0/5/0.01/1, 
                  2/1/2.5/0.1/0, 
                  3/2.5/1/0.3/0, 
                  4/5.5/2/0.04/0, 
                  5/3/3/0.01/0, 
                  6/7/3.5/0.1/0, 
                  7/4/4.25/0.03/0} 
    \foreach \k/\x/\y/\p/\rej in \nodeList {
        \ifthenelse{\rej=0}{
            \node [hyp] (\k) at (\x, \y) {\normalsize $\k$\\ \tiny ($\p$)};
        }{
            \node [rej hyp] (\k) at (\x, \y) {\normalsize $\k$\\ \tiny ($\p$)};
        }
    }

    \def\edgeList{7/5/1, 5/2/1, 5/3/1, 6/4/1, 6/5/0, 4/3/0} 
    \foreach \s/\e/\w in \edgeList {
        \ifthenelse{\w=1}{
        \draw[arrowF] (\s)--(\e);
        }{
        \draw[arrowG] (\s)--(\e);
        }
    }

    \def\alphaList{7/0.0333, 6/0.0167} 
    \foreach \k/\a in \alphaList {
        \node [text = purple, above left=0.01cm of \k] {\scriptsize $\a$};
    }
    
\end{tikzpicture}}
         \caption{After rejection of node $1$ in the first step, we reallocate the $\alpha$-budget, which allows us to reject node $7$.\\
         \\}
         \label{fig:sparseMG_step2}
     \end{subfigure}%
     \\ \vspace{0.2cm}
     \begin{subfigure}[b]{0.45\textwidth}
         \centering
         \resizebox{\textwidth}{!}{\begin{tikzpicture}
    \draw[boundary box] (-1.5, 0.25) rectangle (7.75, 5.75) {};

    \def\nodeList{1/0/5/0.01/1, 
                  2/1/2.5/0.1/0, 
                  3/2.5/1/0.3/0, 
                  4/5.5/2/0.04/0, 
                  5/3/3/0.01/0, 
                  6/7/3.5/0.1/0, 
                  7/4/4.25/0.03/1} 
    \foreach \k/\x/\y/\p/\rej in \nodeList {
        \ifthenelse{\rej=0}{
            \node [hyp] (\k) at (\x, \y) {\normalsize $\k$\\ \tiny ($\p$)};
        }{
            \node [rej hyp] (\k) at (\x, \y) {\normalsize $\k$\\ \tiny ($\p$)};
        }
    }

    \def\edgeList{7/5/1, 5/2/1, 5/3/1, 6/4/1, 6/5/0, 4/3/0} 
    \foreach \s/\e/\w in \edgeList {
        \ifthenelse{\w=1}{
        \draw[arrowF] (\s)--(\e);
        }{
        \draw[arrowG] (\s)--(\e);
        }
    }

    \def\alphaList{5/0.0333, 6/0.0167} 
    \foreach \k/\a in \alphaList {
        \node [text = purple, above left=0.01cm of \k] {\scriptsize $\a$};
    }
    
\end{tikzpicture}}
         \caption{Now that node $7$ has been rejected, its child $5$ receives $\alpha$-budget sufficiently large for it to be rejected. Although $p_6$ is quite large, $6$ is an ancestor of $5$ in the induced DAG and will hence also be rejected.\\}
         \label{fig:sparseMG_step3}
     \end{subfigure}
     \hspace{0.5cm}
     \begin{subfigure}[b]{0.45\textwidth}
         \centering
         \resizebox{\textwidth}{!}{\begin{tikzpicture}
    \draw[boundary box] (-1.5, 0.25) rectangle (7.75, 5.75) {};

    \def\nodeList{1/0/5/0.01/1, 
                  2/1/2.5/0.1/0, 
                  3/2.5/1/0.3/0, 
                  4/5.5/2/0.04/0, 
                  5/3/3/0.01/1, 
                  6/7/3.5/0.1/1, 
                  7/4/4.25/0.03/1} 
    \foreach \k/\x/\y/\p/\rej in \nodeList {
        \ifthenelse{\rej=0}{
            \node [hyp] (\k) at (\x, \y) {\normalsize $\k$\\ \tiny ($\p$)};
        }{
            \node [rej hyp] (\k) at (\x, \y) {\normalsize $\k$\\ \tiny ($\p$)};
        }
    }

    \def\edgeList{7/5/1, 5/2/1, 5/3/1, 6/4/1, 6/5/0, 4/3/0} 
    \foreach \s/\e/\w in \edgeList {
        \ifthenelse{\w=1}{
        \draw[arrowF] (\s)--(\e);
        }{
        \draw[arrowG] (\s)--(\e);
        }
    }

    \def\alphaList{2/0.0167, 3/0.0167, 4/0.0167} 
    \foreach \k/\a in \alphaList {
        \node [text = purple, above left=0.01cm of \k] {\scriptsize $\a$};
    }
    
\end{tikzpicture}}
         \caption{None of the remaining three nodes, which happen to be the leaf nodes, have a $p$-value smaller than their respective $\alpha$-budgets. Hence, no further rejection is made and the procedure terminates. Nodes $1$, $5$, $6$ and $7$ have been rejected.}
         \label{fig:sparseMG_step4}
     \end{subfigure}
     
        \caption{Illustration of Algorithm~\ref{algo:R_ISS} with $\alpha = 0.05$.  Nodes are numbered according to one potential reverse topological ordering of the induced weighted DAG. The $p$-value for each hypothesis is given in round brackets. Solid arrows represent edges with weight $1$ in the induced polyforest-weighted DAG, whereas dashed arrows represent those with weight $0$. Each iteration of the procedure corresponds to one panel. A filled circle indicates that the hypothesis has been rejected in a previous step. If a node is assigned positive $\alpha$-budget at the current iteration, then its (rounded) budget is given in purple to the top left.}
        \label{fig:sparseMG}
\end{figure}

The DAG testing procedure $\rejectionSet^{\mathrm{ISS}}$ allows us to define the corresponding isotonic subgroup selection set
\begin{align*}
\hat{A}^{\mathrm{ISS}} \equiv
\hat{A}^{\mathrm{ISS}}_{\sigma,\tau,\alpha,m}(\sample):= \bigl\{ x \in \R^d : \! X_{i_0} \preccurlyeq x\text{ for some }i_0 \in \rejectionSet_{\alpha}^{\mathrm{ISS}}\bigl(\mathcal{G}_{\mathrm{W}}(\sample_{X,m}),\bigl(\hat{p}_{\sigma,\tau}(X_i,\sample)\bigr)_{i \in [m]}\bigr)\bigr\}.
\end{align*}
Pseudocode for computing $\hat{A}^{\mathrm{ISS}}$ is given in Algorithm~\ref{Alg:ISS}. In Section~\ref{sec:theory} below, we establish that $\hat{A}^{\mathrm{ISS}}$ controls the Type I error uniformly over appropriate distributional classes, and moreover has optimal worst-case power up to poly-logarithmic factors.

\newcommand\mycommfont[1]{\ttfamily\textcolor{blue}{#1}}
\SetCommentSty{mycommfont}
\DontPrintSemicolon

\begin{algorithm}[H]
    \SetAlgoLined

    \textbf{Input: } $\alpha \in (0, 1)$, polyforest-weighted DAG $(I,E,\bm{w})$, $\bm{p} = (p_i)_{i\in I} \in (0,1]^I$\;
	$G \leftarrow (I,E)$ \tcp{$G$ gives the logical structure of the hypotheses}
	$F \leftarrow (I,\{ e \in E: w_e = 1\})$ \tcp{$F$ determines the allocation of the $\alpha$-budget}
    $R_0 \leftarrow \emptyset$\;
    
        \For{$\ell \in [|I|]$}{
          $S_{\mathrm{L}} \leftarrow L(F) \setminus R_{\ell - 1}$ \tcp{currently unrejected $F$-leaf nodes} 
          $S_{\mathrm{C}} \leftarrow \{i \in I: i \notin R_{\ell - 1}, \pa_F(i) \subseteq R_{\ell - 1}\}$ \tcp{current rejection candidates}  
          $\alpha_i \leftarrow \one_{\{i\in S_{\mathrm{C}}\}} \cdot \bigl|\bigl(\{i\} \cup \de_F(i)\bigr) \cap S_{\mathrm{L}}\bigr| \cdot \alpha/|S_{\mathrm{L}}|$ for all $i\in I$\;  
          $I_\ell \leftarrow \{i \in I: p_i \leq \alpha_i\}$\;
            \uIf{$I_\ell = \emptyset$}{
            $R_{|I|} \leftarrow R_{\ell - 1}$ \tcp{will not make any further rejections}
            \textbf{break}
            } 
            $R_\ell \leftarrow R_{\ell - 1} \cup I_\ell \cup \bigcup_{i \in I_\ell} \an_G(i)$ \tcp{reject $I_\ell$ and its ancestors in $G$}
            \uIf{$R_\ell = I$}{
            $R_{|I|} \leftarrow R_{\ell}$ \tcp{everything has been rejected}
            \textbf{break}
            }
    }
    \KwResult{The set of rejected hypotheses $\mathcal{R}_\alpha^{\mathrm{ISS}}\bigl((I, E, \bm{w}), \bm{p}\bigr) := R_{|I|}$}

    \caption{The DAG testing procedure $\mathcal{R}^{\mathrm{ISS}}$.}
    \label{algo:R_ISS}
\end{algorithm}

\DontPrintSemicolon

\begin{algorithm}[H]
    \SetAlgoLined

    \textbf{Input:} $\sample = (X_i,Y_i)_{i \in [n]} \in (\R^d \times \R)^n$, $\tau \in \R$, $\alpha \in (0,1)$, $\sigma > 0$, $m\in [n]$\;
    $\sample_{X,m} \leftarrow (X_1, \ldots, X_m)$ \tcp{subsample the first $m$ covariate points in $\sample$}
    $p_i \leftarrow \hat{p}_{\sigma, \tau}(X_i)$ for $i \in [m]$ \tcp{calculate $p$-values as in Definition~\ref{defn:pValue}}
    $G \leftarrow \mathcal{G}_\mathrm{W}(\sample_{X,m})$ \tcp{construct induced polyforest-weighted DAG}
    $R \leftarrow \mathcal{R}^{\mathrm{ISS}}_\alpha\bigl(G, (p_1,\ldots, p_m)\bigr)$  \tcp{combine $p$-values via Algorithm~\ref{algo:R_ISS}}
    $\hat{A}^{\mathrm{ISS}}_{\sigma,\tau,\alpha,m}(\sample) \leftarrow \{x \in \R^d: X_i \preccurlyeq x \text{ for some } i \in R\}$ \tcp{\hspace{-0.1cm}select upper hull of $(X_i)_{i \in R}$}
    \KwResult{The selected set $\hat{A}^{\mathrm{ISS}}_{\sigma,\tau,\alpha,m}(\sample)$.}
\caption{The subgroup selection algorithm $\hat{A}^{\mathrm{ISS}}$.}    
    \label{Alg:ISS}
\end{algorithm}

\section{Theory}\label{sec:theory}


\subsection{Type I error control}\label{sec:theory_validity}

We first introduce the class of distributions over which we prove the Type I error control of~$\hat{A}^{\mathrm{ISS}}$.
\begin{defn}\label{defn:classOfMonDistributions}
Given $\sigma > 0$, let $\distributionClassNonDecreasingRegressionFunction$ denote the class of all distributions $P$ on $\R^d\times \R$ with increasing regression function $\eta:\R^d \rightarrow \R$, and for which, when $(X,Y) \sim P$, the conditional distribution of $Y-\eta(X)$ given $X$ is sub-Gaussian with variance parameter~$\sigma^2$.
\end{defn}
Our Type I error control relies on showing first that $\hat{p}_{\sigma, \tau}(x,\mathcal{D})$ is indeed a $p$-value for testing the null hypothesis $\eta(x) < \tau$ when  $P\in\distributionClassNonDecreasingRegressionFunction$ and then that the DAG testing procedure $\mathcal{R}^{\mathrm{ISS}}$ controls the FWER in the sense defined in Definition~\ref{def:FWER_control_DAG} below.  The next lemma accomplishes the first of these tasks (in fact, it shows that $\hat{p}_{\sigma, \tau}(x,\mathcal{D})$ is a $p$-value even conditional on $\mathcal{D}_X$).  We write $P^n$ for the $n$-fold product measure corresponding to $P$. 
\begin{lemma}\label{lemma:validLocalPValStouffer} Given any  $x \in \R^d$, $\tau \in \R$, $\sigma >0$, $P \in \distributionClassNonDecreasingRegressionFunction$ such that $\eta(x) < \tau$ and $\sample = \bigl((X_1,Y_1),\ldots,(X_n,Y_n)\bigr) \sim P^n$, we have $\Prob_P\bigl\{\hat{p}_{\sigma,\tau}(x,\sample) \leq \alpha | \sample_X\bigr\} \leq \alpha$ for all $\alpha \in (0,1)$.
\end{lemma}

We now direct our attention towards the DAG testing procedure $\mathcal{R}^{\mathrm{ISS}}$. Here, it is convenient to introduce the following terminology.
\begin{defn}
    Given a weighted, directed graph $G = (I, E, \bm{w})$, we say that a subset $I_0 \subseteq I$ is \emph{$G$-lower} if whenever $i_0 \in I_0$ we have $\de_G(i_0) \subseteq I_0$. Conversely, $I_0 \subseteq I$ is called \emph{$G$-upper} if $I\setminus I_0$ is $G$-lower.
\end{defn}


Given a finite set $I$, a family of distributions $\mathcal{Q}$ on $(0,1]^I$ and a finite collection of null hypotheses $H_i \subseteq \mathcal{Q}$ for $i\in I$, let $G_0 := (I, E_0)$ with $E_0 := \{(i_0, i_1) \in I^2: H_{i_0} \subseteq H_{i_1}\}$ be the directed graph that encodes all logical relationships between hypotheses. 
Then for any $Q \in \mathcal{Q}$, the true null index set 
$I_0(Q) := \{i \in I: Q \in H_i\}$ 
is necessarily a $G_0$-lower set. Conversely, the index set of false null hypotheses must be a $G_0$-upper set.  We say that a polyforest-weighted DAG $(I, E, \bm{w})$ is \emph{$G_0$-consistent} if $E \subseteq E_0$.  Multiple testing procedures that reject hypotheses corresponding to a $G_0$-upper set are called \emph{coherent} \citep[p.~229]{gabriel1969}, and by construction, $\mathcal{R}^{\mathrm{ISS}}$ is indeed coherent when applied to a $G_0$-consistent polyforest-weighted DAG.

We are now in a position to formalise the concept of FWER control for DAG testing procedures. 
\begin{defn}\label{def:FWER_control_DAG}
A DAG testing procedure $\mathcal{R}$ \emph{controls the FWER} if given any finite set $I$, a family of distributions $\mathcal{Q}$ on $(0,1]^I$, a collection of random variables $\bm{p} = (p_i)_{i \in I}$ taking values in $(0,1]^I$, as well as hypotheses $H_i \subseteq \{Q\in \mathcal{Q}: \Prob_Q(p_i \leq t) \leq t, \forall t \in (0, 1]\}$ for $i\in I$ and any $G_0$-consistent polyforest-weighted DAG $G' = (I,E,\bm{w})$, we have $\Prob_Q\bigl(\mathcal{R}_\alpha(G', \bm{p}) \cap I_0(Q) \neq \emptyset\bigr) \leq \alpha$ for all $\alpha \in (0,1)$ and $Q \in \mathcal{Q}$.
\end{defn}
\begin{lemma}\label{lemma:sparseMG_valid}
The DAG testing procedure $\mathcal{R}^{\mathrm{ISS}}$ defined by Algorithm~\ref{algo:R_ISS} controls the FWER. 
\end{lemma}
The strategy of the proof of Lemma~\ref{lemma:sparseMG_valid} is based on ideas in the proof of \citet[Theorem~1]{goeman2010sequential}.  Combining Lemmas~\ref{lemma:validLocalPValStouffer} and~\ref{lemma:sparseMG_valid} yields our Type I error guarantee:
\begin{theorem}\label{thm:validSelectionSet} For any $d\in \N$, $n\in \N$, $m\in [n]$, $\alpha \in (0,1)$, $\tau \in \R$, $\sigma>0$, and $P \in\distributionClassNonDecreasingRegressionFunction$, along with $\sample = \bigl((X_1,Y_1),\ldots,(X_n,Y_n)\bigr) \sim P^n$, we have 
\[
\Prob_P\bigl( \hat{A}^{\mathrm{ISS}}_{\sigma,\tau,\alpha,m}(\sample)\subseteq \superLevelSet{\tau}{\regressionFunction} \bigm| \sample_X\bigr) \geq 1-\alpha.
\]
\end{theorem}

Let $\hat{\mathcal{A}}_n(\tau, \alpha, \mathcal{P})$ denote the family of \emph{data-dependent selection sets} $\hat{A}$ (i.e.~Borel measurable functions from $(\R^d \times \R)^n$ to the set of Borel subsets of $\R^d$) that control the Type~I error rate at level $\alpha \in (0, 1)$ over the family $\mathcal{P}$ of distributions on $\R^d \times \R$.  In other words, we write $\hat{A} \in \hat{\mathcal{A}}_n(\tau, \alpha, \mathcal{P})$ if 
\begin{equation}
\label{Eq:TypeIErrorControl}
\mathbb{P}_P\bigl(\hat{A}(\sample) \subseteq \superLevelSet{\tau}{\regressionFunction}\bigr) \geq 1 - \alpha
\end{equation}
for all $P \in \mathcal{P}$ with $\sample \sim P^n$.  An immediate consequence of Theorem~\ref{thm:validSelectionSet} is that $\hat{A}_{\sigma, \tau, \alpha, m}^{\mathrm{ISS}} \in \hat{\mathcal{A}}_n\bigl(\tau, \alpha, \distributionClassNonDecreasingRegressionFunction\bigr)$.  In fact, an inspection of the proof of Theorem~\ref{thm:validSelectionSet} (see also Lemma~\ref{lemma:validLocalPValStoufferGeneralised}) reveals that $\hat{A}_{\sigma, \tau, \alpha, m}^{\mathrm{ISS}}$ controls the Type I error over a larger class.  Indeed, writing $\mathcal{P}_{\mathrm{Upp}, d}(\tau,\sigma)$ for the class of distributions of pairs $(X, Y)$ such that the $\tau$-superlevel set of the regression function $\eta:\R^d\rightarrow \R$ is an upper set and, again, the conditional distribution of $Y - \eta(X)$ given $X$ is sub-Gaussian with variance parameter $\sigma^2$, it follows from the proof that  $\hat{A}^{\mathrm{ISS}}_{\sigma, \tau, \alpha, m} \in \hat{\mathcal{A}}_n\bigl(\tau, \alpha, \mathcal{P}_{\mathrm{Upp}, d}(\tau, \sigma)\bigr)$.  We have $\distributionClassNonDecreasingRegressionFunction = \bigcap_{\tau' \in \R} \mathcal{P}_{\mathrm{Upp}, d}(\tau',\sigma)$, but the regression functions of distributions in $\mathcal{P}_{\mathrm{Upp}, d}(\tau,\sigma)$ for a fixed $\tau \in \R$ may deviate from monotonicity as long as the $\tau$-superlevel set remains an upper set.  In this sense, our procedure $\hat{A}_{\sigma, \tau, \alpha, m}^{\mathrm{ISS}}$ is robust to misspecification of the monotonicity of the regression function.

\subsection{Power}\label{sec:theory_power}



Classical results on Gaussian testing reveal that merely asking for $P \in \distributionClassNonDecreasingRegressionFunction$ is insufficient to be able to provide non-trivial uniform power guarantees for data-dependent selection sets with Type I error control (see Proposition~\ref{prop:GaussianTestingMonoNotEnough} in Appendix~\ref{sec:appendix_theory_power} for details).  
The main issue here is that the marginal distribution $\mu$ may place a lot of mass in regions where $\regressionFunction$ is only slightly above $\tau$, and these regions will be hard for a data-dependent selection set to include if it has Type I error control.  In this section, therefore, we introduce a margin condition that controls the $\mu$-measure of these difficult~regions.

\begin{defn}
\label{def:univariateDistributionAssumption}
    For $d \in \N$, $\tau \in \R$, $\marginExponent > 0$ and $\marginConstant > 0$, let $\distributionClassUnivariateMarginCondition[d]$ denote the class of distributions $P$ on $\R^d \times \R$ for which the marginal $\marginalDistribution$ on $\R^d$ and the regression function $\regressionFunction: \mathbb{R}^d \rightarrow \R$ satisfy $\marginalDistribution\bigl(\regressionFunction^{-1}([\tau , \tau +\marginConstant \marginDummy^\marginExponent])\bigr) \leq \marginDummy$ for all $\marginDummy\in (0,1]$.
\end{defn}

\begin{example}
Let $d = 1$ and let $P \in \distributionClassNonDecreasingRegressionFunction[1]$ have uniform marginal distribution $\mu$ on $[0,1]$ and regression function $\eta$.  We then have $P \in \distributionClassNonDecreasingRegressionFunction[1] \cap \distributionClassUnivariateMarginCondition[1]$ if $\eta(x + \marginDummy) \geq \tau + \marginConstant \marginDummy^\marginExponent$ for all $\marginDummy \in (0, 1]$ and $x \in \superLevelSet{\tau}{\regressionFunction}$. 
\end{example}
We now divide our power analysis for $\hat{A}^{\mathrm{ISS}}$ into univariate and multivariate cases, since the natural total order on $\R$ means that our results simplify a little when $d=1$.  Theorem~\ref{thm:powerBound_1d} below provides  high-probability and expectation upper bounds on the \emph{regret} $\mu\bigl(\superLevelSet{\tau}{\regressionFunction}\setminus\hat{A}^{\mathrm{ISS}}\bigr)$ for $P \in \distributionClassNonDecreasingRegressionFunction[1] \cap \distributionClassUnivariateMarginCondition[1]$. 
\begin{theorem}\label{thm:powerBound_1d} Let $\sigma, \marginExponent, \marginConstant > 0$ and $\alpha \in (0,1)$. There exists a universal constant $C > 0$ such that for any distribution $P \in \distributionClassNonDecreasingRegressionFunction[1] \cap \distributionClassUnivariateMarginCondition[1]$ and $\delta \in (0,1)$, we have 
\begin{align*}
\Prob_P\biggl[ \marginalDistribution\bigl(\superLevelSet{\tau}{\regressionFunction}\setminus\hat{A}^{\mathrm{ISS}}_{\sigma,\tau,\alpha, n}(\sample)\bigr) > 1 \wedge C \biggl\{ \biggl(\frac{\sigma^2}{n\marginConstant^2}\log_+\Bigl(\frac{\log_+ n}{\alpha \wedge\delta}\Bigr) \biggr)^{1/(2\marginExponent + 1)} + \frac{\log_+(1/\delta)}{n} \biggr\} \biggr] \leq \delta,
\end{align*}
and
\[
\mathbb{E}_{P}\bigl\{ \marginalDistribution\bigl(\superLevelSet{\tau}{\regressionFunction}\setminus\hat{A}^{\mathrm{ISS}}_{\sigma,\tau,\alpha, n}(\sample)\bigr) \bigr\} \leq 1 \wedge C \biggl\{ \biggl(\frac{\sigma^2}{n\marginConstant^2}\log_+\Bigl(\frac{\log_+ n}{\alpha}\Bigr) \biggr)^{1/(2\marginExponent + 1)} + \frac{1}{n} \biggr\}.
\]
\end{theorem}
From Theorem~\ref{thm:powerBound_1d}, we see that the regret of $\hat{A}^{\mathrm{ISS}}$ decomposes as a sum of two terms: the first reflects the error incurred in determining whether each data point belongs to $\superLevelSet{\tau}{\eta}$, while the second represents the error arising from the uncertainty of whether or not regions between the data points belong to this superlevel set.  The combination of Theorem~\ref{thm:powerBound_1d} with Proposition~\ref{Prop:Containment} below and Theorem~\ref{Thm:LowerBound} in Section~\ref{sec:lowerBound} reveals that the dependence of our bound on the parameters $n$, $\alpha$, $\sigma$, $\marginExponent$ and $\marginConstant$ is optimal up to an iterated logarithmic factor in~$n$. 

In the proof of Theorem~\ref{thm:powerBound_1d}, we exploit the fact that $\sample_{X}$ has a total order in the univariate case, so the corresponding induced polyforest-weighted DAG forms a directed path in which each edge has weight $1$.  Since in our algorithm, the $p$-values $\hat{p}_{\sigma,\tau}$ for coinciding hypotheses are equal, $\mathcal{R}^{\mathrm{ISS}}$ is equivalent to the fixed sequence procedure~$\mathcal{R}^{\mathrm{FS}}$. 

Turning now to the multivariate case, we begin with a negative result, which reveals that we can find distributions in our class for which no data-dependent selection set with Type~I error control performs better than the trivial procedure that ignores the data, and selects the entire domain with probability $\alpha$ and the empty set otherwise. 
\begin{prop}\label{prop:MaxMinNegativeResultMultivariateNoDensityBound}
Let $d \geq 2$, $\tau \in \R$, $\sigma, \marginExponent, \marginConstant > 0$ and $\alpha \in (0, 1)$. Then, writing $\mathcal{P}' := \distributionClassNonDecreasingRegressionFunction[d]\cap\distributionClassUnivariateMarginCondition[d]$, we have for any $n \in \N$ that 
\[
\sup_{P \in \mathcal{P}'}\inf_{\hat{A} \in \hat{\mathcal{A}}_n(\tau,\alpha,\mathcal{P}')}  \mathbb{E}_{P}\bigl\{ \marginalDistribution\bigl(\superLevelSet{\tau}{\regressionFunction}\setminus\hat{A}(\sample)\bigr) \bigr\} \geq 1-\alpha.
\]
\end{prop}

An interesting feature of Proposition~\ref{prop:MaxMinNegativeResultMultivariateNoDensityBound} is the ordering of the supremum over distributions in our class and the infimum over data-dependent selection sets.  Usually, with minimax lower bounds, these would appear in the opposite order, but here we are able to establish the stronger conclusion, because it is the same subfamily of $\distributionClassNonDecreasingRegressionFunction[d]\cap\distributionClassUnivariateMarginCondition[d]$ that causes the poor performance of any data-dependent selection set with Type~I error control.  In fact, by examining the proof, we see that the issue is caused by constructing a marginal distribution $\mu$ that concentrates its mass around a large antichain\footnote{Recall that an \emph{antichain} in $\R^d$ is a set $\mathbb{W}$ such that we do not have $x \preccurlyeq x'$ for any $x,x' \in \mathbb{W}$.  It is the fact that antichains of arbitrary size exist in $[0,1]^d$ when $d \geq 2$ that is essential to this construction; when $d=1$, any antichain must be a singleton.} in $[0,1]^d$, which constitutes the boundary of $\superLevelSet{\tau}{\regressionFunction}$.  This motivates us to regulate the extent to which this is allowed to happen.
\begin{defn}\label{def:multivariateAssumption} Given $d \in \N$, $\tau \in \R$,  $\densityConstant > 1$, $\etaIncreasingExponent, \etaIncreasingConstant > 0$, we let $\distributionClassMultivariateCondition$ denote the class of all distributions $P$ on $\R^d \times \R$ with marginal $\marginalDistribution$ on $\R^d$ and associated regression function~$\regressionFunction$ such that 
\begin{enumerate}[(i)]
\item $\densityConstant^{-1} \cdot r^d \leq \marginalDistribution\bigl(\closedSupNormMetricBall{x}{r}\bigr) \leq \densityConstant \cdot (2r)^d$ for all $x \in \superLevelSet{\tau}{\regressionFunction} \cap \support(\marginalDistribution)$ and $r \in (0,1]$;
\item $\closedSupNormMetricBall{x}{r} \cap \superLevelSet{\tau+\etaIncreasingConstant \cdot r^{\etaIncreasingExponent}}{\regressionFunction} \neq \emptyset$ for all $x \in \superLevelSet{\tau}{\regressionFunction} \cap \support(\marginalDistribution)$ and $r \in (0,1]$.
\end{enumerate}
\end{defn}
For distributions in the $\distributionClassNonDecreasingRegressionFunction[d]$ class, Definition~\ref{def:multivariateAssumption} represents a slight strengthening of the margin condition used in our univariate analysis, as made precise by Proposition~\ref{Prop:Containment} below.
\begin{prop}
\label{Prop:Containment}
Let $d \in \N$, $\tau \in \R$, $\sigma, \etaIncreasingExponent, \etaIncreasingConstant > 0$ and $\densityConstant > 1$. There exists $C \geq 1$, depending only on $(d, \densityConstant)$, such that
\[
\distributionClassNonDecreasingRegressionFunction[d] \cap \distributionClassMultivariateCondition[d] \subseteq \distributionClassNonDecreasingRegressionFunction[d] \cap \mathcal{P}_{\mathrm{Mar}, d}(\tau, \etaIncreasingExponent, \etaIncreasingConstant/C^\etaIncreasingExponent).
\]
\end{prop}

Thus, $(\etaIncreasingExponent,\etaIncreasingConstant)$ in the class $\distributionClassMultivariateCondition[d]$ play a similar but not identical role to $(\marginExponent,\marginConstant)$ in $\distributionClassUnivariateMarginCondition[d]$, in controlling the way in which the regression function is required to grow as we move away from the boundary of the $\tau$-superlevel set.  We are now in a position to state our main result concerning the power of our proposed procedure; the result holds in all dimensions  but our primary interest here is in the multivariate case.
\begin{theorem}\label{thm:powerBound} Let $d \in \N$, $\tau \in \R$, $\sigma,\etaIncreasingExponent,\etaIncreasingConstant > 0$ and $\densityConstant > 1$. There exists $C\geq 1$, depending only on $(d, \densityConstant)$, such that for any $P \in \distributionClassNonDecreasingRegressionFunction[d] \cap \distributionClassMultivariateCondition[d]$, $n \in \N$, $\alpha \in (0,1)$ and $\delta \in (0,1)$, along with $\sample = \bigl((X_1,Y_1),\ldots,(X_n,Y_n)\bigr) \sim P^n$, we have for $m \in [n]$ that
\begin{align*}
\Prob_P\biggl[ \marginalDistribution\bigl(\superLevelSet{\tau}{\regressionFunction}\setminus\hat{A}^{\mathrm{ISS}}_{\sigma,\tau,\alpha,m}(\sample)\bigr) > 1 \wedge C \biggl\{ \biggl(\frac{\sigma^2}{n \etaIncreasingConstant^2}\log_+\Bigl(\frac{m\log_+ n}{\alpha \wedge\delta}\Bigr)\biggr)^{\frac{1}{2\etaIncreasingExponent+d}} +\biggl(\frac{\log_+(m/\delta)}{m}\biggr)^{\frac{1}{d}} \biggr\} \biggr] \leq \delta,
\end{align*}
and 
\begin{align*}
\mathbb{E}_{P}\bigl\{ \marginalDistribution\bigl(\superLevelSet{\tau}{\regressionFunction}\setminus\hat{A}^{\mathrm{ISS}}_{\sigma,\tau,\alpha, m}(\sample)\bigr) \bigr\} &\leq 1 \wedge C \biggl\{ \biggl(\frac{\sigma^2}{n \etaIncreasingConstant^2}\log_+\Bigl(\frac{m\log_+ n}{\alpha}\Bigr) \biggr)^{1/(2\etaIncreasingExponent+d)} +\biggl(\frac{\log_+ m}{m}\biggr)^{1/d} \biggr\}.
\end{align*}
\end{theorem}
The terms in the bound in Theorem~\ref{thm:powerBound} are similar to those in Theorem~\ref{thm:powerBound_1d}, and exhibit the trade-off in the choice of $m$: if we choose it to be small, then there are fewer data points in our subsample that belong to $\superLevelSet{\tau}{\eta}$ and moreover these are less likely to be excluded from $\hat{A}^{\mathrm{ISS}}$ because they are typically assigned greater budget in our DAG testing procedure.  On the other hand, we incur a greater loss in excluding regions between data points in our subsample that belong to this superlevel set.  By specialising Theorem~\ref{thm:powerBound} via a particular choice of $m$, we obtain the following almost immediate corollary.  It is this upper bound to which we will compare our minimax lower bound in Theorem~\ref{Thm:LowerBound}.
\begin{corollary}
\label{Cor:PowerBound}
Under the conditions of Theorem~\ref{thm:powerBound}, if we take $m_0:= n \wedge \lceil  n\etaIncreasingConstant^2/\sigma^2\rceil$, then
\begin{align*}
\mathbb{E}_{P}\bigl\{ \marginalDistribution\bigl(\superLevelSet{\tau}{\regressionFunction}\setminus\hat{A}^{\mathrm{ISS}}_{\sigma,\tau,\alpha, m_0}(\sample)\bigr) \bigr\} &\leq 1 \wedge 4C \biggl\{ \biggl(\frac{\sigma^2}{n \etaIncreasingConstant^2}\log_+\Bigl(\frac{n\lambda^2\log_+ n}{\sigma^2\alpha}\Bigr)\biggr)^{1/(2\etaIncreasingExponent+d)} + \biggl(\frac{\log_+ n}{n}\biggr)^{1/d} \biggr\}.    
\end{align*}
\end{corollary}
As is apparent from the proof of Corollary~\ref{Cor:PowerBound}, a high-probability bound analogous to that in Theorem~\ref{thm:powerBound} also holds, but this is omitted for brevity.  In practice, one can take $m = n$ (as we do in our simulations in Section~\ref{sec:simulations}), with a corresponding power bound obtained as a special case of Theorem~\ref{thm:powerBound}.   Corollary~\ref{Cor:PowerBound} suggests a curse of dimensionality effect in isotonic subgroup selection; this is confirmed as an essential price to pay by Theorem~\ref{Thm:LowerBound} below.

\subsection{Main lower bound}\label{sec:lowerBound}

In order to discuss the optimality of our data-dependent selection set $\hat{A}^{\mathrm{ISS}}$, we present a minimax lower bound that provides a benchmark on the regret that is achievable by any data-dependent selection set with Type I error control.
\begin{theorem}\label{Thm:LowerBound}
    Let $d \in \N$, $\tau \in \R$, $\sigma, \etaIncreasingExponent, \etaIncreasingConstant > 0$ and $\densityConstant > 1$. Then, writing  $\mathcal{P}' := \distributionClassNonDecreasingRegressionFunction[d] \cap \distributionClassMultivariateCondition[d]$, there exists $c \in (0, 1)$, depending only on $(d, \etaIncreasingExponent)$, such that for any $n\in\N$ and $\alpha \in (0, 1/4]$, we have 
    \begin{equation}
    \label{Eq:MLB}
    \inf_{\hat{A} \in \hat{\mathcal{A}}_n(\tau, \alpha, \mathcal{P}')} \sup_{P \in \mathcal{P}'} \mathbb{E}_P \bigl\{\marginalDistribution\bigl(\superLevelSet{\tau}{\regressionFunction}\setminus \hat{A}(\sample)\bigr)\bigr\} \geq c\biggl[1 \wedge \biggl\{\biggl(\frac{\sigma^2}{n\etaIncreasingConstant^2}\log_+\Bigl(\frac{1}{5\alpha}\Bigr)\biggr)^{1/(2\etaIncreasingExponent + d)} + \frac{1}{n^{1/d}}\biggr\}\biggr].
    \end{equation}
\end{theorem}

By comparing the rate in Theorem~\ref{Thm:LowerBound} with those in Theorem~\ref{thm:powerBound_1d} and Corollary~\ref{Cor:PowerBound}, we see that  $\hat{A}^{\mathrm{ISS}}$ attains the optimal regret among procedures with Type~I error control, up to poly-logarithmic factors.  In particular, up to such factors, these results reveal the optimal dependence of the regret not only on $n$, but also on $\sigma$, $\etaIncreasingConstant$ and $\alpha$.  It is interesting to note that Theorem~\ref{Thm:LowerBound} incorporates procedures that are only required to control Type~I error over $\distributionClassNonDecreasingRegressionFunction[d] \cap \distributionClassMultivariateCondition[d]$, whereas $\hat{A}^{\mathrm{ISS}}$ has Type~I error control over the larger class $\distributionClassNonDecreasingRegressionFunction[d]$, by Theorem~\ref{thm:validSelectionSet}.  Thus, $\hat{A}^{\mathrm{ISS}}$ suffers no deterioration in performance for this stronger validity guarantee, at least up to poly-logarithmic factors.

The proof of Theorem~\ref{Thm:LowerBound} combines two minimax lower bounds, given in Propositions~\ref{prop:lowerBound_alpha_multivariate} and~\ref{prop:lowerBoundParametric}, which provide the different terms in the sum in~\eqref{Eq:MLB}.  The main idea in both cases is to divide $[0,1]^d$ into a hypercube lattice, and to construct pairs of distributions where either the regression function (Proposition~\ref{prop:lowerBound_alpha_multivariate}) or the marginal distribution (Proposition~\ref{prop:lowerBoundParametric}) only differ in a single hypercube among a collection whose centres lie on a large antichain in $\R^d$.  Observations outside these critical hypercubes therefore do not help to distinguish between the distributions in a pair, so by choosing the number of hypercubes and the difference in the regression function levels appropriately, we obtain a non-trivial probability of failing to include them in a data-dependent selection set.  Our formal constructions, together with illustrations, are given in Section~\ref{sec:lowerBound_proofs}.

\section{Extensions}
\label{Sec:Extensions}

\newcommand\distributionClassBounded{\mathcal{P}_{\mathrm{Bdd}, d}}
\newcommand\distributionClassMedian{\mathcal{P}_{\mathrm{Med}, d}(\tau)}

\subsection{Choice of \texorpdfstring{$p$}{p}-value construction}
\label{SubSec:pvalues}

Recall that in our isotonic subgroup selection procedure $\hat{A}^{\mathrm{ISS}}$, we propose a $p$-value based on a martingale test in combination with a \emph{finite law of the iterated logarithm (LIL) bound} (Lemma~\ref{lemma:howard_uniform_bound}\emph{(a)}).  The following definition gives an alternative $p$-value construction that uses a different bound and includes a hyperparameter $\rho > 0$. 
\begin{defn}\label{defn:pValue_normalmixture}
    In the setting of Definition~\ref{defn:pValue}, for $\rho > 0$, we define
        \begin{align*}
        \tilde{p}^\rho_{\sigma, \tau}(x) \equiv \tilde{p}^\rho_{\sigma, \tau}(x, \sample) := 1 \wedge \min_{k \in [n(x)]} \sqrt{\frac{k + \rho}{4\rho}} \biggl\{\exp\biggl(\frac{(S_k \vee 0)^2}{2(k+\rho)}\biggr) - 1\biggr\}^{-1}
    \end{align*}
    whenever $n(x) > 0$, and $\tilde{p}^\rho_{\sigma, \tau}(x, \sample) := 1$ otherwise.
\end{defn}
By Lemma~\ref{lemma:howard_uniform_bound}\emph{(b)}, which is due to \citet{howard2021uniform}, in combination with the proof technique of Lemma~\ref{lemma:validLocalPValStouffer}, $\tilde{p}^\rho_{\sigma, \tau}(x)$ is indeed a $p$-value for the null hypothesis $\eta(x) < \tau$ (even conditional on $\sampleX$).  It follows that if we modify our procedure to use these $p$-values instead of those in Definition~\ref{defn:pValue}, then the Type I error guarantee in Theorem~\ref{thm:validSelectionSet} is unaffected.  The objective function being minimised over $k \in [n(x)]$ is a little smaller for the original $p$-value definition when $k$ is sufficiently large, and this therefore leads to a stronger power bound in Theorem~\ref{thm:powerBound}.  Nevertheless, for appropriate values of $\rho > 0$, Definition~\ref{defn:pValue_normalmixture} may yield a slightly smaller objective for small and moderate values of $k$, and hence may be preferable in practice.  Based on some preliminary simulations, we found that the power of our approach varied very little over choices of $\rho \in (0,1]$, and that $\rho=1/2$ was a robust choice that we used throughout our experiments in Section~\ref{sec:simulations} and recommend for practical use.

\newcommand\distributionClassNormal{\mathcal{P}_{\mathrm{N}, d}(\sigma)}

\subsection{Alternative distributional assumptions}\label{Sec:Extension_AlternativeDistributions}

In this subsection, we introduce three variants of $\hat{A}^{\mathrm{ISS}}$, each of which is able to control Type~I error over appropriate classes without knowledge of any nuisance parameter and without the need for sample splitting.  In each case, we retain the same multiple testing component $\mathcal{R}^{\mathrm{ISS}}$ to our procedure, but construct the $p$-values in different ways.  Power results analogous to Theorem~\ref{thm:powerBound} also hold for these versions of $\hat{A}^{\mathrm{ISS}}$, but are omitted for brevity.

\subsubsection{Gaussian noise with unknown variance}

Since Algorithm~\ref{Alg:ISS} takes the sub-Gaussian variance parameter $\sigma^2$ as an input, we present here an adaptive approach for a Gaussian error setting.  More precisely, for $\sigma > 0$, let $\distributionClassNormal$ denote the subset of $\distributionClassNonDecreasingRegressionFunction$ with $Y - \eta(X) | X \sim \mathcal{N}(0, \sigma^2)$. 
\begin{defn}\label{def:pValue_normal}
    In the setting of Definition~\ref{defn:pValue}, let 
    $\hat{\sigma}^2_{0,k} := k^{-1}\sum_{j=1}^k \bigl(Y_{(j)}(x) - \tau\bigr)_+^2$ and $\Bar{Y}_{1, k} :=  k^{-1}\sum_{j=1}^k Y_{(j)}(x)$ for $k\in [n(x)]$ and $\hat{\sigma}^2_{1,k} := k^{-1}\sum_{j=1}^k \bigl(Y_{(j)}(x) - \Bar{Y}_{1,k}\bigr)^2$ for $k \in \{2,\ldots, n(x)\}$. Moreover, we denote $\Bar{Y}_{1,0} := 0$, and $\hat{\sigma}^2_{1,k} :=1$ for $k \in \{0, 1\}$. 
    For $k\in [n(x)]$, define
    \[
    \Bar{p}^k_{\tau}(x) \equiv \Bar{p}^k_\tau(x,\sample) := \frac{1}{\hat{\sigma}^{k}_{0,k}e^{k/2}}\cdot \prod_{j=1}^k \hat{\sigma}_{1, j-1}\exp\biggl\{\frac{\bigl(Y_{(j)}(x) - \Bar{Y}_{1, j-1}\bigr)^2}{2\hat{\sigma}^2_{1,j-1}} \biggr\},
    \]
    where for definiteness $\Bar{p}^k_\tau(x) := 1$ if $\hat{\sigma}_{0,k} = 0$, and $\Bar{p}_{\tau}(x) \equiv  \Bar{p}_{\tau}(x, \sample) := 1 \wedge \min_{k\in [n(x)]}  \Bar{p}^k_{\tau}(x)$.
\end{defn}
The idea here is that $\Bar{p}^k_{\tau}(x)$ exploits the sequential likelihood ratio test principle developed by \cite{wasserman2020universal}, applied to a notion of a $t$-test for a stream of independent normal random variables with varying means.  In particular, $\hat{\sigma}_{0,k}^2$ and $\hat{\sigma}_{1,k}^2$ are maximum likelihood estimators of $\sigma^2$ under the null hypothesis that $\max_{j \in [k]} \eta\bigl(X_{(j)}(x)\bigr) < \tau$ and without this constraint respectively.  The next lemma is analogous to Lemma~\ref{lemma:validLocalPValStouffer} and guarantees that $\Bar{p}_\tau(x)$ is a $p$-value.  
\begin{lemma}\label{lemma:validLocalPValNormality}
    \sloppy Let $x \in \R^d$, $\tau \in [0,1)$, $P \in \cup_{\sigma \in (0,\infty)} \distributionClassNormal$ with $\eta(x) < \tau$ and $\sample = \bigl((X_1,Y_1),\ldots,(X_n,Y_n)\bigr) \sim P^n$.  Then $\Prob_P\bigl\{\Bar{p}_{\tau}(x, \sample) \leq \alpha | \sample_X\bigr\} \leq \alpha$ for all $\alpha \in (0,1)$.
\end{lemma}

It is worth highlighting that the proof of Lemma~\ref{lemma:validLocalPValNormality} relies on the fact that $\{x \in \R^d: \eta(x) < \tau\}$ is a lower set, but otherwise does not use the monotonicity of $\eta$. From this, we can deduce that an analogous robustness to misspecification holds to that presented at the end of Section~\ref{sec:theory_validity}.  

\subsubsection{Classification}\label{sec:extensions_classification}

Our second variant is tailored to the case of bounded responses, which in particular includes classification settings.  Suppose that $(X, Y) \sim P$ for some distribution $P$ on $\R^d \times [0, 1]$ with increasing regression function $\regressionFunction$ on $\R^d$. By Hoeffding's lemma, the sub-Gaussianity condition of Definition~\ref{defn:classOfMonDistributions} is satisfied with $\sigma = 1/2$, so that $\hat{A}^{\mathrm{ISS}}$ may be used to control the Type~I error over $\mathcal{P}_{\mathrm{Mon}, d}(1/2)$.  However, in this context, it suffices to control the Type~I error over the subclass $\distributionClassBounded$ of $\mathcal{P}_{\mathrm{Mon}, d}(1/2)$ consisting of distributions $P$ on $\R^d \times [0, 1]$ with increasing regression function.  In such a setting, we may combine our procedure with the following modified $p$-value construction.  Recall that for $z \in (0,1)$ and $a,b > 0$, the \emph{incomplete beta function} is defined by $\mathrm{B}(z; a, b) := \int_0^z t^{a - 1}(1-t)^{b - 1}\, dt$.
\begin{defn}\label{defn:pValue_binary}
    In the setting of Definition~\ref{defn:pValue}, let $\Check{S}_k \equiv \check{S}_k(x, \sample) := \sum_{j=1}^k Y_{(j)}(x)$ and define
    \begin{align*}
        \check{p}_{\tau}(x) \equiv \check{p}_{\tau}(x, \sample) := 1 \wedge \min_{k \in [n(x)]} \frac{\tau^{\check{S}_k}(1-\tau)^{n-\check{S}_k + 1}}{\mathrm{B}(1-\tau; n - \check{S}_k + 1, \check{S}_k + 1)}.
    \end{align*}
\end{defn}
The following lemma confirms that this indeed defines a $p$-value (even conditional on~$\sampleX$).  Its proof proceeds via a one-sided version of the time-uniform confidence sequence construction of \cite{robbins1970statistical}, which is itself based on earlier work by \cite{ville1939etude} and \cite{wald1947sequential}. 

\begin{lemma}\label{lemma:validLocalPValClassification}\sloppy Let $x \in \R^d$, $\tau \in [0,1)$, $P \in \distributionClassBounded$ with $\eta(x) < \tau$ and $\sample = \bigl((X_1,Y_1),\ldots,(X_n,Y_n)\bigr) \sim P^n$.  Then $\Prob_P\bigl\{\check{p}_{\tau}(x, \sample) \leq \alpha | \sample_X\bigr\} \leq \alpha$ for all $\alpha \in (0,1)$.
\end{lemma}
As a consequence of Lemma~\ref{lemma:validLocalPValClassification}, combining $\hat{A}^{\mathrm{ISS}}$ with the $p$-values $\check{p}_\tau$ still controls the Type~I error over $\distributionClassBounded$.  In a similar spirit to the discussion at the end of Section~\ref{sec:theory_validity}, both this conclusion and Lemma~\ref{lemma:validLocalPValClassification} hold over the even larger class $\mathcal{P}_{\mathrm{BddUpp}, d}(\tau) \supseteq \mathcal{P}_{\mathrm{Bdd}, d}$ of distributions $P$ on $\R^d \times [0,1]$ with regression function $\eta$ such that $\superLevelSet{\tau}{\eta}$ is an upper set (see Lemma~\ref{lemma:validLocalPValClassificationGeneralisation}).

\subsubsection{Increasing conditional quantiles}\label{sec:extensions_nonnegative_median}

\newcommand{\distributionClassQuantileIncreasing}{\mathcal{P}_{\mathrm{Q},d}(\theta)}

Finally in this subsection, we present an alternative assumption on the conditional response distribution that motivates a version of $\hat{A}^{\mathrm{ISS}}$ that is robust to heavy tails.
\begin{defn}
\label{Def:CondQuant}
Given $\theta \in (0,1)$, a distribution $P$ on $\R^d \times \R$ and $(X,Y) \sim P$, let $\zeta_{\theta}:\R^d \rightarrow \R$ denote the conditional $\theta$-quantile  given by $\zeta_{\theta}(x):= \inf\bigl\{ y \in \R : \mathbb{P}_P\bigl(Y \leq y | X = x\bigr)\geq \theta\bigr\}$ for $x \in \R^d$. Now let $\distributionClassQuantileIncreasing$  denote the class of all such distributions $P$ for which $\zeta_{\theta}$ is increasing.
\end{defn}
\begin{lemma}\label{lemma:validLocalPValConditionalQuantile}\sloppy Given $\alpha \in (0,1)$, $\theta \in (0,1)$, $\tau \in \R$, $P \in \distributionClassQuantileIncreasing$, $\sample = \bigl((X_1,Y_1),\ldots,(X_n,Y_n)\bigr) \sim P^n$ and writing $\sample^{\tau} = \bigl((X_1,\one_{\{Y_1 > \tau\}}),\ldots,(X_n,\one_{\{Y_n > \tau\}})\bigr)$, we have whenever $x \notin \mathcal{X}_\tau(\zeta_{\theta})$ that $\Prob_P\bigl\{\hat{p}_{1/2,1-\theta}(x, \sample^{\tau}) \leq \alpha | \sample_X\bigr\} \leq \alpha$ and $\Prob_P\bigl\{\check{p}_{1-\theta}(x, \sample^{\tau}) \leq \alpha | \sample_X\bigr\} \leq \alpha$ for all $\alpha \in (0,1)$. 
\end{lemma}
In particular, if $\eta$ is increasing and the conditional distribution of $Y - \eta(X)$ given $X$ is symmetric about zero, then the distribution of $(X,Y)$ belongs to $\mathcal{P}_{\mathrm{Q},d}(1/2)$.  As such, the modification of $\hat{A}^{\mathrm{ISS}}$ with the $p$-values $\check{p}_{1/2}(x, \sample^{\tau})$ in place of $\hat{p}_{\sigma, \tau}(x, \sample)$ controls the Type~I error at the nominal level.

\subsection{Application to heterogeneous treatment effects}
\label{SubSec:HTE}

We now describe how our proposed procedure can be used to identify subsets of the covariate domain with high treatment effects in randomised controlled trials.  As a model for such a setting, we assume that we observe independent copies $(X_1, T_1, \Tilde{Y}_1), \ldots, (X_n, T_n, \Tilde{Y}_n)$ of the triple $(X, T, \Tilde{Y})$, where $X$ is the covariate vector, $T$ takes values in $\{0, 1\}$ and encodes the assignment to one of two treatment arms, and $\Tilde{Y}$ gives the corresponding response. For $\ell \in \{0, 1\}$, denote by $\Tilde{P}^\ell$ the conditional distribution of $(X, \tilde{Y})$ given that $T = \ell$ and define corresponding regression functions $\Tilde{\eta}^\ell$ by $\Tilde{\eta}^\ell(x) := \mathbb{E}(\Tilde{Y}|X = x, T = \ell)$ for $x\in\R^d$.  We are interested in identifying the $\tau$-superlevel set of the \emph{heterogeneous treatment effect} $\eta$, where $\eta(x) := \Tilde{\eta}^1(x) - \Tilde{\eta}^0(x)$ for $x\in \R^d$.  To this end, observe that writing $\pi(x) := \mathbb{P}(T = 1 | X = x)$ for the \emph{propensity score}, and considering the \emph{inverse propensity weighted response}
\begin{align}
Y := \frac{T - \pi(X)}{\pi(X)\bigl(1-\pi(X)\bigr)}\cdot \Tilde{Y}, \label{eq:inverse_propensity_weighting}    
\end{align}
we have $\mathbb{E}(Y | X = x) = \eta(x)$ for all $x\in \R^d$. Hence, writing $P$ for the distribution of $(X, Y)$ and with $\sample = \bigl((X_1, Y_1), \ldots, (X_n, Y_n)\bigr) \sim P^n$, where $Y_1, \ldots, Y_n$ are the inverse propensity weighted responses obtained from $(X_1, T_1, \Tilde{Y}_1), \ldots, (X_n, T_n, \Tilde{Y}_n)$, we have for any $\alpha \in (0, 1)$ and any $m \in [n]$ that $\mathbb{P}_P\bigl(\hat{A}^\mathrm{ISS}_{\sigma, \tau, \alpha, m}(\sample) \subseteq \superLevelSet{\tau}{\eta}\bigr) \geq 1 - \alpha$ whenever $P \in \distributionClassNonDecreasingRegressionFunction$, by Theorem~\ref{thm:validSelectionSet}.

\subsubsection{Conditional treatment ranking}\label{sec:extensions_noninferiority}

Under assumptions that are in spirit similar to those in Section~\ref{sec:extensions_nonnegative_median}, we can use $\hat{A}^{\mathrm{ISS}}$ to establish non-inferiority of a treatment on a subgroup. Suppose that for $x\in \R^d$, $y \in \R$ and $\ell \in \{0, 1\}$, we have $\mathbb{P}\bigl(\tilde{Y}  - \Tilde{\eta}^\ell(x) \leq y| X = x, T = \ell\bigr) = F(y)$ for some continuous distribution function $F$ with the symmetry property $F(t) = 1 - F(-t)$ for all $t \in \R$. In particular, this includes the case where $\eta$ is increasing and we have homoscedastic Gaussian errors with unknown variance. Let $Y$ be as in \eqref{eq:inverse_propensity_weighting} and observe that whenever $\pi(x) = 1/2$,
\begin{align*}
    \mathbb{P}(Y \geq 0\mid X =x) &= \frac{1}{2} \mathbb{P}\bigl(Y \geq 0 \mid T = 1, X = x\bigr) + \frac{1}{2}\mathbb{P}\bigl(Y \geq 0 \mid T = 0, X = x\bigr) \\
    &= \frac{1}{2} \mathbb{P}\bigl(\Tilde{Y} \geq 0 \mid T = 1, X = x\bigr) + \frac{1}{2} \mathbb{P}\bigl(\Tilde{Y} \leq 0 \mid T = 0, X = x\bigr) \\
    &= \frac{1}{2}\bigl\{1 -  F\bigl(-\Tilde{\eta}^1(x)\bigr)\bigr\} + \frac{1}{2} F\bigl(-\Tilde{\eta}^0(x)\bigr) \\
    &= \frac{1}{2}F\bigl(\Tilde{\eta}^1(x)\bigr) + \frac{1}{2}F\bigl(-\Tilde{\eta}^0(x)\bigr).
\end{align*}
Writing $Y^* := \one_{\{Y \geq 0\}}$ and $\eta^*(x) := \mathbb{E}(Y^* | X = x) = \mathbb{P}(Y \geq 0| X =x)$ for $x \in \mathbb{R}^d$, we have $\eta^*(x) \geq 1/2$ if and only if $\eta(x) = \Tilde{\eta}^1(x) - \Tilde{\eta}^0(x) \geq 0$, so $\superLevelSet{1/2}{\eta^*} = \superLevelSet{0}{\eta}$. Moreover, when $\eta$ is increasing on $\R^d$, the distribution of $(X, Y^*)$ belongs to $\mathcal{P}_{\mathrm{BddUpp}, d}(1/2)$. Hence, in order to estimate $\superLevelSet{0}{\eta}$, we may use $\hat{A}^{\mathrm{ISS}}$ with either $\hat{p}_{1/2, 1/2}$ or $\check{p}_{1/2}$ and retain Type~I error control (see the discussions at the end of Section~\ref{sec:theory_validity} and~\ref{sec:extensions_classification} respectively). An application of this procedure is presented in Section~\ref{sec:application_hte}.

\section{Simulations}\label{sec:simulations}


The aim of this section is to explore the empirical performance of $\hat{A}^{\mathrm{ISS}}$ in a wide range of settings.  Throughout, we took independent pairs $(X_1,Y_1),\ldots,(X_n,Y_n)$, where $X_1,\ldots, X_n \sim \mathrm{Unif}\bigl([0,1]^d\bigr) =: \mu$ with $d \in \{2,3,4\}$ and $n \in \{500,1000,2000,5000\}$.  Rescaled versions of the six functions $f$ in Table~\ref{table:sim_regression_functions} serve as our main regression functions; see Appendix~\ref{Sec:FurtherSimulations} for eight further examples.  More specifically, we let $\regressionFunction(x):= \{f(x) - f(0)\}/\{f(\bm 1_d) - f(0)\}$ for each choice of~$f$, as illustrated in Figure~\ref{fig:sim_regression_functions_2D}.  Further, we let $Y_i | X_i \sim N\bigl(\eta(X_i) , \sigma^2\bigr)$ for $i \in [n]$ with $\sigma = 1/4$ when $d=2$, $\sigma = 1/16$ when $d=3$ and $\sigma=1/64$ when $d=4$.  We set $\alpha = 0.05$ and the thresholds $\tau \equiv \tau(\regressionFunction)$ were chosen such that $\marginalDistribution\bigl(\superLevelSet{\tau}{\regressionFunction}\bigr) = 1/2$; see Table~\ref{table:sim_regression_functions}.  Finally, in this table we also provide 
\[
\etaIncreasingExponent(P) := \inf\biggl\{\etaIncreasingExponent > 0: P \in \bigcup_{\etaIncreasingConstant > 0} \distributionClassMultivariateCondition\biggr\}
\]
for the distribution $P$ associated with each choice of $\mu$ and $\eta$. Since for $\etaIncreasingExponent, \etaIncreasingExponent' > 0$ and $\etaIncreasingConstant, \etaIncreasingConstant' > 0$ such that $\etaIncreasingExponent \leq \etaIncreasingExponent'$ and $\etaIncreasingConstant \geq \etaIncreasingConstant'$, we have $\distributionClassMultivariateCondition \subseteq \mathcal{P}_{\mathrm{Reg}, d}(\tau, \densityConstant, \etaIncreasingExponent', \etaIncreasingConstant')$, this is a natural choice to illustrate the effect of the exponent in Definition~\ref{def:multivariateAssumption}\emph{(ii)} on the rate of convergence.

\begin{table}[htbp]
    \centering
    \begin{tabular}{c|c| c| c}
        Label & Function $f$ & $\tau$ & $\gamma(P)$\\ \hline
        
        (a) & $\sum_{j=1}^d x^{(j)}$ & $1/2$ & $1$\\
        (b) & $\max_{1\leq j\leq d} x^{(j)}$ & $1/2^{1/d}$ & $1$\\
        (c) & $\min_{1\leq j\leq d} x^{(j)}$  & $1 - 1/2^{1/d}$ & $1$\\
        (d) & $\one_{(0.5, 1]}\bigl(x^{(1)}\bigr)$ & $1/2$ & $0$\\
        (e) & $\sum_{j=1}^d \bigl(x^{(j)} - 0.5\bigr)^3$  & $1/2$ & $3$ \\
        (f) & $x^{(1)}$ & $1/2$ & $1$\\
    \end{tabular}
    \caption{Definition of the functions used in the simulations. Here, $x = (x^{(1)}, \ldots, x^{(d)})^\top\in [0,1]^d$.}
    \label{table:sim_regression_functions}
\end{table}

\begin{figure}[hbtp]
    \centering
    \includegraphics[width = 0.85\textwidth]{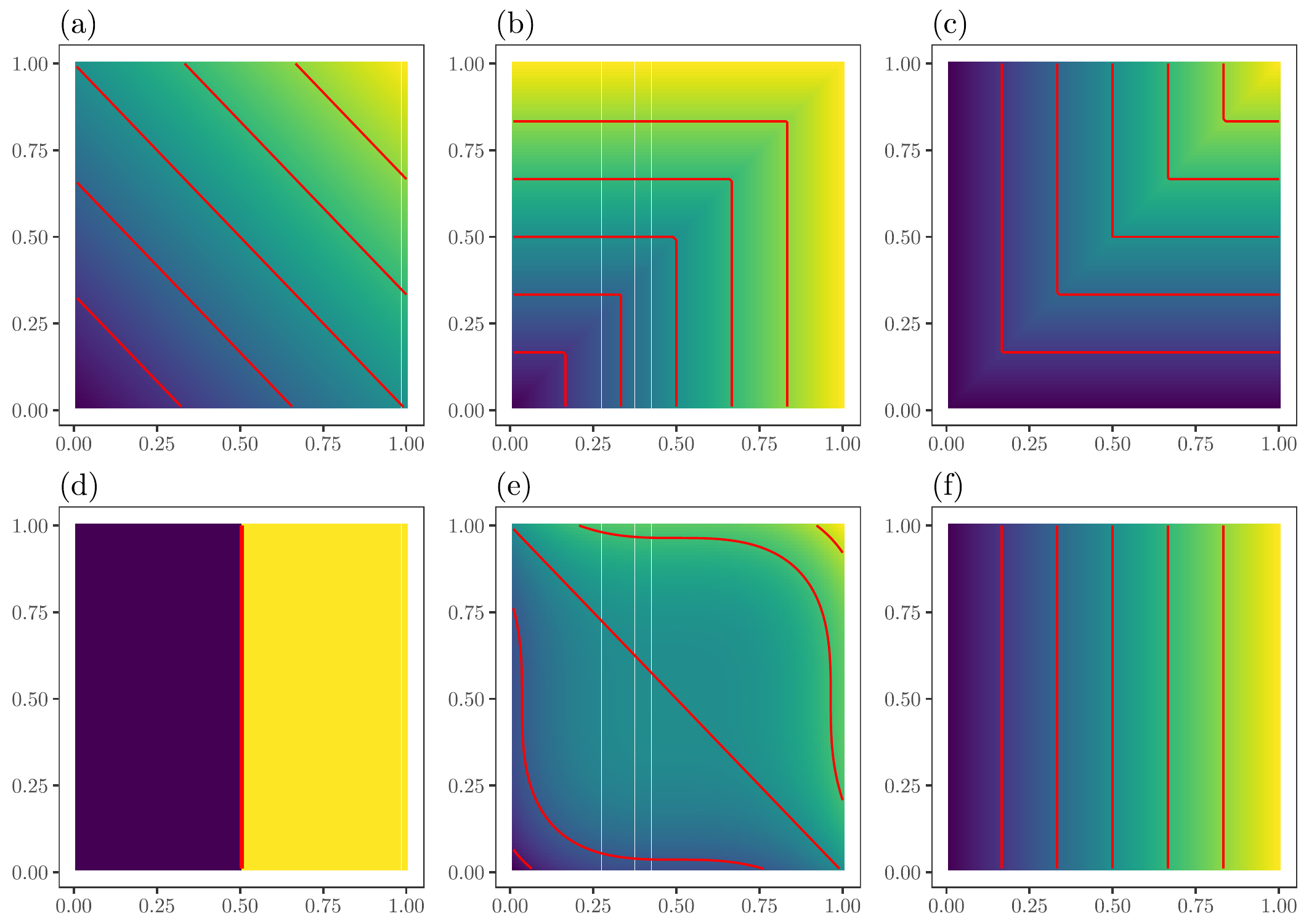}
    \caption{For $d= 2$, the contour lines (red) of the regression functions corresponding to the functions $f$ in Table~\ref{table:sim_regression_functions} at the levels $k/6$ for $k\in [5]$ are shown. The fill colour indicates the function value at the respective position from $0$ (purple) to $1$ (yellow).}
    \label{fig:sim_regression_functions_2D}
\end{figure}

Although we are not aware of other proposed methods for isotonic subgroup selection, there are alternative ways in which we could combine our $p$-values with different DAG testing procedures to form a data-dependent selection set.  For instance, one could apply Holm's procedure \citep{holm1979simple} to ensure FWER control on the data points, and then take our data-dependent selection set to be the upper hull of the set of points in $\sample_{X,m}$ corresponding to rejected hypotheses.  This simple procedure is already a uniform improvement (in terms of the size of the selected set) on the Bonferroni approach to constructing a one-sided confidence band for $\regressionFunction$ that was mentioned in the introduction.  Alternatively, one could combine our $p$-values with either the all-parent or any-parent version of the DAG testing procedure due to \citet{meijer2015multiple}, which is described in detail in Section~\ref{sec:MG_procedure}, and which we applied with uniform weights on the leaf nodes.  Although we are able to prove in Section~\ref{sec:negativeMG_theory} that these latter procedures have sub-optimal worst-case performance, they remain a natural approach to controlling the FWER for DAG-structured hypotheses.  We refer to these three alternative versions of our procedure as $\hat{A}^{\mathrm{ISS},\mathrm{H}}$, $\hat{A}^{\mathrm{ISS},\mathrm{All}}$ and $\hat{A}^{\mathrm{ISS},\mathrm{Any}}$ respectively.  In all cases, we took $m = n$, used $\Tilde{p}^{1/2}$ given in Definition~\ref{defn:pValue_normalmixture} as $p$-values and for each data-dependent selection set $\hat{A}$ we estimated $\mathbb{E}\bigl\{\mu\bigl(\superLevelSet{\tau}{\eta} \setminus \hat{A}\bigr)\bigr\}$ using a Monte Carlo approximation based on $10^5$ independent draws from~$\mu$ for each data realisation, averaged over 100 repetitions of each experiment.  A comparison of the running times given in Appendix~\ref{Sec:FurtherSimulations_runtime} shows that $\hat{A}^{\mathrm{ISS}}$ can be as much as 10 times faster to compute than $\hat{A}^{\mathrm{ISS},\mathrm{All}}$ and $\hat{A}^{\mathrm{ISS},\mathrm{Any}}$, though it is not as fast as the more naive $\hat{A}^{\mathrm{ISS},\mathrm{H}}$.

The results for regression functions (a)--(f) are presented in Figures~\ref{Fig:atofd2},~\ref{Fig:atofd3} and~\ref{Fig:atofd4} respectively.  Corresponding results for the other eight regression functions defined in Appendix~\ref{Sec:FurtherSimulations}, which are qualitatively similar, are given in Figures~\ref{Fig:gtond2},~\ref{Fig:gtond3} and~\ref{Fig:gtond4}.  Moreover, in Figures~\ref{Fig:gtond2Split} and~\ref{Fig:gtond4Split}, we compare $\hat{A}^{\mathrm{ISS}}$ with two different possible approaches based on sample splitting that are of a similar flavour to the two-stage approaches mentioned in the introduction.  These were omitted from our earlier comparisons for visual clarity, and because their performance turns out not to be competitive.  From all of these figures, we see that $\hat{A}^{\mathrm{ISS}}$ is the most effective of these approaches for combining our $p$-values with a DAG testing procedure.  The differences between $\hat{A}^{\mathrm{ISS}}$ and the other approaches are more marked when $d=2$ than in higher dimensions.  It is also notable that regression functions with smaller values of $\etaIncreasingExponent(P)$ such as (d) yield much smaller estimates of $\mathbb{E}\bigl\{\mu\bigl(\superLevelSet{\tau}{\eta} \setminus \hat{A}\bigr)\bigr\}$ that decay more rapidly with the sample size.  Conversely, for settings with larger values of $\etaIncreasingExponent(P)$, such as (e), the decay of our estimates of $\mathbb{E}\bigl\{\mu\bigl(\superLevelSet{\tau}{\eta} \setminus \hat{A}\bigr)\bigr\}$ is much slower.  These observations are in agreement with our theory in Section~\ref{sec:theory_power}.  Finally, we remark that our procedures appear to adapt well to settings where the regression function depends only on a subset of the $d$ variables, as can be seen for instance by comparing the results in~(a) and~(f). 

\begin{figure}
    \centering
    \includegraphics[width = 0.98\textwidth]{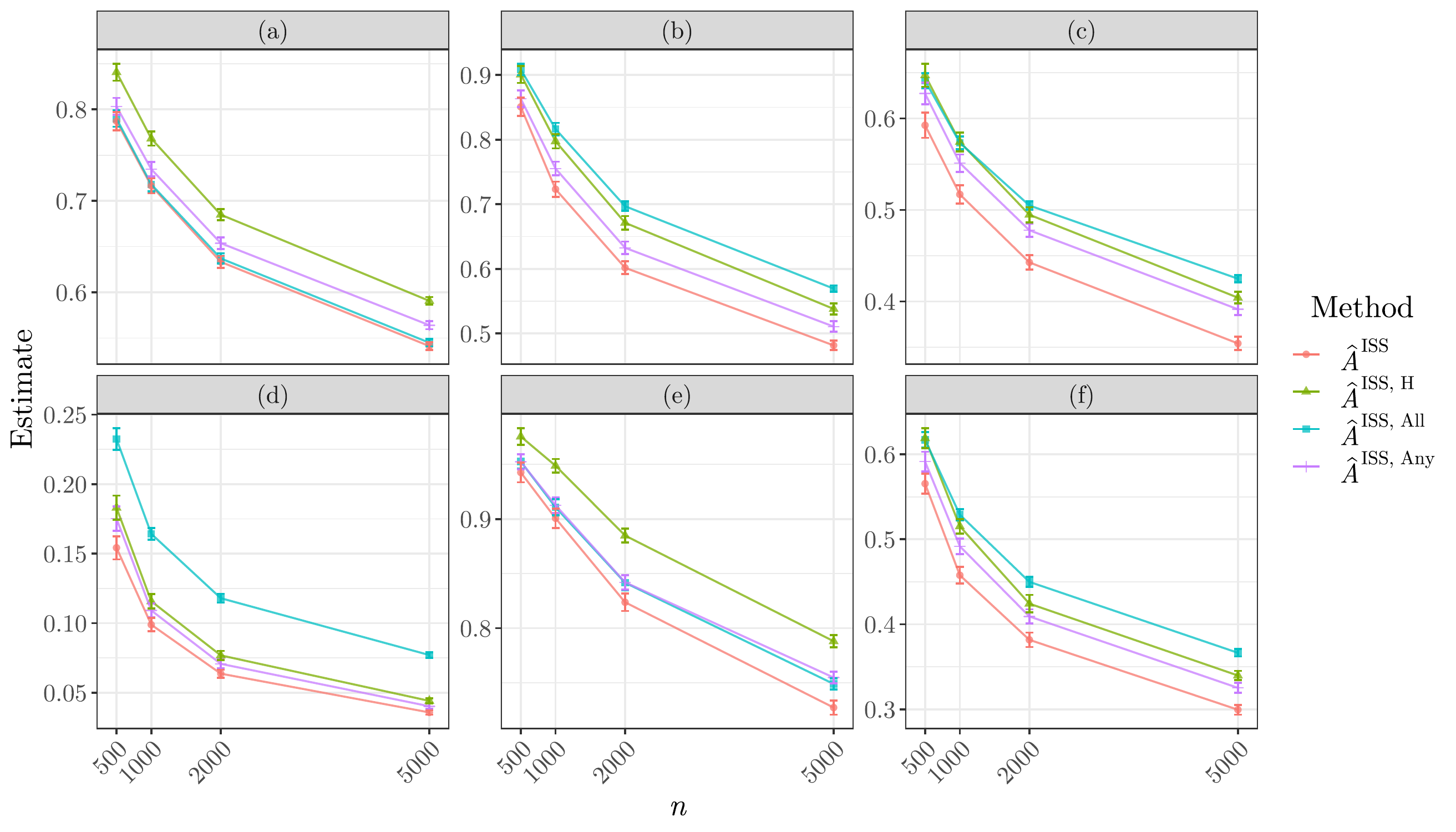}
    \caption{Estimates of $\mathbb{E}\bigl\{\mu\bigl(\superLevelSet{\tau}{\eta} \setminus \hat{A}\bigr)\bigr\}$ for $d = 2$ and $\sigma = 1/4$.}
    \label{Fig:atofd2}
\end{figure}

\begin{figure}
    \centering
    \includegraphics[width = 0.98\textwidth]{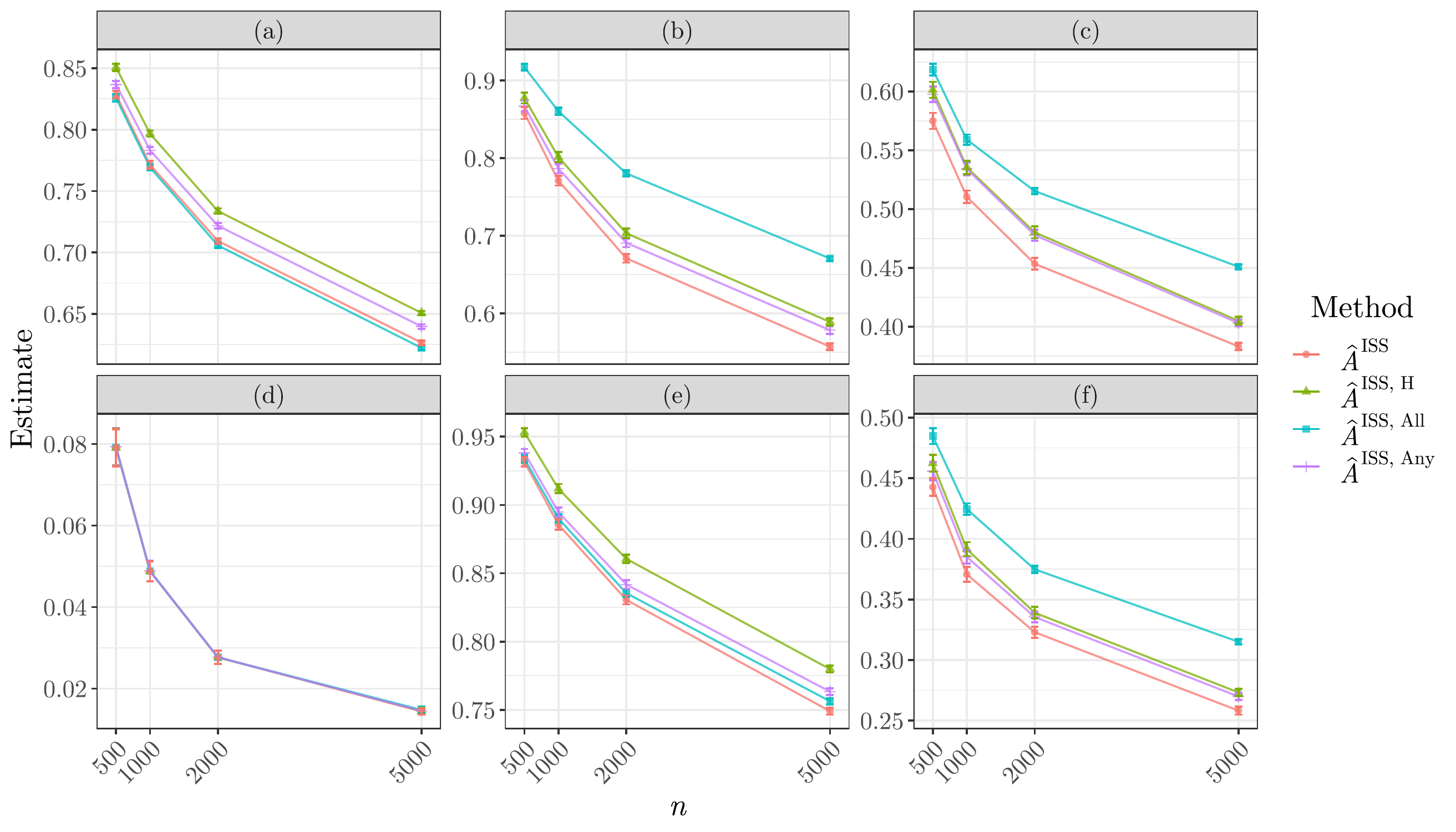}
    \caption{Estimates of $\mathbb{E}\bigl\{\mu\bigl(\superLevelSet{\tau}{\eta} \setminus \hat{A}\bigr)\bigr\}$ for $d = 3$ and $\sigma = 1/16$.}
    \label{Fig:atofd3}
\end{figure}

\begin{figure}
    \centering
    \includegraphics[width = 0.98\textwidth]{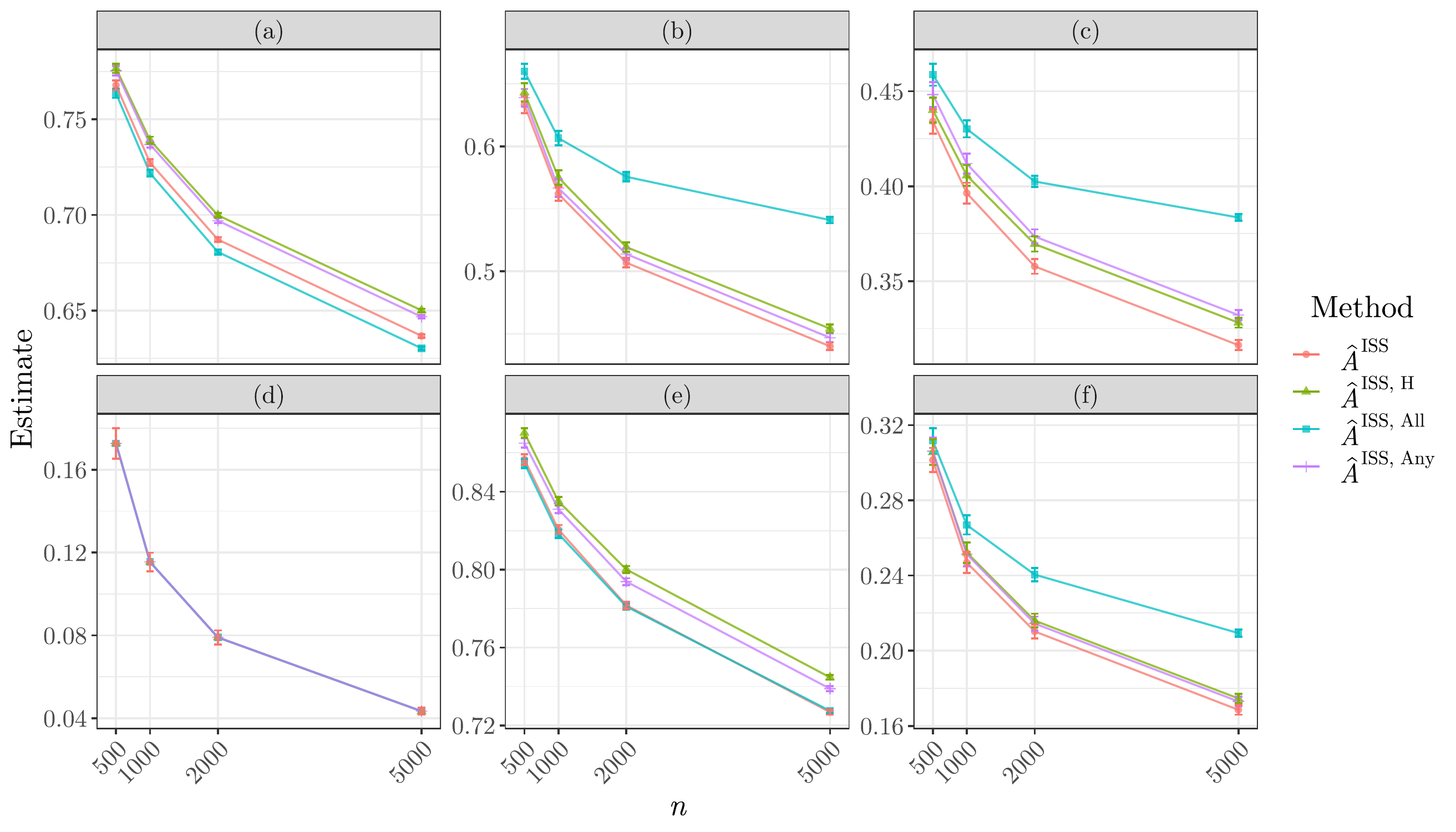}
    \caption{Estimates of $\mathbb{E}\bigl\{\mu\bigl(\superLevelSet{\tau}{\eta} \setminus \hat{A}\bigr)\bigr\}$ for $d = 4$ and $\sigma = 1/64$.}
    \label{Fig:atofd4}
\end{figure}




\section{Real data applications}\label{sec:application}

\subsection{ACTG 175}

As a first illustration of our procedure on real data, we consider the AIDS Clinical Trials Group Study 175 (ACTG 175) data.  This was a randomised controlled trial in which HIV-1 patients whose CD4 cell counts at screening were between 200 to 500 cells per cubic millimetre and who had no history of AIDS-defining events were randomly assigned to one of four treatment groups \citep{hammer1996trial}.  We restrict our attention to two of these four treatment arms, comparing the effects of monotherapy through the antiretroviral medication zidovudine against the effects of a combination therapy of zidovudine together with zalcitabine.  At the time, patient heterogeneity with respect to the response to these treatments was not well understood \citep{burger1994pharmacokinetic}. Moreover, prior studies suggested that the beneficial effects of zidovudine fade with time and that this could be remedied with multitherapy \citep{hammer1996trial}.  ACTG 175 aimed to investigate treatment effect heterogeneity, in particular with respect to prior drug exposure, among patients with less advanced HIV disease. Besides relevant parts of their medical records, covariates including age, weight and ethnicity were recorded.  The primary end point of the study was defined as a reduction of the CD4 cell count by at least 50\%, development of AIDS, or death, with a median follow-up duration of 143 weeks \citep{hammer1996trial}.  The data for 2139 patients are freely available in the \texttt{R} package  \texttt{speff2trial} \citep{speff2trial}.  

\subsubsection{Risk group estimation}

We first consider the task of identifying the patient subgroup whose probability of not reaching the primary endpoint when receiving zidovudine alone (532 patients in total) is at least $\tau = 0.5$ based on their age, which the study's eligibility criteria required to be at least $12$ years.  To that end, for $i\in [n]$, let $Y_i = 1$ if the $i$th patient did not reach the primary endpoint and $Y_i = 0$ otherwise. Furthermore, let $X_i$ denote the $i$th patient's age (multiplied by $-1$), since a decrease in age is expected to correspond to an increased probability of avoiding the primary end point across the eligible age range.  Under the assumption that $\sample = \bigl((X_1, Y_1), \ldots, (X_n, Y_n)\bigr) \sim P^n$, we then have that $P \in \mathcal{P}_{\mathrm{Mon}, 1}(1/2)$, so Type~I error control for our procedure $\hat{A}^{\mathrm{ISS}}$ is guaranteed by Theorem~\ref{thm:validSelectionSet}.  The left panel of Figure~\ref{Fig:RiskGroup} illustrates the data-dependent selection set that we output with $\alpha = 0.05$, indicating that not reaching the primary endpoint is the more likely outcome for patients aged 39 and under. 

For a bivariate illustration, we use age multiplied by $-1$ and CD4 cell count at the trial onset as covariates.  A high initial CD4 cell count is expected to be associated with a lower risk of reaching the primary endpoint.  Thus, we assume that $\sample \sim P^n$ for some $P \in \mathcal{P}_{\mathrm{Mon}, 2}(1/2)$, and the right panel of Figure~\ref{Fig:RiskGroup} illustrates the output $\hat{A}^{\mathrm{ISS}}$ for $\tau = 0.5$ and $\alpha = 0.05$.  The fact that the left-hand extreme of this selected set is slightly below 39 years is a reflection of the stronger form of Type~I error control sought in the larger dimension.

\begin{figure}
    \begin{subfigure}[b]{0.45\textwidth}
        \centering
        \includegraphics[width = 0.9\textwidth]{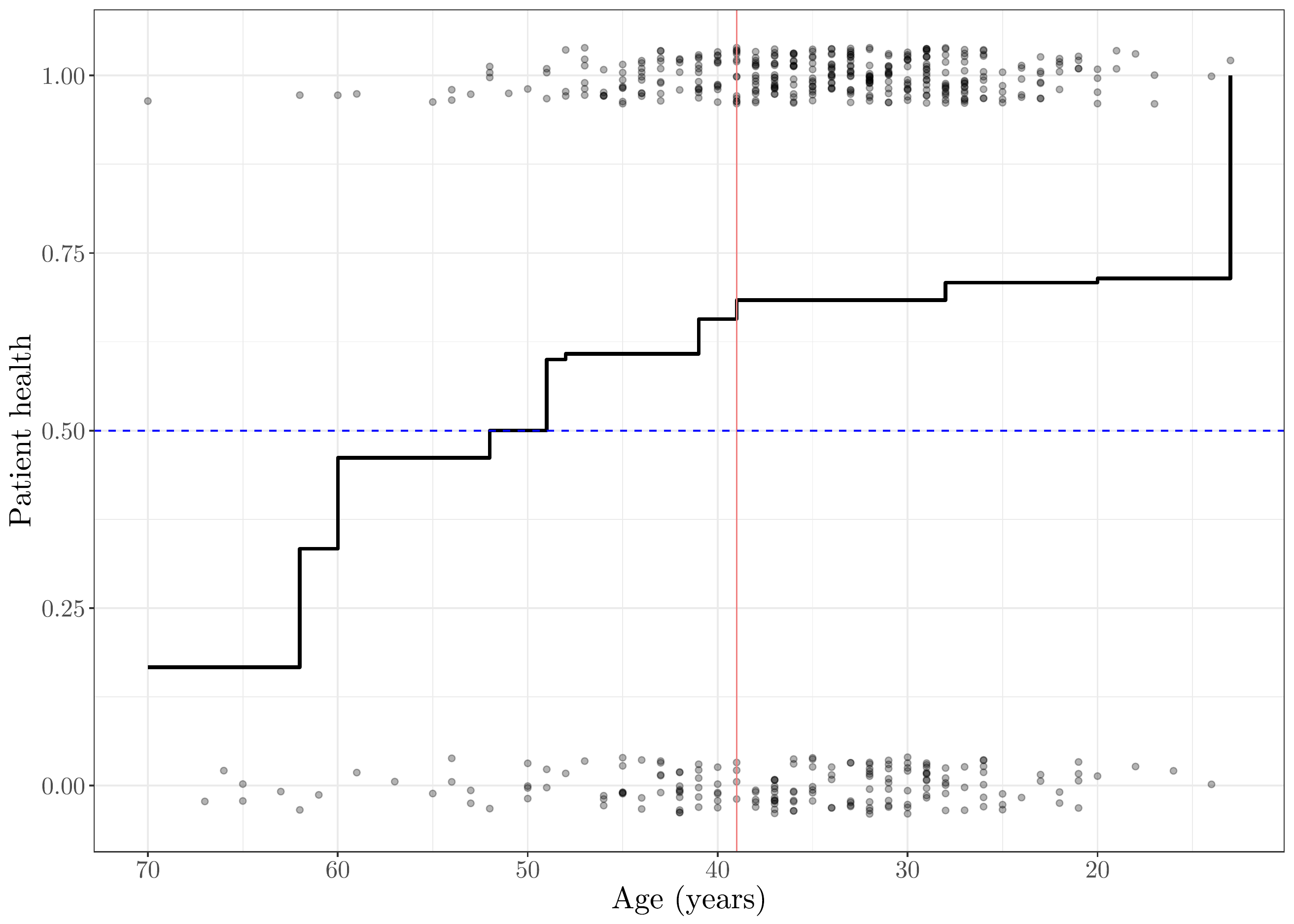}
        \label{fig:univariate_riskgroup}
     \end{subfigure}
     \hspace{0.5cm}
     \begin{subfigure}[b]{0.45\textwidth}
        \centering
        \includegraphics[width = 0.9\textwidth]{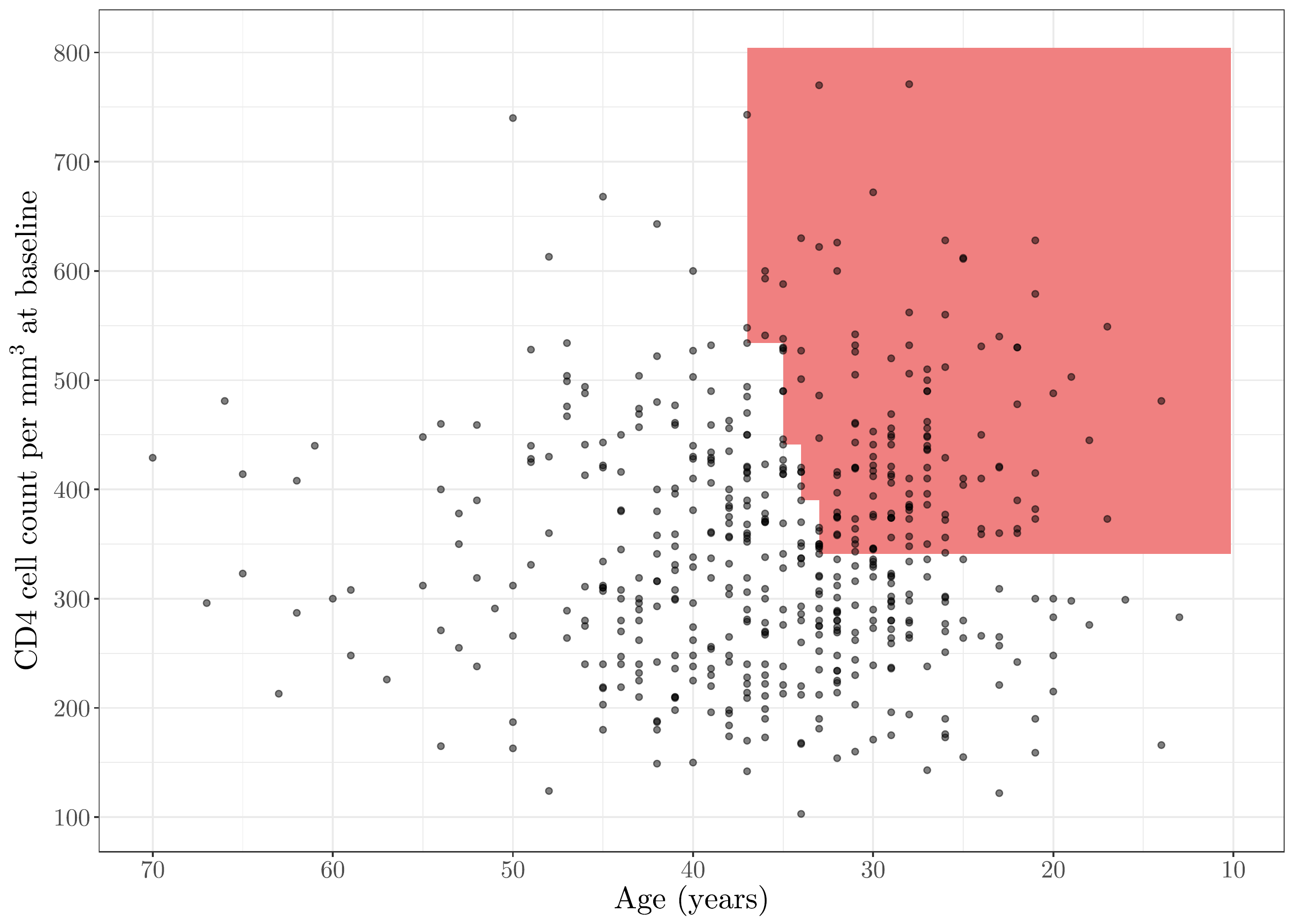}
        \label{fig:bivariate_riskgroup}
     \end{subfigure}
     \caption{\label{Fig:RiskGroup}Left: Each black dot represents one patient in $\sample$ (with the response jittered vertically for clarity) and the red line gives the lower bound of $\hat{A}^{\mathrm{ISS}}_{\sigma, \tau, \alpha, m}(\sample)$ when applied with the $p$-values $\Tilde{p}^{1/2}$, for $\sigma = 1/2$, $\tau = 1/2$, $\alpha = 0.05$, $m = n = 532$.  The black line shows the isotonic least squares regression estimator and the blue dashed line indicates the level $\tau$.  Right: corresponding bivariate illustration, where the red region indicates $\hat{A}^{\mathrm{ISS}}$.}
\end{figure}

\subsubsection{Heterogeneous treatment effects}\label{sec:application_hte}

To illustrate an application of the methodology of Section~\ref{sec:extensions_noninferiority}, we take the change in CD4 cell count from trial onset to week 20 ($\pm$5 weeks) as the measured response $\Tilde{Y}_i$ for the $i$th patient.  We further set $T_i = 0$ if the $i$th patient was in the control group receiving monotherapy with zidovudine (532 patients) and $T_i = 1$ if they were assigned to receive multitherapy with zidovudine and zalcitabine (524 patients).  We are interested in identifying the subgroup for which multitherapy is at least as good as monotherapy, in the sense that the CD4 cell count is decreased by less, based on the patient's age (again multiplied by $-1$), denoted by $X_i$ for the $i$th patient. This means that, conditional on treatment $T_i$ and age $X_i$, $\Tilde{\eta}^{T_i}(X_i)$ is the expected change in CD4 cell count in the first 20 weeks, and we assume that the observed response~$\Tilde{Y}_i$ is conditionally symmetrically distributed around $\Tilde{\eta}^{T_i}(X_i)$.  Thus $\eta(x) := \Tilde{\eta}^1(x) - \tilde{\eta}^0(x)$ gives the heterogeneous treatment effect on the change in CD4 cell count for patients of~$x$ years of age, and we are interested in identifying $\superLevelSet{0}{\eta}$ under the assumption that this is an upper set. Here, $\pi(x) = 1/2$ for all $x$, so that from~\eqref{eq:inverse_propensity_weighting}, $Y_i = (4T_i - 2)\cdot \Tilde{Y}_i$ for all $i$. Defining now $Y^*_i := \one_{\{Y_i \geq 0\}}$ and assuming $\sample^0 = \bigl((X_1, Y^*_1), \ldots, (X_n, Y^*_n)\bigr)\sim P^n$, we have $P \in \mathcal{P}_{\mathrm{Upp}, 1}(1/2, 1/2)$, so that $\hat{A}^{\mathrm{ISS}}$ when applied with $\check{p}_{1/2}$ controls the Type I error by the discussion in Section~\ref{sec:extensions_noninferiority}. See Figure~\ref{fig:actg175-hte} for a visualisation of the result. We conclude that among patients aged 25 or younger, replacing monotherapy by multitherapy is uniformly associated with a neutral or beneficial effect on the stability of CD4 cell count.

\begin{figure}
     \centering
        \includegraphics[width = 0.65\textwidth]{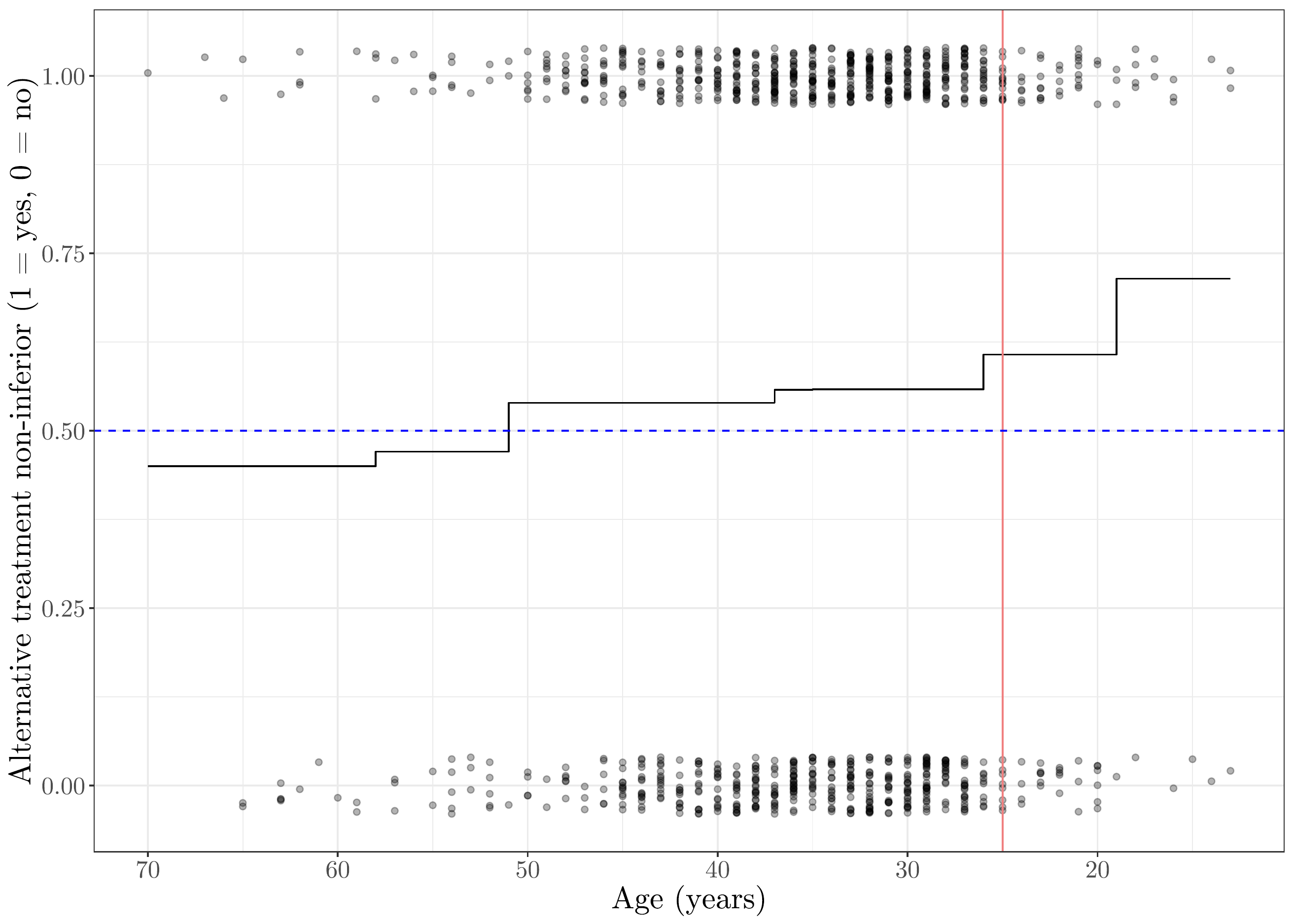}
     \caption{Each black dot represents one patient in $\sample$ (with $Y_i^*$ jittered vertically for clarity) and the red line gives the lower bound of $\hat{A}^{\mathrm{ISS}}_{\sigma, \tau, \alpha, m}(\sample)$ when applied with  $\Check{p}_\tau$, $\tau = 1/2$, $\alpha = 0.05$, $m = n = 1056$. Further, the black line shows the isotonic least squares regression estimator. The blue dashed line indicates the level $\tau$.}
     \label{fig:actg175-hte}
\end{figure}

\subsection{Fuel consumption dataset} 

Here we consider the \texttt{Auto MPG} dataset\footnote{See \href{https://archive.ics.uci.edu/ml/datasets/auto+mpg}{https://archive.ics.uci.edu/ml/datasets/auto+mpg}.} that was popularised by \cite{quinlan1993combining} and that is available through the UCI Machine Learning Repository \citep{dua2019UCI}. This dataset contains information on $n = 398$ cars, including their urban fuel consumption, weight and engine displacement.  We would like to identify the combinations of car weight and engine displacement for which the probability of fuel efficiency being at least 15mpg is at least $\tau = 0.5$.  To this end, we set $Y_i = 1$ if the $i$th car's fuel efficiency is at least 15mpg and $Y_i = 0$ otherwise.  Since increases in weight and engine displacement can be assumed to decrease the conditional probability of high fuel efficiency, we let the two components of $X_i \in (-\infty, 0)^2$ give the $i$th car's weight and engine displacement (multiplied by $-1$). We then have that $\sample = \bigl((X_1, Y_1), \ldots (X_n, Y_n)\bigr) \sim P^n$ with $P\in \mathcal{P}_{\mathrm{Mon, 2}}(1/2)$, so Type~I error control for our procedure $\hat{A}^{\mathrm{ISS}}$ is guaranteed by Theorem~\ref{thm:validSelectionSet}.  The output set from the $\hat{A}^{\mathrm{ISS}}$ algorithm is shown in Figure~\ref{fig:autompg}.  The strong sample correlation of 0.93 between weight and engine displacement contributes to the data-dependent selection set being almost rectangular, with a weight of under 3400lbs and an engine displacement of under 250 cubic inches being sufficient to be fairly confident that the fuel consumption is more likely than not to be at least 15mpg. 

\begin{figure}
     \centering
        \includegraphics[width = 0.65\textwidth]{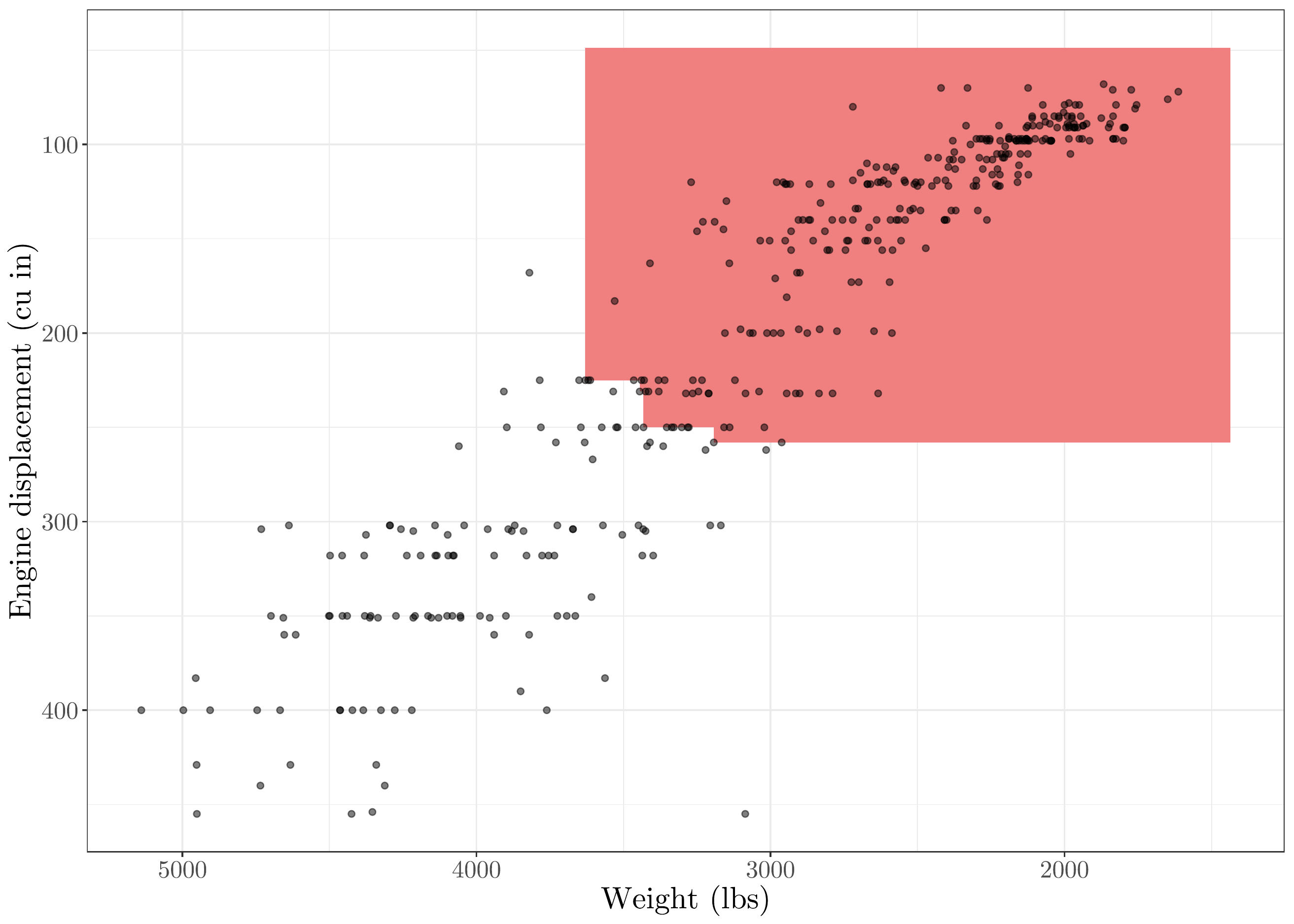}
     \caption{Each black dot represents one car in $\sample_X$ and the red area gives the output of $\hat{A}^{\mathrm{ISS}}_{\sigma, \tau, \alpha, m}(\sample)$ applied with $\tilde{p}_{\sigma,\tau}^{1/2}$ for $\sigma = 1/2$, $\tau = 1/2$, $\alpha = 0.05$, $m = n = 398$.}
     \label{fig:autompg}
\end{figure}

\appendix

\section*{Appendix}

Section~\ref{Sec:Proofs} of this appendix consists of proofs of all of our main results, as well as statements and proofs of intermediate results.  Section~\ref{Sec:FurtherSimulations} presents further simulations, while Section~\ref{sec:MG_appendix} contains a discussion of an alternative and general approach to combining the $p$-values due to \citet{meijer2015multiple}.  Finally, in Section~\ref{Sec:Auxiliary}, we give a few auxiliary results.

\section{Proofs}
\label{Sec:Proofs}

We begin with some additional notation used in the appendix.  For a set $A \subseteq \R^d$, let $\powerSet(A)$ denote the power set of $A$.  Denote by $\|\cdot\|_2$ the Euclidean norm on $\R^d$ and for $x \in \R^d$ and $r >0$, define the closed Euclidean norm ball by $\closedEuclideanMetricBall{x}{r}:=\{z \in \R^d : \|z-x\|_2\leq r\}$.  Given $A_0$, $A_1 \subseteq \R^d$, we write $A_0 \preccurlyeq A_1$ if $x_0 \preccurlyeq x_1$ for every pair $(x_0,x_1) \in A_0 \times A_1$.  We write $\Lebesgue$ for Lebesgue measure on $\R^d$. Finally, for $r\geq 0$ and $x\in \R^d$, we denote $\mathcal{I}_{r}(x) \equiv \mathcal{I}_{r}(x, \sampleX) := \{i \in [n]: X_i \preccurlyeq x, \|X_i - x\|_\infty \leq r\}$.

\subsection{Proofs from  Section~\ref{sec:theory_validity}}

In order to verify Lemma~\ref{lemma:validLocalPValStouffer} it will be convenient to prove the following small generalisation.
\begin{lemma}\label{lemma:validLocalPValStoufferGeneralised} Let $\sigma >0$, $\tau \in \R$ and let $P$ be a distribution on $\R^d \times \R$ with regression function $\eta$ such that if $(X,Y) \sim P$, then $Y-\eta(X)$ is conditionally sub-Gaussian with variance parameter $\sigma^2$ given~$X$.  Fix $x \in \R^d$ and suppose that $\eta(x') \leq \tau$ for all $x' \preccurlyeq x$.  Given $\sample = \bigl((X_1,Y_1),\ldots,(X_n,Y_n)\bigr) \sim P^n$, we have $\Prob_P\bigl\{\hat{p}_{\sigma,\tau}(x,\sample) \leq \alpha | \sample_X\bigr\} \leq \alpha$ for all $\alpha \in (0,1)$.
\end{lemma}

\begin{proof}[Proof of Lemma~\ref{lemma:validLocalPValStoufferGeneralised} (and hence Lemma~\ref{lemma:validLocalPValStouffer})] 
Throughout the proof, we operate conditional on $\sample_X$ and consider the setting of Definition~\ref{defn:pValue}. If $\mathcal{I}(x) = \emptyset$, then $\hat{p}_{\sigma, \tau}(x, \sample) = 1$ and the result follows, so suppose henceforth that $n(x)\geq 1$. Define the $\sigma$-algebra $\mathcal{F}_0$ generated by $\{(X_i)_{i \in \mathcal{I}(x)}\}$ and the $\sigma$-algebras $\mathcal{F}_k$ generated by $\bigl\{\bigl(Y_{(j)}(x)\bigr)_{j \in [k]} \cup \mathcal{F}_0\bigr\}$ for $k\in [n(x)]$. Similarly to the proof of \citet[Theorem~3]{duan2020interactive}, we first show that $(S_k)_{k \in \{0\}\cup [n(x)]}$,where $S_0 := 0$, is a supermartingale with respect to the filtration $(\mathcal{F}_k)_{k \in \{0\} \cup [n(x)]}$. Since for $k \in [n(x)]$, $S_{k-1}$ is measurable with respect to $\mathcal F_{k-1}$ and the ordering $Y_{(1)}, \ldots, Y_{(n(x))}$ is fixed conditional on $\sample_X$, we have for any $k\in[n(x)]$ that
\begin{align}
    \E(S_k \mid \mathcal F_{k-1}) &= S_{k-1} + \frac{1}{\sigma}\bigl\{\E\bigl(Y_{(k)}(x) \mid \mathcal F_{k-1}\bigr) - \tau\bigr\} \nonumber\\
    &= S_{k-1} + \frac{1}{\sigma}\bigl\{\eta\bigl(X_{(k)}(x)\bigr) - \tau\bigr\} \leq S_{k-1}, \nonumber
\end{align}
where in the last step we used the fact that  $\eta(X_{i}) \leq \eta(x) \leq \tau$ for $i \in \mathcal{I}(x)$.  Since the integrability of $S_k$ follows from the sub-Gaussianity of the increments, the sequence $(S_k)_{k \in \{0\}\cup [n(x)]}$ is a supermartingale.  Moreover, its increments satisfy $Z'_k := S_k - S_{k-1} = \bigl(Y_{(k)}(x) - \tau\bigr)/\sigma  \leq \bigl\{Y_{(k)}(x) - \eta\bigl(X_{(k)}(x)\bigr)\bigr\}/\sigma =: Z_k$ for $k\in[n(x)]$ and the random variables $(Z'_k)_{k\in [n(x)]}$ are independent conditional on $\sampleX$, as are $(Z_k)_{k \in [n(x)]}$.  Thus, we have by Lemma~\ref{lemma:howard_uniform_bound}\emph{(a)} and with $u_\alpha(\cdot)$ as defined there that
\begin{align*}
\Prob\biggl(\bigcup_{k=1}^{n(x)} \bigl\{S_k \geq u_\alpha(k)\} \, \biggm|\, \sample_X\biggr) &= \Prob\biggl(\bigcup_{k=1}^{n(x)} \biggl\{\sum_{j=1}^k Z'_j \geq u_\alpha(k)\biggr\} \,\biggm|\, \sample_X\biggr) \nonumber\\
&\leq \Prob\biggl(\bigcup_{k=1}^{n(x)} \biggl\{\sum_{j=1}^k Z_j \geq u_\alpha(k)\biggr\} \,\biggm|\, \sample_X\biggr) \leq \alpha.
\end{align*}
Hence, for $\alpha \in (0,1)$,
\begin{align*}
\Prob\bigl(\hat{p}_{\sigma, \tau}(x,\sample) \leq \alpha \mid \sample_X\bigr) &= \Prob\bigl(\hat{p}_{\sigma, \tau}(x,\sample) \leq \alpha,n(x) > 0 \mid \sample_X\bigr) \\
&= \Prob\biggl(\max_{k \in [n(x)]} \frac{S_k}{u_\alpha(k)} \geq 1,n(x) > 0 \Bigm| \sample_X \biggr) \leq \alpha,
\end{align*}
as required. 
\end{proof}

\begin{proof}[Proof of Lemma~\ref{lemma:sparseMG_valid}]
Fix a finite set $I$, a family of distributions $\mathcal{Q}$ on $(0,1]^I$, a collection of random variables $\bm{p} = (p_i)_{i \in I}$ taking values in $(0,1]^I$, as well as hypotheses $H_i \subseteq \{Q\in \mathcal{Q}: \Prob_Q(p_i \leq t) \leq t, \forall t \in (0, 1]\}$ for $i\in I$ and any $G_0$-consistent polyforest-weighted DAG $G' = (I,E,\bm{w})$.  Throughout this proof, define $G := (I, E)$ and $F := (I, \{e \in E: w_e = 1\})$ as in Algorithm~\ref{algo:R_ISS}, and write $J^c := I \setminus J$ for any $J \subseteq I$.  Fix $Q_0 \in \mathcal{Q}$, so that $I_0 \equiv I_0(Q_0) \subseteq I$ is a $G$-lower set giving the indices of true null hypotheses. If $I_0 = \emptyset$, then no Type I error can be made and the proof is complete.  We therefore suppose henceforth that $|I_0| > 0$.  
For a proper subset $J$ of $I$, it is convenient to define
\begin{align}
\label{Eq:alpha}
    \alpha(i, J) \equiv \alpha(i, J, F) := \begin{cases} 
    \frac{|(\{i\}\cup \de_F(i)) \cap L(F) \cap J^c|}{|L(F) \cap J^c|}\cdot \alpha\quad &\text{if $i \notin J$, $\pa_F(i) \subseteq J$}\\
    0 &\text{otherwise.}\end{cases}
\end{align}
Thus, given a set of rejected hypotheses $J$ at a particular iteration of Algorithm~\ref{algo:R_ISS} and writing $\mathcal{N}_0(J) := \{i \in J^c: p_i \leq \alpha(i, J)\}$, the hypotheses in $\mathcal{N}(J) := \mathcal{N}_0(J) \cup \bigcup_{j \in \mathcal{N}_0(J)} \an_G(j)$ will be rejected at the next iteration.  Hence, the set of rejected hypotheses at the $\ell$-th iteration in Algorithm~\ref{algo:R_ISS} can be written as $R_\ell = R_{\ell - 1} \cup \mathcal{N}(R_{\ell - 1})$ with $R_0 = \emptyset$.  We first claim that 
\begin{align}
    \mathcal{N}(I_1) \subseteq \mathcal{N}(I_2) \cup I_2 \label{eq:srp_monotonicity}
\end{align}
for all $G$-upper proper subsets $I_1 \subseteq I_2$ of $I$.  To see this, fix such $I_1, I_2$ and any $i \in \mathcal{N}(I_1)$. The result is immediate if $i \in I_2$, so suppose that $i \in I_2^c$.  If $i \in \mathcal{N}_0(I_1)$, then $\alpha(i, I_1) > 0$, so $\pa_F(i) \subseteq I_1$. Since $I_1$ and $I_2$ are $G$-upper, they are also $F$-upper. Hence $\{i\} \cup \de_F(i) \subseteq I^c_2 \subseteq I^c_1$, so that
\begin{equation}
\label{Eq:alphainequality}
\alpha(i, I_1) = \frac{\bigl|\bigl(\{i\}\cup \de_F(i)\bigr) \cap L(F)\bigr|}{|L(F) \cap I_1^c|}\cdot \alpha \leq \frac{\bigl|\bigl(\{i\}\cup \de_F(i)\bigr) \cap L(F)\bigr|}{|L(F) \cap I_2^c|}\cdot \alpha = \alpha(i, I_2),
\end{equation}
and we deduce that $i \in \mathcal{N}_0(I_2)\subseteq \mathcal{N}(I_2)$. If instead $i \in \an_G(j_0)$ for some $j_0 \in \mathcal{N}_0(I_1)$, then  since $I_2$ is $G$-upper and $i \in I_2^c$, we have $j_0 \in I_2^c$. Following the same line of reasoning as in~\eqref{Eq:alphainequality}, we see that $0 < \alpha(j_0, I_1) \leq \alpha(j_0, I_2)$, so that $j_0 \in \mathcal{N}_0(I_2)$ and consequently $i \in \an_G(j_0) \subseteq \mathcal{N}(I_2)$.  This establishes the claim in \eqref{eq:srp_monotonicity}. 

Our second claim is that
\begin{align}
\mathbb{P}_{Q_0}\bigl(\mathcal{N}(I^c_0) = \emptyset\bigr) \geq 1 - \alpha. \label{eq:srp_singlestep}    
\end{align}
To see this, first note that since $I_0^c$ is $G$-upper, it is $F$-upper. Moreover, for any $i \in I_0$, we have $\alpha(i, I_0^c) > 0$ only if $\pa_F(i) \subseteq I_0^c$, so the ancestors of any element of $I_* := \{i \in I_0: \alpha(i, I_0^c) > 0\}$ belong to $I_0^c$, and we deduce that $I_*$ is an antichain in $F$.  Combining this with the fact that $F$ is a polyforest in which each node has at most one parent, we see that if $i_1,i_2 \in I_*$ are distinct, then $\{i_1\}\cup\de_F(i_1)$ and $\{i_2\}\cup\de_F(i_2)$ are disjoint.  Hence,
\begin{align*}
    \mathbb{P}_{Q_0}\bigl(\mathcal{N}(I^c_0) \neq \emptyset\bigr) &= \mathbb{P}_{Q_0}\biggl(\bigcup_{i \in I_0} \bigl\{p_i \leq \alpha(i, I_0^c)\bigr\}\biggr) \leq \sum_{i\in I_0} \mathbb{P}_{Q_0}\bigl(p_i \leq \alpha(i, I_0^c)\bigr) \\
    &\leq \sum_{i\in I_0} \alpha(i, I_0^c) = \sum_{i\in I_*} \frac{\bigl|\bigl(\{i\}\cup \de_F(i)\bigr) \cap L(F)\bigr|}{|L(F) \cap I_0|} \cdot \alpha \leq \alpha,
\end{align*}
as required.  Writing $\Omega_0 := \bigl\{\mathcal{N}(I_0^c)\cap I_0 = \emptyset\bigr\}$ and using \eqref{eq:srp_monotonicity}, we see that $\mathcal{N}(R_{\ell-1}) \subseteq \mathcal{N}(I_0^c) \cup I_0^c$, so on $\Omega_0$, we have $\mathcal{N}(R_{\ell-1}) \subseteq I_0^c$.  We deduce that  
\[
\Omega_0 \cap \{R_{\ell - 1} \cap I_0 = \emptyset \} = \Omega_0 \cap \{\mathcal{N}(R_{\ell - 1}) \cap I_0 = \emptyset\} \cap \{R_{\ell - 1} \cap I_0 = \emptyset  \} = 
\Omega_0 \cap \{R_{\ell} \cap I_0 = \emptyset\}. 
\]
Since $R_0 = \emptyset$, we have $\Omega_0 = \Omega_0 \cap \{R_0 \cap I_0 = \emptyset \}$, which yields by induction that $\Omega_0 = \Omega_0 \cap \{R_{|I|} \cap I_0 = \emptyset\} \subseteq \{R_{|I|} \cap I_0 = \emptyset\}$.
Combining this with \eqref{eq:srp_singlestep} we conclude that
\[
\mathbb{P}_{Q_0}\bigl(R_{|I|} \cap I_0 = \emptyset \bigr) \geq \mathbb{P}_{Q_0}(\Omega_0) = \mathbb{P}_{Q_0}\bigl(\mathcal{N}(I_0^c) = \emptyset\bigr) \geq 1 - \alpha,
\]
as required.
\end{proof}

\begin{proof}[Proof of Theorem~\ref{thm:validSelectionSet}]
If $(\mathcal{X},\mathcal{A})$ and $(\mathcal{Y},\mathcal{B})$ are measurable spaces, $f:\mathcal{X} \rightarrow \mathcal{Y}$ is measurable and $\pi$ is a distribution on $\mathcal{X}$, let $f\sharp\pi$ denote the pushforward measure on $\mathcal{Y}$ of $\pi$ under $f$; i.e., if $Z \sim \pi$ then $f(Z) \sim f\sharp\pi$. We condition on $\sample_X$ throughout this proof and denote $\hat{\bm{p}}^*\equiv \hat{\bm{p}}^*(\cdot) = \bigl(\hat{p}_i^*(\cdot)\bigr)_{i\in [m]} := \bigl(\hat{p}_{\sigma, \tau}(X_i, \cdot)\bigr)_{i\in [m]}$. Write $\mathcal{Q} := \{\hat{\bm{p}}^*\sharp \tilde{P}^n: \tilde{P} \in \distributionClassNonDecreasingRegressionFunction\}$ for a family of distributions over $(0, 1]^m$ induced by $\distributionClassNonDecreasingRegressionFunction$.  Further, for $i \in [m]$, let $H^*_i := \{\tilde{P} \in \distributionClassNonDecreasingRegressionFunction: \mathbb{E}_{\tilde{P}}(Y_i|X_i) < \tau\}$ and $H_i := \{\hat{\bm{p}}^*\sharp \tilde{P}^n: \tilde{P} \in H_i^*\} \subseteq \mathcal{Q}$, so that for $Q := \hat{\bm{p}}^*\sharp P^n$ we have $I_0(P) := \{i \in [m]: P \in H^*_i\} \subseteq \{i \in [m]: Q \in H_i\} =: I_0(Q)$. Lemma~\ref{lemma:validLocalPValStouffer} then shows that $H_i \subseteq \{\tilde{Q}\in \mathcal{Q}: \Prob_{\tilde{Q}}(\hat{p}^*_i \leq t | \sample_X) \leq t \ \forall t \in (0, 1]\}$.  Now define $G_0^* := ([m],E_0^*)$, where $E_0^* := \{(i_0,i_1) \in [m]^2: H_{i_0}^* \subseteq H_{i_1}^*\}$ and $G_0 := ([m],E_0)$, where $E_0 := \{(i_0,i_1) \in [m]^2: H_{i_0} \subseteq H_{i_1}\}$.  We claim that $E_0^* \subseteq E_0$.  To see this, fix $(i_0,i_1) \in E_0^*$, so that $H_{i_0}^* \subseteq H_{i_1}^*$, and suppose that $Q_0 \in H_{i_0}$.  Then we can find $P_0 \in H_{i_0}^*$ such that $Q_0 = \hat{\bm{p}}^*\sharp P_0^n$.  But since $P_0 \in H_{i_1}^*$, we must have that $Q_0 \in H_{i_1}$, and this establishes our claim.  By construction, $\mathcal{G}_{\mathrm{W}}(\sample_{X,m})$ is a $G_0^*$-consistent polyforest-weighted DAG, and we can therefore deduce from our claim that it is also a $G_0$-consistent polyforest-weighted DAG.  Hence, by Lemma~\ref{lemma:sparseMG_valid}, 
\begin{align*}
\mathbb{P}_P\bigl\{ \rejectionSet_{\alpha}^{\mathrm{ISS}}\bigl(\mathcal{G}_{\mathrm{W}}(\sample_{X,m}),&\hat{\bm{p}}^*(\sample)\bigr)\cap I_0(P) = \emptyset\bigm|\sampleX\bigr\} \\
&\geq \mathbb{P}_Q\bigl\{ \rejectionSet_{\alpha}^{\mathrm{ISS}}\bigl(\mathcal{G}_{\mathrm{W}}(\sample_{X,m}),\hat{\bm{p}}^*\bigr)\cap I_0(Q) = \emptyset\bigm|\sampleX\bigr\} \geq 1 - \alpha.
\end{align*}
Moreover, $\superLevelSet{\tau}{\regressionFunction}$ is an upper set because $P\in\distributionClassNonDecreasingRegressionFunction$, and we conclude that
    \[
        \Prob_P\bigl( \hat{A}^{\mathrm{ISS}}_{\sigma,\tau,\alpha,m}(\sample) \subseteq \superLevelSet{\tau}{\regressionFunction} \bigm| \sample_X\bigr) \geq \mathbb{P}_P\bigl\{ \rejectionSet_{\alpha}^{\mathrm{ISS}}\bigl(\mathcal{G}_{\mathrm{W}}(\sample_{X,m}),\hat{\bm{p}}^*(\sample)\bigr)\cap I_0(P) = \emptyset\bigm|\sampleX\bigr\} \geq 1-\alpha,
    \]
    as required.    
\end{proof}

\subsection{Proofs from  Section~\ref{sec:theory_power}}\label{sec:appendix_theory_power}

The following proposition shows that if we only know that $P \in \distributionClassNonDecreasingRegressionFunction$, then it is impossible to provide non-trivial uniform power guarantees for data-dependent selection sets with Type I error control.

\begin{prop}\label{prop:GaussianTestingMonoNotEnough} Let $d \in \N$, $\tau \in \R$, $\sigma > 0$ and $\alpha \in (0, 1)$. Then, for any $n \in \N$, 
\[
\sup_{P \in \distributionClassNonDecreasingRegressionFunction[d]}\inf_{\hat{A} \in \hat{\mathcal{A}}_n(\tau,\alpha,\distributionClassNonDecreasingRegressionFunction[d])}  \mathbb{E}_{P}\bigl\{ \marginalDistribution\bigl(\superLevelSet{\tau}{\regressionFunction}\setminus\hat{A}(\sample)\bigr) \bigr\} \geq 1-\alpha.
\]
\end{prop}
\begin{proof}[Proof of Proposition \ref{prop:GaussianTestingMonoNotEnough}] Fix a Borel probability measure $\mu$ on $\R^d$. For $\Delta \in \R$, let $\eta_\Delta:\R^d \rightarrow \R$ denote the constant function satisfying $\eta_\Delta(x) := \tau+\Delta$ for all $x \in \R^d$, and let $P_\Delta$ denote the distribution on $\R^d \times \R$ of $(X,Y)$, where $X \sim \mu$ and $Y|X \sim \mathcal{N}\bigl(\eta_{\Delta}(X),\sigma^2\bigr)$.  Thus $\{P_\Delta:\Delta \in \R\} \subseteq \distributionClassNonDecreasingRegressionFunction[d]$. Moreover, for any $\Delta >0$, we have by Pinsker's inequality that 
\begin{align*}
\mathrm{TV}( P_{-\Delta}^n, P_{0}^n) &\leq \sqrt{\frac{n}{2} \cdot \mathrm{KL}(P_{-\Delta},P_{0})} = \sqrt{\frac{n }{2} \cdot \mathrm{KL}\bigr(\mathcal{N}(\tau-\Delta,\sigma^2),\mathcal{N}(\tau,\sigma^2)\bigr)} \leq \frac{\sqrt{n}\Delta}{2\sigma}.  
\end{align*}
Now fix $\Delta>0$, and suppose that $\hat{A} \in \hat{\mathcal{A}}_n\bigl(\tau,\alpha,\distributionClassNonDecreasingRegressionFunction[d]\bigr)$.  Since $\superLevelSet{\tau}{\eta_{-\Delta}} = \emptyset$, we have for every $x \in \R^d$ that
\begin{align*}
\Prob_{P_{0}}\bigl(x \in \hat{A}(\sample)\bigr) &\leq \Prob_{P_{-\Delta}}\bigl(x \in \hat{A}(\sample)\bigr) + \mathrm{TV}\bigl(P_{-\Delta}^n,P_{0}^n\bigr) \\ &\leq \Prob_{P_{-\Delta}}\bigl( \hat{A}(\sample) \nsubseteq \superLevelSet{\tau}{\eta_{-\Delta}}\bigr)+\frac{\sqrt{n}\Delta}{2\sigma}\leq \alpha+\frac{\sqrt{n}\Delta}{2\sigma}.
\end{align*}
Hence, by Fubini's theorem,
\begin{align*}
\E_{P_{0}}\bigl\{ \mu\bigl(\hat{A}(\sample)\bigr) \bigr\} = \E_{P_{0}}\biggl( \int_{\R^d} \one_{\{x \in \hat{A}(\sample)\}} \, d\mu(x)\biggr) =\int_{\R^d}
\Prob_{P_{0}}\bigl(x \in \hat{A}(\sample)\bigr) \, d\mu(x) \leq \alpha+\frac{\sqrt{n}\Delta}{2\sigma}.
\end{align*}
Moreover, by our choice of $\eta_{0}$, we have $\superLevelSet{\tau}{\eta_{0}}=\R^d$, and hence
\begin{align*}
\E_{P_{0}}\bigl\{ \mu\bigl(\superLevelSet{\tau}{\eta_0}\setminus\hat{A}(\sample)\bigr) \bigr\}&= 1-  \E_{P_{0}}\bigl\{ \mu\bigl(\hat{A}(\sample)\bigr) \bigr\} \geq 1-\alpha - \frac{\sqrt{n}\Delta}{2\sigma} .
\end{align*}
The result follows by taking an infimum over $\hat{A} \in \hat{\mathcal{A}}_n\bigl(\tau,\alpha, \distributionClassNonDecreasingRegressionFunction[d]\bigr)$, and then letting $\Delta \rightarrow 0$.
\end{proof}

\begin{proof}[Proof of Theorem~\ref{thm:powerBound_1d}]
Let us define $C_0 := 962 \geq 2 \{16 \vee 481^{1/(2\marginExponent + 1)}\}$ and $C := 3C_0$.  Further, let
\[
\marginDummy := \frac{C_0}{2} \biggl\{\biggl(\frac{\sigma^2}{n\marginConstant^2} \log_+\Bigl(\frac{\log_+ n}{\alpha\wedge\delta}\Bigr)\biggr)^{1/(2\marginExponent + 1)} + \frac{\log_+(1/\delta)}{n}\biggr\}.
\]
By the choice of $\marginDummy$, the result holds if $\marginalDistribution\bigl(\superLevelSet{\tau}{\regressionFunction}\bigr) \leq 2\marginDummy$.  We therefore suppose henceforth that~$\marginDummy$ is such that $\marginDummy < 1/2$ and $\marginalDistribution\bigl(\superLevelSet{\tau}{\regressionFunction}\bigr) > 2\marginDummy$. Then, since $P\in\distributionClassUnivariateMarginCondition$, we have that $\marginalDistribution\bigl(\superLevelSet{\tau+\marginConstant\marginDummy^\marginExponent}{\regressionFunction}\bigr) > \marginDummy$, so $x_0 := \inf\{x\in\superLevelSet{\tau+\marginConstant\marginDummy^\marginExponent}{\regressionFunction}: \marginalDistribution\bigl(\superLevelSet{\tau+\marginConstant\marginDummy^\marginExponent}{\regressionFunction} \cap (-\infty, x]\bigr) \geq \marginDummy\}$ is finite. For $I_\marginDummy := \superLevelSet{\tau+\marginConstant\marginDummy^\marginExponent}{\regressionFunction} \cap (-\infty, x_0]$, it then holds that $\marginalDistribution(I_\marginDummy) \geq \marginDummy$. Further, by Lemma~\ref{lemma:log_logplus_facts}\emph{(i)} and the fact that $C_0 \geq 32$,
\[
n\marginDummy \geq 16 \log_+\Bigl(\frac{1}{\delta}\Bigr)\geq 8\log\Bigl(\frac{2}{\delta}\Bigr).
\]
Writing $\Omega_0 := \bigl\{n^{-1}\sum_{i=1}^n\one_{\{X_i \in I_\marginDummy\}} \geq \marginDummy/2\bigr\}$, it follows by a multiplicative Chernoff bound \citep[][Theorem~2.3(c)]{mcdiarmid1998concentration} that 
\[
\Prob_P\bigl(\Omega_0^c \bigr) \leq e^{-n\marginDummy/8} \leq \frac{\delta}{2}.
\]
By the choice of $\marginDummy$, it holds on $\Omega_0$ that
\[
\sum_{i=1}^n\one_{\{X_i \in I_\marginDummy\}} \geq \frac{n\marginDummy}{2} \geq \frac{1}{2}\cdot\Bigl(\frac{C_0}{2}\Bigr)^{2\marginExponent + 1} \cdot \frac{\sigma^2}{\marginConstant^2 \marginDummy^{2\marginExponent}} \log_+\Bigl(\frac{\log_+ n}{\alpha\wedge\delta}\Bigr);
\]
in particular, since $n\marginDummy \geq 8\log(2/\delta)$, it holds on this event that $\sum_{i=1}^n\one_{\{X_i \in I_\marginDummy\}} \geq 1$. Thus, we can fix any $i_1 \in [n]$ such that $X_{i_1} = \max(\sample_X \cap I_\marginDummy)$. Furthermore, let $i_1,\ldots, i_K \in [n]$ with $K := |\{i\in[n]: X_i \geq X_{i_1}\}|$ be the maximal set of indices such that $X_{i_k} \geq X_{i_1}$ for all $k \in [K]$. Writing $r_k := X_{i_k} - \min (\sample_X \cap I_\marginDummy)$ and noting that $\mathcal{I}_{r}(x) = \{i \in [n]: x-r \leq X_i \leq x\}$ for $r\geq 0$ and $x\in \R$, we have on $\Omega_0$ that for all $k \in [K]$,
\[
n \geq |\mathcal{I}_{r_k}(X_{i_k})| \geq |\mathcal{I}_{r_1}(X_{i_1})| \geq \frac{1}{2}\cdot\Bigl(\frac{C_0}{2}\Bigr)^{2\marginExponent + 1} \cdot \frac{\sigma^2}{\marginConstant^2\marginDummy^{2\marginExponent}} \log_+\Bigl(\frac{\log_+ n}{\alpha\wedge\delta}\Bigr).
\]
Hence, writing $u_{\delta'}(\ell) := 1.7\sqrt{\ell \bigl\{0.72 \log(5.2/\delta')+ \log \log (2\ell)\bigr\}}$ for $\delta'\in (0,1)$ and $\ell \in \mathbb{N}$ as in Lemma~\ref{lemma:howard_uniform_bound}\emph{(a)}, we have for all $k\in [K]$ that 
\begin{align*}
    \hspace{0.6cm}&\hspace{-0.6cm}\biggl(\frac{u_\alpha(|\mathcal{I}_{r_k}(X_{i_k})|) + u_{\delta/2}(|\mathcal{I}_{r_k}(X_{i_k})|)}{|\mathcal{I}_{r_k}(X_{i_k})|}\biggr)^2 \\
    &\leq \frac{2\cdot 1.7^2 \cdot 0.72 \cdot \bigl\{\log(5.2/\alpha) + \log(10.4/\delta)\bigr\} + 2\cdot 2\cdot 1.7^2 \cdot \log\log\bigl(2\cdot |\mathcal{I}_{r_k}(X_{i_k})|\bigr)}{|\mathcal{I}_{r_k}(X_{i_k})|} \\
    &\leq \frac{2\cdot 1.7^2 \cdot 0.72 \cdot 2\cdot \log\bigl(\frac{10.4}{\alpha \wedge \delta}\bigr) + 4\cdot 1.7^2 \cdot \log \bigl\{2\log_+\bigl(|\mathcal{I}_{r_k}(X_{i_k})|\bigr)\bigr\}}{|\mathcal{I}_{r_k}(X_{i_k})|} \\
    &\leq 11.56 \cdot \frac{\log\bigl(\frac{20.8\log_+ n}{\alpha \wedge \delta}\bigr)}{|\mathcal{I}_{r_1}(X_{i_1})|} \leq 11.56 \cdot 20.8 \cdot \frac{\log_+\bigl(\frac{\log_+ n}{\alpha \wedge \delta}\bigr)}{|\mathcal{I}_{r_1}(X_{i_1})|} \\
    &\leq 481 \cdot \Bigl(\frac{2}{C_0}\Bigr)^{2\marginExponent + 1} \cdot \frac{\marginConstant^2\marginDummy^{2\marginExponent}}{\sigma^2} \leq \frac{\marginConstant^2\marginDummy^{2\marginExponent}}{\sigma^2},
\end{align*}
where we used the inequality $(a+b)^2 \leq 2(a^2 +b^2)$ for $a,b \geq 0$, Lemma~\ref{lemma:log_logplus_facts}\emph{(i)} and the fact that $C_0 \geq 2 \cdot 481^{1/(2\marginExponent+1)}$.  For
\[
\Omega_1\bigl(X_{i_1}\bigr) := \bigcap_{k=1}^K \biggl\{\sum_{i \in \mathcal{I}_{r_k}(X_{i_k})} \frac{Y_i - \tau}{\sigma}\geq u_\alpha\bigl(|\mathcal{I}_{r_k}(X_{i_k})|\bigr)\biggr\},
\]
we therefore have on $\Omega_0$ that
\begin{align*}
    \mathbb P_P\Bigl(\Omega_1\bigl(X_{i_1}\bigr)^c  &\bigm| \sample_X \Bigr) \\
    &= \mathbb P_P\biggl(\bigcup_{k=1}^K \biggl\{\sum_{i \in \mathcal{I}_{r_k}(X_{i_k})} \frac{Y_{i} - (\tau + \marginConstant\marginDummy^\marginExponent)}{\sigma}< u_\alpha\bigl(|\mathcal{I}_{r_k}(X_{i_k})|\bigr) - |\mathcal{I}_{r_k}(X_{i_k})|\frac{\marginConstant\marginDummy^\marginExponent}{\sigma}\biggr\} \biggm| \sample_X \biggr)\\
    &\leq \mathbb P_P\biggl(\bigcup_{k=1}^K \biggl\{\sum_{i \in \mathcal{I}_{r_k}(X_{i_k})} \frac{Y_{i} - (\tau + \marginConstant\marginDummy^\marginExponent)}{\sigma} <  -u_{\delta/2}\bigl(|\mathcal{I}_{r_k}(X_{i_k})|\bigr) \biggr\} \biggm| \sample_X \biggr) \leq \frac{\delta}{2},
\end{align*}
where the last inequality follows from Lemma~\ref{lemma:howard_uniform_bound}\emph{(a)}.  Let $n(x)$ and $\bigl(Y_{(j)}(x)\bigr)_{j \in [n(x)]}$ be as in Definition~\ref{defn:pValue}.  For $k \in [K]$, write $n_k \equiv n_k(x, \sampleX) := |\mathcal{I}_{r_k}(x)|$ and
\[
\hat{p}^{r_k}_{\sigma, \tau}(x,\sample) := 5.2\exp\biggl\{-\frac{\max\bigl(\sum_{j =1}^{n_k} Y_{(j)}(x)-\tau n_k,0\bigr)^2}{2.0808 \sigma^2 n_k} +\frac{ \log\log(2n_k)}{0.72}\biggr\}.
\]
We have on $\Omega_0 \cap \Omega_1\bigl(X_{i_1}\bigr)$ that
\[
\max_{k \in [K]} \hat{p}_{\sigma,\tau}(X_{i_k},\sample) \leq \max_{k \in [K]} \hat{p}_{\sigma,\tau}^{r_k}(X_{i_k},\sample) \leq \alpha,
\]
so that $\{i_k: k \in [K]\} \subseteq \mathcal{R}^{\mathrm{ISS}}_\alpha\bigl(\mathcal{G}_{\mathrm{W}}(\sample_X), \bigl(\hat{p}_{\sigma,\tau}(X_{i},\sample)\bigr)_{i \in [n]}\bigr)$.  Since $X_{i_1} \in I_\marginDummy$, we have on $\Omega_0 \cap \Omega_1\bigl(X_{i_1}\bigr)$ that
\[
     [x_0, \infty) \subseteq \bigl[X_{i_1}, \infty\bigr) \subseteq \hat{A}^{\mathrm{ISS}}_{\sigma,\tau,\alpha, n}(\sample).
\]
It follows that
\begin{align*}
    \marginalDistribution\bigl(\superLevelSet{\tau}{\regressionFunction}\setminus \hat{A}^{\mathrm{ISS}}_{\sigma,\tau,\alpha, n}(\sample)\bigr) &\leq \marginalDistribution\bigl(\superLevelSet{\tau}{\regressionFunction}\setminus \superLevelSet{\tau + \marginConstant\marginDummy^\marginExponent}{\regressionFunction}\bigr) + \marginalDistribution\bigl(\superLevelSet{\tau + \marginConstant\marginDummy^\marginExponent}{\regressionFunction}\setminus \hat{A}^{\mathrm{ISS}}_{\sigma,\tau,\alpha, n}(\sample)\bigr) \\
    &\leq \marginDummy + \marginalDistribution\bigl(\superLevelSet{\tau + \marginConstant\marginDummy^\marginExponent}{\regressionFunction}\cap (-\infty, X_{i_1})\bigr) \\
    &\leq \marginDummy + \marginalDistribution\bigl(\superLevelSet{\tau + \marginConstant\marginDummy^\marginExponent}{\regressionFunction}\cap (-\infty, x_0)\bigr) \leq 2\marginDummy,
\end{align*}
since $P\in\distributionClassUnivariateMarginCondition[1]$. 
We conclude that 
\begin{align*}
\Prob_P\biggl[ \marginalDistribution\bigl(\superLevelSet{\tau}{\regressionFunction}\setminus\hat{A}^{\mathrm{ISS}}_{\sigma,\tau,\alpha, n}(\sample)\bigr) &> 1 \wedge C_0 \biggl\{\biggl(\frac{\sigma^2}{n\marginConstant^2} \log_+\Bigl(\frac{\log_+ n}{\alpha\wedge\delta}\Bigr)\biggr)^{1/(2\marginExponent + 1)} +\frac{\log_+(1/\delta)}{n}\biggr\}\biggr] \nonumber\\
&= \Prob_P\Bigl( \marginalDistribution\bigl(\superLevelSet{\tau}{\regressionFunction}\setminus \hat{A}^{\mathrm{ISS}}_{\sigma,\tau,\alpha, n}(\sample)\bigr) > 2\marginDummy\Bigr) \leq \Prob_P\bigl(\Omega_0^c \cup \Omega_{1}(X_{i_1})^c\bigr) \leq \delta.
\end{align*}
This proves the first statement in the theorem, and we deduce the second result by integrating our tail bound over $\delta \in (0, 1)$.  Since this part of the calculation is an identical argument to that in the multivariate case, we refer the reader to~\eqref{Eq:Integrate1},~\eqref{Eq:Integrate2} and~\eqref{Eq:Integrate3} in the proof of Theorem~\ref{thm:powerBound} for details.  Since $C = 3C_0$, the result follows.
\end{proof}

\begin{proof}[Proof of Proposition~\ref{prop:MaxMinNegativeResultMultivariateNoDensityBound}] Take $q \in \N$ and let $\mathbb{W}_{q,d} \subseteq [q]^d$ and $(\mathcal{H}_{\bm{j}}^q: \bm{j}\in \mathbb{W}_{q,d})$ be as in Section~\ref{sec:lowerBound}.  Let $\mu_q$ denote the uniform distribution on $\bigcup_{\bm{j}\in \mathbb{W}_{q,d}}\mathcal{H}_{\bm{j}}^q$. For each $\bm{j}=(j_1,\ldots,j_d) \in \mathbb{W}_{q,d}$, we define $\eta_{q,\bm{j}}:\R^d\rightarrow \R$ by
\begin{align*}
\eta_{q,\bm{j}}(x) \equiv \eta_{q,\bm{j}}(x_1,\ldots,x_d) :=\begin{cases} \tau - \marginConstant &\text{ if } x_\ell < j_\ell/q \text{ for all } \ell \in [d]\\
\tau + \marginConstant &\text{ otherwise. }
\end{cases}
\end{align*}
We also define $\eta_{q,\ast}:\R^d\rightarrow \R$ to be the constant function $\eta_{q,\ast}(x) := \tau + \marginConstant$.  For $\bm{j} \in \mathbb{W}_{q,d} \cup \{\ast\}$, let $P_{q,\bm{j}} \in \distributionClassNonDecreasingRegressionFunction[d]$ denote the distribution on $[0,1]^d \times \R$ of $(X,Y)$, where $X \sim \mu_q$ and $Y|X \sim \mathcal{N}\bigl(\eta_{q,\bm{j}}(X),\sigma^2\bigr)$.  Moreover, $\mu_q\bigl(\eta_{q,\bm{j}}^{-1}([\tau,\tau+\marginConstant\marginDummy^\marginExponent])\bigr) = \mu_q(\emptyset) = 0 \leq \marginDummy$ for all $\marginDummy < 1$.  On the other hand, if $\marginDummy = 1$ then $\mu_q\bigl(\eta_{q,\bm{j}}^{-1}([\tau,\tau+\marginConstant\marginDummy^\marginExponent])\bigr) \leq 1 = \marginDummy$. Thus,  $P_{q,\bm{j}} \in \distributionClassNonDecreasingRegressionFunction[d]\cap \distributionClassUnivariateMarginCondition[d]$ for all $\bm{j} \in \mathbb{W}_{q,d} \cup \{\ast\}$. In addition, given any $\bm{j}=(j_1,\ldots,j_d) \in \mathbb{W}_{q,d}$, and $\bm{j}'=(j'_1,\ldots,j'_d) \in \mathbb{W}_{q,d}\setminus \{\bm{j}\}$ we must have $j'_{\ell'}>j_{\ell'}$ for some $\ell' \in [d]$ by the antichain property.  Hence $x_{\ell'} \geq (j'_{\ell'}-1)/q \geq j_{\ell'}/q$ for all $x = (x_1,\ldots,x_d)^\top \in \mathcal{H}_{\bm{j}'}^q$, so $\eta_{q,\bm{j}}(x)=\tau+\marginConstant=\eta_{q,\ast}(x)$ for such $x$. Consequently, for each $\bm{j} \in \mathbb{W}_{q,d}$, we have
\begin{align*}
\mathrm{KL}(P_{q,\ast},P_{q,\bm{j}}) &= \int_{\R^d}\mathrm{KL}\bigl\{\mathcal{N}\bigl(\eta_{q,\ast}(x),\sigma^2\bigr),\mathcal{N}\bigl(\eta_{q,\bm{j}}(x),\sigma^2\bigr)\bigr\} \, d\mu_q(x)\\
&=\mu_q(\mathcal{H}_{\bm{j}}^q)\cdot \frac{2\marginConstant^2}{\sigma^2} = \frac{1}{|\mathbb{W}_{q,d}|} \cdot \frac{2\marginConstant^2}{\sigma^2} \leq \frac{2\marginConstant^2 d}{q^{d-1}\sigma^2}.
\end{align*}
Hence, by Pinsker's inequality, 
\begin{align*}
\mathrm{TV}\bigl(P_{q,\ast}^n,P_{q,\bm{j}}^n\bigr) \leq \sqrt{\frac{1}{2}\cdot \mathrm{KL}\bigl(P_{q,\ast}^n,P_{q,\bm{j}}^n\bigr)} 
= \sqrt{\frac{n}{2}\cdot \mathrm{KL}\bigl(P_{q,\ast},P_{q,\bm{j}}\bigr)}
\leq \sqrt{\frac{n\marginConstant^2 d}{q^{d-1}\sigma^2} }.
\end{align*}
To complete the proof, consider $\hat{A} \in \hat{\mathcal{A}}_n(\tau,\alpha,\mathcal{P}')$. 
Let $x \in \bigcup_{\bm{j}\in \mathbb{W}_{q,d}}\mathcal{H}_{\bm{j}}^q$, so we can find $\bm{j}_x \in \mathbb{W}_{q,d}$ such that $x \in \mathcal{H}_{\bm{j}_x}^q$. Then $x \notin \superLevelSet{\tau}{\eta_{q,\bm{j}_x}}$, so 
\begin{align*}
\Prob_{P_{q,\ast}}\bigl(x \in \hat{A}(\sample)\bigr) &\leq \Prob_{P_{q,\bm{j}_x}}\bigl(x \in \hat{A}(\sample)\bigr) + \mathrm{TV}\bigl(P_{q,\ast}^n,P_{q,\bm{j}_x}^n\bigr)\\ 
&\leq \Prob_{P_{q,\bm{j}_x}}\bigl( \hat{A} \nsubseteq \superLevelSet{\tau}{\eta_{q,\bm{j}_x}}\bigr)+\sqrt{\frac{n\marginConstant^2 d}{q^{d-1}\sigma^2} } \leq \alpha+\sqrt{\frac{n\marginConstant^2 d}{q^{d-1}\sigma^2} }.
\end{align*}
Hence, by Fubini's theorem,
\begin{align*}
\E_{P_{q,\ast}}\bigl\{ \mu_q\bigl(\hat{A}(\sample)\bigr) \bigr\}&= \E_{P_{q,\ast}}\biggl( \int_{\R^d} \one_{\{x \in \hat{A}(\sample)\}} \, d\mu_q(x)\biggr)\\ &=\int_{\R^d}
\Prob_{P_{q,\ast}}\bigl(x \in \hat{A}(\sample)\bigr) \, d\mu_q(x) \leq \alpha+\sqrt{\frac{n\marginConstant^2 d}{q^{d-1}\sigma^2} }.
\end{align*}
By our choice of $\eta_{q,\ast}$, we have $\superLevelSet{\tau}{\eta_{q,\ast}}=\R^d$, and hence
\begin{align*}
\E_{P_{q,\ast}}\bigl\{ \mu_q\bigl(\superLevelSet{\tau}{\eta_\ast}\setminus\hat{A}(\sample)\bigr) \bigr\}&= 1-  \E_{P_{q,\ast}}\bigl\{ \mu_q\bigl(\hat{A}(\sample)\bigr) \bigr\} \geq 1-\alpha - \sqrt{\frac{n\marginConstant^2 d}{q^{d-1}\sigma^2} }.
\end{align*}
The result follows by taking an infimum over $\hat{A} \in \hat{\mathcal{A}}_n(\tau,\alpha, \mathcal{P}')$, and then letting $q \rightarrow \infty$. 
\end{proof}

\begin{proof}[Proof of Proposition~\ref{Prop:Containment}]
Fix $P \in \distributionClassNonDecreasingRegressionFunction[d] \cap \distributionClassMultivariateCondition[d]$ with marginal distribution $\mu$ on $\R^d$ and regression function $\regressionFunction$, and let $C := 2\cdot 3^{3d + 1} \cdot (1 + d^{1/2})^{d + 1} \cdot \densityConstant^2 \geq 1$.  
Fix $\marginDummy \in (0,1]$, and let $r := 2\marginDummy/C \in (0,1]$. Let $T \subseteq \superLevelSet{\tau + \etaIncreasingConstant r^\etaIncreasingExponent}{\regressionFunction}$ be as in Lemma~\ref{lemma:geomLemmaForPowerBound}, so that, by that same lemma, 
\begin{align*}
\marginalDistribution\bigl(\regressionFunction^{-1}([\tau, \tau + \etaIncreasingConstant \marginDummy^\etaIncreasingExponent/C^\etaIncreasingExponent])\bigr) &= \marginalDistribution\bigl(\regressionFunction^{-1}([\tau, \tau + \etaIncreasingConstant (r/2)^\etaIncreasingExponent])\bigr) \leq \marginalDistribution\bigl(\regressionFunction^{-1}\bigl([\tau, \tau + \etaIncreasingConstant r^\etaIncreasingExponent)\bigr)\bigr) \\&= \marginalDistribution\bigl(\superLevelSet{\tau}{\regressionFunction}\setminus \superLevelSet{\tau + \etaIncreasingConstant r^\etaIncreasingExponent}{\regressionFunction}\bigr) \leq \marginalDistribution\bigl(\superLevelSet{\tau}{\regressionFunction}\setminus T\bigr) \\
&\leq 3^{3d + 1} \cdot (1 + d^{1/2})^{d + 1} \cdot \densityConstant^2 \cdot r = \marginDummy,
\end{align*}
as required. 
\end{proof}

\begin{proof}[Proof of Theorem~\ref{thm:powerBound}] 
Let us define $C:= 3^{3d+1 + 1/d} \cdot (1+d^{1/2})^{d+1} \cdot \densityConstant^2 \cdot C_\circ$, where
\begin{align*}
C_\circ &:=  \bigl(2^6 \densityConstant \log(4 \cdot 9^d \cdot \densityConstant)\bigr)^{1/d} 
\geq \bigl(2^4 \densityConstant \cdot \log(4 \cdot 9^d \cdot \densityConstant)\bigr)^{1/d} \vee (175\cdot\densityConstant)^{1/(2\etaIncreasingExponent + d)}.
\end{align*}
Further, let 
\[
r := C_\circ \biggl\{ \biggl(\frac{\sigma^2}{n \etaIncreasingConstant^2}\log_+\Bigl(\frac{m\log_+ n}{\alpha \wedge\delta}\Bigr)\biggr)^{1/(2\etaIncreasingExponent+d)}+\biggl(\frac{\log_+(m/\delta)}{m}\biggr)^{1/d} \biggr\}.
\]
Since $C \geq C_\circ$, the first result is immediate if $r > 1$.  We therefore assume henceforth that $r \leq  1$ (so in particular, $n\etaIncreasingConstant^2 \geq \sigma^2$).  Observe that 
\begin{align*}
m\cdot r^d \geq C_\circ^d \cdot \log_+(m/\delta) &\geq 2^4 \densityConstant \cdot \log(4\cdot 9^d \cdot \densityConstant) \cdot \log_+(m^{(d-1)/d}/\delta) \\
&\geq 8\densityConstant\bigl\{\log(4\cdot 9^d \cdot \densityConstant) + \log(r^{-(d-1)} /\delta)\bigr\} = 8\densityConstant \log (4 \cdot 9^d \cdot \densityConstant \cdot r^{-(d-1)}/\delta),
\end{align*}
since $r \geq m^{-1/d}$ and so $m^{(d-1)/d} \geq r^{-(d-1)}$.  Now, let $\bigl((S^0_j,S^1_j)\bigr)_{j \in [q]}$ be the hypercubes in Lemma~\ref{lemma:geomLemmaForPowerBound} and let $(\Omega_{0,j,k})_{j\in[q], k\in\{0,1\}}$ be the events in Lemma \ref{lemma:ChernoffLemmaForPowerBound}.
Then on $\bigcap_{j \in [q]} \Omega_{0,j,1}$, we have for each $j \in [q]$ that there exists $i_j \in [m]$ with $X_{i_j} \in S_j^1$.  We extend $\{X_{i_1},\ldots,X_{i_q}\}$ to a maximal set $\{X_{i_1},\ldots,X_{i_\ell}\}$ with $\ell \in \{q,q+1,\ldots,m\}$ such that $X_{i_{q+1}}, \ldots, X_{i_\ell} \in \bigcup_{j=1}^q \bigl\{x \in \R^d: X_{i_j} \preccurlyeq x\bigr\}$. 
Note that since $S_j^1 \succcurlyeq S_j^0$ for every $j\in[q]$, for every $s\in[\ell]$ there exists $j_s \in [q]$ such that $\{X_{i_s}\} \succcurlyeq S_{j_s}^0$.  For $s\in [\ell]$ and $x\in\R$, write $r_s := \sup_{x' \in S_{j_s}^0} \|x' - X_{i_s}\|_\infty$, $\mathcal{I}_{r_s}(x) := \{i \in [n]: X_i \preccurlyeq x, \|X_i - x\|_\infty \leq r_s\}$, $n_s := |\mathcal{I}_{r_s}(x)|$ and 
\[
\hat{p}^{r_s}_{\sigma, \tau}(x, \sample) := 5.2\exp\biggl\{-\frac{\max\bigl(\sum_{i \in \mathcal{I}_{r_s}(x)} Y_{i}-\tau n_s,0\bigr)^2}{2.0808 \sigma^2 n_s} +\frac{ \log\log(2n_s)}{0.72}\biggr\}.
\]
By the choice of~$r$, we have on $\bigcap_{j\in[q], k \in \{0, 1\}} \Omega_{0,j,k}$ that 
\begin{align*}
\min_{s \in [\ell]} |\mathcal{I}_{r_s}(X_{i_s})| = \min_{j \in [q]} |\mathcal{I}_{r_j}(X_{i_j})| &\geq \min_{j \in [q]} \sum_{i=1}^n \mathbbm{1}_{\{X_i \in S_j^0\}} \geq \frac{nr^d}{2\densityConstant}\\
&\geq \frac{C_\circ^{2\etaIncreasingExponent+d} \cdot \sigma^2}{ 2 \cdot \densityConstant \cdot  \etaIncreasingConstant^2 \cdot r^{2\gamma}}\cdot \log_+\Bigl(\frac{m\log_+ n}{\alpha \wedge\delta}\Bigr) \\
&\geq \frac{ 2\cdot\sigma^2}{\etaIncreasingConstant^2 \cdot r^{2\gamma}}\cdot 4.2\cdot 5.2\cdot 2 \log_+\Bigl(\frac{m\log_+ n}{\alpha \wedge\delta}\Bigr) \\
&\geq \frac{2\sigma^2}{\etaIncreasingConstant^2 \cdot r^{2\gamma}} \biggl\{4.2 \log\Bigl(\frac{5.2m}{\alpha \wedge\delta}\Bigr)+ 3\log (2\log_+ n)\biggr\} \\
&\geq \frac{2\sigma^2}{\etaIncreasingConstant^2 \cdot r^{2\etaIncreasingExponent}}\biggl\{2\log\Bigl(\frac{2m}{\delta}\Bigr) + 2.0808\log \Bigl(\frac{5.2m}{\alpha}\Bigr)+3\log\log (2n)\biggr\},
\end{align*}
where the last two inequalities follow from Lemma~\ref{lemma:log_logplus_facts}\emph{(i)}.  Let $\bigl(\Omega_{1,s}(\cdot)\bigr)_{s\in[\ell]}$ be as in Lemma~\ref{lemma:HoeffdingLemmaForPowerBound_v2}.
Then, on $\Omega_* := \bigcap_{j \in [q],k \in \{0,1\}} \Omega_{0,j,k} \cap \bigcap_{s \in [\ell]} \Omega_{1,s}(X_{i_s})$, we have
\begin{equation}
\label{Eq:OnEventBound}
\max_{s \in [\ell]} \hat{p}_{\sigma,\tau}(X_{i_s},\sample) \leq \max_{s \in [\ell]} \hat{p}_{\sigma,\tau}^{r_s}(X_{i_s},\sample) \leq \frac{\alpha}{m}.
\end{equation}
We claim that on $\Omega_*$, we have $I_1 := \{i_1,\ldots,i_\ell\} \subseteq  \mathcal{R}^{\mathrm{ISS}}_\alpha\bigl(\mathcal{G}_{\mathrm{W}}(\sample_{X,m}),\bigl(\hat{p}_{\sigma,\tau}(X_i,\sample)\bigr)_{i \in [m]}\bigr)$, and prove this by contradiction.  First, for $([m], E, \bm{w}) := \mathcal{G}_{\mathrm{W}}(\sample_{X,m})$, denote for brevity $G := ([m], E) = \mathcal{G}(\sample_{X,m})$ and $F := ([m], \{e \in E: w_e = 1\}) = \mathcal{G}_{\mathrm{F}}(\sample_{X,m})$ as in Algorithm~\ref{algo:R_ISS}.  Moreover, define $\alpha(\cdot, \cdot)$ as in~\eqref{Eq:alpha} in the proof of Lemma~\ref{lemma:sparseMG_valid} and write $R := \mathcal{R}^{\mathrm{ISS}}_\alpha\bigl(\mathcal{G}_{\mathrm{W}}(\sample_{X,m}),\bigl(\hat{p}_{\sigma,\tau}(X_i,\sample)\bigr)_{i \in [m]}\bigr)$. Suppose now for a contradiction that there exists $s \in [\ell]$ such that $i_s \notin R$ and write $I_* := \{i \in [m]: \alpha(i, R) > 0\} \subseteq [m]\setminus R$.  Now $I_1$ is $G$-upper by construction and hence also $F$-upper.  Consequently, there exists $s' \in [\ell]$ with $i_{s'} \in I_* \cap \bigl(\{i_s\}\cup\an_F(i_s)\bigr) \subseteq I_1\setminus R$ which in turn necessitates by Algorithm~\ref{algo:R_ISS} that $\hat{p}_{\sigma,\tau}(X_{i_{s'}},\sample) > \alpha(i_{s'}, R)$. Moreover, as $R$ is also $G$-upper by construction and therefore $F$-upper, we deduce that $\bigl(\{i_{s'}\}\cup\de_F(i_{s'})\bigr) \cap R = \emptyset$ while $\bigl(\{i_{s'}\}\cup \de_F(i_{s'})\bigr) \cap L(F) \neq \emptyset$, so that $\bigl|\bigl(\{i_{s'}\}\cup\de_F(i_{s'})\bigr)\cap L(F)\setminus R\bigr| \geq 1$.  Thus, by~\eqref{Eq:OnEventBound},
\[
\alpha(i_{s'}, R) = \frac{\bigl|\bigl(\{i_{s'}\}\cup \de_F(i_{s'})\bigr) \cap L(F) \setminus R\bigr|}{|L(F) \setminus R|}\cdot \alpha \geq \frac{\alpha}{m} \geq \hat{p}_{\sigma,\tau}(X_{i_{s'}},\sample),
\]
which establishes our contradiction and therefore proves the claim.  It follows that on $\Omega_*$, we have $\{i_1,\ldots,i_q\} \subseteq \mathcal{R}^{\mathrm{ISS}}_\alpha\bigl(\mathcal{G}_{\mathrm{W}}(\sample_{X,m}),\bigl(\hat{p}_{\sigma,\tau}(X_i,\sample)\bigr)_{i \in [m]}\bigr)$, so taking the Borel measurable $T \subseteq \superLevelSet{\tau}{\regressionFunction} \cap \support(\marginalDistribution)$ from Lemma~\ref{lemma:geomLemmaForPowerBound}, we have
\[
T \subseteq \bigcup_{j=1}^q \bigl\{x \in \R^d: X_{i_j} \preccurlyeq x\bigr\} \subseteq \hat{A}^{\mathrm{ISS}}_{\sigma, \tau, \alpha, m}(\sample).
\]
Hence, on $\Omega_*$,
\begin{align*}
    \marginalDistribution\bigl(\superLevelSet{\tau}{\regressionFunction}\setminus \hat{A}^{\mathrm{ISS}}_{\sigma, \tau, \alpha, m}(\sample)\bigr) \leq \marginalDistribution\bigl(\superLevelSet{\tau}{\regressionFunction} \setminus T\bigr) \leq 3^{3d+1} \cdot (1+d^{1/2})^{d+1} \cdot \densityConstant^2 \cdot r, 
\end{align*}
where the second inequality follows from Lemma~\ref{lemma:geomLemmaForPowerBound}. Thus, for any $m\in[n]$, we conclude  by Lemmas~\ref{lemma:ChernoffLemmaForPowerBound} and~\ref{lemma:HoeffdingLemmaForPowerBound_v2} that 
\begin{align*}
\Prob_P\biggl[ \marginalDistribution\bigl(\superLevelSet{\tau}{\regressionFunction}&\setminus \hat{A}^{\mathrm{ISS}}_{\sigma, \tau, \alpha, m}(\sample)\bigr)\! > 1 \wedge \frac{C}{3^{1/d}}  \biggl\{\biggl(\frac{\sigma^2}{n \etaIncreasingConstant^2}\log_+\Bigl(\frac{m\log_+ n}{\alpha \wedge\delta}\Bigr) \biggr)^{1/(2\etaIncreasingExponent+d)}\! \! \! +\biggl(\frac{\log_+(m/\delta)}{m}\biggr)^{1/d} \biggr\} \biggr] \nonumber\\
&= \Prob_P\Bigl\{ \marginalDistribution\bigl(\superLevelSet{\tau}{\regressionFunction} \setminus \hat{A}^{\mathrm{ISS}}_{\sigma, \tau, \alpha, m}(\sample)\bigr) > 3^{3d+1} \cdot (1+d^{1/2})^{d+1} \cdot \densityConstant^2 \cdot r\Bigr\} \leq \Prob_P(\Omega_*^c) \leq \delta.
\end{align*}
This proves the tail bound in the first part of the theorem. We deduce the bound in expectation by integrating over $\delta \in (0, 1)$.  First observe that
\begin{align}
\label{Eq:Integrate1}
   \int_0^1 \log_+\Bigl(\frac{m\log_+ n}{\alpha \wedge\delta}\Bigr)  \, d\delta \leq \int_0^1 \log\biggl(\frac{(m\log_+ n) \vee e}{\alpha \wedge\delta}\biggr)  \, d\delta 
&= \log\biggl(\frac{(m\log_+ n) \vee e}{\alpha}\biggr) + \alpha \nonumber \\
&\leq 3\log_+\Bigl(\frac{m\log_+ n}{\alpha}\Bigr).
\end{align}
Hence, by Jensen's inequality,
\begin{align}
\label{Eq:Integrate2}
\int_0^1 \biggl\{\log_+\Bigl(\frac{m\log_+ n}{\alpha \wedge\delta}\Bigr) \biggr\}^{1/(2\etaIncreasingExponent + d)} \, d\delta \leq 3^{1/(2\etaIncreasingExponent + d)}\biggl\{\log_+\Bigl(\frac{m\log_+ n}{\alpha}\Bigr)\biggr\}^{1/(2\etaIncreasingExponent + d)}. 
\end{align}
At the same time, by Jensen's inequality again, 
\begin{align}
\label{Eq:Integrate3}
    \int_0^1 \log^{1/d}_+\Bigl(\frac{m}{\delta}\Bigr)\, d\delta &\leq \biggl\{\int_0^1 \log_+\Bigl(\frac{m}{\delta}\Bigr)\, d\delta\biggr\}^{1/d} \leq \biggl\{\int_0^1 \log\Bigl(\frac{m \vee e}{\delta}\Bigr)\, d\delta\biggr\}^{1/d} \nonumber \\ 
    &= \bigl(\log_+ m + 1\bigr)^{1/d} \leq 2^{1/d} \log^{1/d}_+ m, 
\end{align}
whence the result follows as $3^{1/(2\etaIncreasingExponent + d)} \vee 2^{1/d} \leq 3^{1/d}$.  
\end{proof}

\begin{proof}[Proof of Corollary~\ref{Cor:PowerBound}]
If $\lambda \geq \sigma$, then $m_0=n$, and the result follows from Theorem~\ref{thm:powerBound}.  On the other hand, if $\lambda < \sigma$, then $m_0 = \lceil n\lambda^2/\sigma^2 \rceil \leq n$.  As in the proof of Theorem~\ref{thm:powerBound}, we may assume that $n\lambda^2 \geq \sigma^2$, so that
\[
\biggl(\frac{\log_+(m_0/\delta)}{m_0}\biggr)^{1/d} \leq \biggl(\frac{\sigma^2}{n\lambda^2} \log_+\Bigl(\frac{2n\lambda^2\log_+ n}{\sigma^2(\alpha \wedge \delta)}\Bigr)\biggr)^{1/d} \leq 2\biggl(\frac{\sigma^2}{n \etaIncreasingConstant^2}\log_+\Bigl(\frac{n\lambda^2\log_+ n}{\sigma^2(\alpha \wedge\delta)}\Bigr)\biggr)^{1/d}.
\]
Since the result is clear if $\frac{\sigma^2}{n \etaIncreasingConstant^2}\log_+\bigl(\frac{n\lambda^2\log_+ n}{\sigma^2(\alpha \wedge\delta)}\bigr) > 1$, we may further assume that this quantity is at most 1.  But then
\[
\biggl(\frac{\log_+(m_0/\delta)}{m_0}\biggr)^{1/d} \leq 2\biggl(\frac{\sigma^2}{n\lambda^2} \log_+\Bigl(\frac{n\lambda^2\log_+ n}{\sigma^2(\alpha \wedge \delta)}\Bigr)\biggr)^{1/(2\gamma+d)}.
\]
The $\log_+\bigl(m_0/(\alpha \wedge \delta)\bigr)$ term can be handled similarly (in fact, in a slightly simpler way), so
\begin{align*}
\Prob_P\biggl[ \marginalDistribution\bigl(\superLevelSet{\tau}{\regressionFunction}\setminus\hat{A}^{\mathrm{ISS}}_{\sigma,\tau,\alpha,m_0}(\sample)\bigr) &> 1 \wedge \frac{4C}{3^{1/d}} \biggl\{ \biggl(\frac{\sigma^2}{n \etaIncreasingConstant^2}\log_+\Bigl(\frac{n\lambda^2\log_+ n}{\sigma^2(\alpha \wedge\delta)}\Bigr)\biggr)^{1/(2\etaIncreasingExponent+d)} \\
&\hspace{5cm}+\biggl(\frac{\log_+(n/\delta)}{n}\biggr)^{1/d} \biggr\} \biggr] \leq \delta.
\end{align*}
We can then deduce the expectation bound using the same techniques as in the proof of Theorem~\ref{thm:powerBound}, and the result follows.
\end{proof}

\begin{lemma}\label{lemma:geomLemmaForPowerBound}  \sloppy Let $d \in \mathbb{N}$, $\tau \in \R$, $\sigma, \etaIncreasingExponent, \etaIncreasingConstant > 0$,  $\densityConstant \in (1,\infty)$ and take $P \in \distributionClassNonDecreasingRegressionFunction \cap \distributionClassMultivariateCondition$.
Given $r \leq 1$, there exist $q \leq \lfloor 9^d \cdot \densityConstant \cdot r^{-(d-1)}\rfloor \in \mathbb{N}$ and pairs of hypercubes $\bigl((S^0_j,S^1_j)\bigr)_{j \in [q]} \in \bigl(\powerSet(\R^d) \times \powerSet(\R^d)\bigr)^q$ such that $S^0_j \preccurlyeq S^1_j$, $S^0_j \subseteq \superLevelSet{\tau+\etaIncreasingConstant \cdot r^{\etaIncreasingExponent}}{\regressionFunction}$, $ \marginalDistribution(S^0_j)\wedge\marginalDistribution(S^1_j) \geq \densityConstant^{-1} \cdot r^d$, along with a Borel measurable set $T \subseteq \superLevelSet{\tau + \etaIncreasingConstant r^\etaIncreasingExponent}{\regressionFunction} \cap \support(\marginalDistribution) \subseteq \superLevelSet{\tau}{\regressionFunction} \cap \support(\marginalDistribution)$ such that for every $x \in T$ there exists $j_x \in [q]$ with $S^1_{j_x} \preccurlyeq \{x\}$, and 
\begin{align*}
\marginalDistribution\bigl(\superLevelSet{\tau}{\regressionFunction} \setminus T\bigr) \leq 3^{3d+1} \cdot (1+d^{1/2})^{d+1} \cdot \densityConstant^2 \cdot r.
\end{align*}
\end{lemma}

\begin{proof}
Without loss of generality, assume that $\support(\marginalDistribution)\cap \superLevelSet{\tau}{\regressionFunction} \neq \emptyset$.  Write $V_1:=\{ s \cdot \bm{1}_{d}:s \in \R\}$ with orthogonal complement $V_1^\perp$, and write $\Pi_{V_1}:\R^d \rightarrow V_1$ and~$\Pi_{V_1^\perp}:\R^d \rightarrow V_1^\perp$ for the orthogonal projections onto $V_1$ and $V_1^\perp$ respectively.  Fix $r \leq 1$.  We begin by showing that $\Pi_{V_1^\perp}\bigl(\support(\marginalDistribution)\cap \superLevelSet{\tau}{\regressionFunction}\bigr)$ can be covered by $q \leq \lfloor 3^{d-1} \cdot 2^{3d/2} \cdot \densityConstant \cdot r^{-(d-1)}\rfloor$ closed Euclidean balls of radius $r$. To see this, first let $(z_\ell)_{\ell \in [p]}$ with $p \in \N \cup \{\infty\}$ be a maximal sequence in $\support(\marginalDistribution)\cap \superLevelSet{\tau}{\regressionFunction}$ with $\|z_{j}-z_{j'}\|_{\infty} > 2^{-1/2}$ for $j \neq j'$. Then
\[
\support(\marginalDistribution)\cap \superLevelSet{\tau}{\regressionFunction} \subseteq \bigcup_{j=1}^p\closedSupNormMetricBall{z_j}{2^{-1/2}} \subseteq \bigcup_{j=1}^p \closedEuclideanMetricBall{z_j}{(d/2)^{1/2}}.
\] 
Moreover, since $\bigl(\closedSupNormMetricBall{z_j}{2^{-3/2}}\bigr)_{j \in [p]}$ are disjoint and $P \in \distributionClassMultivariateCondition$, we have 
\[
1 \geq \sum_{j=1}^p \marginalDistribution\bigl(\closedSupNormMetricBall{z_j}{2^{-3/2}}\bigr) \geq \frac{p}{2^{3d/2}\theta},
\]
so $p \leq 2^{3d/2}\theta$. Each projected Euclidean ball $\Pi_{V_1^\perp}\bigl(\closedEuclideanMetricBall{z_j}{(d/2)^{1/2}}\bigr)$ can in turn be covered by $(3/r)^{d-1}$
closed Euclidean balls\footnote{Here we use the fact that given any $d' \in \N$ and $\epsilon \in (0,1]$, the closed Euclidean unit ball in $\R^{d'}$ may be covered by at most $(3/\epsilon)^{d'}$ closed Euclidean balls of radius $\epsilon$.  Indeed, if $w_1,\ldots,w_M \in B_{2,d'}(0,1)$ satisfy $\|w_j - w_{j'}\|_2 > \epsilon$ for $j \neq j'$, then $\cup_{j \in M} B_{2,d'}(w_j,\epsilon/2) \subseteq B_{2,d'}(0,1+\epsilon/2) \subseteq B_{2,d'}(0,3/2)$, so $M (\epsilon/2)^{d'} \leq (3/2)^{d'}$, and the result follows.} of radius $(d/2)^{1/2}r \leq (d/2)^{1/2}$.  It follows that we can find a sequence $x_1,\ldots,x_q \in V_1^\perp$ with $q \leq 3^{d-1} \cdot 2^{3d/2} \cdot \densityConstant \cdot r^{-(d-1)}$ and $\Pi_{V_1^\perp}\bigl(\support(\marginalDistribution)\cap \superLevelSet{\tau}{\regressionFunction}\bigr)\subseteq \bigcup_{j \in [q]}\Pi_{V_1^\perp}\bigl(\closedEuclideanMetricBall{x_j}{(d/2)^{1/2}r}\bigr)$.  We deduce that 
\begin{align*}
\support(\marginalDistribution) \cap \superLevelSet{\tau}{\regressionFunction} &\subseteq \bigcup_{j \in [q]} \Pi_{V_1^\perp}\bigl(\closedEuclideanMetricBall{x_j}{(d/2)^{1/2}r}\bigr) \oplus \bigcup_{\ell \in \mathbb{Z}} \Pi_{V_1}\bigl(\closedEuclideanMetricBall{\ell r \cdot \bm{1}_d}{(d/2)^{1/2}r}\bigr) \\
&\subseteq \bigcup_{(j,\ell) \in [q]\times \Z}\closedEuclideanMetricBall{x_j+\ell r \cdot \bm{1}_d}{d^{1/2}r} \\
&\subseteq \bigcup_{(j,\ell) \in [q]\times \Z}\closedSupNormMetricBall{x_j+\ell r \cdot \bm{1}_d}{d^{1/2}r}.
\end{align*}
Now, for each $j \in [q]$, choose 
\begin{align*}
\ell_{0,j}&:=\min\bigl\{ \ell \in \Z : \closedSupNormMetricBall{x_j+\ell r\cdot \bm{1}_d}{d^{1/2}r} \cap \superLevelSet{\tau}{\regressionFunction} \cap \support(\marginalDistribution) \neq \emptyset \bigr\},
\end{align*}
with the convention that $\min \emptyset := \infty$, and the minimum of a set with no lower bound is $-\infty$. Note that since $\support(\marginalDistribution) \cap \superLevelSet{\tau}{\regressionFunction} \subseteq \bigcup_{j' \in [p]}\closedSupNormMetricBall{z_{j'}}{2^{-1/2}}$, we must have $\ell_{0,j} \in \Z \cup \{\infty\}$.  Let $\mathcal{J}_0 := \{j \in [q]:\ell_{0,j} \in \mathbb{Z}\}$.  By construction, for each $j \in \mathcal{J}_0$ there exists $z_{0,j} \in \closedSupNormMetricBall{x_j+\ell_{0,j}r\cdot \bm{1}_d}{d^{1/2}r} \cap \superLevelSet{\tau}{\regressionFunction} \cap \support(\marginalDistribution)$. Hence $z_{0,j} + r \cdot \bm{1}_d \in  \closedSupNormMetricBall{z_{0,j}}{r} \cap \superLevelSet{\tau+\etaIncreasingConstant \cdot r^{\etaIncreasingExponent}}{\regressionFunction} \subseteq \closedSupNormMetricBall{x_j+\ell_{0,j}r\cdot \bm{1}_d}{(1+d^{1/2})r} \cap \superLevelSet{\tau+\etaIncreasingConstant \cdot r^{\etaIncreasingExponent}}{\regressionFunction}$ since $P \in \distributionClassMultivariateCondition$ and $r \leq 1$.  But for all $x' \in \closedSupNormMetricBall{x_j+\ell \cdot r\cdot \bm{1}_d}{(1+d^{1/2}) \cdot r}$ with $\ell \in \Z \cap [\ell_{0,j}+2(1+d^{1/2}),\infty)$, we have $z_{0,j} + r \cdot \bm{1}_d \preccurlyeq x_j+(\ell_{0,j}+1+d^{1/2}) \cdot r \cdot \bm{1}_d \preccurlyeq  x_j+(\ell-1-d^{1/2}) \cdot r\cdot \bm{1}_d \preccurlyeq x'$.   Since $\eta$ is increasing, we deduce that $\bigcup_{\ell \in \Z \cap [\ell_{0,j}+2(1+d^{1/2}),\infty)} \closedSupNormMetricBall{x_j+\ell r\cdot \bm{1}_d}{(1+d^{1/2}) \cdot r} \subseteq \superLevelSet{\tau+\etaIncreasingConstant \cdot r^{\etaIncreasingExponent}}{\regressionFunction}$.  Next, define 
\begin{align*}
\ell_{1,j}&:=\min\bigl\{ \ell \in \Z \cap [\ell_{0,j}+2(1 + d^{1/2}),\infty) : \closedSupNormMetricBall{x_j+\ell r\cdot \bm{1}_d}{d^{1/2}r} \cap  \support(\marginalDistribution) \neq \emptyset \bigr\},\\
\ell_{2,j}&:=\min\bigl\{ \ell \in \Z \cap [\ell_{1,j}+2(1+ d^{1/2}),\infty) : \closedSupNormMetricBall{x_j+\ell r\cdot \bm{1}_d}{d^{1/2}r} \cap  \support(\marginalDistribution) \neq \emptyset \bigr\},
\end{align*}
and set $\mathcal{J}_k := \{j \in [q]:\ell_{k,j} \in \mathbb{Z}\}$ for $k \in \{1,2\}$.  Then for $j \in \mathcal{J}_1$, there exists $z_{1,j} \in \closedSupNormMetricBall{x_j+\ell_{1,j}r\cdot \bm{1}_d}{d^{1/2} r} \cap  \support(\marginalDistribution)$ with $S^0_j:=\closedSupNormMetricBall{z_{1,j}}{r}\subseteq \closedSupNormMetricBall{x_j+\ell_{1,j} r\cdot \bm{1}_d}{(1+d^{1/2}) \cdot r} \subseteq  \superLevelSet{\tau+\etaIncreasingConstant \cdot r^{\etaIncreasingExponent}}{\regressionFunction}$ and $\marginalDistribution(S_0^j) \geq \densityConstant^{-1} \cdot r^d$, since $P \in \distributionClassMultivariateCondition$ and $r \leq 1$. Similarly, for $j \in \mathcal{J}_2$ there exists $z_{2,j} \in \closedSupNormMetricBall{x_j+\ell_{2,j} r\cdot \bm{1}_d}{d^{1/2} r} \cap  \support(\marginalDistribution)$ with $S^1_j:=\closedSupNormMetricBall{z_{2,j}}{r}\subseteq \closedSupNormMetricBall{x_j+\ell_{2,j}r\cdot \bm{1}_d}{(1+d^{1/2}) \cdot r} \subseteq  \superLevelSet{\tau+\etaIncreasingConstant \cdot r^{\etaIncreasingExponent}}{\regressionFunction}$ and $\marginalDistribution(S_1^j) \geq \densityConstant^{-1} \cdot r^d$. We claim moreover that $S^0_j \preccurlyeq S^1_j$ for each $j \in \mathcal{J}_2$. Indeed, given $a_0 \in S^0_j \subseteq \closedSupNormMetricBall{x_j+\ell_{1,j} r\cdot \bm{1}_d}{(1+d^{1/2})\cdot r}$ and $a_1 \in S^1_j \subseteq \closedSupNormMetricBall{x_j+\ell_{2,j} r\cdot \bm{1}_d}{(1+d^{1/2}) \cdot r}$ we have  $a_0 \preccurlyeq x_j +(\ell_{1,j}+1+d^{1/2}) \cdot r \cdot \bm{1}_d \preccurlyeq x_j+(\ell_{2,j}-1-d^{1/2}) \cdot r \cdot \bm{1}_d \preccurlyeq a_1$, as required. Similarly, $S^1_j \preccurlyeq \bigcup_{\ell \in \Z \cap [\ell_{2,j}+2(1 + d^{1/2}),\infty)}\closedSupNormMetricBall{x_j+\ell \cdot r \cdot \bm{1}_d}{d^{1/2} r}$. Thus, letting 
\begin{align*}
T:= \biggl(\bigcup_{j \in [q], \ell \in \Z \cap [\ell_{2,j}+2(1 + d^{1/2}),\infty)}\closedSupNormMetricBall{x_j+\ell \cdot r \cdot \bm{1}_d}{d^{1/2} r}\biggr) \cap \support(\marginalDistribution) \subseteq \superLevelSet{\tau}{\regressionFunction},
\end{align*}
there exists $j_x \in [q]$ with $S_{j_x}^1 \preccurlyeq \{x\}$ for every $x \in T$. Moreover, since $P \in \distributionClassMultivariateCondition$ we have 
\[
\marginalDistribution\bigl(\superLevelSet{\tau}{\regressionFunction} \setminus T\bigr) \leq  q \cdot 6(1+d^{1/2}) \cdot \lceil d^{1/2} \rceil^d \cdot \densityConstant \cdot (2r)^d \leq 3^{3d+1} \cdot (1+d^{1/2})^{d+1} \cdot \densityConstant^2 \cdot r,
\]
as required.
\end{proof}


\begin{lemma}\label{lemma:ChernoffLemmaForPowerBound}
Fix $\delta \in (0,1]$, $n \in \N$, $\tau \in \R$, $\sigma, \etaIncreasingExponent, \etaIncreasingConstant > 0$, $\densityConstant \in (1,\infty)$, $m \in [n]$ and take $P \in \distributionClassNonDecreasingRegressionFunction \cap \distributionClassMultivariateCondition$.   Fix $r \leq 1$, and let $\bigl((S^0_j,S^1_j)\bigr)_{j \in [q]} \in \bigl(\powerSet(\R^d) \times \powerSet(\R^d)\bigr)^q$ denote the hypercubes in Lemma~\ref{lemma:geomLemmaForPowerBound}.  Let $\sample = \bigl((X_1,Y_1),\ldots,(X_n,Y_n)\bigr) \sim P^n$, and for $j \in [q]$, let
\[
\Omega_{0,j,0} := \biggl\{\frac{1}{n}\sum_{i=1}^n \mathbbm{1}_{\{X_i \in S_j^0\}} \geq \frac{r^d}{2\theta}\biggr\} \quad \text{and} \quad \Omega_{0,j,1} := \biggl\{\frac{1}{m}\sum_{i=1}^m \mathbbm{1}_{\{X_i \in S_j^1\}} \geq \frac{r^d}{2\theta}\biggr\}.
\]
Then, for $m \cdot r^d \geq 8\theta \cdot \log(4 \cdot 9^d \cdot \densityConstant \cdot r^{-(d-1)} \cdot \delta^{-1})$ that
\[
\Prob_P\biggl(\bigcup_{j \in [q],k \in \{0,1\}} \Omega_{0,j,k}^c \biggr) \leq \frac{\delta}{2}.
\]
\end{lemma}
\begin{proof}
By Lemma~\ref{lemma:geomLemmaForPowerBound}, we have $\marginalDistribution(S_j^k) \geq r^d/\theta$ for every $j \in [q]$ and $k \in \{0,1\}$, and moreover, $q \leq  9^d \cdot \densityConstant \cdot r^{-(d-1)}$.  Hence, by the multiplicative Chernoff bound \citep[][Theorem~2.3(c)]{mcdiarmid1998concentration}, we have
\[
\Prob_P\biggl(\bigcup_{j \in [q],k \in \{0,1\}} \Omega_{0,j,k}^c \biggr) \leq 2q \cdot \exp\biggl(-\frac{m\cdot r^d}{8\theta}\biggr) \leq \frac{\delta}{2},
\]
as required.
\end{proof}

\begin{lemma}\label{lemma:HoeffdingLemmaForPowerBound_v2} 
Fix $\delta \in (0,1]$, $n \in \N$, $\tau \in \R$, $\sigma, \etaIncreasingExponent, \etaIncreasingConstant > 0$, $\densityConstant \in (1,\infty)$, $\alpha \in (0,1)$, $m \in [n]$, $\ell \in [m]$ and take $P \in \distributionClassNonDecreasingRegressionFunction \cap \distributionClassMultivariateCondition$.   Fix $r \leq 1$, and let $(S^0_j)_{j \in [q]} \in \bigl(\powerSet(\R^d)\bigr)^q$ be as in Lemma~\ref{lemma:geomLemmaForPowerBound}.  Let $\sample = \bigl((X_1,Y_1),\ldots,(X_n,Y_n)\bigr) \sim P^n$, and for each $s \in [\ell]$, find $(x_s,j_s) \in \R^d \times [q]$ such that $\{x_s\} \succcurlyeq S_{j_s}^0$.  Now let $r_s := \sup_{x \in S_{j_s}^0} \|x - x_s\|_\infty$, and let
\[
\Omega_{1,s}(x_s) := \biggl\{\frac{1}{|\mathcal{I}_{r_s}(x_s)|}\sum_{i \in \mathcal{I}_{r_s}(x_s)} Y_i \geq \tau + \sigma\sqrt{\frac{2.0808\log (5.2m/\alpha)+3\log\log (2n)}{|\mathcal{I}_{r_s}(x_s)|}}\biggr\}.
\]
If $\min_{s \in [\ell]} |\mathcal{I}_{r_s}(x_s)| \geq \frac{2\sigma^2}{\etaIncreasingConstant^2 \cdot r^{2\etaIncreasingExponent}}\bigl\{2\log(2m/\delta) + 2.0808\log (5.2m/\alpha)+3\log\log (2n)\bigr\}$, then
\[
\Prob_P\biggl(\bigcup_{s=1}^\ell \Omega_{1,s}(x_s)^c  \biggm| \sampleX\biggr) \leq \frac{\delta}{2}.
\]
\end{lemma}
\begin{proof}
As shorthand, write $w_{n,m,\alpha} := 2.0808\log (5.2m/\alpha)+3\log\log (2n)$.  Then, by Hoeffding's inequality,
\[
\Prob_P\biggl(\bigcup_{s=1}^\ell \Omega_{1,s}(x_s)^c  \biggm| \sampleX\biggr) \leq \sum_{s=1}^\ell \exp\biggl\{-\frac{|\mathcal{I}_{r_s}(x_s)|}{2\sigma^2}\biggl(\lambda \cdot r^\gamma - \sigma\sqrt{\frac{w_{n,m,\alpha}}{|\mathcal{I}_{r_s}(x_s)|}}\biggr)^2\biggr\} \leq \frac{\delta}{2},
\]
where we have used the fact that $(a+b)^{1/2} \geq (a/2)^{1/2} + (b/2)^{1/2}$ for $a,b \geq 0$.
\end{proof}

\subsection{Proofs from Section~\ref{sec:lowerBound}}\label{sec:lowerBound_proofs}

The proof of Theorem~\ref{Thm:LowerBound} involves combining three different lower bounds that emphasise different aspects of the challenge in isotonic subgroup selection.  However, there are some commonalities to these three lower bound constructions, so we explain the key ideas here.  Fix $q \in \mathbb{N}$, and note that if $x=(x_1,\ldots,x_d), x'=(x_1',\ldots,x_d') \in [q]^d$ and $x \succcurlyeq x'$, then the length of any chain from $x$ to $x'$ is at most $\sum_{j=1}^d (x_j - x_j') \leq (q-1) \cdot d$, because successive elements within the chain must decrease at least one coordinate by at least 1.  Now let $\mathbb{W}_{q,d} \subseteq [q]^d$ be an antichain  of maximal cardinality.  Dilworth's theorem \citep{dilworth1950decomposition} states that we can partition $[q]^d$ into $|\mathbb{W}_{q,d}|$ chains, so\footnote{In fact, this bound is fairly sharp.  Indeed, define the \emph{width} of a partially ordered set $R$, denoted $\mathrm{wd}(R)$, to be the maximum cardinality of an antichain in $R$.  Then, for any two finite partially ordered sets $(R_1,\preccurlyeq_1), (R_2,\preccurlyeq_2)$, we have $\mathrm{wd}(R_1 \times R_2) \leq \min\{ |R_1|\cdot \mathrm{wd}(R_2), |R_2|\cdot \mathrm{wd}(R_1)\}$, where the Cartesian product $R_1 \times R_2$ is equipped with the order relation $\preccurlyeq_{1 \times 2}$, where for $r_1,r_1' \in R_1$ and $r_2,r_2' \in R_2$, we define $(r_1,r_2) \preccurlyeq_{1 \times 2} (r_1',r_2')$ if and only if $r_1 \preccurlyeq_1 r_1'$ and $r_2 \preccurlyeq_2 r_2'$.  It therefore follows by induction that $|\mathbb{W}_{q,d}| \leq q^{d-1}$.}
$|\mathbb{W}_{q,d}| \geq q^d/\{(q-1) \cdot d\} \geq q^{d-1}/d$.  For each $\bm{j}=(j_1,\ldots,j_d) \in \mathbb{W}_{q,d}$, define a hypercube
\begin{align*}
\mathcal{H}_{\bm{j}}^q:= \prod_{\ell \in [d]}\biggl[\frac{{j_\ell}-1}{q},\frac{{j_\ell}}{q}\biggr). 
\end{align*}
We also set
\begin{align*}
\mathcal{H}_{\infty}^q:=\bigcup_{\bm{j}\in \mathbb{W}_{q,d}} \bigl\{ x \in \R^d \setminus \mathcal{H}_{\bm{j}}^q : x' \preccurlyeq x \text{ for some }x' \in \mathcal{H}_{\bm{j}}^q \bigr\},  \end{align*}
and let $\mathcal{H}_{-\infty}^q:=\R^d \setminus \bigcup_{\bm{j} \in \mathbb{W}_{q,d} \cup \{\infty\}}\mathcal{H}_{\bm{j}}^q$.  By Lemma~\ref{lemma:partition}, the sets $\bigl\{\mathcal{H}^q_{\bm{j}}:\bm{j} \in \mathbb{W}_{q,d} \cup \{-\infty,\infty\}\bigr\}$ form a partition of $\R^d$.  For each $S \subseteq \mathbb{W}_{q,d}$ and for $\tau \in \R$, $\etaIncreasingExponent, \etaIncreasingConstant > 0$, define $\eta_S:\mathbb{R}^d \rightarrow \mathbb{R}$ by
\begin{align}\label{def:regFuncDef}
\eta_S(x) \equiv \eta_{S, \mathbb{W}_{q,d}, q, \tau, \etaIncreasingExponent, \etaIncreasingConstant}(x) :=\begin{cases} \tau -\frac{\etaIncreasingConstant}{q^{\etaIncreasingExponent}}  & \text{ if }x \in \bigcup_{\bm{j} \in \left(\mathbb{W}_{q,d} \cup \{-\infty\}\right) \setminus S}\mathcal{H}_{\bm{j}}^q\\
\tau+\frac{\etaIncreasingConstant}{q^{\etaIncreasingExponent}} &\text{ if }x \in \bigcup_{\bm{j} \in  S}\mathcal{H}_{\bm{j}}^q\\
\tau + \etaIncreasingConstant &\text{ if $x \in \mathcal{H}_{\infty}^q$.}\end{cases}  
\end{align}
The intuition is that if $\bm{j}, \bm{j}'$ are distinct elements of $\mathbb{W}_{q,d}$, then the response at any $x \in \mathcal{H}_{\bm{j}}^q$ provides no information on whether or not $x' \in \mathcal{H}_{\bm{j}'}^q$ belongs to $\superLevelSet{\tau}{\eta}$.  Any data-dependent selection set will therefore struggle to identify $S$ from the data, and since $\mathbb{W}_{q,d}$ is a large antichain, the $\mu$-measure of this difficult set may be quite large. 

\begin{lemma}\label{lemma:partition}
For any $d, q \in \N$ and antichain $\mathbb{W}_{q,d} \subseteq [q]^d$, the sets $\bigl\{\mathcal{H}^q_{\bm{j}}:\bm{j} \in \mathbb{W}_{q,d} \cup \{-\infty,\infty\}\bigr\}$ form a partition of $\R^d$.
\end{lemma}
\begin{proof}
The fact that these sets cover $\R^d$ follows by definition of $\mathcal{H}_{-\infty}^q$.  Since the sets $\bigl\{\mathcal{H}_{\bm{j}}^q:\bm{j} \in \mathbb{W}_{q,d} \cup \{-\infty\}\bigr\}$ are disjoint, and $\mathcal{H}_{-\infty}^q \cap \mathcal{H}_{\infty}^q = \emptyset$, we need only check that $\mathcal{H}_{\bm{j}}^q \cap \mathcal{H}_{\infty}^q = \emptyset$ when $\bm{j} \in \mathbb{W}_{q,d}$.  To this end, suppose for a contradiction that $x \in \mathcal{H}_{\bm{j}}^q \cap \mathcal{H}_{\infty}^q$ for some $\bm{j} \in \mathbb{W}_{q,d}$, and take $\bm{j}' \in \mathbb{W}_{q,d} \setminus \{\bm{j}\}$ and $x \in \R^d \setminus \mathcal{H}_{\bm{j}'}^q$ but $x' \preccurlyeq x$ for some $x' \in \mathcal{H}_{\bm{j}'}^q$.  Then there exists $\delta \in (0,1)$ such that $\bm{j}'-\bm{1}_d \preccurlyeq q \cdot x' \preccurlyeq q\cdot x \preccurlyeq \bm{j} - \delta \cdot \bm{1}_d$ so that $\bm{j}' - (1-\delta)\cdot \bm{1}_d \preccurlyeq \bm{j}$.  But the coordinates of $\bm{j}$ and $\bm{j}'$ are positive integers, so we must have $\bm{j}' \preccurlyeq \bm{j}$, which contradicts $\mathbb{W}_{q,d}$ being an antichain.
\end{proof}

\begin{figure}

    \tikzdeclarepattern{
      name=mylines,
      parameters={
          \pgfkeysvalueof{/pgf/pattern keys/size},
          \pgfkeysvalueof{/pgf/pattern keys/angle},
          \pgfkeysvalueof{/pgf/pattern keys/line width},
      },
      bounding box={
        (0,-0.5*\pgfkeysvalueof{/pgf/pattern keys/line width}) and
        (\pgfkeysvalueof{/pgf/pattern keys/size},
         0.5*\pgfkeysvalueof{/pgf/pattern keys/line width})},
      tile size={(\pgfkeysvalueof{/pgf/pattern keys/size},
                  \pgfkeysvalueof{/pgf/pattern keys/size})},
      tile transformation={rotate=\pgfkeysvalueof{/pgf/pattern keys/angle}},
      defaults={
        size/.initial=5pt,
        angle/.initial=45,
        line width/.initial=.4pt,
    }, code={
    \draw [line width=\pgfkeysvalueof{/pgf/pattern keys/line width}] (0,0) -- (\pgfkeysvalueof{/pgf/pattern keys/size},0);
    }, }
    
    \definecolor{superlevelsetcolor}{rgb}{0.1523438, 0.5039062, 0.8984375} 
    \definecolor{sublevelsetcolor}{rgb}{0.7, 0.7, 0.7}  
    

    \tikzstyle{superlevel}=[draw = none, fill = superlevelsetcolor]
    \tikzstyle{superlevel boundary}=[draw = none, fill = superlevelsetcolor!50]
    \tikzstyle{sublevel}=[draw = none, fill = sublevelsetcolor]
    
    \tikzstyle{superlevel no support} = [draw = none]
    \tikzstyle{sublevel no support} = [draw = none]
    
    \tikzstyle{black line}=[line cap = rect, line width = 0.2pt]
    \tikzstyle{dashed line}=[line cap = rect, line width = 0.2pt, dashed]
    
    \tikzstyle{axis} = [->, thick]
     \centering
     \begin{subfigure}[b]{0.45\textwidth}
         \centering
         \resizebox{\textwidth}{!}{\begin{tikzpicture}

    \clip (-1,-1) rectangle (6, 6);

    \draw[draw = none] (-6, -6) rectangle (6, 6) {};

    \draw[superlevel] (1, 5) -- (1, 4) -- (2, 4) -- (2, 3) -- (3, 3) -- (3, 2) -- (4, 2) -- (4, 1) -- (5, 1) -- (5, 5) -- cycle;
    \draw[superlevel no support] (0, 15) -- (0, 5) -- (5, 5) -- (5, 0) -- (15, 0) -- (15, 15) -- cycle;

    \draw[superlevel boundary] (0, 5) rectangle (1, 4) {};
    \draw[superlevel boundary] (3, 2) rectangle (4, 1) {};
    \draw[superlevel boundary] (4, 1) rectangle (5, 0) {};
    \draw[sublevel] (4, 0) -- (4, 1) -- (3, 1) -- (3, 2) -- (2, 2) -- (2, 3) -- (1, 3) -- (1, 4) -- (0, 4) -- (0, 0) -- cycle;
    \draw[sublevel] (1, 4) rectangle (2, 3) {} ;
    \draw[sublevel] (2, 3) rectangle (3, 2) {} ;
    \draw[sublevel no support] (-15, 15) -- (-15, -15) -- (15, -15) -- (15, 0) -- (0,0) -- (0, 15) -- cycle;

    \draw[axis] (0, -6) -> (0, 6) {};
    \draw[axis] (-6, 0) -> (6, 0) {};

    \node at (0.5, 4.5) {\small $\mathcal{H}_{(1,5)}^5$};
    \node at (1.5, 3.5) {\small $\mathcal{H}_{(2,4)}^5$};
    \node at (2.5, 2.5) {\small $\mathcal{H}_{(3,3)}^5$};
    \node at (3.5, 1.5) {\small $\mathcal{H}_{(4,2)}^5$};
    \node at (4.5, 0.5) {\small $\mathcal{H}_{(5,1)}^5$};


    \draw[black line] (0, 0) rectangle (5, 5) {};

    \draw[dashed line] (0, 4) -- (2, 4) {};
    \draw[dashed line] (2, 4) -- (2, 2);
    \draw[dashed line] (2, 2) -- (4, 2);
    \draw[dashed line] (4, 2) -- (4, 0.0);

    \draw[dashed line] (1, 5) -- (1, 3) {};
    \draw[dashed line] (1, 3) -- (3, 3);
    \draw[dashed line] (3, 3) -- (3, 1);
    \draw[dashed line] (3, 1) -- (5, 1);

\end{tikzpicture}}
     \end{subfigure}
     \hspace{0.5cm}
     \begin{subfigure}[b]{0.45\textwidth}
         \centering
         \resizebox{\textwidth}{!}{\begin{tikzpicture}

    \clip (-5,-4) rectangle (6, 6);

    \draw[draw = none] (-6, -6) rectangle (6, 6) {};

    \draw[superlevel] (1, 5) -- (1, 4) -- (2, 4) -- (2, 3) -- (3, 3) -- (3, 2) -- (4, 2) -- (4, 1) -- (5, 1) -- (5, 5) -- cycle;
    \draw[superlevel no support] (0, 15) -- (0, 5) -- (5, 5) -- (5, 0) -- (15, 0) -- (15, 15) -- cycle;

    \draw[superlevel boundary] (0, 5) rectangle (1, 4) {};
    \draw[superlevel boundary] (3, 2) rectangle (4, 1) {};
    \draw[superlevel boundary] (4, 1) rectangle (5, 0) {};
    \draw[sublevel] (4, 0) -- (4, 1) -- (3, 1) -- (3, 2) -- (2, 2) -- (2, 3) -- (1, 3) -- (1, 4) -- (0, 4) -- (0, 0) -- cycle;
    \draw[sublevel no support] (1, 4) rectangle (2, 3) {} ;
    \draw[sublevel no support] (2, 3) rectangle (3, 2) {} ;
    \draw[sublevel no support] (-15, 15) -- (-15, -15) -- (15, -15) -- (15, 0) -- (0,0) -- (0, 15) -- cycle;
    \draw[sublevel] (-4, -1) rectangle (-3, -2) {};
    \draw[sublevel] (-3, -2) rectangle (-2, -3) {};
    \draw[dashed line] (-4, -1) rectangle (-3, -2) {};
    \draw[dashed line] (-3, -2) rectangle (-2, -3) {};

    \draw[axis] (0, -6) -> (0, 6) {};
    \draw[axis] (-6, 0) -> (6, 0) {};

    \node at (0.5, 4.5) {\small $\mathcal{H}_{(1,5)}^5$};
    \node at (1.5, 3.5) {\small $\mathcal{H}_{(2,4)}^5$};
    \node at (2.5, 2.5) {\small $\mathcal{H}_{(3,3)}^5$};
    \node at (3.5, 1.5) {\small $\mathcal{H}_{(4,2)}^5$};
    \node at (4.5, 0.5) {\small $\mathcal{H}_{(5,1)}^5$};

    \node at (-2.3, -0.7) {\small $\mathcal{H}_{(2,4)}^5 - \bm{1}_2$};
    \draw [->] (-2.25, -1.05) to [out = 270, in = 0] (-2.85, -1.5);
    
    \node at (-1.3, -1.7) {\small $\mathcal{H}_{(3,3)}^5 - \bm{1}_2$};
    \draw [->] (-1.25, -2.05) to [out = 270, in = 0] (-1.85, -2.5);

    \draw[black line] (0, 0) rectangle (5, 5) {};

    \draw[dashed line] (0, 4) -- (2, 4) {};
    \draw[dashed line] (2, 4) -- (2, 2);
    \draw[dashed line] (2, 2) -- (4, 2);
    \draw[dashed line] (4, 2) -- (4, 0.0);

    \draw[dashed line] (1, 5) -- (1, 3) {};
    \draw[dashed line] (1, 3) -- (3, 3);
    \draw[dashed line] (3, 3) -- (3, 1);
    \draw[dashed line] (3, 1) -- (5, 1);

\end{tikzpicture}}
     \end{subfigure}
     
        \caption{Lower bound constructions in the proofs of Propositions~\ref{prop:lowerBound_alpha_multivariate} (left) and~\ref{prop:lowerBoundParametric} (right).  The grey regions do not belong to $\superLevelSet{\tau}{\eta}$, while light blue and dark blue regions correspond to areas where $\regressionFunction$ is slightly above and comfortably above $\tau$ respectively.  White regions have no marginal mass.  In both panels, $q=5$, $d=2$, $\mathbb{W}_{q,d} = \{(1,5),(2,4),(3,3),(4,2),(5,1)\}$ and $S = \{(1,5),(4,2),(5,1)\}$.}
        \label{fig:lowerBoundConstructions}
\end{figure}

\begin{prop}\label{prop:lowerBound_alpha_multivariate}
Fix $d\in\N$, $\alpha \in (0, 1/4]$, $\tau \in \R$, $\sigma, \etaIncreasingExponent, \etaIncreasingConstant > 0$ and $\densityConstant > 1$. For any $n \geq (8\cdot 2^d)^{(2\etaIncreasingExponent +d)/(2\etaIncreasingExponent)}\bigl(\frac{13\etaIncreasingConstant^2}{\sigma^2 \log_+\{1/(5\alpha)\}}\bigr)^{d/(2\etaIncreasingExponent)}$, we have
\[
\sup_{P \in \mathcal{P}'} \inf_{\hat{A}\in \hat{\mathcal{A}}_n(\tau,\alpha,\mathcal{P}')} \E_P\bigl\{\mu\bigl(\superLevelSet{\tau}{\eta}\setminus \hat{A}(\sample)\bigr)\bigr\} \geq \frac{1}{40d}\biggl\{\frac{\sigma^2}{13n\etaIncreasingConstant^2}\log_+\Bigl(\frac{1}{5\alpha}\Bigr) \wedge 1\biggr\}^{1/(2\etaIncreasingExponent + d)},
\]
where $\mathcal{P}' := \distributionClassNonDecreasingRegressionFunction[d] \cap \distributionClassMultivariateCondition[d]$. 
\end{prop}
\begin{proof}
Suppose first that 
\[
n \geq \frac{\sigma^2}{13\etaIncreasingConstant^2}\log_+\Bigl(\frac{1}{5\alpha}\Bigr),
\]
so that 
\begin{align}
q := \Biggl\lceil \biggl\{\frac{13n\etaIncreasingConstant^2}{\sigma^2\log_+\bigl(1/(5\alpha)\bigr)}\biggr\}^{1/(2\etaIncreasingExponent + d)}\Biggr\rceil \leq 2\biggl\{\frac{13n\etaIncreasingConstant^2}{\sigma^2\log_+\bigl(1/(5\alpha)\bigr)}\biggr\}^{1/(2\etaIncreasingExponent + d)}.
\label{Eq:lowerBound_alpha_qChoice}    
\end{align}
Let $\mathbb{W}_{q,d} \subseteq [q]^d$ be an antichain with $|\mathbb{W}_{q,d}| \geq q^{d-1}/d$.  For each $S \subseteq \mathbb{W}_{q,d}$, let $P_S$ denote the joint distribution of $(X,Y)$, where $X \sim \mathrm{Unif}\bigl([0,1]^d\bigr) =: \mu$ and $Y|X \sim \mathcal{N}\bigl(\eta_S(X),\sigma^2\bigr)$, with $\eta_S$ defined by \eqref{def:regFuncDef}. Then, by Lemma~\ref{lemma:antichain_distributions_properties}, $P_S \in \mathcal{P}'$ for every $S\subseteq \mathbb{W}_{q,d}$. For ease of notation, we write $S^* := \mathbb{W}_{q,d}$ and $S^*_{-\bm{j}} := \mathbb{W}_{q,d}\setminus\{\bm{j}\}$ for $\bm{j} \in \mathbb{W}_{q,d}$, so that $\mu\bigl(\superLevelSet{\tau}{\eta_{S^*}}\setminus\superLevelSet{\tau}{\eta_{S^*_{-\bm{j}}}}\bigr) = \mu(\mathcal{H}_{\bm{j}}^q) = 1/q^d$ for all $\bm{j} \in \mathbb{W}_{q,d}$. 
Hence, for any Borel set $A\subseteq \R^d$,
\begin{align}
\mu\bigl(\superLevelSet{\tau}{\eta_{S^*}}\setminus A\bigr) \geq \sum_{\bm{j} \in S^*} \frac{1}{q^d}\one_{\{A \cap \mathcal{H}_{\bm{j}}^q = \emptyset\}}.
\label{eq:lower_bound_error_additivity}
\end{align}
Note that by the upper bound on $q$ in~\eqref{Eq:lowerBound_alpha_qChoice} and the lower bound on $n$ in the statement of the proposition we have
\[
\frac{n}{q^d} \geq \frac{n^{2\etaIncreasingExponent/(2\etaIncreasingExponent + d)}}{2^d}\cdot \Bigl\{\frac{\sigma^2 \log_+\bigl(1/(5\alpha)\bigr)}{13\etaIncreasingConstant^2}\Bigr\}^{d/(2\etaIncreasingExponent + d)} \geq 8.
\]
Moreover, by~\eqref{Eq:lowerBound_alpha_qChoice},
\[
\Delta := \frac{2\etaIncreasingConstant}{q^\etaIncreasingExponent} \leq \frac{\sigma}{\sqrt{3.2n/q^d}}\log_+^{1/2}\Bigl(\frac{1}{5\alpha}\Bigr)
\]
Fix $\bm{j} \in S^*$ and a data-dependent selection set $\hat{A} \in \hat{\mathcal{A}}_n(\tau, \alpha, \mathcal{P}')$.  Now define $\psi_{\hat{A}}(\sample) := \one_{\{\hat{A}(\sample) \cap \mathcal{H}_{\bm{j}}^q \neq \emptyset\}}$, which satisfies $\mathbb{P}_{P_{S^*_{-\bm{j}}}}\bigl(\psi_{\hat{A}}(\sample) = 1\bigr) \leq \alpha$.   Then, by Corollary~\ref{Cor:GaussianTesting2} with $t = \tau - \etaIncreasingConstant/q^\etaIncreasingExponent$ and $p = 1/q^d$, we have
\[
\mathbb{P}_{P_S^*}\bigl(\hat{A}(\sample) \cap \mathcal{H}_{\bm{j}}^q = \emptyset \bigr) = \mathbb{P}_{P_S^*}\bigl(\psi_{\hat{A}}(\sample) = 0 \bigr) \geq \frac{1}{20}.
\]
In combination with~\eqref{eq:lower_bound_error_additivity}, we deduce that 
\begin{align*}
    \sup_{P\in\mathcal{P}'} \inf_{\hat{A}\in \hat{\mathcal{A}}_n(\tau,\alpha,\mathcal{P}')} \mathbb{E}_P\bigl\{\mu\bigl(\superLevelSet{\tau}{\eta} \setminus \hat{A}(\sample)\bigr)\bigr\} &\geq  \inf_{\hat{A}\in \hat{\mathcal{A}}_n(\tau,\alpha,\mathcal{P}')} \mathbb{E}_{P_{S^*}}\bigl\{\mu\bigl(\superLevelSet{\tau}{\eta_{S^*}} \setminus \hat{A}(\sample)\bigr)\bigr\}\\
    &\geq  \inf_{\hat{A}\in \hat{\mathcal{A}}_n(\tau,\alpha,\mathcal{P}')} \frac{1}{q^d}\sum_{\bm{j} \in \mathbb{W}_{q,d}}\mathbb{P}_{P_{S^*}}\bigl\{\hat{A}(\sample) \cap \mathcal{H}_{\bm{j}}^q = \emptyset\bigr\} \\
    &\geq \frac{|\mathbb{W}_{q,d}|}{20q^d} \geq \frac{1}{40d}\biggl\{\frac{\sigma^2}{13n\etaIncreasingConstant^2}\log_+\Bigl(\frac{1}{5\alpha}\Bigr)\biggr\}^{1/(2\etaIncreasingExponent + d)}.
\end{align*}
Finally, if
\[
n < \frac{\sigma^2}{13\etaIncreasingConstant^2}\log_+\Bigl(\frac{1}{5\alpha}\Bigr),
\]
then
\[
\frac{1}{40d}\biggl\{\frac{\sigma^2}{13n\etaIncreasingConstant^2}\log_+\Bigl(\frac{1}{5\alpha}\Bigr)\biggr\}^{1/(2\etaIncreasingExponent + d)} \geq \frac{1}{40d},
\]
as required.
\end{proof}

\begin{lemma}\label{lemma:antichain_distributions_properties}
For any $d, q\in \N$, $\tau \in \R$, $\sigma, \etaIncreasingExponent, \etaIncreasingConstant > 0$, $\densityConstant > 1$, an antichain $\mathbb{W}_{q,d} \subseteq [q]^d$ and $S\subseteq \mathbb{W}_{q,d}$, we have that $P_S\equiv P_{S, \mathbb{W}_{q,d}, q, \sigma, \tau, \etaIncreasingExponent, \etaIncreasingConstant}$ defined as in the proof of Proposition~\ref{prop:lowerBound_alpha_multivariate} satisfies $P_S \in \distributionClassNonDecreasingRegressionFunction[d] \cap \distributionClassMultivariateCondition[d]$.
\end{lemma}

\begin{proof}
Fix $S \subseteq \mathbb{W}_{q,d}$.  We first prove that $P_S \in \distributionClassNonDecreasingRegressionFunction[d]$.  Since the sub-Gaussianity condition is satisfied, it suffices to show that $\eta_S$ is coordinate-wise increasing in $\R^d$.  To this end, first note that for $x_0 \in \bigcup_{\bm{j} \in (\mathbb{W}_{q,d} \cup \{-\infty\}) \setminus S}\mathcal{H}_{\bm{j}}^q$ and $x_1 \succcurlyeq x_0$, we have $\eta_S(x_0) \leq \eta_S(x_1)$ since $\eta_S(x_0) = \inf_{x\in\R^d} \eta_S(x)$. Next, consider the case $x_0 \in \bigcup_{\bm{j} \in S} \mathcal{H}_{\bm{j}}^q$ and let $\bm{j}_0 \in S$ be such that $x_0 \in \mathcal{H}_{\bm{j}_0}^q$.  If $x_1 \succcurlyeq x_0$, then either $x_1 \in \mathcal{H}_{\bm{j}_0}^q$, in which case $\eta_S(x_0) = \eta_S(x_1)$, or $x_1 \in \R^d\setminus \mathcal{H}_{\bm{j}_0}^q$, in which case $x_1 \in \mathcal{H}_\infty^q$ so that $\eta_S(x_0) = \tau + \etaIncreasingConstant/q^\etaIncreasingExponent \leq \tau + \etaIncreasingConstant = \eta_S(x_1)$.  Finally, suppose that $x_0 \in \mathcal{H}_\infty^q$ and find $\bm{j}_0 = (j_{0,1},\ldots,j_{0,d}) \in \mathbb{W}_{q,d}$ and $x' \preccurlyeq x_0$ such that $x_0 \in \R^d \setminus \mathcal{H}_{\bm{j}_0}^q$ and $x' \in \mathcal{H}_{\bm{j}_0}^q$. The fact that $\bm{j}_0 - \bm{1}_d \preccurlyeq q\cdot x' \preccurlyeq q\cdot x_0$ in conjunction with $x_0 = (x_{0,1}, \ldots, x_{0,d})^\top \in \R^d \setminus \mathcal{H}_{\bm{j}_0}^q$ means that there exists $\ell_0 \in [d]$ such that $j_{0,\ell_0} \leq q\cdot x_{0,\ell_0}$. Thus, for any $x_1 \succcurlyeq x_0$, it follows that $x_1 \in \R^d \setminus \mathcal{H}_{\bm{j}_0}^q$. Moreover, $x' \preccurlyeq x_0 \preccurlyeq x_1$, so that $x_1 \in \mathcal{H}_\infty^q$, whence $\eta_S(x_0) = \eta_S(x_1)$. This establishes that $P_S \in \distributionClassNonDecreasingRegressionFunction[d]$.

We now show that $P_S \in \distributionClassMultivariateCondition$, which requires us to verify the conditions in Definition~\ref{def:multivariateAssumption}\emph{(i)} and~\emph{(ii)}. We start by showing \emph{(i)}.  Indeed, for any $x_0 \in [0,1]^d = \support(\mu)$ and $r \in (0,1]$, we have $r^d \leq \mu\bigl(\closedSupNormMetricBall{x_0}{r}\bigr) \leq (2r)^d$. This establishes the condition in Definition~\ref{def:multivariateAssumption}\emph{(i)}.

It remains to show that Definition~\ref{def:multivariateAssumption}\emph{(ii)} is satisfied.  Note that $\superLevelSet{\tau}{\eta_S} = \bigcup_{\bm{j}\in S} \mathcal{H}_{\bm{j}}^q \cup \mathcal{H}_\infty^q$.  Suppose first that $x_0 \in \mathcal{H}_{\bm{j}_0}^q$ for some $\bm{j}_0 \in S$.  If $r \leq 1/q$, then $\tau + \etaIncreasingConstant \cdot r^\etaIncreasingExponent \leq \tau + \etaIncreasingConstant/q^\etaIncreasingExponent = \eta_S(x_0)$.    On the other hand, if $r \in (1/q, 1]$, then $x_1 := x_0 + r\cdot \bm{1}_d \succcurlyeq (\bm{j}_0 - \bm{1}_d)/q + r\cdot \bm{1}_d \succcurlyeq \bm{j}_0/q$, so that $x_1 \in \mathcal{H}_\infty^q$. Since $x_1 \in \closedSupNormMetricBall{x_0}{r}$ and $\eta_S(x_1) = \tau + \etaIncreasingConstant \geq \tau + \etaIncreasingConstant\cdot r^\etaIncreasingExponent$ for all $r\in (0, 1]$ the claim is shown for $x_0 \in \bigcup_{\bm{j}\in S} \mathcal{H}_{\bm{j}}^q$.  Now, suppose $x_0 \in \mathcal{H}_\infty^q$. Then, similarly to before, $\eta_S(x_0) = \tau + \etaIncreasingConstant \geq \tau + \etaIncreasingConstant\cdot r^\etaIncreasingExponent$ for all $r\in (0, 1]$. This establishes that $P_S \in \distributionClassMultivariateCondition[d]$ and hence completes the proof.
\end{proof}

Our second construction proceeds similarly, but we now also vary the marginal distribution. 
\begin{prop}\label{prop:lowerBoundParametric}
    Let $d \in \N$, $\tau \in \R$, $\sigma, \etaIncreasingExponent, \etaIncreasingConstant > 0$ and $\densityConstant > 1$. Then, writing  $\mathcal{P}' := \distributionClassNonDecreasingRegressionFunction[d] \cap \distributionClassMultivariateCondition[d]$, we have for any $n\in\N$ and $\alpha \in (0, 1/4]$ that
    \[
    \inf_{\hat{A} \in \hat{\mathcal{A}}_n(\tau, \alpha, \mathcal{P}')} \sup_{P \in \mathcal{P}'} \mathbb{E}_P \bigl\{\marginalDistribution\bigl(\superLevelSet{\tau}{\regressionFunction}\setminus \hat{A}(\sample)\bigr)\bigr\} \geq \frac{1}{352 \cdot d \cdot 4^{1/d}} \cdot  \frac{1}{n^{1/d}}.
    \]
\end{prop}

\begin{proof}
Let $q:= \lceil (4 n )^{1/d}\rceil$ and let $\mathbb{W}_{q,d} \subseteq [q]^d$ be an antichain with $|\mathbb{W}_{q,d}| \geq q^{d-1}/d$. For each $S \subseteq \mathbb{W}_{q,d}$, we define a Borel subset $J_S$ of $\R^d$ by
\begin{align*}
J_S:= \biggl\{ [0,1]^d \setminus \biggl( \bigcup_{ \bm{j} \in \mathbb{W}_{q,d} \setminus S} \mathcal{H}^q_{\bm{j}} \biggr) \biggr\} \cup \biggl\{ \bigcup_{ \bm{j} \in \mathbb{W}_{q,d} \setminus S}  (\mathcal{H}^q_{\bm{j}} - \bm{1}_d)\biggr\},
\end{align*}
where $\bm{1}_d \in \R^d$ denotes the all-ones vector; see the right-hand panel in Figure~\ref{fig:lowerBoundConstructions} for an illustration.  Note that $\Lebesgue(J_S)=1$, so we can define a Borel probability measure $\mu_S$ on $\R^d$ by $\mu_S(A) := \Lebesgue(A\cap J_S)$ for Borel sets $A \subseteq \R^d$.  Now let $P_S$ denote the joint distribution of $(X,Y)$, where $X \sim \mu_S$ and $Y|X \sim \mathcal{N}\bigl(\eta_S(X),\sigma^2\bigr)$, with $\eta_S$ defined by \eqref{def:regFuncDef}.  Then $P_S \in \distributionClassNonDecreasingRegressionFunction[d]\cap \distributionClassMultivariateCondition[d]$ by Lemma~\ref{lemma:antichain_distributions_properties_varyingMarginal}. Given $\hat{A} \in \hat{\mathcal{A}}_n(\tau, \alpha, \mathcal{P}')$, let 
\begin{align*}
\hat{S}:= 
 \bigl\{ \bm{j} \in \mathbb{W}_{q,d} : \hat{A}\cap \mathcal{H}_{\bm{j}}^q \neq \emptyset\bigr\}.
\end{align*}
Note that for each $S_0$, $S_1 \subseteq \mathbb{W}_{q,d}$ with $|S_0\triangle S_1|=1$, we have by Lemma~\ref{lemma:TVProductMeasure}\emph{(b)} that
\begin{align*}
\mathrm{TV}\bigl(P_{S_0}^n,P_{S_1}^n\bigr) \leq n \cdot\mathrm{TV}\bigl(P_{S_0},P_{S_1}\bigr) = \frac{n}{q^d}.
\end{align*}
Thus, by Assouad's lemma again, there exists $S_* \subseteq \mathbb{W}_{q,d}$ such that
\begin{align*}
\E_{P_{S_*}}\bigl( |\hat{S}\triangle S_*|\bigr) \geq \frac{|\mathbb{W}_{q,d}|}{2} \cdot \biggl( 1-\frac{n}{ q^d }\biggr) \geq \frac{3|\mathbb{W}_{q,d}|}{8},
\end{align*}
by the choice of $q$.  Hence, writing $Z := |\hat{S}\triangle S_*|/|\mathbb{W}_{q,d}|$ and $E := \{Z \geq 1/11\}$, we have
\begin{align*}
\Prob_{P_{S_*}}(E) = \frac{11}{10}\biggl\{\mathbb{P}_{S_*}(E) + \frac{1}{11}\mathbb{P}_{S_*}(E^c) - \frac{1}{11}\biggr\} \geq \frac{11}{10}\biggl\{\mathbb{E}_{P_{S_*}}(Z\mathbbm{1}_{E}) + \mathbb{E}_{P_{S_*}}(Z\mathbbm{1}_{E^c}) - \frac{1}{11}\biggr\} \geq \frac{5}{16}.
\end{align*}
Now $\Prob_{P_{S_*}}(\hat{S}\subseteq S_*) \geq 3/4$ because $\alpha \in (0,1/4]$, so
\begin{align*}
\Prob_{P_{S_*}}\biggl( \frac{|S_* \setminus \hat{S}|}{|\mathbb{W}_{q,d}|}\geq \frac{1}{11}  \biggr) &\geq \Prob_{P_{S_*}}\bigl( E \cap \{\hat{S} \subseteq S_*\} \bigr) \geq \Prob_{P_{S_*}}(E) - \Prob_{P_{S_*}}\bigl(\{\hat{S} \nsubseteq S_*\}\bigr) \geq \frac{1}{16}. 
\end{align*}
Thus,  
\begin{align*}
\E_{P_{S_*}}\bigl\{ \mu\bigl(\superLevelSet{\eta_{S_*}}{\tau}\setminus \hat{A}\bigr) \bigr\} \geq \frac{|\mathbb{W}_{q,d}|}{176 \cdot q^d}  \geq \frac{1}{176 \cdot d \cdot  q}\geq  \frac{1}{352\cdot d \cdot 4^{1/d}} \cdot  \frac{1}{n^{1/d}},
\end{align*}
as required.
\end{proof}

\begin{lemma}\label{lemma:antichain_distributions_properties_varyingMarginal}
For any $d, q\in \N$, $\tau \in \R$, $\sigma, \etaIncreasingExponent, \etaIncreasingConstant > 0$, $\densityConstant > 1$, an antichain $\mathbb{W}_{q,d} \subseteq [q]^d$ and $S\subseteq \mathbb{W}_{q,d}$, we have that $P_S\equiv P_{S, \mathbb{W}_{q,d}, q, \sigma, \tau, \etaIncreasingExponent, \etaIncreasingConstant}$ defined as in the proof of Proposition~\ref{prop:lowerBoundParametric} satisfies $P_S \in \distributionClassNonDecreasingRegressionFunction[d] \cap \distributionClassMultivariateCondition[d]$.
\end{lemma}

\begin{proof}
Since the regression function $\eta_S$ associated to $P_S$ is identical to that in Proposition~\ref{prop:lowerBound_alpha_multivariate}, we follow the same steps as in the proof of Lemma~\ref{lemma:antichain_distributions_properties} to show that $P_S \in \distributionClassNonDecreasingRegressionFunction$ and that the condition in Definition~\ref{def:multivariateAssumption}\emph{(ii)} is satisfied. It remains to show that $P_S$ has the property specified in Definition~\ref{def:multivariateAssumption}\emph{(i)}. Indeed, for $J_S$ as in Proposition~\ref{prop:lowerBoundParametric}, any $x_0 \in \support(\mu_S)$ and any $r \in (0, 1]$, we have $\mu_S\bigl(\closedSupNormMetricBall{x_0}{r}\bigr) = \Lebesgue\bigl(\closedSupNormMetricBall{x_0}{r} \cap J_S\bigr) \leq \Lebesgue\bigl(\closedSupNormMetricBall{x_0}{r}\bigr) = (2r)^d$.  Moreover there exists $x_1 \in \closedSupNormMetricBall{x_0}{r} \cap J_S$ such that 
$\closedSupNormMetricBall{x_0}{r} \cap J_S \supseteq \closedSupNormMetricBall{x_1}{r/2}$, so $\Lebesgue\bigl(\closedSupNormMetricBall{x_0}{r} \cap J_S\bigr) \geq \Lebesgue\bigl(\closedSupNormMetricBall{x_1}{r/2}\bigr) \geq r^d$, as required. 
\end{proof}
We are now in a position to prove Theorem~\ref{Thm:LowerBound}.
\begin{proof}[Proof of Theorem~\ref{Thm:LowerBound}]
Let $c := 1/(1408 \cdot d\cdot 32^{1/d}\cdot 13^{1/(2\etaIncreasingExponent + d)})$. Suppose first that
\[
    n < (8\cdot 2^d)^{(2\etaIncreasingExponent + d)/(2\etaIncreasingExponent)}\biggl(\frac{13\etaIncreasingConstant^2}{\sigma^2\log_+\bigl(1/(5\alpha)\bigr)}\biggr)^{d/(2\etaIncreasingExponent)},
\]
so that
\[
\frac{1}{704\cdot d\cdot 4^{1/d}} \cdot \frac{1}{n^{1/d}} > c\biggl(\frac{\sigma^2}{n\etaIncreasingConstant^2}\log_+\Bigl(\frac{1}{5\alpha}\Bigr)\biggr)^{1/(2\etaIncreasingExponent + d)}.
\]
By Proposition~\ref{prop:lowerBoundParametric}, we have
\begin{align*}
    \inf_{\hat{A} \in \hat{\mathcal{A}}_n(\tau, \alpha, \mathcal{P}')} \sup_{P \in \mathcal{P}'} \mathbb{E}_P \bigl\{\marginalDistribution\bigl(\superLevelSet{\tau}{\regressionFunction}\setminus \hat{A}(\sample)\bigr)\bigr\} 
    &\geq \frac{1}{352\cdot d\cdot 4^{1/d}}\cdot \frac{1}{n^{1/d}} \\
    &\geq \frac{1}{704\cdot d\cdot 4^{1/d}}\cdot \frac{1}{n^{1/d}} + \frac{c}{n^{1/d}}\\
    &\geq c\biggl[1 \wedge \biggl\{\biggl(\frac{\sigma^2}{n\etaIncreasingConstant^2}\log_+\Bigl(\frac{1}{5\alpha}\Bigr)\biggr)^{1/(2\etaIncreasingExponent + d)} + \frac{1}{n^{1/d}}\biggr\}\biggr].
\end{align*}
We may therefore suppose that 
\[
    n \geq (8\cdot 2^d)^{(2\etaIncreasingExponent + d)/(2\etaIncreasingExponent)}\biggl(\frac{13\etaIncreasingConstant^2}{\sigma^2\log_+\bigl(1/(5\alpha)\bigr)}\biggr)^{d/(2\etaIncreasingExponent)}.
\]
Then, by Propositions~\ref{prop:lowerBound_alpha_multivariate} and~\ref{prop:lowerBoundParametric}, we have
\begin{align*}
    \inf_{\hat{A} \in \hat{\mathcal{A}}_n(\tau, \alpha, \mathcal{P}')} \sup_{P \in \mathcal{P}'} \mathbb{E}_P \bigl\{\marginalDistribution\bigl(&\superLevelSet{\tau}{\regressionFunction}\setminus \hat{A}(\sample)\bigr)\bigr\} \\
    &\geq \frac{1}{40d} \biggl\{\frac{\sigma^2}{13n\etaIncreasingConstant^2}\log_+\Bigl(\frac{1}{5\alpha}\Bigr)\wedge 1\biggr\}^{1/(2\etaIncreasingExponent + d)} \vee  \frac{1}{352\cdot d\cdot 4^{1/d} \cdot n^{1/d}} \\
    &\geq \frac{1}{40d}\wedge \biggl\{\frac{1}{80d}\biggl(\frac{\sigma^2}{13n\etaIncreasingConstant^2}\log_+\Bigl(\frac{1}{5\alpha}\Bigr)\biggr)^{1/(2\etaIncreasingExponent + d)} +  \frac{1}{704\cdot d\cdot 4^{1/d} \cdot n^{1/d}}\biggr\} \\
    &\geq c\biggl[1 \wedge \biggl\{\biggl(\frac{\sigma^2}{n\etaIncreasingConstant^2}\log_+\Bigl(\frac{1}{5\alpha}\Bigr)\biggr)^{1/(2\etaIncreasingExponent + d)} + \frac{1}{n^{1/d}}\biggr\}\biggr],
\end{align*}
as required.
\end{proof}

\subsection{Proofs from Section~\ref{Sec:Extensions}}\label{sec:extension_proofs}

\begin{proof}[Proof of Lemma~\ref{lemma:validLocalPValNormality}]
We condition on $\sampleX$ throughout this proof.  Let $\Theta_0 := (-\infty,\tau)^{n(x)} \times (0,\infty)$ and $P \in \mathcal{P}_{\mathrm{N},d}(\sigma_*)$.  Write $\varphi(\cdot; a, \sigma^2)$ for the density function of the $\mathcal{N}(a, \sigma^2)$ distribution.  We fix $k \in [n(x)]$ and initially operate on the event $\bigl\{\max_{j \in [k]} Y_{(j)}(x) > \tau\bigr\} = \{\hat{\sigma}^2_{0,k} > 0\}$. We claim that $\bigl((Y_{(j)}(x) \wedge \tau)_{j\in [k]}, \hat{\sigma}^2_{0,k}\bigr)$ then maximises the conditional likelihood $L(t, \sigma^2) := \prod_{j=1}^k \varphi\bigl(Y_{(j)}(x); t_j, \sigma^2\bigr)$ over $\Theta_{0,k} := (-\infty, \tau]^k\times [0,\infty)$, where $L(t, 0) := \lim_{\sigma \searrow 0} L(t, \sigma^2) = 0$ for $t\in (-\infty, \tau]^k$. To see this, note first that $L(t, 0) = 0 < \sup_{\sigma > 0} L(t, \sigma^2)$ for $t\in (-\infty, \tau]^k$, so any maximiser must be contained in $\Theta_{0,k}^0 := (-\infty, \tau]^k\times (0,\infty)$.  Moreover, for any $(t_1,\ldots,t_{j-1},t_{j+1},\ldots,t_k,\sigma^2) \in (-\infty,\tau]^{k-1} \times (0,\infty)$, the unique maximiser $t_j^0$ of $t_j \mapsto L(t,\sigma^2)$ satisfies $t_j^0 = Y_j(x) \wedge \tau$.  It therefore suffices to maximise $\sigma^2 \mapsto L(t^0,\sigma^2)$ over $(0,\infty)$, where $t^0 := (t_1^0,\ldots,t_k^0)^\top \in (-\infty,\tau]^k$, and the unique maximiser is given by $\sigma_0^2 := \hat{\sigma}^2_{0,k}$.

Hence, writing $t_j^* := \regressionFunction\bigl(X_{(j)}(x)\bigr)$ for $j \in [n(x)]$, when $\hat{\sigma}^2_{0,k} > 0$, we have for $k \in [n(x)]$ that
\begin{align*}
\frac{1}{\Bar{p}^k_\tau(x, \sample)} &= \frac{(2\pi)^{-k/2}\prod_{j=1}^k \hat{\sigma}^{-1}_{1,j-1}\exp\bigl\{-\bigl(Y_{(j)}(x) - \bar{Y}_{1,j-1}\bigr)^2/\bigl(2\hat{\sigma}^2_{1,j-1}\bigr)\bigr\}}{(2\pi\hat{\sigma}^2_{0,k})^{-k/2}\exp\bigl\{-\sum_{j=1}^k\bigl(Y_{(j)}(x) - (Y_{(j)}(x) \wedge \tau)\bigr)^2/\bigl(2\hat{\sigma}^2_{0,k}\bigr)\bigr\}} \\
&= \frac{\prod_{j=1}^k \varphi\bigl(Y_{(j)}(x); \Bar{Y}_{1, j-1}, \hat{\sigma}^2_{1,j-1}\bigr)}{\sup_{(t, \sigma^2) \in \Theta_{0,k}} \prod_{j=1}^k \varphi\bigl(Y_{(j)}(x); t_j, \sigma^2\bigr)} \leq \frac{\prod_{j=1}^k \varphi\bigl(Y_{(j)}(x); \Bar{Y}_{1, j-1}, \hat{\sigma}^2_{1,j-1}\bigr)}{\prod_{j=1}^k \varphi\bigl(Y_{(j)}(x); t_j^*, \sigma_*^2\bigr)} =: \Lambda_k.
\end{align*}
We now claim that the process given by $(\Lambda_k)_{k\in [n(x)]}$ defines a martingale with respect to the filtration $(\mathcal{F}_j)_{j \in \{0\} \cup [n(x)]}$, where $\mathcal{F}_0$ is the trivial $\sigma$-algebra and where $\mathcal{F}_j$ denotes the $\sigma$-algebra generated by $\bigl(Y_{(\ell)}(x)\bigr)_{\ell \in [j]}$, with $\mathbb{E}(\Lambda_1|\sampleX) = 1$.  To see this, observe that
\begin{align*}
\mathbb{E}(\Lambda_{k+1} \mid \mathcal{F}_{k}, \sampleX) &= \Lambda_{k}\mathbb{E}\biggl( \frac{\varphi\bigl(Y_{(k+1)}(x); \Bar{Y}_{1, k}, \hat{\sigma}^2_{1,k}\bigr)}{\varphi\bigl(Y_{(k+1)}(x); t_{k+1}^*, \sigma^2_*\bigr)}\biggm| \mathcal{F}_{k}, \sampleX\biggr)\\
&= \Lambda_{k}\int_{-\infty}^\infty \frac{\varphi\bigl(y; \Bar{Y}_{1, k}, \hat{\sigma}^2_{1,k}\bigr)}{\varphi\bigl(y; t_{k+1}^*, \sigma^2_*\bigr)}\cdot \varphi\bigl(y; t_{k+1}^*, \sigma^2_*\bigr) \, dy\\ 
&= \Lambda_{k}\int_{-\infty}^\infty \varphi\bigl(y; \Bar{Y}_{1, k}, \hat{\sigma}^2_{1,k}\bigr) \, dy = \Lambda_{k}    
\end{align*}
for $k\in [n(x) - 1]$.  Hence, by Ville's inequality \citep{ville1939etude}, for any $\alpha \in (0, 1)$,
\begin{align*}
    \mathbb{P}\bigl(\Bar{p}_\tau(x, \sample) \leq \alpha \bigm| \sampleX \bigr) &= \mathbb{P}\biggl(\bigcup_{k\in [n(x)]} \Bigl\{ \Bar{p}^k_\tau(x, \sample) \leq \alpha, \hat{\sigma}^2_{0,k} > 0 \Bigr\}\Bigm| \sampleX \biggr)\\
    &= \mathbb{P}\biggl(\bigcup_{k \in [n(x)]}\Bigl\{ \Lambda_k \geq 1/\alpha, \hat{\sigma}^2_{0,k} > 0\Bigr\} \Bigm| \sampleX\biggr)\\
    &\leq \mathbb{P}\Bigl(\max_{k \in [n(x)]} \Lambda_k \geq 1/\alpha \Bigm| \sampleX\Bigr) \leq \alpha,
\end{align*}
as required.
\end{proof}

We prove Lemma~\ref{lemma:validLocalPValClassification} by establishing the generalisation given by Lemma~\ref{lemma:validLocalPValClassificationGeneralisation} below.  This latter result is stated for more general $p$-values that we now define.  
\begin{defn}\label{defn:pValue_binary_generalised}
    In the setting of Definition~\ref{defn:pValue}, let $\nu$ be a measure supported on $[\tau,1]$, let $\check{S}_k := \sum_{j = 1}^k Y_{(j)}(x)$ and define
    \begin{align*}
 \check{p}_{\tau,\nu}(x, \sample) := 1 \wedge \min_{k \in [n(x)]} \frac{\tau^{\check{S}_k}(1-\tau)^{n-\check{S}_k}}{\int_{[\tau,1]} t^{\check{S}_k}(1-t)^{n - \check{S}_k}\, d\nu(t)}.
    \end{align*}
\end{defn}
If we take $\nu$ to be the $\mathrm{Unif}[\tau,1]$ distribution in Definition~\ref{defn:pValue_binary_generalised}, then we recover the $p$-values from Definition~\ref{defn:pValue_binary} that are employed in Lemma~\ref{lemma:validLocalPValClassification}.
\begin{lemma}\label{lemma:validLocalPValClassificationGeneralisation}
Let $\tau \in \R$, let $\nu$ be a measure supported on $[\tau, 1]$ and let $P$ be a distribution on $\R^d \times [0,1]$ with regression function $\eta$. Fix $x \in \R^d$ and suppose that $\eta(x') \leq \tau$ for all $x' \preccurlyeq x$.  Given $\sample = \bigl((X_1,Y_1),\ldots,(X_n,Y_n)\bigr) \sim P^n$, we have $\Prob_P\bigl\{\check{p}_{\tau,\nu}(x,\sample) \leq \alpha | \sample_X\bigr\} \leq \alpha$ for all $\alpha \in (0,1)$.   
\end{lemma}
\begin{proof}[Proof of Lemma~\ref{lemma:validLocalPValClassificationGeneralisation} (and hence Lemma~\ref{lemma:validLocalPValClassification})]
Fix $k\in\N$. For $t \in [0,1]$ and $z_1, \ldots, z_k \in \{0,1\}$, write $L_k(t;z_1, \ldots, z_k) := t^{s_k}(1-t)^{k - s_k}$, where $s_k := \sum_{j=1}^k z_j$, for the likelihood function of an independent sample of~$k$ Bernoulli random variables with success probability $t$. Further, let $\Bar{L}_{k,\nu}(z_1, \ldots, z_k) := \int_{[\tau,1]}  L_k(t;z_1, \ldots, z_k) \, d\nu(t)$.  Finally, let $(Z_j)_{j \in \N}$ be a sequence of independent $[0,1]$-valued random variables with $\tilde{\tau}_j:=\mathbb{E}(Z_j) \leq  \tau$.  We claim that the likelihood ratio sequence $\bigl(\Lambda_k(Z_1,\ldots, Z_k)\bigr)_{k \in \mathbb{N}}$ given by
\[
\Lambda_k(Z_1,\ldots, Z_k) := \frac{\Bar{L}_{k,\nu}(Z_1, \ldots, Z_k)}{L_k(\tau;Z_1, \ldots, Z_k)} = \int_{[\tau,1]} \frac{t^{\check{S}_k}(1-t)^{k - \check{S}_k}}{\tau^{\check{S}_k}(1-\tau)^{k - \check{S}_k}} \, d\nu(t)
\]
defines a non-negative super-martingale with respect to the filtration $(\mathcal{F}_k)_{k \in \N_0}$, where $\mathcal{F}_0$ denotes the trivial $\sigma$-algebra and $\mathcal{F}_k := \sigma(Z_1, \ldots, Z_k)$ for $k \in \N$. Indeed, by Fubini's theorem 
\begin{align*}
\mathbb{E}\bigl( &\Lambda_k(Z_1,\ldots, Z_k) \mid \mathcal{F}_{k-1}\bigr) = \int_{[\tau,1]} \mathbb{E}\biggl\{\frac{t^{\sum_{j=1}^k Z_j}(1-t)^{k-\sum_{j=1}^k Z_j} }{\tau^{\sum_{j=1}^k Z_j}(1-\tau)^{k-\sum_{j=1}^k Z_j}} \biggm| \mathcal{F}_{k-1}\biggr\} \, d\nu(t)\\
&=    \int_{[\tau,1]} \bigg( \frac{t^{\sum_{j=1}^{k-1} Z_j}(1-t)^{(k-1)-\sum_{j=1}^{k-1} Z_j} }{\tau^{\sum_{j=1}^{k-1} Z_j}(1-\tau)^{(k-1)-\sum_{j=1}^{k-1} Z_j}} \cdot \frac{1-t}{1-\tau} \cdot \mathbb{E}\biggl\{ \biggl( \frac{t(1-\tau)}{\tau(1-t)} \bigg)^{Z_k} \biggm| \mathcal{F}_{k-1} \biggr\} \bigg) \, d\nu(t)\\
& \leq \int_{[\tau,1]} \bigg( \frac{t^{\sum_{j=1}^{k-1} Z_j}(1-t)^{(k-1)-\sum_{j=1}^{k-1} Z_j} }{\tau^{\sum_{j=1}^{k-1} Z_j}(1-\tau)^{(k-1)-\sum_{j=1}^{k-1} Z_j}} \cdot \biggl\{ \tilde{\tau}_k \cdot \frac{t}{\tau}+ (1-\tilde{\tau}_k) \cdot \frac{1-t}{1-\tau} \biggr\} \biggr) \, d\nu(t) \\
& \leq \int_{[\tau,1]} \bigg( \frac{t^{\sum_{j=1}^{k-1} Z_j}(1-t)^{(k-1)-\sum_{j=1}^{k-1} Z_j} }{\tau^{\sum_{j=1}^{k-1} Z_j}(1-\tau)^{(k-1)-\sum_{j=1}^{k-1} Z_j}} \biggr) \, d\nu(t) = \Lambda_{k-1}(Z_1,\ldots, Z_{k-1}),
\end{align*}
where we have applied \citet[Lemma 9]{garivier2011kl} in the first inequality.  Now let $(Z_j)$ be an independent sequence of independent $[0,1]$-valued random variables so that $Z_j$ has the same distribution as the conditional distribution of $Y_{(j)}(x)$ given $\mathcal{D}_X$ for $j \in [n(x)]$, and $Z_j = 0$ almost surely for $j > n(x)$.  We conclude by Ville's inequality \citep{ville1939etude} that
\begin{align*}
    \mathbb{P}_P\biggl(\check{p}_{\tau,\nu}(x, \sample) \leq \alpha \biggm | \sampleX \biggr) 
    &= \mathbb{P}_P\biggl(\max_{k\in [n(x)]} \Lambda_k\bigl(Y_{(1)}(x),\ldots, Y_{(k)}(x)\bigr) \geq \frac{1}{\alpha} \biggm| \sample_X\biggr) \\
    &\leq \mathbb{P}\biggl(\sup_{k\in\N}\Lambda_k(Z_1,\ldots, Z_k) \geq \frac{1}{\alpha}\biggr) \leq \alpha,
\end{align*}
for any $\alpha \in (0,1)$, as required.
\end{proof}

\begin{proof}[Proof of Lemma~\ref{lemma:validLocalPValConditionalQuantile}] Let $\tilde{Y}_i := \one_{\{Y_i >\tau\}}$ for $i\in [n]$. Suppose that $x' \preccurlyeq x$, so that $\zeta_{\theta}(x') \leq \zeta_{\theta}(x)<\tau$.  Then $\mathbb{P}_P\bigl(Y_i \leq \tau | X_i = x'\bigr)\geq \theta$ and so
\begin{align*}
\mathbb E_P(\tilde{Y}_i\mid X_i = x') = \mathbb P_P(Y_i > \tau\mid X_i = x') \leq 1- \theta.
\end{align*}
By Hoeffding's lemma, the conditional distribution of $\tilde{Y}_i - \mathbb E_P(\tilde{Y}_i| X_i)$ given $X_i$ is sub-Gaussian with variance parameter $1/4$. The first result now follows by Lemma \ref{lemma:validLocalPValStoufferGeneralised}, and the second follows similarly from Lemma~\ref{lemma:validLocalPValClassificationGeneralisation}.
\end{proof}

\section{Further simulation results}\label{Sec:FurtherSimulations}

\subsection{Further performance comparisons}
\label{Sec:FurtherSimulations_FurtherComparisons}
To expand on the simulations in Section~\ref{sec:simulations}, we illustrate the performance of our procedure on eight more regression functions, which are presented in Table~\ref{table:sim_regression_functions_appendix} and illustrated in Figure~\ref{fig:sim_regression_functions_2D_appendix}.  Other than the choice of $f$, the simulations were carried out in identical fashion to that described in Section~\ref{sec:simulations}, and the results are illustrated in Figures~\ref{Fig:gtond2},~\ref{Fig:gtond3} and~\ref{Fig:gtond4}.

\begin{table}[htbp]
    \centering  
    \begin{tabular}{c|c| c| c }
        Label & Function $f$ & $\tau$ & $\etaIncreasingExponent(P)$\\ \hline
        (g) & $\exp\bigl(\sum_{j=1}^d x^{(j)}\bigr)$ & $\frac{e^{d/2} - 1}{e^d - 1}$ & $1$\\
        (h) & $\{1 + \exp\bigl(-4\cdot \sum_{j=1}^d (x^{(j)}-0.5)\bigr)\}^{-1}$ & $1/2$ & $1$\\
        (i) & $\sum_{j=1}^d \bigl(x^{(j)}\bigr)^3$ & $\frac{1}{d}\Bigl(\frac{\Gamma(1 + d/3)}{2\Gamma(4/3)^d}\Bigr)^{3/d}$ & $1$\\
        (j) & $\sum_{j=1}^d \bigl(x^{(j)} - 1\bigr)^3$ & $1 - \frac{1}{d}\Bigl(\frac{\Gamma(1 + d/3)}{2\Gamma(4/3)^d}\Bigr)^{3/d}$ & $1$\\
        (k) & $\sum_{j=1}^d \lceil 6\cdot x^{(j)}\rceil /6$ & $7/12$ & $0$\\
        (l) & $\sqrt{x^{(1)}} + x^{(2)}$ & $0.584$ & $1$\\

        (m) & $\bigl(\sum_{j=1}^d (x^{(j)} - 0.5)\bigr)^{1/3}$ & $1/2$ & $1/3$\\
        (n) & $\bigl(\sum_{j=1}^d (x^{(j)} - 0.5)\bigr)^{3}$ & $1/2$ & $3$
    \end{tabular}
    \caption{Definition of the functions used in the simulations. Here, $x = (x^{(1)}, \ldots, x^{(d)})^\top\in [0,1]^d$.  For the regression function $\eta$ given by $\eta(x) := \bigl(f(x) - f(0)\bigr)/\bigl(f(\bm{1}_d) - f(0)\bigr)$, we have $\marginalDistribution\bigl(\superLevelSet{\tau}{\eta}\bigr) = 1/2$, except for cases (i), (j), (k), where $\marginalDistribution\bigl(\superLevelSet{\tau}{\eta}\bigr) \approx 1/2$ for the considered values of $d$.}
    \label{table:sim_regression_functions_appendix}
\end{table}

\begin{figure}[hbtp]
    \centering
    \includegraphics[width = 0.85\textwidth]{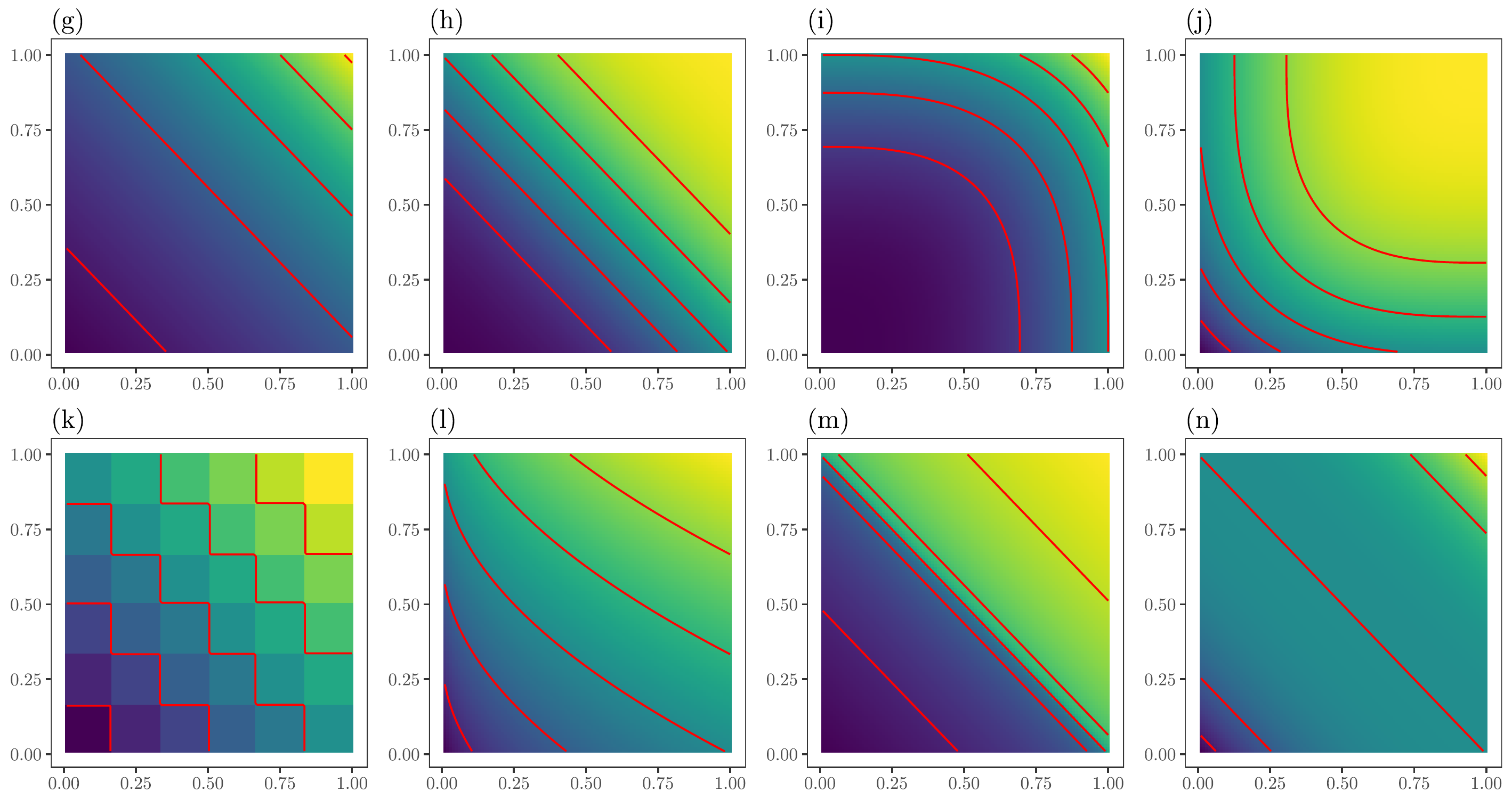}
    \caption{For $d= 2$, the contour lines (red) of the regression functions corresponding to the functions $f$ in Table~\ref{table:sim_regression_functions_appendix} at the levels $k/6$ for $k\in [5]$ are shown. The fill colour indicates the function value at the respective position from $0$ (purple) to $1$ (yellow).}
    \label{fig:sim_regression_functions_2D_appendix}
\end{figure}

\begin{figure}
    \centering
    \includegraphics[width = 0.98\textwidth]{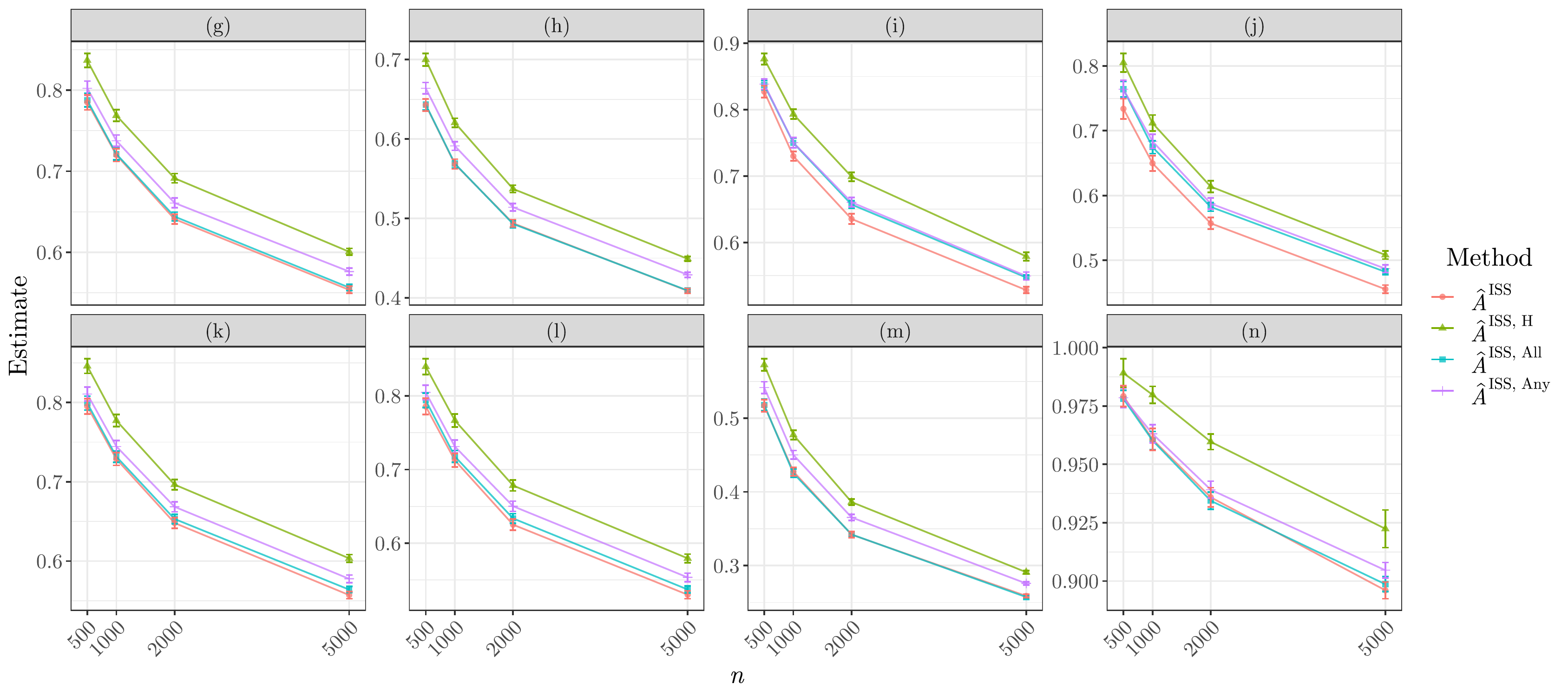}
    \caption{Estimates of $\mathbb{E}\bigl\{\mu\bigl(\superLevelSet{\tau}{\eta} \setminus \hat{A}\bigr)\bigr\}$ for $d = 2$ and $\sigma = 1/4$.}
    \label{Fig:gtond2}
\end{figure}

\begin{figure}
    \centering
    \includegraphics[width = 0.98\textwidth]{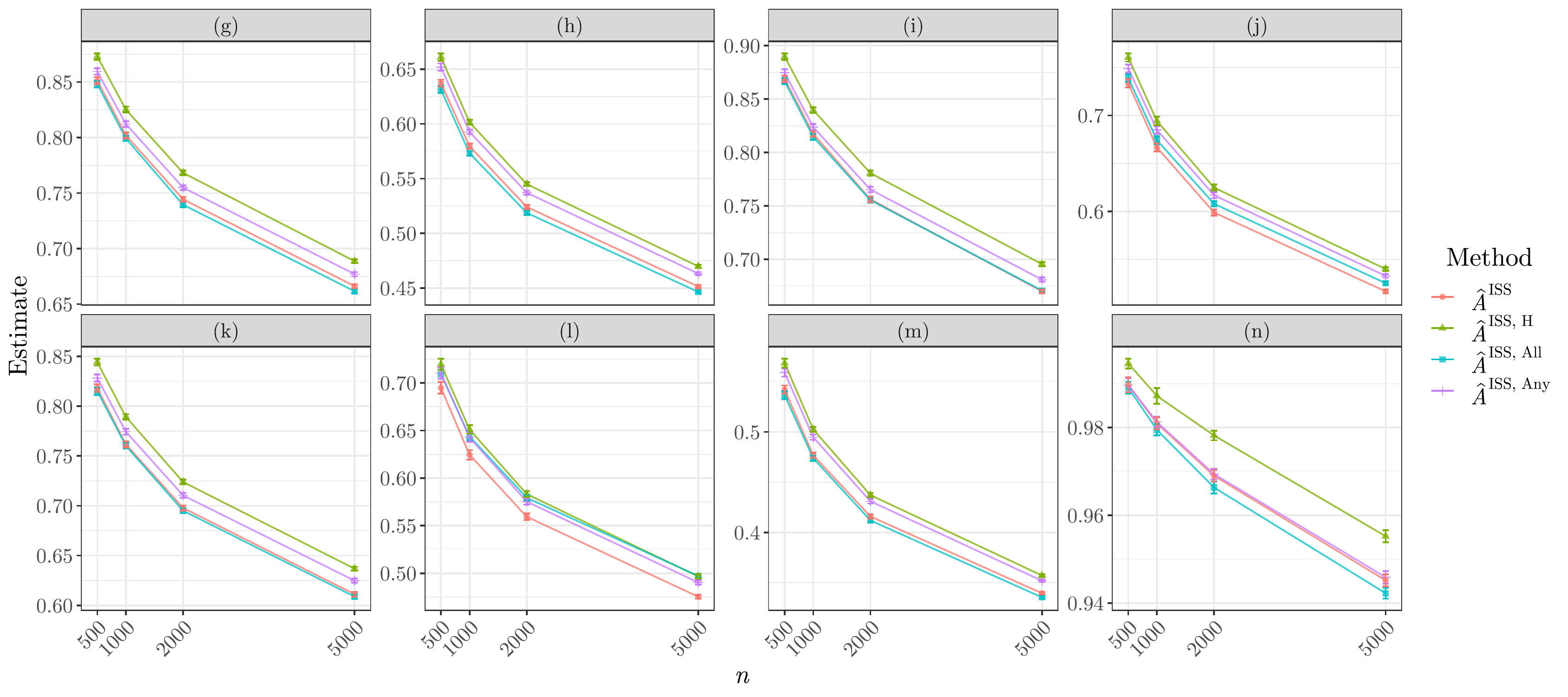}
    \caption{Estimates of $\mathbb{E}\bigl\{\mu\bigl(\superLevelSet{\tau}{\eta} \setminus \hat{A}\bigr)\bigr\}$ for $d = 3$ and $\sigma = 1/16$.}
    \label{Fig:gtond3}
\end{figure}

\begin{figure}
    \centering
    \includegraphics[width = 0.98\textwidth]{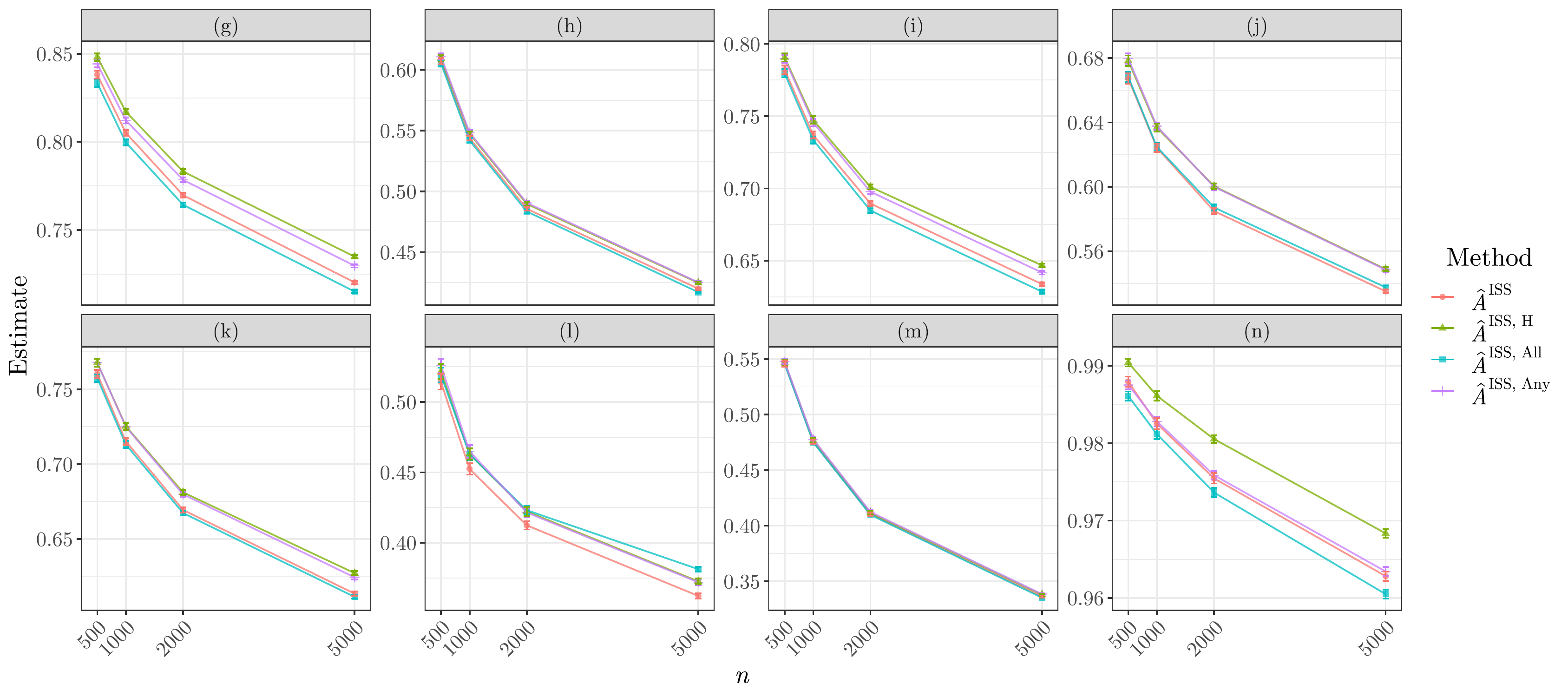}
    \caption{Estimates of $\mathbb{E}\bigl\{\mu\bigl(\superLevelSet{\tau}{\eta} \setminus \hat{A}\bigr)\bigr\}$ for $d = 4$ and $\sigma = 1/64$.}
    \label{Fig:gtond4}
\end{figure}

We now consider a comparison of $\hat{A}^{\mathrm{ISS}}$ with two possible procedures based on sample splitting. Indeed, a natural approach to combining our $p$-values is to employ a fixed sequence testing procedure \citep{hsu1999stepwise,westfall2001optimally} as we do in the univariate setting in Section \ref{sec:theory_power}.  However, since our sequence must be specified independently of the data used for testing, and there is no canonical total ordering in the multivariate case, sample splitting offers a potential way forward.  We consider two such procedures: in the first, denoted $\hat{A}^{\mathrm{Split}}$, we use the first half of the data to compute $p$-values at each of our~$n$ data points.  These are then ordered from smallest to largest, and this determines the ordering for our fixed sequence testing based on $p$-values computed on the second half of the data.  The second procedure, denoted $\hat{A}^{\mathrm{Split},\mathrm{OR}}$, discards the first half of the data and instead uses an oracle ordering of the data points using the underlying knowledge of the regression function; the second stage of the procedure is then identical to $\hat{A}^{\mathrm{Split}}$.  Results comparing $\hat{A}^{\mathrm{ISS}}$ with $\hat{A}^{\mathrm{Split}}$ and $\hat{A}^{\mathrm{Split},\mathrm{OR}}$ are presented in Figures~\ref{Fig:gtond2Split} and~\ref{Fig:gtond4Split}, which indicate that both of these sample-splitting variants have considerably worse empirical performance than $\hat{A}^{\mathrm{ISS}}$.  This is perhaps surprising given the impressive numerical results for sample splitting in conjunction with fixed sequence testing reported by \cite{angelopoulos2021learn}.  However, the performance of sample-splitting approaches is highly dependent on the procedure used to determine the ordering of the hypotheses from the first split of the data. Even exact knowledge of the regression function may be insufficient to determine an ordering with high conditional power, as the distribution of the $p$-values on the second half of the data also depends on the marginal distribution of the covariates.  This is reflected in the fact that $\hat{A}^{\mathrm{Split},\mathrm{OR}}$ has worse performance than $\hat{A}^{\mathrm{Split}}$ in some cases, especially when the regression function depends only on a strict subset of the $d$ variables, such as case (l) in Figure~\ref{Fig:gtond4Split}.

\begin{figure}
    \centering
    \includegraphics[width = 0.98\textwidth]{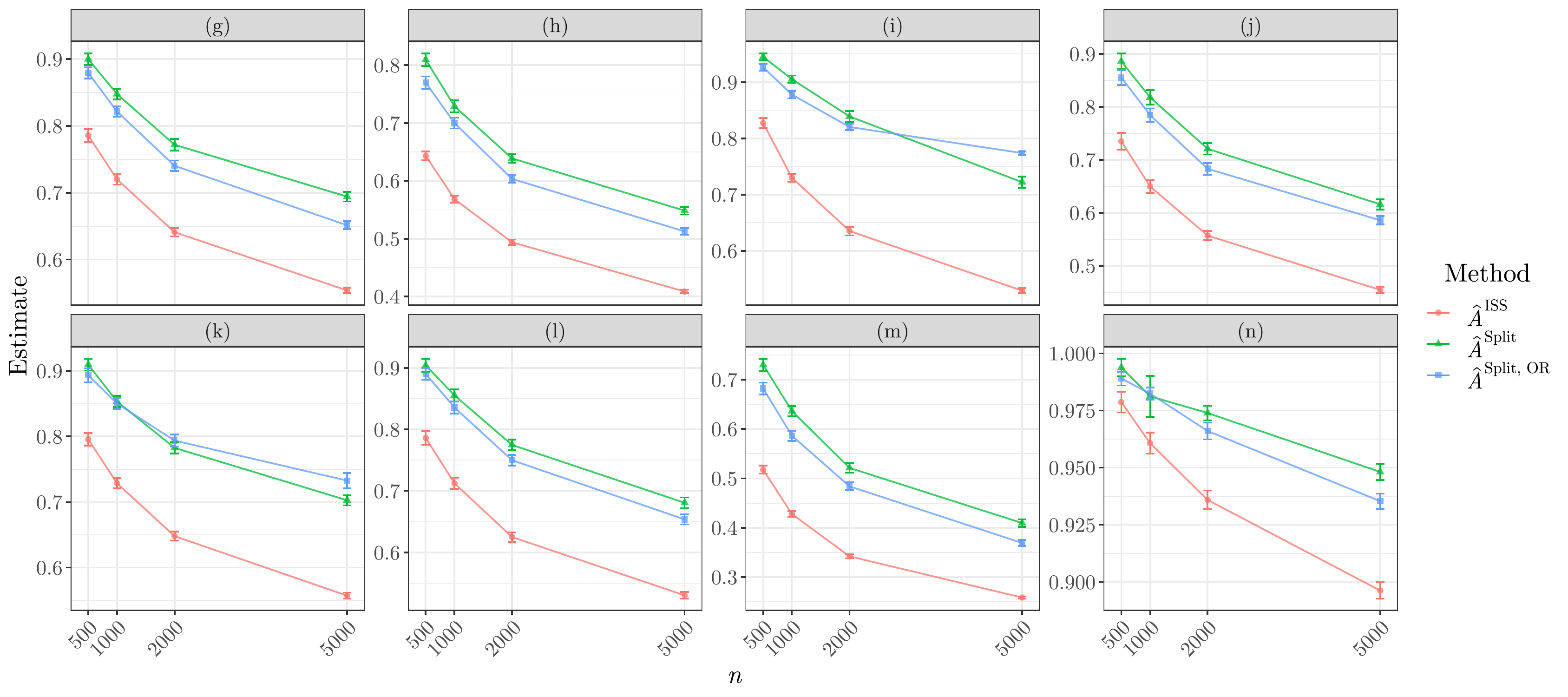}
    \caption{Estimates of $\mathbb{E}\bigl\{\mu\bigl(\superLevelSet{\tau}{\eta} \setminus \hat{A}\bigr)\bigr\}$ for $d = 2$ and $\sigma = 1/4$.}
    \label{Fig:gtond2Split}
\end{figure}

\begin{figure}
    \centering
    \includegraphics[width = 0.98\textwidth]{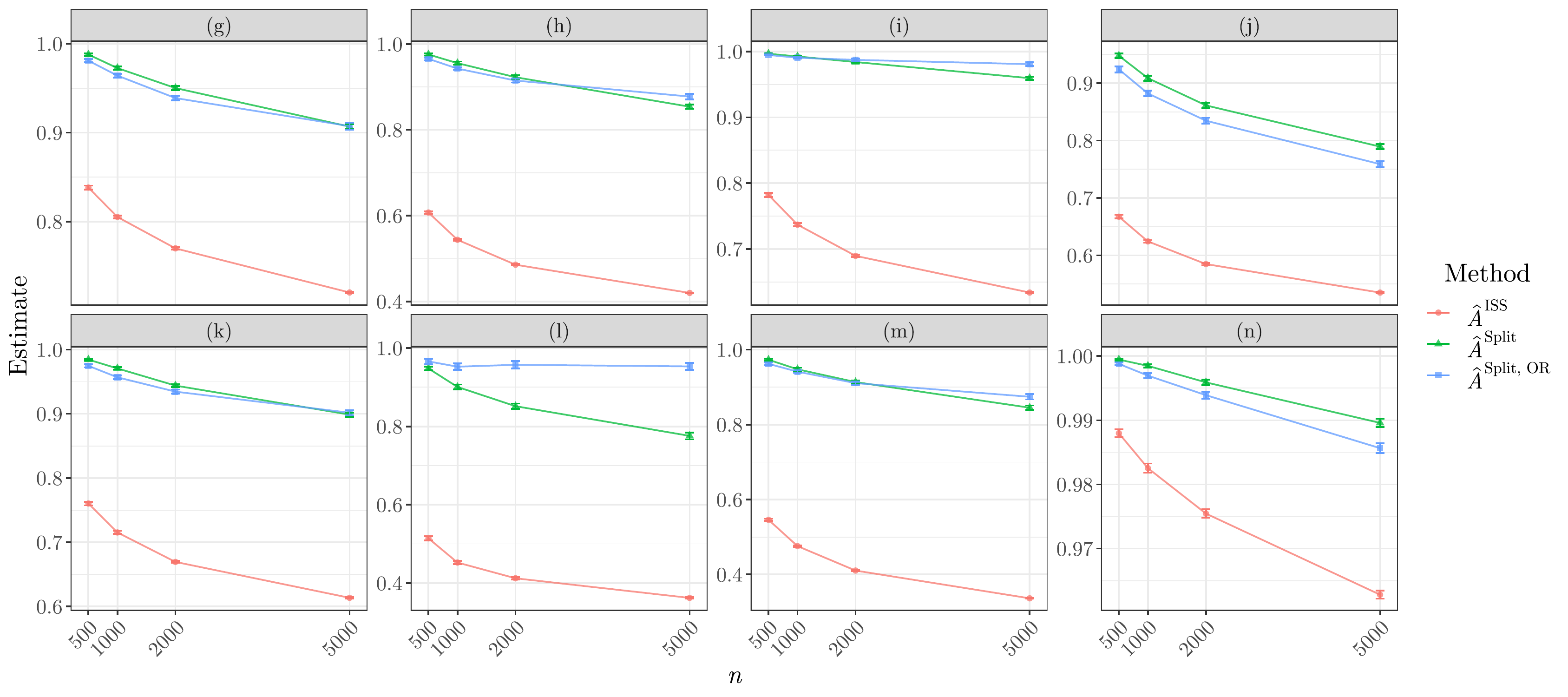}
    \caption{Estimates of $\mathbb{E}\bigl\{\mu\bigl(\superLevelSet{\tau}{\eta} \setminus \hat{A}\bigr)\bigr\}$ for $d = 4$ and $\sigma = 1/64$.}
    \label{Fig:gtond4Split}
\end{figure}

\subsection{Computation time}\label{Sec:FurtherSimulations_runtime}

In Figures~\ref{Fig:runtimed2} and~\ref{Fig:runtimed4}, we present the average computation time of the different procedures studied in Section~\ref{sec:simulations} for dimensions $d = 2,4$ and regression functions (g)--(n) above.  These reveal that the computation time varies quite substantially across the different regression functions but does not even necessarily increase at all with dimension.  These effects are related to the depth of the DAG induced by the observations as well as the power of the procedures in the different settings.

\begin{figure}
    \centering
    \includegraphics[width = 0.98\textwidth]{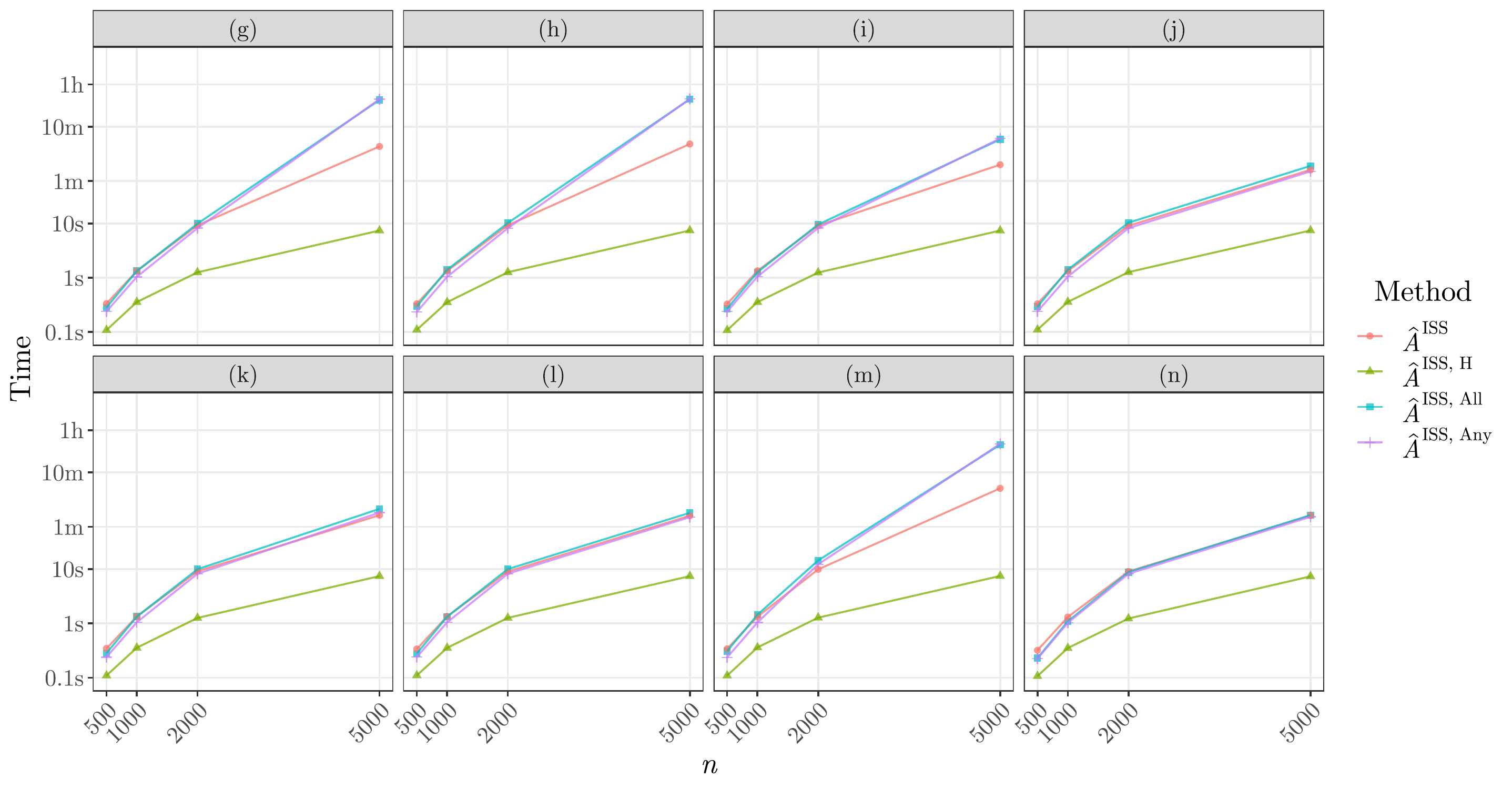}
    \caption{Computation time of the different estimates of $\mathbb{E}\bigl\{\mu\bigl(\superLevelSet{\tau}{\eta} \setminus \hat{A}\bigr)\bigr\}$ for $d = 2$ and $\sigma = 1/4$.}
    \label{Fig:runtimed2}
\end{figure}

\begin{figure}
    \centering
    \includegraphics[width = 0.98\textwidth]{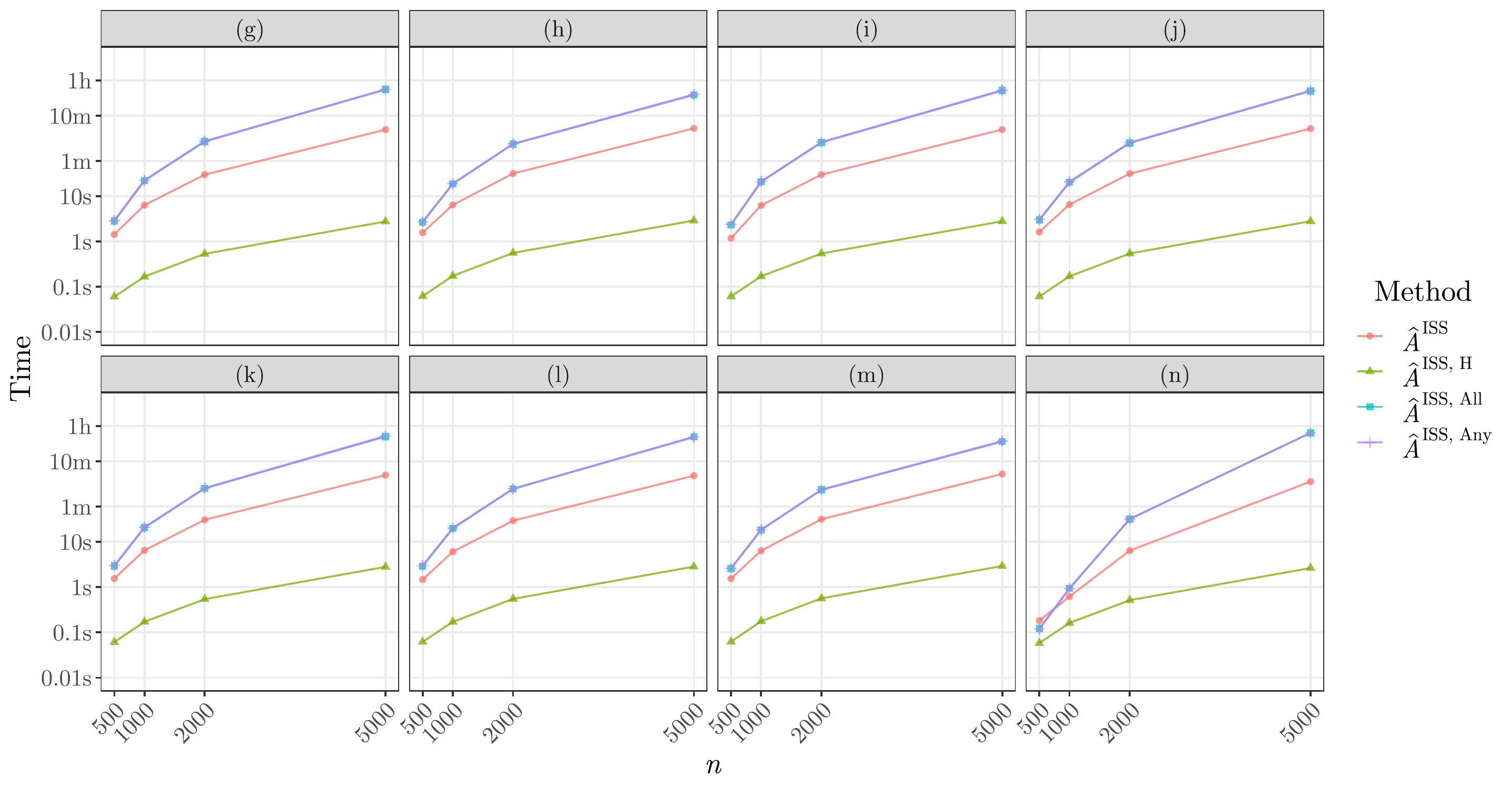}
    \caption{Computation time of the different estimates of $\mathbb{E}\bigl\{\mu\bigl(\superLevelSet{\tau}{\eta} \setminus \hat{A}\bigr)\bigr\}$ for $d = 4$ and $\sigma = 1/64$.}
    \label{Fig:runtimed4}
\end{figure}

\section{Comparison with procedure based on Meijer \& Goeman (2015)}
\label{sec:MG_appendix}
As discussed in Section~\ref{Sec:Methodology}, our proposed procedure consists of two steps; calculating $p$-values to test whether the regression function $\eta$ exceeds the threshold $\tau$ at $m$ given points and then controlling the FWER over $\distributionClassNonDecreasingRegressionFunction$ through a DAG testing procedure.  For the second step, an alternative approach would be to use the algorithm introduced by \cite{meijer2015multiple}. Indeed, the empirical results in Section~\ref{sec:simulations} suggest that such a procedure can work well in some cases.  However, we show in this section that it fails to attain the optimal worst-case regret over $\distributionClassNonDecreasingRegressionFunction\cap\distributionClassMultivariateCondition$.  

\subsection{Description of procedure}\label{sec:MG_procedure}

The iterative algorithm of \citet{meijer2015multiple} is an application of the sequential rejection principle \citep{goeman2010sequential} to hypotheses indexed by elements of $I = [m]$, for some $m\in \N$, that are a priori arranged as a DAG $G = (I, E)$\footnote{Indeed, in order to fit the more general notion of Definition~\ref{def:graphMultipilictyProcedure}, we may assume that this DAG is weighted, although the weights will be irrelevant for the procedure.}.  Inputs to the algorithm include a fixed significance level $\alpha \in (0, 1)$, a vector $\bm{v} \in (0, \infty)^m$ and $(p_i)_{i\in I} \in (0, 1]^m$, with the latter thought of as a collection of $p$-values.  Any choice of $\bm{v}$ will correspond to a DAG testing procedure, as defined by Definition~\ref{def:graphMultipilictyProcedure}.  Each iteration of the procedure comprises three steps: the first assigns to each unrejected hypothesis (or, equivalently, the corresponding node) a proportion $\alpha_i$ of the $\alpha$-budget; the second step rejects any hypothesis $i \in I$ for which $p_i \leq \alpha_i$; and the third rejects all ancestors of rejected hypotheses.  The procedure terminates if no new rejections are made in the second step of an iteration or if every hypothesis has been rejected, and hence takes at most $m$ iterations.  In more detail, in the first step, the $\alpha$-budget is split among the unrejected leaf nodes in proportion to the corresponding elements of $\bm{v}$. These budgets are then propagated from the leaf nodes towards  currently unrejected ancestor nodes. \cite{meijer2015multiple} suggest two variants for this, which we enumerate by $\omega \in \{0, 1\}$ and call the \emph{all-parent variant} ($\omega = 0$) and the \emph{any-parent variant} ($\omega = 1$). In the all-parent variant of the procedure, a node's entire budget is evenly distributed among its unrejected parents (keeping nothing for itself), whereas in the any-parent variant, the budget that would go to rejected parents if it were evenly distributed among all parents simply stays at the node and only the remaining budget is evenly distributed among the unrejected parents. Importantly, the order in which the nodes pass their budgets to their parents follows a reverse topological ordering of $G$; all reverse topological orderings $\pi_G$ lead to the same output in Algorithm~\ref{algo:MG}, making the specific choice immaterial.  Thus, a node only distributes its budget once all of its descendants have distributed theirs.  Once this budget propagation has terminated, we move to the second step and reject all hypotheses whose $p$-value does not exceed the assigned budget. Finally, the third step is only relevant in the any-parent variant of the procedure, and rejecting the ancestors of nodes rejected at the second step does not increase the Type~I error rate when $G$ is $G_0$-consistent for a directed graph $G_0$ encoding all logical relationships between hypotheses (see Section~\ref{sec:theory_validity}).  A concise formal description of the \citet{meijer2015multiple} procedure, which outputs a set $\mathcal{R}_\alpha^{\mathrm{MG}, \omega, \bm{v}}(G,\bm{p})$ of rejected hypotheses, is given in Algorithm~\ref{algo:MG}. 

\cite{meijer2015multiple} prove that Algorithm~\ref{algo:MG} satisfies the two sufficient conditions for controlling the FWER described by \cite{goeman2010sequential}.  The DAG testing procedures $\mathcal{R}^{\mathrm{MG}, \omega, \bm{v}}$ for $\omega \in \{0, 1\}$ motivate the following selection sets\footnote{We deviate slightly from the notation in Section~\ref{sec:simulations}:  $\hat{A}^{\mathrm{ISS},\mathrm{All}}\equiv \hat{A}^{\mathrm{ISS},0}$ and $\hat{A}^{\mathrm{ISS},\mathrm{Any}}\equiv \hat{A}^{\mathrm{ISS},1}$.}
\begin{align*}
\hat{A}^{\mathrm{ISS}, \omega} &\phantom{:}\equiv
\MGselectionset \\
&:= \bigl\{ x \in \R^d : \! X_{i_0} \preccurlyeq x\text{ for some }i_0 \in \rejectionSet_{\alpha}^{\mathrm{MG}, \omega, \bm{v}}\bigl(\mathcal{G}_{\mathrm{W}}(\sample_{X,m}),\bigl(\hat{p}_{\sigma,\tau}(X_i,\sample)\bigr)_{i \in [m]}\bigr)\bigr\}.
\end{align*}
Indeed, by a proof analogous to that of Theorem~\ref{thm:validSelectionSet}, we have $\mathbb{P}_P\bigl(\MGselectionset \subseteq \superLevelSet{\tau}{\eta} | \sample_X\bigr) \geq 1 - \alpha$ whenever $P \in \distributionClassNonDecreasingRegressionFunction$.  However, the budget propagation mechanism in the first step of each iteration has an important drawback: if $I_0 \subseteq I$ is such that there exists $i_* \in I$ with $\{i_*\} = \ch_G(i)$ for all $i\in I_0$, then the sum of the budgets assigned to the nodes in $I_0$ can never exceed the budget that passes through node $i_*$. Moreover, the same conclusion holds for ancestors of nodes in $I_0$ that do not have descendants belonging to an antichain with $i_*$. Intuitively, this can make $i_*$ a bottleneck in the sense that the potentially large number of hypotheses $I_0$ may each only receive a fraction of the budget propagated through $i_*$.



\begin{algorithm}[H]
    \small
    \SetAlgoLined
        \textbf{Input: } $\omega \in \{0, 1\}$, $\alpha \in (0, 1)$, $m\in\N$, a weighted DAG $G = ([m], E, \bm{w})$, $\bm{p} = (p_i)_{i \in [m]} \in (0,1]^m$, $\bm{v} = (v_i)_{i \in [m]} \in (0, \infty)^{m}$\;
        $\pi_G \leftarrow $ a topological ordering of $G$\;
        
        $R^\omega_0 \leftarrow \emptyset$\;
        
        \For{$\ell \in [m]$}{
        
           $S_\mathrm{L} \leftarrow L(G) \setminus R^\omega_{\ell - 1}$ \tcp{$S_\mathrm{L}$ is the set of currently unrejected leaf nodes} 
            $v_i^* \leftarrow v_i/\bigl(\sum_{i'\in S_\mathrm{L}} v_{i'}\bigr)$ for all $i\in S_\mathrm{L}$\; 
            $\alpha_{\ell, 0}^\omega(i)\leftarrow \mathbbm{1}_{\{i \in S_\mathrm{L}\}}\cdot \alpha \cdot v^*_i$ for all $i\in [m]$\;
            \tcp{iteratively distribute the $\alpha$-budget:}
            \For{$k \in [m]$}{
            $i \leftarrow \pi_G^{-1}(k)$ \tcp{iterate through the nodes in order} 
            $R_\mathrm{P} \leftarrow \pa_G(i) \cap R^\omega_{\ell - 1}$ \tcp{currently rejected parents of node $i$}
            $S_\mathrm{P} \leftarrow \pa_G(i) \setminus R^\omega_{\ell - 1}$ \tcp{currently unrejected parents of node $i$}
            
            \uIf{$S_\mathrm{P} \neq \emptyset$}{
            
                \uIf{$\omega = 0$}{
                \tcp{evenly distribute the entire $\alpha$-budget among nodes in $S_\mathrm{P}$:}
                $\alpha_{\ell, k}^\omega(j) \leftarrow \alpha_{\ell, k-1}^\omega(j) + \frac{\alpha_{\ell, {k-1}}^\omega(i)}{|S_\mathrm{P}|}$ for all $j \in S_\mathrm{P}$\;
                \tcp{no $\alpha$-budget remains in the node $i$ itself:}
                $\alpha_{\ell, k}^\omega(i) \leftarrow 0$\;
                } \Else {
                \tcp{divide $\alpha$-budget among $R_\mathrm{P} \cup S_\mathrm{P}$, but only distribute to $S_\mathrm{P}$:}
                $\alpha_{\ell, k}^\omega(j) \leftarrow \alpha_{\ell, k-1}^\omega(j) + \frac{\alpha_{\ell, {k-1}}^\omega(i)}{|\pa_G(i)|}$ for all $j \in S_\mathrm{P}$\;
                \tcp{keep the $\alpha$-budget that would go to nodes in $R_\mathrm{P}$ in node $i$:}
                $\alpha_{\ell, k}^\omega(i) \leftarrow |R_\mathrm{P}| \cdot \frac{\alpha_{\ell, {k-1}}^\omega(i)}{|\pa_G(i)|}$\;
                }
                
            } \Else {
                
                $\alpha_{\ell, k}^\omega(i) \leftarrow \alpha_{\ell, k - 1}^\omega(i)$\;
                    
            }
            }
            
            $\alpha_{\ell}^\omega(i) \leftarrow \alpha_{\ell, m}^\omega(i)$ for all $i \in [m]$\;
            \tcp{reject nodes based on the final distribution of the $\alpha$-budget:}
            $N \leftarrow \{i \in [m]: p_i \leq \alpha_{\ell}^\omega(i)\}$\;
            \If{$\omega = 1$}{$N\leftarrow N \cup \bigcup_{i\in N} \an_G(i)$\;}
            \If{$N = \emptyset$}{
            $R^\omega_{m} \leftarrow R^\omega_{\ell - 1}$\;
            \textbf{break}\;
            }
            $R^\omega_\ell \leftarrow R^\omega_{\ell - 1} \cup N$\;
        }
        \KwResult{The set of rejected hypotheses $\mathcal{R}^{\mathrm{MG}, \omega, \bm{v}}_\alpha(G, \bm{p}) := R^\omega_{m}$}
        
        \caption{The \cite{meijer2015multiple} one-way logical relation DAG testing procedure $\mathcal{R}^{\mathrm{MG}}$.}
        \label{algo:MG}
\end{algorithm}

\subsection{Sub-optimal worst-case performance}\label{sec:negativeMG_theory}

The following proposition illustrates that using the \cite{meijer2015multiple} procedure in our setting leads to a sub-optimal worst-case rate, as seen by comparison with the upper bound for $\hat{A}^{\mathrm{ISS}}$ established in Theorem~\ref{thm:powerBound}.
\begin{prop}\label{prop:negative_res_MG_finite_sample} Let $d \geq 2$, $\tau \in \R$, $\sigma,\etaIncreasingExponent, \etaIncreasingConstant > 0$,  $\densityConstant \in [2^d,\infty)$, $\alpha \in (0, 1/4]$ and $\omega \in \{0, 1\}$. There exists $c > 0$, depending only on $d$, $\alpha$, $\sigma$, $\etaIncreasingConstant$ and $\etaIncreasingExponent$, such that for every $n \in \N$,
\begin{align*}
    \min_{m\in [n]} \sup_{P\in \mathcal{P}'} \inf_{\bm{v} \in (0, \infty)^m} \E_P\bigl\{ \marginalDistribution\bigl(\superLevelSet{\tau}{\regressionFunction}\setminus \MGselectionset\bigr)\bigr\} \geq \frac{c}{n^{1/(2\etaIncreasingExponent + d + 1)}(\log_+ n)^{2/d}},
\end{align*}
where $\mathcal{P}' := \distributionClassNonDecreasingRegressionFunction \cap \distributionClassMultivariateCondition$. 
\end{prop}

\begin{figure}
    \centering
    \begin{subfigure}[b]{0.45\textwidth}
    \resizebox{\textwidth}{!}{
        \begin{tikzpicture}
        \tikzstyle{black circle}=[fill=black, draw=black, shape=circle, minimum size=2pt,inner sep=2pt]
        \tikzstyle{black dot}=[fill=black, draw=black, shape=circle, minimum size=0.75pt,inner sep=0pt]
        
        \draw[gray!80] (-1.5, -0.5) rectangle (10.5, 10.5);
        
        \draw [draw = black, fill = Cerulean!40] (0, 5) rectangle (5,10);
        \node (1) at (2, 5.35) {\footnotesize $A = \superLevelSet{\tau}{\regressionFunction} \cap \support(\marginalDistribution)$};
        \draw [draw = black, fill = Cerulean!80] (1, 6) rectangle (5,10);
        \node (2) at (3, 6.35) {\footnotesize $A_\epsilon := \superLevelSet{\tau + \epsilon}{\regressionFunction} \cap \support(\marginalDistribution)$};
        
        \foreach \jOne in {1, 2}{
            \foreach \jTwo in {1, 2}{
                \node [style=black circle, label = 180:{\large $z_{(\jOne, \jTwo)}$}] (\jOne\jTwo) at ({10 * (\jOne - 1)}, {6 * (\jTwo - 1) / 10}) {};
            }
            
            \node [style=black circle, label = 180:{\large $z_{(\jOne, q)}$}] (\jOne5) at ({10 * (\jOne - 1)}, {6 * (5 - 1) / 10}) {};
            
            \foreach \jTwo in {3, 3.5, 4}{
                \node [style=black dot] (\jOne\jTwo) at ({10 * (\jOne - 1)}, {6 * (\jTwo - 1) / 10}) {};
            }
        }
        
    \end{tikzpicture}
    }
    \subcaption{The support of $\marginalDistribution$ and the shape of the $\tau$- and $(\tau+\epsilon)$-superlevel set of $\regressionFunction$ for $\epsilon \in [0, \etaIncreasingConstant/2^\etaIncreasingExponent]$ in the case $d = 2$.}
    \label{fig:negativeMG_unitsquare_subfigure}
    \end{subfigure}\hspace{0.5cm}%
    \begin{subfigure}[b]{0.45\textwidth}
    \resizebox{\textwidth}{!}{
    \begin{tikzpicture}[shorten >= 2pt, shorten <= 2pt, ->,x = 1cm, y = 1cm]
        \tikzstyle{vertex}=[circle,fill=black!10,minimum size=12pt,inner sep=0pt, font = \scriptsize]
        \tikzstyle{dot}=[circle,fill=black!10,minimum size=2pt,inner sep=0pt]
        \tikzstyle{dark dot}=[circle,fill=black!30,minimum size=2pt,inner sep=0pt]
        \tikzstyle{vertex superlevelset}=[circle,fill=black!10,minimum size=12pt,inner sep=0pt, font = \scriptsize]
        \tikzstyle{dot superlevelset}=[circle,fill=black!10,minimum size=2pt,inner sep=0pt, font = \scriptsize]
        \tikzstyle{fitellipse}=[rotate fit = 45, fill = black!30, ellipse, inner sep = 3pt]
        
        \draw[gray!80] (-1.5, -0.75) rectangle (10.5, 10.25);
        
        \node [vertex] (1_a) at (0, 0) {$i_{\mathrm{L}}$};
        \node [vertex] (2_a) at (1, 2) {$\cdot$};
        \node [vertex] (3_a) at (2, 5) {$\cdot$};
        \node [vertex] (4_a) at (6, 1) {$\cdot$};
        \node [vertex] (5_a) at (7, 3) {$\cdot$};
        \node [vertex] (6_a) at (8, 6) {$\cdot$};
        
        \node [dark dot] (left_middle) at (2, 4) {};
        \node [dark dot] (left_upper) at (2, 4.2) {};
        \node [dark dot] (left_lower) at (2, 3.8) {};

        \node [dark dot] (right_middle) at (8, 5) {};
        \node [dark dot] (right_upper) at (8, 5.2) {};
        \node [dark dot] (right_lower) at (8, 4.8) {};
        
        \foreach \block in {1,...,6}{
        \node [dot] (\block_b) at ($ (\block_a) + (0.5, 0.5) $) {};
        \node [dot, below left = 0.1cm of \block_b] (\block_lower) {};
        \node [dot, above right = 0.1cm of \block_b] (\block_upper) {};
        \node [vertex] (\block_c) at ($ (\block_a) + (1,1) $) {$\cdot$};
        }
        
        \node [vertex] (1_c_label) at (1_c) {$i_1$};
        \node [vertex] (2_c_label) at (2_c) {$i_2$};
        \node [vertex] (3_c_label) at (3_c) {$i_q$};
        
        \node [vertex] (6_c_label) at (6_c) {$i_2^*$};        
        
        \foreach \block in {1,...,6}{
            \draw [black] (\block_c) -- (\block_upper);
            \draw [black] (\block_lower) -- (\block_a);
        }
        
        \foreach \from/\to in {2_a/1_c, left_lower/2_c, 5_a/2_c, 3_a/left_upper, 6_a/3_c, 4_a/1_c, 5_a/4_c, right_lower/5_c, 6_a/right_upper}{
            \draw [black] (\from) -- (\to);
        }

        \node [vertex superlevelset] (lower_1) at (3, 8.5) {$\cdot$};
        \node [vertex superlevelset] (lower_2) at ($(lower_1) +(1, -1)$) {$\cdot$};
        \node [vertex superlevelset] (upper_1) at ($(lower_1) + (1, 1)$) {$\cdot$};
        \node [vertex superlevelset] (upper_2) at ($(lower_2) + (1, 1)$) {$\cdot$};

        \node [dot superlevelset] (lower_dot_middle) at ($(lower_1) + (0.5, -0.5)$) {}; 
        \node [dot superlevelset] (lower_dot_left) at ($(lower_dot_middle) + (-0.15, 0.15)$) {}; 
        \node [dot superlevelset] (lower_dot_right) at ($(lower_dot_middle) + (0.15, -0.15)$) {}; 
        
        \node [dot superlevelset] (upper_dot_middle) at ($(upper_1) + (0.5, -0.5)$) {}; 
        \node [dot superlevelset] (upper_dot_left) at ($(upper_dot_middle) + (-0.15, 0.15)$) {}; 
        \node [dot superlevelset] (upper_dot_right) at ($(upper_dot_middle) + (0.15, -0.15)$) {}; 
        
        \node [dot superlevelset] (left_dot_middle) at ($(lower_1) + (0.5, 0.5)$) {}; 
        \node [dot superlevelset] (left_dot_lower) at ($(left_dot_middle) + (-0.15, -0.15)$) {}; 
        \node [dot superlevelset] (left_dot_upper) at ($(left_dot_middle) + (0.15, 0.15)$) {}; 
        
        \node [dot superlevelset] (right_dot_middle) at ($(lower_2) + (0.5, 0.5)$) {}; 
        \node [dot superlevelset] (right_dot_lower) at ($(right_dot_middle) + (-0.15, -0.15)$) {}; 
        \node [dot superlevelset] (right_dot_upper) at ($(right_dot_middle) + (0.15, 0.15)$) {}; 
        
        \node [dot superlevelset] (center_dot) at ($(lower_1) + (1, 0)$) {}; 
        \node [dot superlevelset] (center_dot_upperleft) at ($(center_dot) + (0, 0.15)$) {}; 
        \node [dot superlevelset] (center_dot_upperright) at ($(center_dot) + (0.15, 0)$) {}; 
        \node [dot superlevelset] (center_dot_lowerright) at ($(center_dot) + (0, -0.15)$) {}; 
        \node [dot superlevelset] (center_dot_lowerleft) at ($(center_dot) + (-0.15, 0)$) {}; 
        
        \foreach \to/\from in {3_c/lower_1, 3_c/lower_2, lower_1/left_dot_lower, lower_1/center_dot_lowerleft, lower_2/center_dot_lowerright, lower_2/right_dot_lower, left_dot_upper/upper_1, center_dot_upperleft/upper_1, center_dot_upperright/upper_2, right_dot_upper/upper_2} {
        \draw [black] (\from) -- (\to);
        }

        \begin{pgfonlayer}{background}
        \foreach \block/\jOne/\jTwo in {1/1/1, 2/1/2, 3/1/q, 4/2/1, 5/2/2, 6/2/q}{
        \node [rotate fit = 45, fit = (\block_a)(\block_b)(\block_c), label = {$I_{(\jOne, \jTwo)}$}, shape = rectangle, fill = black!20, inner sep = 6pt, rounded corners = 0.2cm] {};
        }
        \node [rotate fit = 45, fit = (lower_1)(lower_2)(upper_1)(upper_2), shape = rectangle, fill = black!20, inner sep = 6pt, label = {$I_A$}, rounded corners = 0.2cm] {};
        \end{pgfonlayer}

    \end{tikzpicture}
    }
    \caption{The DAG induced by the points in $\sample_{X,m}$ in Lemma~\ref{lemma:negativeMG_problematic_setup} for any $d \geq 2$. Note $i_q = i_1^*$. See also the proof of Lemma~\ref{lemma:negativeMG_problematic_setup} for notation.}
    \label{fig:negativeMG_DAG_subfigure}
    \end{subfigure}    
    
    \caption{Illustration of the probability distribution defined in Section~\ref{sec:negativeMG_theory} and the resulting induced graph.  The construction demonstrates a problematic consequence of the bottleneck effect described at the end of Section~\ref{sec:MG_procedure}: in order to identify the superlevel set, we need to reject nodes in $I_A$, but unless rejections have been made in previous iterations of Algorithm~\ref{algo:MG}, the combined budget of the nodes in $I_A$ cannot exceed what is passed through $i_q$, which will be very little, as most is propagated towards $i_2^*$.}
    \label{fig:negativeMG_constructed_distribution}
\end{figure}

The main idea of the proof of Proposition~\ref{prop:negative_res_MG_finite_sample} is to construct a distribution in $\distributionClassNonDecreasingRegressionFunction \cap \distributionClassMultivariateCondition$, for which the \citet{meijer2015multiple} algorithm propagates little budget to points in the $\tau$-superlevel set of the regression function $\regressionFunction$.  This distribution, which belongs to $\distributionClassNonDecreasingRegressionFunction \cap \distributionClassMultivariateCondition$ (Lemma~\ref{lemma:negativeMG_distribution}), is illustrated in Figure~\ref{fig:negativeMG_constructed_distribution}.  It consists of $q$ pairs of atoms, where the regression function $\eta$ is well below $\tau$, as well as an absolutely continuous component, where $\eta$ is at least $\tau$ (see Figure~\ref{fig:negativeMG_unitsquare_subfigure}).  The probability masses at each atom are sufficiently large to ensure that, with high probability, we see at least one observation at each of them (Lemma~\ref{lemma:negativeMG_chernoff}).  On this high probability event, the observations therefore induce the DAG illustrated in  Figure~\ref{fig:negativeMG_DAG_subfigure}.  Moreover, the regression function at $i_2^*$ and $i_q$ is sufficiently below $\tau$ that the corresponding $p$-values exceed $\alpha$ with high probability (Lemma~\ref{lemma:negativeMG_grid_pvalues}).  At the same time, the marginal distribution and regression function on the set $A$ in Figure~\ref{fig:negativeMG_unitsquare_subfigure} are chosen so that all of the $p$-values corresponding to points in $A \setminus A_\epsilon$ exceed $\alpha/2^{q-1}$ with high probability (Lemmas~\ref{lemma:negativeMG_chernoff_superlevelset} and~\ref{lemma:negativeMG_uniform}).  But, as we argue in Lemma~\ref{lemma:negativeMG_problematic_setup}, the budget propagation of the \citet{meijer2015multiple} procedure means that a budget of at most  $\alpha/2^{q-1}$ is passed into $A$.  It then follows that with high probability, we can only reject hypotheses corresponding to points in $A_\epsilon$, and in that case the corresponding data-dependent selection set returned will omit $A \setminus A_\epsilon$.  These ideas establish the result when $n$ and $m$ are sufficiently large; when $n$ is small, we can apply our earlier bound in Theorem~\ref{Thm:LowerBound} and when $m$ is small we can apply Proposition~\ref{prop:lowerBound_m}, which provides a lower bound for the worst-case performance of any data-dependent selection set that returns the upper hull of $m$ observations.

To begin our construction, let $A := [0, 1/2] \times [1/2, 1] \times [0,1/2]^{d-2}$ and $q \in \mathbb{N}$.  For $\bm{j} = (j_1, j_2)^\top \in \{1,2\}\times [q]$, define 
\begin{align}
    z_{\bm{j}} := \biggl(j_1-1, \frac{j_2}{2q} - 1, 0\ldots, 0\biggr)^\top \in \mathbb{R}^d.
\end{align}
For $d \geq 2$ and $q \in \N$, let $\mu_q$ denote the distribution on $\R^d$ satisfying:
\begin{itemize}
    \item[$\bullet$] $\mu_q(\{z_{\bm{j}}\}) = (2^d - 1)/(2^{d+1}q)$ for all $\bm{j} \in \{1, 2\} \times [q]$;
    \item[$\bullet$] $\mu_q(A) = 1/2^d$;
    \item[$\bullet$] $X | X\in A \sim \mathrm{Unif}(A)$ when $X \sim \mu_q$.
\end{itemize}
Thus $\sum_{\bm{j} \in \{1,2\}\times [q]} \mu_q(\{z_{\bm{j}}\}) = 1 - 1/2^d$ and $\mu_q(B \cap A) = \Lebesgue(B \cap A)$ for any Borel set $B \subseteq \R^d$.  Write $x_A := (0, 1/2, 0, \ldots, 0)^\top \in \R^d$, so that $x_A \in A$ and $x \succcurlyeq x_A$ for all $x\in A$. For $q\in\N$, $M >0$, $\tau \in \R$, $\etaIncreasingExponent, \etaIncreasingConstant > 0$ define $\eta_{q,M}\equiv \eta_{q, M, \tau, \etaIncreasingConstant, \etaIncreasingExponent}: \R^d \rightarrow \R$ by
\begin{align*}
\eta_{q,M}(x^{(1)},\ldots,x^{(d)}) := 
\begin{cases} 
\tau + \etaIncreasingConstant \cdot \min_{j \in [d]} \bigl(x^{(j)} - x_A^{(j)} \bigr)^\etaIncreasingExponent &\quad\text{ if } x\succcurlyeq x_A\\
\tau - M &\quad\text{ otherwise},
\end{cases}    
\end{align*}
where $x_A^{(j)}$ denotes the $j$th coordinate of $x_A$.  Finally, for $q\in\N$, $M >0$, $\sigma > 0$, $\tau \in \R$, $\etaIncreasingExponent, \etaIncreasingConstant > 0$, let $P_{q,M} \equiv P_{q,M,\sigma,\tau, \etaIncreasingConstant,\etaIncreasingExponent}$ denote any joint distribution of $(X, Y)$ such that $X$ has marginal distribution $\mu_q$, and $Y| X \sim \mathcal{N}\bigl(\eta_{q,M}(X), \sigma^2\bigr)$.
\begin{lemma}\label{lemma:negativeMG_distribution}
    For $d \geq 2$, $\tau \in \R$, $\sigma, \etaIncreasingExponent, \etaIncreasingConstant > 0$, $q \in \N$ and $M > 0$, we have $P_{q,M} \equiv P_{q,M,\sigma,\tau,\etaIncreasingConstant,\etaIncreasingExponent} \in \distributionClassNonDecreasingRegressionFunction \cap \distributionClassMultivariateCondition$ for all $\densityConstant \geq 2^d$.
\end{lemma}

\begin{proof}
    We first prove that $P_{q,M} \in \distributionClassNonDecreasingRegressionFunction$. Since the sub-Gaussianity condition is satisfied by construction, it suffices to show that $\eta_{q,M}$ is coordinate-wise increasing on $\R^d$. Whenever $x_0 \not\succcurlyeq x_A$, we have $\eta_{q,M}(x_0) = \inf_{x\in\R^d} \eta_{q,M}(x)$.  On the other hand, for $x_0, x_1 \in \R^d$ with $x_A \preccurlyeq x_0 \preccurlyeq x_1$, we have 
    \[
    \eta_{q,M}(x_0) = \tau + \lambda \cdot \min_{j\in [d]}\bigl(x_0^{(j)} - x_A^{(j)}\bigr)^\etaIncreasingExponent \leq \tau + \lambda \cdot \min_{j\in [d]}\bigl(x_1^{(j)} - x_A^{(j)}\bigr)^\etaIncreasingExponent = \eta_{q,M}(x_1),
    \]
    as required.

    We now show that $P_{q,M} \in \distributionClassMultivariateCondition$ and start by establishing that the condition in Definition~\ref{def:multivariateAssumption}\emph{(i)} is satisfied. For any $x \in \superLevelSet{\tau}{\eta} \cap \support(\mu_q)= A$ and $r \in (0, 1]$, we have $\mu_q\bigl(\closedSupNormMetricBall{x}{r}\bigr) \geq \bigl(r \wedge (1/2)\bigr)^d \geq (r/2)^d$.  For $r \geq 1/4^d$, we have $\mu_q\bigl(\closedSupNormMetricBall{x}{r}\bigr) \leq 1 \leq \theta \cdot (2r)^d$, so let $r \in (0, 1/4^d)$. We then have for any $x\in A$ that $\closedSupNormMetricBall{x}{r} \cap \bigcup_{\bm{j}\in \{1,2\}\times [q]} \{z_{\bm{j}}\} = \emptyset$ and hence $\mu_q\bigl(\closedSupNormMetricBall{x}{r}\bigr) = \mu_q\bigl(\closedSupNormMetricBall{x}{r} \cap A\bigr) \leq (2r)^d$.  Finally, for Definition~\ref{def:multivariateAssumption}\emph{(ii)}, observe that for any $x \in \superLevelSet{\tau}{\eta} \cap \support(\mu_q) = A$ and $r\in (0, 1]$, we have $x_0 := x + r\bm{1}_d \in \closedSupNormMetricBall{x}{r}$ satisfies $x_0 - x_A \succcurlyeq r\bm{1}_d$ and hence $\eta_{q,M}(x_0) \geq \tau + \etaIncreasingConstant r^\etaIncreasingExponent$, as required.
\end{proof}

\begin{lemma}\label{lemma:negativeMG_chernoff}
Fix $d \geq 2$, $\delta \in (0,1)$, positive integers $m \leq n$ and $q \leq \bigl\lfloor \frac{m}{32\log_+(m/\delta)} \bigr\rfloor$.  If $\sample_{X,m} = \bigl(X_1, \ldots, X_m\bigr) \sim \mu_q^m$, and we define $\Omega_1 := \bigcap_{\bm{j} \in \{1, 2\} \times [q]} \bigl\{\{z_{\bm{j}}\} \cap \sample_{X, m} \neq \emptyset\bigr\}$, then
$\mathbb{P}_{\mu_q}(\Omega_1^c) \leq \delta/4$.
\end{lemma}

\begin{proof}
For $\bm{j} \in \{1, 2\} \times [q]$, let
\[
\Omega_{1, \bm{j}} := \biggl\{\frac{1}{m}\sum_{i=1}^m \mathbbm{1}_{\{X_i = z_{\bm{j}}\}} \geq \frac{2^d - 1}{2^{d+2}q}\biggr\}.
\]
Then by the multiplicative Chernoff bound \citep[][Theorem~2.3(c)]{mcdiarmid1998concentration}, the fact that $(2^d - 1)/2^{d+4} \geq 1/32$ and the choice of $q$, we have
\[
\mathbb{P}_{\mu_q}\biggl(\bigcup_{\bm{j} \in \{1, 2\} \times [q]} \Omega^c_{1, \bm{j}}\biggr) \leq 2q \cdot \exp\biggl(-\frac{2^d - 1}{2^{d+4}q} \cdot m\biggr) \leq \frac{m}{4}\cdot \exp\Bigl\{-\log_+\Bigl(\frac{m}{\delta}\Bigr)\Bigr\} \leq \frac{\delta}{4}.
\]
Moreover, 
\[
\frac{(2^d-1)m}{2^{d+2}q} \geq \frac{m}{8q} \geq 4\log_+(m/\delta) \geq 1,
\]
whence
\[
\mathbb{P}_{\mu_q}\bigl(\Omega_1^c\bigr) = \mathbb{P}_{\mu_q}\biggl(\bigcup_{\bm{j} \in \{1, 2\} \times [q]} \biggl\{\frac{1}{m}\sum_{i=1}^m \mathbbm{1}_{\{X_i = z_{\bm{j}}\}} = 0\biggr\}\biggr) \leq \mathbb{P}_{\mu_q}\biggl(\bigcup_{\bm{j} \in \{1, 2\} \times [q]} \Omega_{1,\bm{j}}^c\biggr) \leq \frac{\delta}{4},
\]
as required.
\end{proof}

\begin{lemma}\label{lemma:negativeMG_grid_pvalues} Fix $d \geq 2$, $\alpha \in (0,1)$, $\delta \in (0,1]$, $n \in \N$, $m \in [n]$, $\tau \in \R$, $\sigma, \etaIncreasingExponent, \etaIncreasingConstant > 0$, $q\in\N$ and $M \geq 1.7\sigma\sqrt{\log(41.6/\delta)}$.  Let $\sample = \bigl((X_1,Y_1),\ldots,(X_n,Y_n)\bigr) \sim P_{q, M}^n$, and suppose that $\bigl\{i \in [m]:X_i = z_{(j,q)}\bigr\} \neq \emptyset$ for $j \in \{1, 2\}$.  Write $i^*_j := \max\{i \in [m]: X_i = z_{(j,q)}\}$ for $j \in \{1,2\}$, and let 
\[
\Omega_2 := \bigcap_{j=1}^2 \bigl\{\hat{p}_{\sigma, \tau} (X_{i^*_j}, \sample) > \alpha \bigr\}.
\]
Then $\Prob_{P_{q,M}}\bigl(\Omega_2^c | \sample_X\bigr) \leq \delta/4$.
\end{lemma}
\begin{proof} 
For $x \in \R^d$ and $r > 0$, define $\mathcal{I}_r(x) := \{i \in [n]: X_i \preccurlyeq x, \|X_i - x\|_\infty \leq r\}$.  Fix $j \in \{1,2\}$, and note that 
\begin{align*}
\frac{\sigma}{|\mathcal{I}_r(X_{i^*_j})|} \cdot u_{\delta/8}
\bigl(|\mathcal{I}_r(X_{i^*_j})|\bigr) &\leq 1.7\sigma\sqrt{0.2 + 0.72\log(41.6/\delta)} \leq 1.7\sigma \sqrt{\log(41.6/\delta)} \leq M
\end{align*}
for all $r > 0$.  It follows by Lemma~\ref{lemma:howard_uniform_bound}\emph{(a)} that, with probability at least $1 - \delta/8$ given $\sampleX$, we have simultaneously for all $r > 0$ that
\begin{align*}
\sum_{i\in \mathcal{I}_r(X_{i^*_j})} \frac{Y_i - \tau}{\sigma} &= \sum_{i\in \mathcal{I}_r(X_{i^*_j})} \frac{Y_i - (\tau-M)}{\sigma} - |\mathcal{I}_r(X_{i^*_j})|\cdot \frac{M}{\sigma} \\
&\leq u_{\delta/8}
\bigl(|\mathcal{I}_r(X_{i^*_j})|\bigr) - |\mathcal{I}_r(X_{i^*_j})|\cdot \frac{M}{\sigma} \leq 0,
\end{align*}
so that $\hat{p}_{\sigma, \tau} (X_{i^*_j}, \sample) = 1$, and thus in particular $\hat{p}_{\sigma, \tau} (X_{i^*_j}, \sample) > \alpha$. Hence, the result follows by a union bound over $j\in\{1,2\}$.
\end{proof}

\begin{lemma}\label{lemma:negativeMG_chernoff_superlevelset}
Fix $d \geq 2$, $\alpha \in (0,1)$, $\delta \in (0,1/4]$, $n \in \N$, $\sigma, \etaIncreasingExponent, \etaIncreasingConstant > 0$, $s \in (0, 1/2]$ and $q\in \N$. Let $\sample_X = (X_1, \ldots, X_n)\sim \mu_q^n$ and let $B_j := x_A + [0, 1/2]^{j-1} \times [0, s] \times [0, 1/2]^{d-j}$ for $j \in [d]$. Denote $w_{n,m,\delta} := 173.13 \bigl(\log_+\log n + \log_+(m/\delta)\bigr)$. If
\[
\frac{8}{3n}\log\Bigl(\frac{4d}{\delta}\Bigr) \leq \frac{s}{2^{d-1}} \leq \frac{\sigma^2}{2n\etaIncreasingConstant^2 s^{2\etaIncreasingExponent}}\bigl(0.72q - w_{n,m,\delta}\bigr)
\]
then, writing
\[
\Omega_3 := \bigcap_{j\in[d]} \biggl\{|\sample_X \cap B_j| < \frac{\sigma^2}{\etaIncreasingConstant^2 s^{2\etaIncreasingExponent}}\bigl(0.72q - w_{n,m,\delta}\bigr) \biggr\},
\]
we have $\Prob_{\mu_q}(\Omega_3^c) \leq \delta/4$.
\end{lemma}

\begin{proof} Fix any $j\in [d]$ and note that $\mu_q(B_j) = s/2^{d-1}$. By the upper bound on $s$, a multiplicative Chernoff bound \citep[][Theorem~2.3(b)]{mcdiarmid1998concentration} and the lower bound on $s$, we have
\begin{align*}
\Prob_{\mu_q} \Bigl(|\sample_X \cap B_j| \geq \frac{\sigma^2}{\etaIncreasingConstant^2 s^{2\etaIncreasingExponent}}\bigl(0.72q - w_{m,n,\delta}\bigr)\Bigr)
&\leq \Prob_{\mu_q} \Bigl(|\sample_X \cap B_j| \geq \frac{n s}{2^{d-2}}\Bigr) \\
&\leq \exp\Bigl(-\frac{3ns}{8\cdot 2^{d-1}}\Bigr) \leq \frac{\delta}{4d}.
\end{align*}
The result therefore follows by a union bound.
\end{proof}

\begin{lemma}\label{lemma:negativeMG_uniform}
Fix $d \geq 2$, $\alpha \in (0,1)$, $\delta \in (0,1]$, $n \in \N$,  $m \in [n]$, $\tau \in \R$, $\sigma, \etaIncreasingExponent, \etaIncreasingConstant > 0$, $q\in \N$, $M > 0$ and $s>0$. Let $\sample = \bigl((X_1,Y_1),\ldots,(X_n,Y_n)\bigr) \sim P_{q,M}^n$, and let $1 \leq i_1 < \ldots < i_K \leq m$ be such that $\{i_1,\ldots,i_K\} := \bigl\{i \in [m]:X_i \in \superLevelSet{\tau}{\eta_{q,M}}\setminus\superLevelSet{\tau + \lambda s^\gamma}{\eta_{q,M}}\bigr\}$.  Denote further $w_{n,m,\delta} := 173.13 \bigl(\log_+\log n + \log_+(m/\delta)\bigr)$ as in Lemma~\ref{lemma:negativeMG_chernoff_superlevelset}. If $q \geq w_{n,m,\delta}/0.72 \geq 3\log(5.2\cdot 8 \cdot m/\delta) + 2\log\log(2n)$ and
\[
\max_{k \in [K]} \bigl|\{i \in [n]: X_i \in A, X_i \preccurlyeq X_{i_k}\}\bigr| \leq \frac{\sigma^2}{\etaIncreasingConstant^2 s^{2\etaIncreasingExponent}} \bigl(0.72q - w_{n,m,\delta}\bigr)
\]
then writing 
\[
\Omega_4 := \bigcap_{k\in[K]}\Bigl\{\hat{p}_{\sigma, \tau}(X_{i_k}, \sample) > \frac{\alpha}{2^{q-1}}\Bigr\},
\]
we have $\Prob_{P_{q,M}}\bigl(\Omega_4^c | \sample_X\bigr) \leq \delta/4$.
\end{lemma}

\begin{proof}
When $K=0$, i.e.~$\bigl\{i \in [m]:X_i \in \superLevelSet{\tau}{\eta_{q,M}}\setminus\superLevelSet{\tau + \lambda s^\gamma}{\eta_{q,M}}\bigr\} = \emptyset$, then $\Omega_4^c = \emptyset$ and there is nothing to prove, so assume that $K \in [m]$. For $x \in \R^d$ and $r > 0$, define $\mathcal{I}_r(x) := \{i \in [n]: X_i \preccurlyeq x, \|X_i - x\|_\infty \leq r\}$, $\mathcal{I}^A_r(x) := \{i \in \mathcal{I}_r(x): X_i \in A\}$ and accordingly $\mathcal{I}^{A^c}_r(x) := \{i \in \mathcal{I}_r(x): X_i \notin A\}$.  Fix any $k \in [K]$ and note first that by assumption,
\[
|\mathcal{I}^A_r(X_{i_k})| \leq \bigl|\{i \in [n]: X_i \in A, X_i \preccurlyeq X_{i_k}\}\bigr| \leq \frac{\sigma^2}{\etaIncreasingConstant^2 s^{2\etaIncreasingExponent}}\bigl(0.72q - w_{n,m,\delta}\bigr)
\]
for all $r > 0$.  Hence 
\begin{align}
\label{Eq:utildeminusu}
    \frac{\etaIncreasingConstant^2 s^{2\etaIncreasingExponent}}{\sigma^2}\bigl|\mathcal{I}^A_r(X_{i_k})\bigr|^2 &\leq \bigl|\mathcal{I}^A_r(X_{i_k})\bigr| \cdot  \bigl(0.72q - w_{n,m,\delta}\bigr) \nonumber \\
    &\leq \frac{|\mathcal{I}_r(X_{i_k})|}{2}\biggl\{2.0808\cdot q \cdot \log 2 -6\cdot 1.7^2\cdot\log_+\log n \nonumber \\
    &\hspace{6.5cm} - 4\cdot 1.7^2\cdot 0.72\cdot 5.2 \cdot 8 \cdot \log_+\Bigl(\frac{m}{\delta}\Bigr)\biggr\} \nonumber \\
    &\leq \frac{|\mathcal{I}_r(X_{i_k})|}{2}\biggl\{2.0808\log\Bigl(\frac{5.2}{\alpha} \cdot 2^{q-1}\Bigr) + \Bigl(1.7^2 - 4 \cdot 1.7^2\Bigr)\log\log\bigl(2|\mathcal{I}_r(X_{i_k})|\bigr)\nonumber \\
    &\hspace{6.5cm} - 4\cdot 1.7^2\cdot 0.72\cdot \log\Bigl(\frac{5.2 \cdot 8 \cdot K}{\delta}\Bigr)\biggr\} \nonumber \\
    &\leq |\mathcal{I}_r(X_{i_k})|\biggl\{\sqrt{1.7^2\cdot 0.72\log\Bigl(\frac{5.2}{\alpha} \cdot 2^{q-1}\Bigr) + 1.7^2 \log\log\bigl(2|\mathcal{I}_r(X_{i_k})|\bigr)} \nonumber \\
    &\hspace{1.5cm}- \sqrt{2 \cdot 1.7^2 \log\log\bigl(2|\mathcal{I}_r(X_{i_k})|\bigr) + 2\cdot 1.7^2\cdot 0.72\cdot \log\Bigl(\frac{5.2 \cdot 8 \cdot K}{\delta}\Bigr)}\biggr\}^2 \nonumber \\
    &= \Bigl(u_{\alpha/2^{q-1}} \bigl(|\mathcal{I}_r(X_{i_k})|\bigr) - \sqrt{2} \cdot u_{\delta/(8K)}\bigl(|\mathcal{I}_r(X_{i_k})|\bigr)\Bigr)^2,
\end{align}
where in the final inequality, we used the fact that $(a-2b)/2 \leq \bigl(\sqrt{a} - \sqrt{b}\bigr)^2$ for $a,b \geq 0$.  By Lemma~\ref{lemma:howard_uniform_bound}\emph{(a)}, with probability at least $1 - \delta/(4K)$ conditional on $\sampleX$, we have simultaneously for all $r > 0$ that
\begin{align*}
    \sum_{i\in\mathcal{I}_r(X_{i_k})} \frac{Y_i - \tau}{\sigma} &\leq \sum_{i\in\mathcal{I}^A_r(X_{i_k})} \frac{Y_i - (\tau + \etaIncreasingConstant s^\etaIncreasingExponent)}{\sigma} + \sum_{i\in\mathcal{I}^{A^c}_r(X_{i_k})}  \frac{Y_i - (\tau - M)}{\sigma} + \frac{\etaIncreasingConstant s^\etaIncreasingExponent}{\sigma}|\mathcal{I}^A_r(X_{i_k})| \\
    &< u_{\delta/(8K)}\bigl(|\mathcal{I}^A_r(X_{i_k})|\bigr) + u_{\delta/(8K)}\bigl(|\mathcal{I}^{A^c}_r(X_{i_k})|\bigr) + \frac{\etaIncreasingConstant s^\etaIncreasingExponent}{\sigma}|\mathcal{I}^A_r(X_{i_k})| \\
    &\leq \sqrt{2}\cdot  u_{\delta/(8K)}\bigl(|\mathcal{I}_r(X_{i_k})|\bigr) +  \frac{\etaIncreasingConstant s^\etaIncreasingExponent}{\sigma}|\mathcal{I}^A_r(X_{i_k})| \\
    &\leq u_{\alpha/2^{q-1}} \bigl(|\mathcal{I}_r(X_{i_k})|\bigr),
\end{align*}
where the third inequality follows from that fact that $\sqrt{a} + \sqrt{b} \leq \sqrt{2}\cdot \sqrt{a + b}$ for all $a, b \geq 0$, and the fourth follows from~\eqref{Eq:utildeminusu} and the fact that $u_{\alpha/2^{q-1}} \bigl(|\mathcal{I}_r(X_{i_k})|\bigr) \geq \sqrt{2} \cdot u_{\delta/(8K)}\bigl(|\mathcal{I}_r(X_{i_k})|\bigr)$ since $q \geq 3\log(5.2\cdot 8 \cdot m/\delta) + 2\log\log(2n)$.  But
\[
\biggl\{\sum_{i\in\mathcal{I}_r(X_{i_k})} \frac{Y_i - \tau}{\sigma} < u_{\alpha/2^{q-1}} \bigl(|\mathcal{I}_r(X_{i_k})|\bigr)\biggr\} = \bigl\{\hat{p}_{\sigma, \tau}(X_{i_k}, \sample) > \alpha/2^{q-1}\bigr\},
\]
so the result follows by a union bound over $k\in[K]$.  
\end{proof}
\begin{lemma} \label{lemma:negativeMG_problematic_setup}
Fix $m, q \in \mathbb{N}$ and suppose that $\mathcal{D}_{X,m} = \{X_i:i \in [m]\} \subseteq \bigl\{z_{\bm{j}}:\bm{j} \in \{1,2\} \times [q]\bigr\} \cup A$ with $\{z_{\bm{j}}\} \cap \sample_{X, m} \neq \emptyset$ for all $\bm{j} \in \{1, 2\} \times [q]$.  Fix $\alpha \in (0,1)$ and let $(p_i)_{i \in [m]} \in (0, 1]^m$ be such that $\min_{j\in\{1,2\}} p_{i^*_j} > \alpha$, where $i^*_j := \max\{i\in[m]: X_i = z_{(j, q)}\}$. Then for $\omega \in \{0, 1\}$ and $\bm{v} \in (0,\infty)^m$, we have
\[
\mathcal{R}^{\mathrm{MG}, \omega, \bm{v}}_\alpha\bigl(\mathcal{G}_{\mathrm{W}}(\sample_{X,m}), (p_i)_{i \in [m]}\bigr) \cap \{i \in [m]: p_i > \alpha/2^{q - 1} \text{ and } X_i \in A\} = \emptyset.
\]
\end{lemma}

\begin{proof}
Let $R_0^\omega := \emptyset$, and for $\ell \in [m]$ and $k \in [m]\cup \{0\}$, let $R_\ell^\omega$, $\alpha^\omega_\ell$ and $\alpha_{\ell, k}^\omega$ be as in Algorithm~\ref{algo:MG}. Furthermore, for ease of notation, let $I_A := \{i\in [m]: X_i \in A\}$, $I_{\bm{j}} := \{i\in [m]: X_i = z_{\bm{j}}\}$ for $\bm{j} \in \{1, 2\} \times [q]$ and $I_{A^c} := \{i\in [m]: X_i \notin A\} = \bigcup_{\bm{j}\in \{1, 2\} \times [q]} I_{\bm{j}}$; see Figure~\ref{fig:negativeMG_DAG_subfigure}.  For each $\ell \in [m]\cup \{0\}$ let $P(\ell)$ denote the proposition that
\begin{align*}
R_\ell^\omega \cap \Bigl(\bigl\{i \in I_A: p_i > \alpha (1/2)^{q - 1}\bigr\} \cup I_{A^c} \Bigr) =\emptyset.
\end{align*}
Since $\mathcal{R}^{\mathrm{MG}, \omega, \bm{v}}_\alpha\bigl(\mathcal{G}_{\mathrm{W}}(\sample_{X,m}), (p_i)_{i \in [m]}\bigr) = R_m^\omega$, the result will follow if $P(\ell)$ is true for all $\ell \in [m]\cup \{0\}$, and we prove this by induction on $\ell$. First, note that $P(0)$ is true since $R^\omega_0 = \emptyset$.  Now, fix any $\ell_0 \in [m-1]\cup \{0\}$ such that $P(\ell_0)$ holds true.  In particular, this means that no hypothesis corresponding to a node in $I_{A^c}$ has been rejected in the first $\ell_0$ steps of the algorithm.  For $j \in [q]$, define $i_j := \max I_{(1, j)}$, so that $i_q = i_1^*$, and let $G := \mathcal{G}_{\mathrm{W}}(\sample_{X,m})$.  Since $\mathrm{pa}_G(i) \subseteq I_{A^c}$ whenever $i \in I_{A^c} \setminus \{i_1^*\}$, we have that $P(\ell_0 + 1)$ is true if  $i_1^*, i_2^* \notin R^\omega_{\ell_0 + 1}$ and $\alpha^\omega_{\ell_0+1}(i)\leq \alpha/2^{q-1}$ for all $i \in I_A$.  Regarding the first of these conditions, we have $\sum_{i\in[m]} \alpha^\omega_\ell(i) = \alpha$ for all $\ell \in [m]\cup \{0\}$, so in particular $\alpha^\omega_{\ell_0 + 1}(i_j^*) \leq \alpha < p_{i_j^*}$ for $j \in \{1,2\}$ and hence $i_1^*, i_2^* \notin R^\omega_{\ell_0 + 1}$.  For the second condition, note that by definition, $L(G) = \{i_{\mathrm{L}}\}$ with $i_{\mathrm{L}} := \min I_{(1,1)}$.  There exists exactly one directed path from $i_j$ to $i_{\mathrm{L}}$, unless $j = 1$ and $|I_{(1,1)}| = 1$.  Thus, for any $j\in [q]$, it follows that $k_j := \min\bigl\{k \in [m]\cup\{0\}:  \alpha_{\ell_0+1, k}^\omega (i_j) > 0\bigr\}$ is the maximiser of $k \mapsto \alpha_{\ell_0+1, k}^\omega (i_j)$ over $k \in [m]\cup\{0\}$. We now claim that $\alpha_{\ell_0 + 1, k_j}^\omega (i_j) = \alpha/2^{j-1}$ and show this by another induction, this time on $j \in [q]$. First, each node in $I_{(1,1)}\setminus \{i_1\}$ has exactly one parent and this parent is itself contained in $I_{(1,1)}$, so that $\alpha_{\ell_0 + 1, k_1}^\omega (i_1) = \alpha_{\ell_0 + 1, 0}^\omega (i_{\mathrm{L}}) = \alpha$ for any $\bm{v}$.  If $q=1$, this establishes the claim; otherwise, fix $j_0 \in [q-1]$ for which $\alpha_{\ell_0 + 1, k_{j_0}}^\omega (i_{j_0}) = \alpha/2^{j_0-1}$. By construction, $\pa_G(i_{j_0}) = \{\min I_{(1, j_0+1)}, \min I_{(2, j_0)}\}$ and each node in $I_{(1, j_0+1)} \setminus \{i_{j_0+1}\}$ has again exactly one parent, which is contained in $I_{(1, j_0+1)}$, while at the same time each node in $I_{(1, j_0+1)}\setminus \{\min I_{(1,j_0+1)}\}$ has exactly one child, which is also contained in $I_{(1, j_0+1)}$.  Hence $\alpha_{\ell_0 + 1, k_{j_0+1}}^\omega (i_{j_0+1}) = \alpha_{\ell_0 + 1, k_{j_0}}^\omega(i_{j_0})/2 = \alpha(1/2)^{j_0}$, which completes the induction on $j \in [q]$.  Since for any $i \in I_A$, any directed path in $G$ from $i$ to $i_{\mathrm{L}}$ necessarily contains $i_q$, we deduce that
\[
\max_{i\in I_A} \alpha_{\ell_0+1}^\omega(i) \leq \sum_{i\in I_A} \alpha_{\ell_0+1}^\omega(i) \leq \alpha_{\ell_0+1, k_q }^\omega(i_q) = \frac{\alpha}{2^{q-1}},
\]
which completes the induction on $\ell \in [m]\cup\{0\}$ and hence the proof. 
\end{proof}
For $m\in[n]$, let $\hat{\mathcal{A}}^{\mathrm{U}}_{n,m}(\tau, \alpha, \mathcal{P}) \subseteq \hat{\mathcal{A}}_n(\tau, \alpha, \mathcal{P})$ denote the subfamily of data-dependent selection sets that control the Type I error at level $\alpha$ over $\mathcal{P}$ and for which $\hat{A}(\sample)$ is almost surely the upper hull of a subset of $\sample_{X,m}$.  Thus, for example, $\MGselectionset \in \hat{\mathcal{A}}^{\mathrm{U}}_{n,m}\bigl(\tau, \alpha, \distributionClassNonDecreasingRegressionFunction\bigr)$ and $\hat{A}^{\mathrm{ISS}}_{\sigma, \tau, \alpha, m}(\sample) \in \hat{\mathcal{A}}^{\mathrm{U}}_{n,m}\bigl(\tau, \alpha, \distributionClassNonDecreasingRegressionFunction\bigr)$.  
\begin{prop}\label{prop:lowerBound_m}
Fix $d\in\N$, $\alpha \in (0, 1/4]$, $n\in\N$, $m\in [n]$, $\tau \in \R$, $\sigma, \etaIncreasingExponent, \etaIncreasingConstant > 0$ and $\densityConstant > 1$. There exists $c \in (0, 1)$, depending only on $d$, such that
\[
\sup_{P \in \mathcal{P}'} \inf_{\hat{A}\in\hat{\mathcal{A}}^{\mathrm{U}}_{n,m}(\tau, \alpha, \mathcal{P}')} \E_P\bigl\{\mu\bigl(\superLevelSet{\tau}{\eta}\setminus \hat{A}(\sample)\bigr)\bigr\} \geq c\cdot \frac{1}{m^{1/d}},
\]
where $\mathcal{P}' := \distributionClassNonDecreasingRegressionFunction \cap \distributionClassMultivariateCondition$.
\end{prop}

\begin{proof}
For $q \in \N$, let the antichain $\mathbb{W}_{q,d}$, hypercubes $\mathcal{H}_{\bm{j}}^q$ for $\bm{j} \in \mathbb{W}_{q,d}$ as well as $P_S$, $\eta_S$ for $S\subseteq \mathbb{W}_{q,d}$ be defined as in Section~\ref{sec:lowerBound_proofs} and let $\mu := \mathrm{Unif}\bigl([0,1]^d\bigr)$. For ease of notation, we write $P_* := P_S$ and $\eta_* := \eta_S$ when $S = \mathbb{W}_{q,d}$. Note first that for any $\hat{A} \in \hat{\mathcal{A}}^{\mathrm{U}}_{n,m}(\tau, \alpha, \mathcal{P}')$ we have on $\{\hat{A}(\sample) \subseteq \superLevelSet{\tau}{\eta_*}\}$ that
\[
\hat{S} := \{\bm{j} \in \mathbb{W}_{q,d}:\hat{A}(\sample)\cap \mathcal{H}_{\bm{j}}^q \neq \emptyset\} \subseteq \{\bm{j} \in \mathbb{W}_{q,d}: \sample_{X,m} \cap \mathcal{H}_{\bm{j}}^q \neq \emptyset\} =: \Tilde{S}.
\]
Now, for any $q\in\N$ and $\hat{A} \in \hat{\mathcal{A}}^{\mathrm{U}}_{n,m}(\tau, \alpha, \mathcal{P}')$, we have $\mu\bigl(\superLevelSet{\tau}{\eta_*}\setminus\hat{A}(\sample)\bigr) \geq  |\mathbb{W}_{q,d}\setminus \hat{S}|/{q^d}$ and $|\mathbb{W}_{q,d}|\geq q^{d-1}/d$, so that
\begin{align}
\E_{P_*}\bigl\{\mu\bigl(\superLevelSet{\tau}{\eta_*}\setminus \hat{A}(\sample)\bigr)\bigr\} \geq \frac{|\mathbb{W}_{q,d}|}{q^d} \cdot \E_{P_*}\biggl(\frac{|\mathbb{W}_{q,d}\setminus \hat{S}|}{|\mathbb{W}_{q,d}|}\biggr) \geq \frac{1}{d\cdot q} \cdot \E_{P_*}\biggl(\frac{|\mathbb{W}_{q,d}\setminus \hat{S}|}{|\mathbb{W}_{q,d}|}\biggr).
\label{eq:lowerBound_m}    
\end{align}
On the other hand, 
\begin{align*}
    \E_{P_*}\bigl(|\mathbb{W}_{q,d}\setminus \Tilde{S}|\bigr) = \E_{P_*}\biggl(\sum_{\bm{j} \in \mathbb{W}_{q,d}} \one_{\{\mathcal{H}_{\bm{j}}^q \cap \sample_{X,m} = \emptyset\}}\biggr) = |\mathbb{W}_{q,d}| \Bigl(1 - \frac{1}{q^d}\Bigr)^m. 
\end{align*}
Hence, when setting $q = \lceil (2m)^{1/d} \rceil$ and writing $S_* := \mathbb{W}_{\lceil(2m)^{1/d}\rceil, d}$, we find that $\E_{P_*}\bigl(|S_*\setminus \Tilde{S}| / |S_*| \bigr) \geq \bigl(1 - 1/(2m)\bigr)^m \geq 1/2$, so that $\Prob_{P_*}\bigl(|S_*\setminus \Tilde{S}| / |S_*|  \geq 1/5 \bigr) \geq 3/8$. Since $\Prob_{P_*}\bigl(\hat{A}(\sample) \subseteq \superLevelSet{\tau}{\eta_*}\bigr) \geq 3/4$ as $\alpha \in (0, 1/4]$, it follows that 
\begin{align*}
    \Prob_{P_*}\biggl(\frac{|S_*\setminus \hat{S}|}{|S_*|} \geq \frac{1}{5}\biggr) &\geq \Prob_{P_*}\biggl(\biggl\{\frac{|S_*\setminus \Tilde{S}|}{|S_*|} \geq \frac{1}{5}\biggr\}\cap \bigl\{\hat{A}(\sample) \subseteq \superLevelSet{\tau}{\eta_*}\bigr\}\biggr)\\
    &\geq \Prob_{P_*}\biggl(\frac{|S_*\setminus \Tilde{S}|}{|S_*|} \geq \frac{1}{5}\biggr) - \Prob_{P_*} \bigl(\hat{A}(\sample) \nsubseteq \superLevelSet{\tau}{\eta_*}\bigr) \geq \frac{3}{8} - \frac{1}{4} \geq \frac{1}{8}.
\end{align*}
Hence, $\E_{P_*}\bigl(|S_*\setminus \hat{S}|/ |S_*|\bigr) \geq 1/40$. This holds uniformly over $\hat{A} \in \hat{\mathcal{A}}^{\mathrm{U}}_{n,m}(\tau, \alpha, \mathcal{P}')$, since~$\Tilde{S}$ does not depend on the specific choice of $\hat{A}$. Combining this with~\eqref{eq:lowerBound_m} for the specified choice of $q$, we have 
\begin{align*}
    \sup_{P \in \mathcal{P}'} \inf_{\hat{A} \in \hat{\mathcal{A}}^{\mathrm{U}}_{n,m}(\tau, \alpha, \mathcal{P}')} \E_P\bigl\{\bigl(\superLevelSet{\tau}{\eta}\setminus \hat{A}(\sample)\bigr)\bigr\} &\geq \inf_{\hat{A} \in \hat{\mathcal{A}}^{\mathrm{U}}_{n,m}(\tau, \alpha, \mathcal{P}')} \E_{P_*}\bigl\{\mu\bigl(\superLevelSet{\tau}{\eta_*}\setminus \hat{A}(\sample)\bigr)\bigr\} \\
    &\geq \frac{1}{2d(2m)^{1/d}} \cdot\inf_{\hat{A} \in \hat{\mathcal{A}}^{\mathrm{U}}_{n,m}(\tau, \alpha, \mathcal{P}')}  \E_{P_*}\biggl(\frac{|S_*\setminus \hat{S}|}{|S_*|}\biggr) \\
    &\geq \frac{1}{80\cdot2^{1/d}d}\cdot \frac{1}{m^{1/d}},
\end{align*}
which yields the result with $c = 1/(80\cdot2^{1/d}d)$.
\end{proof}
We are now in a position to prove Proposition~\ref{prop:negative_res_MG_finite_sample}.
\begin{proof}[Proof of Proposition~\ref{prop:negative_res_MG_finite_sample}] Fix $\delta \in (0, 1)$. To begin with, we consider cases arising when either $n$ or $m$ are small, before the main part of the proof deals with $m$ and $n$ sufficiently large. First, suppose that 
\[
n <  \exp\Bigl(\frac{\sigma^2}{\etaIncreasingConstant^2}\cdot 2^{2\etaIncreasingExponent - 13}\Bigr) \vee \biggl\{2^{d+1} \Bigl(\frac{\etaIncreasingConstant^2}{\sigma^2}\Bigr)^{1/(2\etaIncreasingExponent + d)}\log\Bigl(\frac{4d}{\delta}\Bigr)\biggr\}^{(2\etaIncreasingExponent + d + 1)/(2\etaIncreasingExponent + d)}  \vee 2^{(2\etaIncreasingExponent+d+1)/d}.
\]
Then, by Theorem~\ref{Thm:LowerBound}, there exists $c_1'(\delta) \equiv c_1'(\delta,\alpha, d,  \sigma, \etaIncreasingConstant, \etaIncreasingExponent) > 0$ such that
\begin{align*}
\sup_{P\in \mathcal{P}'}  \E_P\bigl\{ \marginalDistribution \bigl(\superLevelSet{\tau}{\regressionFunction}\setminus \MGselectionset\bigr)\bigr\} \geq c'_1(\delta).
\end{align*}
Second, if $m < 2^{15}(1\vee \etaIncreasingConstant^2/\sigma^2)^{(d-1)/(2\etaIncreasingExponent + d)} \cdot n^{d/(2\etaIncreasingExponent + d + 1)}\log_+^2\bigl(n/(\alpha \wedge \delta)\bigr)$, then we have by the proof of Proposition~\ref{prop:lowerBound_m} that there exists $c_2'(\delta) \equiv c_2'(\delta, \alpha, d, \sigma, \etaIncreasingConstant, \etaIncreasingExponent) > 0$ such that
\[
\sup_{P\in \mathcal{P}'}  \E_P\bigl\{ \marginalDistribution \bigl(\superLevelSet{\tau}{\regressionFunction}\setminus \MGselectionset\bigr)\bigr\} \geq \frac{1}{80\cdot2^{1/d}d m^{1/d}} \geq \frac{c'_2(\delta)}{n^{1/(2\etaIncreasingExponent + d + 1)}(\log_+ n)^{2/d}}.
\]
Hence, we may suppose for the remainder of the proof that 
\[
m \geq 2^{15}\Bigl(1\vee \frac{\etaIncreasingConstant^2}{\sigma^2}\Bigr)^{(d-1)/(2\etaIncreasingExponent + d)} \cdot n^{d/(2\etaIncreasingExponent + d + 1)}\log_+^2\Bigl(\frac{n}{\alpha \wedge \delta}\Bigr) \geq n^{d/(2\etaIncreasingExponent + d + 1)}
\] 
and 
\begin{align*}
n &\geq  \exp\Bigl(\frac{\sigma^2}{\etaIncreasingConstant^2}\cdot 2^{2\etaIncreasingExponent - 13}\Bigr) \vee \biggl\{2^{d+1} \Bigl(\frac{\etaIncreasingConstant^2}{\sigma^2}\Bigr)^{1/(2\etaIncreasingExponent + d)}\log\Bigl(\frac{4d}{\delta}\Bigr)\biggr\}^{(2\etaIncreasingExponent + d + 1)/(2\etaIncreasingExponent + d)} \vee 2^{(2\etaIncreasingExponent+d+1)/d}.
\end{align*}
Write 
\[
\rho_0 := \frac{1}{\log n} \log\biggl(\frac{m}{2^{15} (1\vee \etaIncreasingConstant^2/\sigma^2)^{(d-1)/(2\etaIncreasingExponent+d)}\log^2\bigl(n/(\alpha\wedge \delta)\bigr)}\biggr)
\]
and $\rho := (1-\rho_0)\cdot(2\etaIncreasingExponent + d)/(2\etaIncreasingExponent + 1)$. By our assumption on $m$, we have $\rho_0 \geq d/(2\etaIncreasingExponent + d + 1)$ and hence $\rho \leq (2\etaIncreasingExponent + d)/(2\etaIncreasingExponent + d + 1)$.  Moreover, by definition of $\rho_0$ we have that
\begin{align*}
q := \bigg\lceil 484\Bigl(1\vee \frac{\etaIncreasingConstant^2}{\sigma^2}\Bigr)^{(d-1)/(2\etaIncreasingExponent + d)}\cdot n^{\rho_0} \log\Bigl(\frac{n}{\alpha \wedge \delta}\Bigr)\bigg\rceil &\leq 2^9 \Bigl(1\vee \frac{\etaIncreasingConstant^2}{\sigma^2}\Bigr)^{(d-1)/(2\etaIncreasingExponent + d)}\cdot n^{\rho_0} \log\Bigl(\frac{n}{\alpha \wedge \delta}\Bigr) \\
&\phantom{:}\leq \frac{m}{64\log\bigl(n/(\alpha \wedge \delta)\bigr)} \leq \bigg\lfloor \frac{m}{32\log(m/\delta)} \bigg\rfloor.
\end{align*}
Next, let
\[
s := \biggl(\frac{2\sigma^2}{n^\rho \etaIncreasingConstant^2}\log\Bigl(\frac{n}{\alpha \wedge\delta}\Bigr)\biggr)^{1/(2\etaIncreasingExponent+d)} \geq \biggl(\frac{\sigma^2}{n^\rho \etaIncreasingConstant^2}\biggr)^{1/(2\etaIncreasingExponent + d)}.
\]
Note also that
\begin{align*}
s &\leq \biggl(\frac{2\sigma^2}{ \etaIncreasingConstant^2}\biggr)^{1/(2\etaIncreasingExponent+d)} n^{-(1 - \rho_0)/(2\etaIncreasingExponent+1)} \log^{1/(2\gamma+1)}\Bigl(\frac{n}{\alpha \wedge\delta}\Bigr) \\
&= \biggl(\frac{2\sigma^2}{ \etaIncreasingConstant^2}\biggr)^{1/(2\etaIncreasingExponent+d)} \biggl\{\frac{m/n}{2^{15} (1\vee \etaIncreasingConstant^2/\sigma^2)^{(d-1)/(2\etaIncreasingExponent+d)}\log^2\bigl(n/(\alpha\wedge \delta)\bigr)} \cdot \log\Bigl(\frac{n}{\alpha \wedge\delta}\Bigr)\biggr\}^{1/(2\etaIncreasingExponent+1)} \\
&\leq 2^{1/(2\etaIncreasingExponent+d) - 15/(2\etaIncreasingExponent+1)} \biggl(\frac{(\sigma^2/ \etaIncreasingConstant^2)^{(2\etaIncreasingExponent + 1)/(2\etaIncreasingExponent + d)}}{( \etaIncreasingConstant^2/\sigma^2)^{(d-1)/(2\etaIncreasingExponent+d)}\log n}\biggr)^{1/(2\etaIncreasingExponent+1)}\\
&\leq 2^{1/(2\etaIncreasingExponent+d) - 15/(2\etaIncreasingExponent+1)} \biggl(\frac{\sigma^2}{\etaIncreasingConstant^2\log n}\biggr)^{1/(2\etaIncreasingExponent+1)} \leq 2^{1/(2\etaIncreasingExponent+d) - 15/(2\etaIncreasingExponent+1) - (2\etaIncreasingExponent - 13)/(2\etaIncreasingExponent + 1)} \leq 1/2.
\end{align*}
Writing $w_{n,m,\delta} := 173.13 \bigl(\log_+\log n + \log_+(m/\delta)\bigr) \leq 173.13 \cdot 2\log\bigl(n/(\alpha\wedge\delta)\bigr)$, we claim that 
\begin{align}\label{eq:sLowerUpperBoundsClaim}
\frac{8}{3n}\log\Bigl(\frac{4d}{\delta}\Bigr) \leq \frac{s}{2^{d-1}} \leq \frac{\sigma^2}{2n\etaIncreasingConstant^2 s^{2\etaIncreasingExponent}}\bigl(0.72q - w_{n,m,\delta}\bigr).
\end{align}
To see this, first note for the lower bound that
\begin{align*}
    ns &\geq n\biggl(\frac{\sigma^2}{n^\rho \etaIncreasingConstant^2}\biggr)^{1/(2\etaIncreasingExponent + d)} \geq \biggl(\frac{\sigma^2}{\etaIncreasingConstant^2}\biggr)^{1/(2\etaIncreasingExponent + d)} n^{(2\etaIncreasingExponent + d)/(2\etaIncreasingExponent + d + 1)} \geq 2^{d+1} \log\Bigl(\frac{4d}{\delta}\Bigr) \geq 2^{d-1}\cdot \frac{8}{3}\log\Bigl(\frac{4d}{\delta}\Bigr).
\end{align*}
As for the upper bound, note first that since $n^{\rho_0} \geq 2$ and the lower bound on $s$,
\begin{align*}
0.72q &\geq 
    \biggl\{173.13 +  \biggl(\frac{\etaIncreasingConstant^2}{\sigma^2}\biggr)^{(d-1)/(2\etaIncreasingExponent + d)}n^{\rho_0}\biggr\}\cdot 2\log\Bigl(\frac{n}{\alpha \wedge\delta}\Bigr) \geq w_{n,m,\delta} + \frac{4n^{1-\rho}}{(2s)^{d-1}}\log\Bigl(\frac{n}{\alpha \wedge\delta}\Bigr).
\end{align*}
Hence
\[
2\log\Bigl(\frac{n}{\alpha \wedge\delta}\Bigr) \leq \frac{(2s)^{d-1}}{2n^{1-\rho}}(0.72q - w_{n,m,\delta}).
\]
This yields that
\begin{align*}
    \frac{s}{2^{d-1}} = \frac{1}{2^{d-1} s^{2\etaIncreasingExponent + d - 1}} \cdot  \frac{2\sigma^2}{n^\rho \etaIncreasingConstant^2} \log\Bigl(\frac{n}{\alpha \wedge\delta}\Bigr)
    \leq \frac{1}{2n} \cdot  \frac{\sigma^2}{ \etaIncreasingConstant^2 s^{2\etaIncreasingExponent}} \bigl(0.72q - w_{n,m,\delta}\bigr),
\end{align*}
thus establishing the claim \eqref{eq:sLowerUpperBoundsClaim}. 

Now fix $M := 1.7\sigma\sqrt{\log(41.6/\delta)}$.  By Lemma~\ref{lemma:negativeMG_chernoff}, we have $\Prob_{P_{q,M}}\bigl(\Omega_1\bigr) \geq 1 - \delta/4$ for $\Omega_1$ defined as in that lemma. For $\Omega_2$ defined as in Lemma~\ref{lemma:negativeMG_grid_pvalues}, it then holds that $\Prob_{P_{q,M}}\bigl(\Omega_2|\sample_X\bigr) \geq 1- \delta/4$ on $\Omega_1$ by that same lemma.  Let $(B_j)_{j\in[d]}$ and $\Omega_3$ be as in Lemma~\ref{lemma:negativeMG_chernoff_superlevelset}, so that $\Prob_{P_{q,M}}\bigl(\Omega_3\bigr) \geq 1- \delta/4$. Moreover, let $1\leq i_1 < \ldots < i_K \leq m$ be such that $\{i_1, \ldots, i_K\} := \{i\in [m]: X_i \in \superLevelSet{\tau}{\eta_{q,M}}\setminus\superLevelSet{\tau + \etaIncreasingConstant s^\etaIncreasingExponent}{\eta_{q,M}}\}$ as in Lemma~\ref{lemma:negativeMG_uniform}.  Note that $\superLevelSet{\tau + \lambda s^\gamma}{\eta_{q,M}} = \{x\in\R^d: x \succcurlyeq (0, 1/2, 0,\ldots,0)^\top + s\cdot \bm{1}_d\}$, so that $B_1,\ldots, B_d$ define a covering of $\superLevelSet{\tau}{\eta_{q,M}}\setminus \superLevelSet{\tau + \etaIncreasingConstant s^\etaIncreasingExponent}{\eta_{q,M}}$. Hence, for every $k \in [K]$, there exists $j_k \in [d]$ such that $X_{i_k} \in B_{j_k}$ and
\[
\bigl|\{i \in [n]: X_i \in A, X_i \preccurlyeq X_{i_k}\}\bigr| \leq \bigl|\{i\in [n]: X_i \in B_{j_k}\}\bigr|,
\]
so that 
\[
\Omega_3 \subseteq \bigcap_{k\in[K]} \biggl\{\bigl|\{i \in [n]: X_i \in A, X_i \preccurlyeq X_{i_k}\}\bigr| \leq \frac{\sigma^2}{\etaIncreasingConstant^2 s^{2\etaIncreasingExponent}} \bigl(0.72q - w_{n,m,\delta}\bigr)\biggr\} =: \Omega_3^*.
\]
Moreover, we have by Lemma~\ref{lemma:negativeMG_uniform} that 
the set $\Omega_4$ defined therein satisfies $\Prob_{P_{q,M}}\bigl(\Omega_4 | \sample_X\bigr) \geq 1 - \delta/4$ on $\Omega_3^*$.  Hence, by Lemma~\ref{lemma:negativeMG_problematic_setup}, we have on $\Omega_1 \cap \Omega_2 \cap \Omega_3 \cap \Omega_4$ that $\MGselectionset \subseteq \superLevelSet{\tau+\etaIncreasingConstant s^\etaIncreasingExponent}{\eta_{q,M}}$, so
\begin{align*}
\mu_q\bigl(\superLevelSet{\tau}{\eta_{q,M}}\setminus \MGselectionset\bigr) &= \frac{1}{2^d} - \mu_q\bigl(\MGselectionset\bigr)\\
&\geq \frac{1}{2^d} - \mu_q\bigl(\superLevelSet{\tau + \lambda s^\gamma}{\eta_{q,M}}\bigr) = \frac{1- (1-2s)^d}{2^d} \geq \frac{s}{2^{d-1}},
\end{align*}
where the final inequality uses the fact that $s \leq 1/2$. It follows that 
\begin{align*}
    \Prob_{P_{q,M}}\biggl\{ \mu_q\bigl(\superLevelSet{\tau}{\regressionFunction}\setminus\MGselectionset\bigr) \geq \frac{s}{2^{d-1}}\biggr\} 
    \geq \Prob_{P_{q,M}}\bigl(\Omega_1 \cap \Omega_2 \cap \Omega_3 \cap \Omega_4\bigr) \geq 1 - \delta.
\end{align*}
Setting $\delta = 1/2$, we see that there exists $c_3 > 0$, depending only on $d$, $\sigma$, $\etaIncreasingConstant$ and $\etaIncreasingExponent$, such that
\begin{align*}
    \E_{P_{q,M}}\bigl\{ \mu_q\bigl(\superLevelSet{\tau}{\regressionFunction}\setminus\MGselectionset\bigr)\bigr\} &\geq \frac{s}{2^d}
    \geq \frac{c_3}{n^{1/(2\etaIncreasingExponent + d + 1)}},
\end{align*}
so that the result follows for $c := c_1'(1/2) \wedge c'_2(1/2) \wedge c_3$. 
\end{proof}

\section{Auxiliary Results}\label{Sec:Auxiliary}
\begin{lemma}[\citealp{howard2021uniform}]\label{lemma:howard_uniform_bound}
Let $(Z_j)_{j \in \mathbb{N}}$ be a sequence of independent, sub-Gaussian random variables with variance parameter 1. 

\begin{enumerate}[(a)] 
\item For any $\alpha \in (0, 1)$, 
    \[
\Prob\biggl(\bigcup_{k=1}^\infty \biggl\{\sum_{j=1}^k Z_j \geq u_\alpha(k)\biggr\}\biggr) \leq \alpha,
\]
where $u_\alpha(k) := 1.7 \sqrt{k \bigl\{\log \log (2k)+0.72 \log(5.2/\alpha)\bigr\}}$. 

\item For any $\alpha \in (0, 1)$ and $\rho > 0$, 
    \[
\Prob\biggl(\bigcup_{k=1}^\infty \biggl\{\sum_{j=1}^k Z_j \geq u^{\mathrm{NM}}_{\alpha, \rho}(k)\biggr\}\biggr) \leq \alpha,
\]
where $u^{\mathrm{NM}}_{\alpha, \rho}(k) := \sqrt{2(k+\rho)\log\Bigl(\frac{1}{2\alpha}\sqrt{\frac{k+\rho}{\rho}} + 1\Bigr)}$. 
\end{enumerate}
\end{lemma}

The following simple lemma on testing Gaussian distributions is used in the proof of one of our minimax lower bounds (Proposition~\ref{prop:lowerBound_alpha_multivariate}).

\begin{lemma}
\label{Lemma:GaussianTesting}
Fix $\alpha \in (0,1/4]$, $\sigma > 0$, $n \in \mathbb{N}$ and $\Delta \in \bigl(0,\frac{\sigma}{\sqrt{1.6n}}\log^{1/2}\bigl(\frac{1}{2\alpha}\bigr)\bigr]$.  Let $P_0 = \mathcal{N}(0,\sigma^2)$ and $P_1 = \mathcal{N}(\Delta,\sigma^2)$, and let $Z_1,\ldots,Z_n \stackrel{\mathrm{iid}}{\sim} P$ for some $P \in \{P_0,P_1\}$.  If $\psi:\mathbb{R}^n \rightarrow \{0,1\}$ is a Borel measurable function satisfying $\mathbb{P}_{P_0}\bigl(\psi(Z_1,\ldots,Z_n) = 1\bigr) \leq \alpha$, then $\mathbb{P}_{P_1}\bigl(\psi(Z_1,\ldots,Z_n) = 1\bigr) \leq 1/2$.
\end{lemma}
\begin{proof}
Write $\Phi$ and $\phi$ for the standard normal distribution and density function respectively. By \cite{gordon1941values},
\[
1 - \Phi(z) > \frac{\phi(z)}{z + 1/z} \geq \frac{\phi(z)}{3.2 z}
\]
for all $z \geq \sqrt{5/11}$.  Hence for all $z \geq \sqrt{5/11}$, we have $G(z) := \bigl(1 - \Phi(z)\bigr)e^{1.6z^2} > 0$ and 
\[
G'(z) = 3.2ze^{1.6z^2}\Bigl(1 - \Phi(z) - \frac{\phi(z)}{3.2z}\Bigr) > 0,
\]
so that $1 - \Phi(z) \geq G(\sqrt{5/11})\exp(-1.6z^2)$.  Since $G(\sqrt{5/11}) \geq 1/2$, it follows with $z_\alpha := 1.6^{-1/2} \log^{1/2}\bigl(\frac{1}{2\alpha}\bigr) \geq \sqrt{5/11}$ that $1 - \Phi(z_\alpha) \geq \alpha$. By the Neyman--Pearson lemma, for any Borel measurable function $\psi:\mathbb{R}^n \rightarrow \{0,1\}$ satisfying $\mathbb{P}_{P_0}\bigl(\psi(Z_1,\ldots,Z_n) = 1\bigr) \leq \alpha$, we deduce that
\begin{align*}
\mathbb{P}_{P_1}\bigl(\psi(Z_1,\ldots,Z_n) = 1\bigr) &\leq \mathbb{P}_{P_1}\bigl(n^{1/2}\bar{Z} > \sigma\Phi^{-1}(1-\alpha)\bigr) = 1 - \Phi\biggl(\Phi^{-1}(1-\alpha) - \frac{n^{1/2}\Delta}{\sigma}\biggr) \\
&\leq 1 - \Phi\Bigl(\Phi^{-1}(1-\alpha) - z_\alpha\Bigr) \leq \frac{1}{2},
\end{align*}
as required.
\end{proof}

\begin{corollary}
    \label{Cor:GaussianTesting2} Suppose that $\alpha \in (0,2/3]$, $\sigma > 0$, $t \in \R$, $n \in \mathbb{N}$, $p \in [8/n,1]$ and $\Delta \in \bigl(0,\frac{\sigma}{\sqrt{3.2np}}\log_+^{1/2}\bigl(\frac{1}{5\alpha}\bigr)\bigr]$, and let $S\subseteq [0,1]^d$ be a Borel set. Let $P_0$ and $P_1$ denote Borel probability distributions over random pairs $(X,Y)$ taking values in $[0,1]^d \times \R$. For $\omega \in \{0,1\}$, let $P^\omega_X$ denote the corresponding marginal distribution over $X$ and for $x \in [0,1]^d$ let $P^\omega_{Y|X=x}$ denote the corresponding conditional distribution of $Y$ given $X=x$.  Assume that $P^0_X = P^1_X$ and $p = P^0_X(S)=P^1_X(S)$, and that for all $x \in [0,1]^d \setminus S$ we have $P^0_{Y|X=x}=P^1_{Y|X=x}$.  Suppose further that $P^\omega_{Y|X=x}=\mathcal{N}(t+\omega\cdot\Delta,\sigma^2)$ for $\omega \in \{0,1\}$ and $x \in S$.  Let $\sample = \bigl((X_1,Y_1),\ldots,(X_n,Y_n)\bigr) \sim P^n$ for some $P \in \{P_0,P_1\}$.  If $\psi:([0,1]^d \times \mathbb{R})^n \rightarrow \{0,1\}$ is a Borel measurable function satisfying $\mathbb{P}_{P_0}\bigl(\psi(\sample) = 1\bigr) \leq \alpha$, then $\mathbb{P}_{P_1}\bigl(\psi(\sample) =0\bigr) \geq 1/20$.
\end{corollary}
\begin{proof} First suppose that $\alpha \in (0,1/10]$ and define Borel subsets $A_0$, $A_1 \subseteq ([0,1]^d)^n$ by 
\begin{align*}
A_0 &:=\biggl\{(x_i)_{i\in [n]} \in ([0,1]^d)^n:\sum_{i=1}^n \one_{\{x_i \in S\}}\leq 2np\biggr\}, \\
A_1 &:=\biggl\{(x_i)_{i\in [n]} \in ([0,1]^d)^n:\mathbb{P}_{P_0}\bigl\{\psi(\sample) = 1 \bigm| (X_i)_{i \in [n]}=(x_i)_{i\in [n]}\bigr\} \leq \frac{5\alpha}{2}\biggr\}.
\end{align*}
Now, for $(x_i)_{i \in [n]} \in ([0,1]^d)^n$, the Radon–Nikodym derivative of the conditional distribution of $\sample \sim P_1^n$ given $(X_i)_{i \in [n]}=(x_i)_{i \in [n]}$ with respect to the corresponding conditional distribution for $\sample \sim P_0^n$ is equal to the product of $\sum_{i=1}^n \one_{\{x_i \in S\}}$ likelihood ratios between two univariate Gaussians with common variance $\sigma^2$ and mean differing by $\Delta$. Hence, for $(x_i)_{i \in [n]} \in A_0 \cap A_1$, an application of Lemma~\ref{Lemma:GaussianTesting} with $\sum_{i=1}^n \one_{\{x_i \in S\}} \leq 2np$ in place of $n$ and $5\alpha/2$ in place of $\alpha$ yields
\begin{align}\label{ineq:condAppGaussTest2}
\mathbb{P}_{P_1}\bigl\{\psi(\sample) = 1|(X_i)_{i \in [n]}=(x_i)_{i\in [n]}\bigr\} \leq \frac{1}{2}.
\end{align}  Moreover, by the multiplicative Chernoff bound \citep[Theorem~2.3(b)]{mcdiarmid1998concentration}, we have 
\begin{align*}
\mathbb{P}_{P_1}\bigl\{(X_i)_{i \in [n]}\notin A_0\bigr\} \leq e^{-3np/8} \leq \frac{1}{20}.
\end{align*}
Moreover, by Markov's inequality we have $\mathbb{P}_{P_1}\{(X_i)_{i \in [n]}\notin A_1\} \leq 2/5$. Combining with \eqref{ineq:condAppGaussTest2} yields $\mathbb{P}_{P_1}\bigl(\psi(\sample) = 1\bigr) \leq 19/20$, as required.

Now suppose that $\alpha \in (1/10,2/3]$. By Pinsker's inequality we have
\begin{align*}
\mathrm{TV}\left( P_0^n, P_1^n\right) &\leq \sqrt{\frac{n}{2} \cdot \mathrm{KL}(P_0,P_1)} = \sqrt{\frac{n p}{2} \cdot \mathrm{KL}\bigr(\mathcal{N}(t,\sigma^2),\mathcal{N}(t+\Delta,\sigma^2)\bigr)} \leq  \frac{\sqrt{np}\cdot \Delta}{2 \sigma} \leq \frac{1}{2\sqrt{3.2}}.  
\end{align*}
Hence, for any Borel measurable function satisfying $\mathbb{P}_{P_0}\bigl(\psi(\sample) = 1\bigr) \leq \alpha \leq 2/3$, we have
\begin{align*}
\mathbb{P}_{P_1}\bigl(\psi(\sample) =0\bigr) & \geq \mathbb{P}_{P_0}\bigl(\psi(\sample) =0\bigr)- \mathrm{TV}\left( P_0^n, P_1^n\right) \geq 1-\alpha - \frac{1}{2\sqrt{3.2}}  \geq \frac{1}{20},
\end{align*}
as required.
\end{proof}

\begin{lemma}\label{lemma:TVProductMeasure} 
\begin{enumerate}[(a)]
\item Let $P, Q$ denote probability measures on a measurable space $(\mathcal{X},\mathcal{A})$.  Then
\[
\mathrm{TV}(P,Q) = \inf_{(X,Y) \sim (P,Q)} \mathbb{P}(X \neq Y),
\]
where the infimum is taken over all pairs of random variables $X \sim P$ and $Y \sim Q$ defined on the same probability space, and where the infimum is attained.
\item For $i \in [n]$, let $P_i, Q_i$ denote probability measures on a measurable space $(\mathcal{X}_i,\mathcal{A}_i)$, and let $P := \times_{i=1}^n P_i$ and $Q := \times_{i=1}^n Q_i$ denote the corresponding product measures.  Then
\[
\mathrm{TV}(P,Q) \leq \sum_{i=1}^n \mathrm{TV}(P_i,Q_i).
\]
\end{enumerate}
\end{lemma}
\begin{proof}
    \emph{(a)} Let $X \sim P$ and $Y \sim Q$ be defined on the same probability space.  Then for any $A \in \mathcal{A}$, \[ P(A) - Q(A) = \mathbb{P}(X \in A) - \mathbb{P}(Y \in A) \leq \mathbb{P}(X \in A, Y \notin A) \leq \mathbb{P}(X \neq Y). \] Similarly, $Q(A) - P(A) \leq \mathbb{P}(X \neq Y)$, so since $A \in \mathcal{A}$ was arbitrary, we see that $\mathrm{TV}(P,Q) \leq \mathbb{P}(X \neq Y)$.  This bound holds for all couplings of $X \sim P$ and $Y \sim Q$, so \[ \mathrm{TV}(P,Q) \leq \inf_{(X,Y) \sim (P,Q)} \mathbb{P}(X \neq Y). \] To see that this bound is in fact an equality, and that the infimum is attained, let $p, q$ denote the respective densities of $P$ and $Q$ with respect to $P+Q$.  We construct a coupling of $X \sim P$ and $Y \sim Q$ as follows: with probability $1 - \mathrm{TV}(P, Q)$, sample $X = Y$ from a distribution having density $(p \wedge q)/\bigl(1-\mathrm{TV}(P, Q)\bigr)$ with respect to $P+Q$, and otherwise sample $X$ and~$Y$ independently from distributions having respective densities $(p-q)\mathbbm{1}_{\{p > q\}}/\mathrm{TV}(P, Q)$ and $(q-p)\mathbbm{1}_{\{q > p\}}/\mathrm{TV}(P,Q)$ with respect to $P+Q$. The fact that the given expressions are indeed densities with respect to $P+Q$ follows because $\mathrm{TV}(P,Q) = \frac{1}{2} \int_{\mathcal{X}} |p-q| \, d(P+Q)$.  We then have for any $A \in \mathcal{A}$ that \[ \mathbb{P}(X \in A) = \int_A (p \wedge q) \, d(P+Q) + \int_A (p-q)\mathbbm{1}_{\{p > q\}} \, d(P+Q) = \int_A p \, d(P+Q) = P(A), \] and similarly $\mathbb{P}(Y \in A) = Q(A)$.  Thus $X \sim P$ and $Y \sim Q$, and since $\mathbb{P}(X \neq Y) \leq \mathrm{TV}(P, Q)$, the result follows.\medskip

    \noindent \emph{(b)} By \emph{(a)}, there exist independent pairs $(X_1, Y_1), \ldots, (X_n, Y_n)$ with $X_i \sim P_i$, $Y_i \sim Q_i$ and $\mathbb{P}(X_i \neq Y_i) = \mathrm{TV}(P_i,Q_i)$ for $i \in [n]$.  Then $X := (X_1, \ldots, X_n) \sim P$ and $Y := (Y_1, \ldots, Y_n) \sim Q$, and by a union bound,
    \[
         \mathrm{TV}(P, Q) \leq \mathbb{P}(X \neq Y) = \mathbb{P}\biggl(\bigcup_{i=1}^n \{X_i \neq Y_i\}\biggr) \leq \sum_{i=1}^n \mathbb{P}(X_i \neq Y_i) = \sum_{i=1}^n \mathrm{TV}(P_i, Q_i).
    \]
\end{proof}

\begin{lemma}\label{lemma:log_logplus_facts} The following inequalities hold:
\begin{enumerate}[(i)]
    \item $\log_+ (x y) \leq y\cdot\log_+ x$ for $x > 0$ and $y \geq 1$.
    \item $\log_+(x^a)\leq a\cdot\log_+ x$ for $x>0$ and $a\geq 1$.
    \item $\log(x y)\leq y\cdot\log x$ for $x \geq 2$ and $y \geq 2$.
    \item $\log(x y)\leq y\cdot \log x$ for $x\geq e$, $y\geq 1$.
    \item $\log_+(x y)\leq\log_+ x + \log_+ y$ for $x,y>0$.     
\end{enumerate}
\end{lemma}
\begin{proof} 
\emph{(i)} Suppose first that $x \in (0, e)$. Then, $y\log_+ x = y \geq 1 + \log y = \log_+(ey) \geq \log_+(xy)$. If on the other hand $x \geq e$, then $xy \geq e$ and the result is a consequence of \emph{(iv)} below. \medskip

\noindent \emph{(ii)} We have $\log_+(x^a) = \log(x^a \vee e) \leq \log(x^a \vee e^a) = a \log_+ x$. \medskip

\noindent \emph{(iii)} As $z \mapsto (z-1)/\log z$ is increasing for $z > 1$, we have $(y-1)/\log y \geq 1/\log 2$. Thus $\log(xy) \leq \log x + (y-1)\log 2 \leq y \log x$. \medskip

\noindent \emph{(iv)} Since $\log x \geq 1$, we have $\log(xy) \leq \log x + (y-1) \leq y\log x$. \medskip
 
\noindent \emph{(v)} We have $\log_+(xy) = \log(xy \vee e) \leq \log\bigl((x\vee e)(y\vee e)\bigr) = \log_+ x + \log_+ y$.
\end{proof}

\textbf{Acknowledgements:} The research of TIC was supported by Engineering and Physical Sciences Research Council (EPSRC) New Investigator Award EP/V002694/1. The research of RJS was supported by EPSRC Programme grant EP/N031938/1 and ERC Advanced Grant 101019498.

\bibliography{mybib}
\bibliographystyle{apalike}

\end{document}